\documentclass[twocolumn]{svjour3}
\usepackage[boxed]{algorithm2e}
\usepackage{citesort,cite}
\usepackage{algorithmicx}
\usepackage{algpseudocode}
\usepackage{graphicx}
\usepackage{url}
\usepackage{times,amsmath,amssymb,amsfonts,epsfig,graphicx}
\usepackage{enumerate}
\usepackage{stackrel}
\usepackage{multirow}
\usepackage{subfigure}
\usepackage{psfrag}
\usepackage{color}
\usepackage{comment}
\usepackage{framed}

\newcommand{\vectornorm}[1]{\|#1\|}
\newcommand{\vectornormbig}[1]{\big\|#1\big\|}

\newcommand{\obs}{\boldsymbol{y}}
\newcommand{\sensing}{\boldsymbol{\mathcal{A}}}

\newcommand{\signal}{\boldsymbol{X}}
\newcommand{\bestsignal}{\boldsymbol{X}^\ast}
\newcommand{\xtrue}{\bestsignal}
\newcommand{\noise}{\boldsymbol{\varepsilon}}
\newcommand{\dimension}{m \times n}
\newcommand{\numsam}{p}
\newcommand{\sparsity}{s}
\newcommand{\id}{\mathbf{I}}
\newcommand{\rank}{k}

\DeclareMathOperator*{\argmin}{arg\,min}

\smartqed

\author{\IEEEauthorblockN{Anastasios Kyrillidis \and Volkan Cevher \\}}

\begin{document}
\title{Matrix Recipes for \\ Hard Thresholding Methods}

\author{Anastasios Kyrillidis \and Volkan Cevher}

\institute{A. Kyrillidis \at
              Laboratory for Information and Inference Systems, Ecole Polytechnique Federale de Lausanne \\
              Tel.: +41 21 69 31154\\
              \email{anastasios.kyrillidis@epfl.ch}
           \and
           V. Cevher \at
              Laboratory for Information and Inference Systems, Ecole Polytechnique Federale de Lausanne \\
              Tel.: +41 21 69 31101\\
              \email{volkan.cevher@epfl.ch}
}

\date{Received: date / Accepted: date}

\maketitle

\begin{abstract}
In this paper, we present and analyze a new set of low-rank recovery algorithms for linear inverse problems within the class of hard thresholding methods. We provide strategies on how to set up these algorithms via basic ingredients for different configurations to achieve complexity vs. accuracy tradeoffs. Moreover, we study acceleration schemes via memory-based techniques and randomized, $\epsilon$-approximate matrix projections to decrease the computational costs in the recovery process. For most of the configurations, we present theoretical analysis that guarantees convergence under mild problem conditions.  Simulation results demonstrate notable performance improvements as compared to state-of-the-art algorithms both in terms of reconstruction accuracy and computational complexity.

\keywords{Affine rank minimization \and hard thresholding \and $ \epsilon $-approximation schemes \and randomized algorithms.}
\end{abstract}

\section{Introduction}

In this work, we consider the general affine rank minimization (ARM) problem, described as follows:
\vskip.04in
\noindent \textsc{The ARM Problem:} {\it Assume $ \bestsignal \in  \mathbb{R}^{\dimension} $ is a rank-$\rank$ matrix of interest ($ \rank \ll \min\lbrace m, n \rbrace $) and let $ \sensing: \mathbb{R}^{\dimension} \rightarrow \mathbb{R}^\numsam $ be a known linear operator. Given a set of observations as $ \obs = \sensing \bestsignal + \noise \in \mathbb{R}^{\numsam}$, we desire to recover $\bestsignal $from $\obs$ in a scalable and robust manner.}
\vskip.04in
The challenge in this problem is to recover the true low-rank matrix in subsampled settings where $ \numsam \ll m \cdot n $. In such cases, we typically exploit the prior information that $\bestsignal $ is low-rank and thus, we are interested in finding a matrix $ \signal $ of rank at most $\rank$ that minimizes the data error $ f(\signal) :=  \vectornorm{\obs - \sensing \signal}_{2}^2 $ as follows:
\begin{equation}
	\begin{aligned}
	& \underset{\signal \in \mathbb{R}^{\dimension}}{\text{minimize}}
	& & f(\signal) \\
	& \text{subject to}
	& & \text{rank}(\signal) \leq \rank.
	\end{aligned} \label{opt:05}
\end{equation} The ARM problem appears in many applications; low dimensional embedding \cite{baraniuk2010low}, matrix completion \cite{candès2009exact}, image compression \cite{SVP}, function learning \cite{TyagiCevherRidge, TyagiCevherRidgeII} just to name a few. We present below important ARM problem cases, as characterized by the nature of the linear operator $\sensing$.

\textbf{General linear maps:} In many ARM problem cases, $\sensing$ or $\sensing^{\ast}$ has a dense range, satisfying specific incoherence or restricted isometry properties (discussed later in the paper); here, $\sensing^{\ast}$ is the adjoint operator of $\sensing$. In Quantum Tomography, \cite{liuuniversal} studies the Pauli operator, a {\it compressive} linear map $\sensing$ that consists of the kronecker product of $2 \times 2$ matrices and obeys restricted isometry properties, defined later in the paper. Furthermore, recent developments indicate connections of ridge function learning \cite{TyagiCevherRidge, hemant2012active} and phase retrieval \cite{candes2012solving} with the ARM problem where $\sensing$ is a Bernoulli and a Fourier operator, respectively. 

\textbf{Matrix Completion (MC):} Let $ \Omega $ be the set of ordered pairs that represent the coordinates of the observable entries in $\bestsignal$. Then, the set of observations satisfy $\obs = \sensing_{\Omega} \bestsignal + \noise$ where $\sensing_{\Omega} $ defines a linear mask over the observable entries $\Omega $. To solve the MC problem, a potential criterion is given by (\ref{opt:05}) \cite{candès2009exact}. As a motivating example, consider the famous Netflix problem \cite{netflix}, a recommender system problem where users' movie preferences are inferred by a limited subset of entries in a database. 

\textbf{Principal Component Analysis:} In Principal Component Analysis (PCA), we are interested in identifying a low rank subspace that best explains the data in the Euclidean sense from the observations $\obs = \sensing \bestsignal $ where $\sensing: \mathbb{R}^{\dimension} \rightarrow \mathbb{R}^p$ is an identity linear map that stacks the columns of the matrix $\bestsignal$ into a 
single column vector with $p = m \cdot n$. We observe that the PCA problem falls under the ARM criterion in (\ref{opt:05}). While (\ref{opt:05}) is generally NP-hard to solve optimally, PCA can be solved in polynomial time using the truncated Singular Value Decomposition (SVD) of $\sensing^{\ast} \obs$. As an extension to the PCA setting, \cite{candes2011robust} considers the Robust PCA problem where $\boldsymbol{y}$ is further corrupted by gross sparse noise. We extend the framework proposed in this paper for the RPCA case and its generalizations in \cite{KyrillidisCevherSSP}.

For the rest of the paper, we consider only the low rank estimation case in (\ref{opt:05}). As running test cases to support our claims, we consider the MC setting as well as the general ARM setting where $\sensing$ is constituted by permuted subsampled noiselets \cite{sparcs}. 

\subsection{Two camps of recovery algorithms}

~~~~\textbf{Convex relaxations:} In \cite{brecht2010}, the authors study the nuclear norm $ \vectornorm{\signal}_{\ast} $ $:= \sum_{i = 1}^{\text{rank}(\signal)} \sigma_{i} $ as a convex surrogate of $ \text{rank}(\signal) $ operator so that we can leverage convex optimization approaches, such as interior-point methods---here, $\sigma_i$ denotes the $i$-th singular value of $\signal$. Under basic incoherence properties of the sensing linear mapping $ \sensing $, \cite{brecht2010} provides provable guarantees for unique low rank matrix recovery using the nuclear norm.

Once (\ref{opt:05}) is relaxed to a convex problem, decades of knowledge on convex analysis and optimization can be leveraged. Interior point methods find a solution with fixed precision in polynomial time but their complexity might be prohibitive even for moderate-sized problems \cite{Liu2009, Fazel2010}. More suitable for large-scale data analysis, first-order methods constitute low-complexity alternatives but most of them introduce complexity vs. accuracy tradeoffs \cite{SVT, ParallelRecht, ALM, APG}.

\vskip.04in
\textbf{Non-convex approaches:} In contrast to the convex relaxation approaches, iterative greedy algorithms maintain the nonconvex nature of (\ref{opt:05}).
Unfortunately, solving (\ref{opt:05}) optimally is in general NP-hard \cite{natarajan1995sparse}. Due to this computational intractability, the algorithms in this class greedily refine a rank-$\rank$ solution using only ``local'' information available at the current iteration \cite{admira2010, Goldfarb:2011, beck2011linearly}.

\subsection{Contributions}

In this work, we study a special class of iterative greedy algorithms known as hard thresholding methods. Similar results have been derived for the vector case \cite{KyrillidisCevherRecipes}. Note that the transition from sparse vector approximation to ARM is {\it non-trivial}; while $ \sparsity $-sparse signals ``live'' in the union of finite number of subspaces, the set of rank-$ \rank $ matrices expands to infinitely many subspaces. Thus, the selection rules do not generalize in a straightforward way.

Our contributions are the following:

\textbf{Ingredients of hard thresholding methods:} We analyze the behaviour and performance of hard thresholding methods from a global perspective. Five building blocks are studied: $ i) $ step size selection $ \mu_i $, $ ii) $ gradient or least-squares updates over restricted low-rank subspaces (e.g., adaptive block coordinate descent), $ iii) $ memory exploitation, $iv) $ active low-rank subspace tracking and, $v)$ low-rank matrix approximations (described next). We highlight the impact of these key pieces on the convergence rate and signal reconstruction performance and provide optimal and/or efficient strategies on how to set up these ingredients under different problem conditions. 

\textbf{Low-rank matrix approximations in hard thresholding methods:} 
In \cite{clash}, the authors show that the solution efficiency can be significantly improved by $\epsilon$-approximation algorithms. Based on similar ideas, we analyze the impact of $\epsilon$-approximate low rank-revealing schemes in the proposed algorithms with well-characterized time and space co- mplexities. Moreover, we provide extensive analysis to prove convergence using $\epsilon$-approximate low-rank projections.

\textbf{Hard thresholding-based framework with improved convergence conditions:} We study hard thresholding variants that provide salient computational tradeoffs for the class of greedy methods on low-rank matrix recovery. These methods, as they iterate, exploit the non-convex scaffold of low rank subspaces on which the approximation problem resides. Using simple analysis tools, we derive improved conditions that guarantee convergence, compared to state-of-the-art approaches. 

The organization of the paper is as follows. In Section \ref{sec:prel}, we set up the notation and provide some definitions and properties, essential for the rest of the paper. In Section \ref{sec:ALPS}, we describe the basic algorithmic frameworks in a nutshell, while in Section {\ref{section:ingredients}} we provide important ``ingredients'' for the class of hard-thresholding methods; detailed convergence analysis proofs are provided in Section \ref{section:convergence}. The complexity analysis of the proposed algorithms is provided in Section \ref{section:complexity}. We study two acceleration schemes in Sections \ref{sec:memory} and \ref{section:approximate}, based on memory utilization and $\epsilon$-approximate low-rank projections, respectively. We further improve convergence speed by exploiting randomized low rank projections in Section 9, based on power iteration-based subspace finder tools \cite{findingstructure}. We provide empirical support for our claims through experimental results on synthetic and real data in Section \ref{section:experiments}. Finally, we conclude with future work directions in Section \ref{sec: conc}.

\section{Elementary Definitions and Properties}{\label{sec:prel}}

We reserve lower-case and bold lower-case letters for scalar and vector variable representation, respectively. Bold upper-case letters denote matrices while bold calligraphic upper-case letters represent linear operators. We use calligraphic upper-case letters for set representations. We use $ \signal(i) $ to represent the matrix estimate at the $i$-th iteration. 

The rank of $ \signal $ is denoted as $ \text{rank}(\signal) \leq \min\lbrace m, n \rbrace $. The empirical data error is denoted as $f(\signal) := \vectornorm{\obs - \sensing \signal}_2^2 $ with gradient $ \nabla f(\signal) := -2 \sensing^\ast(\obs - \sensing \signal)$,  where $^\ast $ is the adjoint operation over the linear mapping $ \sensing $. The inner product between matrices $ \boldsymbol{A},~\boldsymbol{B} \in \mathbb{R}^{\dimension} $ is denoted as $ \left\langle \boldsymbol{A}, \boldsymbol{B} \right\rangle = \text{trace}(\boldsymbol{B}^T\boldsymbol{A})$, where $ ^T $ represents the transpose operation. $ \mathbf{I} $ represents an identity matrix with dimensions apparent from the context. 

Let $\mathcal{S} $ be a set of orthonormal, rank-1 matrices that span an arbitrary subspace in $\mathbb{R}^{\dimension}$. We reserve $ \text{span}(\mathcal{S}) $ to denote the subspace spanned by $\mathcal{S}$. With slight abuse of notation, we use:
\begin{align}
\text{rank}(\text{span}(\mathcal{S})) \equiv \max_{\boldsymbol{X}} \left \{ \text{rank}(\boldsymbol{X}):~\boldsymbol{X} \in \text{span}(\mathcal{S}) \right \},
\end{align} to denote the {\it maximum} rank a matrix $\boldsymbol{X} \in \mathbb{R}^{\dimension}$ can have such that $\boldsymbol{X}$ lies in the subspace spanned by the set $ \mathcal{S} $. Given a finite set $ \mathcal{S} $, $ |\mathcal{S}| $ denotes the cardinality of $ \mathcal{S} $. For any matrix $ \boldsymbol{X} $, we use $ R(\boldsymbol{X}) $ to denote its range.

We define a {\it minimum cardinality} set of orthonormal, rank-1 matrices that span the subspace induced by a set of rank-1 (and possibly non-orthogonal) matrices $\mathcal{S}$ as:
\begin{align}
\text{ortho}(\mathcal{S}) \in \argmin_{\mathcal{T}} \lbrace |\mathcal{T}|: \mathcal{T} \subseteq \mathcal{U}~\text{s.t.}~ \text{span}(\mathcal{T}) = \text{span}(\mathcal{S}) \rbrace, \nonumber
\end{align} where $ \mathcal{U} $ denotes the superset that includes all the sets of {\it orthonormal}, rank-1 matrices in $\mathbb{R}^{\dimension}$ such that $ \langle \boldsymbol{T}_i, \boldsymbol{T}_j \rangle = 0,~i \neq j, $ $\forall \boldsymbol{T}_i,\boldsymbol{T}_j \in \mathcal{T} $ and, $\vectornormbig{\boldsymbol{T}_i}_F = 1,~\forall i$. In general, $ \text{ortho}(\mathcal{S}) $ is not unique.

A well-known lemma used in the convergence rate proofs of this class of greedy hard thresholding algorithms is defined next.

\begin{lemma}{\label{lemma:5}}\cite{Bertsekas} \textit{Let $ \mathcal{J} \subseteq \mathbb{R}^{\dimension} $ be a closed convex set and $ f: \mathcal{J} \rightarrow \mathbb{R} $ be a smooth objective function defined over $ \mathcal{J} $. Let $ \signal^\ast \in \mathcal{J} $ be a local minimum of the objective function $ f $ over the set $ \mathcal{J} $. Then}
\begin{align}
\langle \nabla f(\signal^\ast), \signal - \signal^\ast \rangle \geq 0, \;\; \forall \signal \in \mathcal{J}.
\end{align}
\end{lemma}

\subsection{Singular Value Decomposition (SVD) and its properties}

\begin{definition}{\label{def:svd}}[SVD]
Let $ \signal \in \mathbb{R}^{\dimension} $ be a rank-$ l $ ($ l < \min $ $\lbrace m, n \rbrace $) matrix. Then, the SVD of $ \signal $ is given by:
\begin{align}
\signal = \boldsymbol{U} \boldsymbol{\Sigma} \boldsymbol{V}^{T} = \begin{bmatrix} \boldsymbol{U}_{\alpha} & \boldsymbol{U}_{\beta} \end{bmatrix} \begin{bmatrix} \widetilde{\boldsymbol{\Sigma}} & \boldsymbol{0} \\ \boldsymbol{0} & \boldsymbol{0} \end{bmatrix} \begin{bmatrix} \boldsymbol{V}_{\alpha}^{T} \\ \boldsymbol{V}_{\beta}^T \end{bmatrix}, \label{prel:eq:00}
\end{align} where $ \boldsymbol{U}_{\alpha} \in \mathbb{R}^{m \times l}, \boldsymbol{U}_{\beta} \in \mathbb{R}^{m \times (m - l)}, \boldsymbol{V}_{\alpha} \in \mathbb{R}^{n \times l}, \boldsymbol{V}_{\beta} \in \mathbb{R}^{n \times (n - l)} $ and $ \widetilde{\boldsymbol{\Sigma}} = \text{diag}(\sigma_1, \dots, \sigma_l) \in \mathbb{R}^{l \times l} $ for $ \sigma_1, \dots, $ $\sigma_l \in \mathbb{R}_{+} $. Here, the columns of $ \boldsymbol{U}, \boldsymbol{V} $ represent the set of left and right singular vectors, respectively, and $ \sigma_1, \dots, \sigma_l $ denote the singular values.
\end{definition}

For any matrix $ \boldsymbol{X} \in \mathbb{R}^{\dimension} $ with arbitrary $ \text{rank}(\boldsymbol{X}) \leq \min \lbrace m, n \rbrace $, its best orthogonal projection $ \mathcal{P}_{\rank}(\signal) $ onto the set of rank-$ \rank $ ($ k < \text{rank}(\boldsymbol{X}) $) matrices $ \mathcal{C}_\rank := \lbrace \boldsymbol{A} \in \mathbb{R}^{m \times n}: \text{rank}(\boldsymbol{A}) \leq k \rbrace $ defines the optimization problem:
\begin{align}
\mathcal{P}_k(\boldsymbol{X}) \in \argmin_{\boldsymbol{Y} \in \mathcal{C}_{\rank}} \vectornormbig{\boldsymbol{Y} - \boldsymbol{X}}_{F}. \label{eq:svd_proj}
\end{align} According to the Eckart-Young theorem \cite{horn1990matrix}, the best rank-$ \rank $ approximation of a matrix $ \boldsymbol{X} $ corresponds to its truncated SVD: if $ \boldsymbol{X} = \boldsymbol{U} \boldsymbol{\Sigma} \boldsymbol{V}^{T} $, then $ \mathcal{P}_k(\boldsymbol{X}) := \boldsymbol{U}_k \boldsymbol{\Sigma}_k \boldsymbol{V}_k^{T} $ where $ \boldsymbol{\Sigma}_k \in \mathbb{R}^{\rank \times \rank} $ is a diagonal matrix that contains the first $ \rank $ diagonal entries of $ \boldsymbol{\Sigma} $ and $\boldsymbol{U}_k, ~\boldsymbol{V}_k$ contain the corresponding left and right singular vectors, respectively. Moreover, this projection is not always unique. In the case of multiple identical singular values, the lexicographic approach is used to break ties. In any case, $\vectornormbig{\mathcal{P}_k(\boldsymbol{X}) - \boldsymbol{X}}_F \leq \vectornormbig{\boldsymbol{W} - \boldsymbol{X}}_F$ for any rank-$\rank$ $\boldsymbol{W} \in \mathbb{R}^{\dimension}$.

\subsection{Subspace projections}
Given a set of orthonormal, rank-1 matrices $ \mathcal{S} $, we denote the orthogonal projection operator onto the subspace induced by $ \mathcal{S} $ as $ \mathcal{P}_{\mathcal{S}} $\footnote{The distinction between $ \mathcal{P}_{\mathcal{S}} $ and $ \mathcal{P}_{\rank} $ for $ \rank $ positive integer is apparent from context.} which is an idempotent linear transformation; furthermore, we denote the orthogonal projection operator onto the orthogonal subspace of $ \mathcal{S} $ as $ \mathcal{P}_{\mathcal{S}^{\bot}} $. 
We can always decompose a matrix $ \signal \in \mathbb{R}^{\dimension} $ into two matrix components, as follows:
\begin{align}
\boldsymbol{X} := \mathcal{P}_{\mathcal{S}} \boldsymbol{X} + \mathcal{P}_{\mathcal{S}^{\bot}} \boldsymbol{X}, ~~\text{such that}~ \langle \mathcal{P}_{\mathcal{S}} \boldsymbol{X}, \mathcal{P}_{\mathcal{S}^{\bot}}\boldsymbol{X} \rangle = 0. \nonumber
\end{align} If $ \signal \in \text{span}(\mathcal{S}) $, the best projection of $ \signal $ onto the subspace induced by $ \mathcal{S} $ is the matrix $ \signal $ itself. Moreover, $ \vectornormbig{\mathcal{P}_{\mathcal{S}}\signal}_F \leq \vectornormbig{\signal}_F $ for any $\mathcal{S}$ and $\signal$.
\begin{definition}{\label{def:svd_proj}}[Orthogonal projections using SVD]
Let $ \signal \in \mathbb{R}^{\dimension} $ be a matrix with arbitrary $ \text{rank} $ and SVD decomposition given by (\ref{prel:eq:00}). Then, $ \mathcal{S}:= \lbrace \boldsymbol{u}_i \boldsymbol{v}_i^T: i = 1,\dots, \rank \rbrace $ ($\rank \leq \text{rank}(\signal)$) constitutes a set of orthonormal, rank-1 matrices that spans the best $k$-rank subspace in $R(\signal)$ and $R(\signal^T)$; here, $ \boldsymbol{u}_i $ and $\boldsymbol{v}_i$ denote the $ i $-th left and right singular vectors, respectively. The orthogonal projection onto this subspace is given by \cite{candès2009exact}:
\begin{align}
\mathcal{P}_{\mathcal{S}} \signal = \mathcal{P}_{\mathcal{U}} \signal + \signal \mathcal{P}_{\mathcal{V}} - \mathcal{P}_{\mathcal{U}} \signal \mathcal{P}_{\mathcal{V}} \label{eq:newproj}
\end{align} where $\mathcal{P}_{\mathcal{U}} = \boldsymbol{U}_{:, 1:\rank} \boldsymbol{U}_{:, 1:\rank}^T$ and $\mathcal{P}_{\mathcal{V}} = \boldsymbol{V}_{:, 1:\rank} \boldsymbol{V}_{:, 1:\rank}^T$ in \textsc{Matlab} notation. Moreover, the orthogonal projection onto the $\mathcal{S}^{\bot}$ is given by:
\begin{align}
\mathcal{P}_{\mathcal{S}^{\bot}} \signal = \signal - \mathcal{P}_{\mathcal{S}} \signal. \label{eq:neworthoproj}
\end{align}
\end{definition}

In the algorithmic descriptions, we use $\mathcal{S} \leftarrow \mathcal{P}_{\rank}\left(\signal\right)$ to denote the set of rank-1, orthonormal matrices as outer products of the $\rank$ left $\boldsymbol{u}_i$ and right $\boldsymbol{v}_i$ principal singular vectors of $\signal$ that span the best rank-$\rank$ subspace of $\signal$; e.g. $\mathcal{S} = \lbrace \boldsymbol{u}_i \boldsymbol{v}_i,~i = 1,\dots,\rank\rbrace$. Moreover, $\widehat{\signal} \leftarrow \mathcal{P}_{\rank}\left(\signal\right)$ denotes a/the best rank-$\rank$ projection matrix of $\signal$. In some cases, we use $\lbrace \mathcal{S},~\widehat{\signal} \rbrace \leftarrow \mathcal{P}_{\rank}\left(\signal\right)$ when we compute both. The distiction between these cases is apparent from the context.

\subsection{Restricted Isometry Property}

Many conditions have been proposed in the literature to establish solution uniqueness and recovery stability such as null space property \cite{cohen06}, exact recovery condition \cite{Tro04:Greed-Good}, etc.
For the matrix case, \cite{brecht2010} proposed the {\it restricted isometry property} (RIP) for the ARM problem. 

\begin{definition}{\label{def:RIP}}[Rank Restricted Isometry Property (R-RIP) for matrix linear operators \cite{brecht2010}] A linear operator $ \sensing: \mathbb{R}^{\dimension} $ $\rightarrow \mathbb{R}^{\numsam} $ satisfies the R-RIP with constant $ \delta_{\rank}(\sensing) \in (0,1)$ if and only if:
\begin{equation}\label{eq:RIP}
  (1-\delta_{\rank}(\sensing))\vectornormbig{\signal}_F^2 \leq \vectornormbig{\sensing \signal}_2^2 \leq (1+\delta_{\rank}(\sensing))\vectornormbig{\signal}_F^2, 
\end{equation} $ ~~\forall \boldsymbol{X} \in \mathbb{R}^{\dimension} ~\text{such that} ~\text{rank}(\signal) \leq \rank. $ We write $\delta_{\rank}$ to mean $\delta_{\rank}(\sensing)$, unless otherwise stated.
\end{definition} \cite{liuuniversal} shows that Pauli operators satisfy the rank-RIP in compressive settings while, in function learning, the linear map $\sensing$ is designed specifically to satisfy the rank-RIP \cite{hemant2012active}.

\subsection{Some useful bounds using R-RIP}

In this section, we present some lemmas that are useful in our subsequent developments---these lemmas are consequen- ces of the R-RIP of $ \sensing $.

\begin{lemma}{\label{lemma:1}}\cite{admira2010}
Let $ \sensing: \mathbb{R}^{\dimension} \rightarrow \mathbb{R}^{\numsam} $ be a linear operator that satisfies the R-RIP with constant $ \delta_{\rank} $. Then, $ \forall \boldsymbol{v} \in \mathbb{R}^{\numsam} $, the following holds true:
\begin{align}
\vectornormbig{\mathcal{P}_{\mathcal{S}}(\sensing^\ast \boldsymbol{v})}_F \leq \sqrt{1 + \delta_{\rank}}\vectornormbig{\boldsymbol{v}}_2, \label{eq:prel:00}
\end{align} where $ \mathcal{S} $ is a set of orthonormal, rank-1 matrices in $\mathbb{R}^{\dimension}$ such that $ \text{rank}(\mathcal{P}_{\mathcal{S}}\signal) \leq \rank, ~\forall \signal \in \mathbb{R}^{\dimension} $.
\end{lemma}

\begin{lemma}{\label{lemma:2}}\cite{admira2010}
Let $ \sensing: \mathbb{R}^{\dimension} \rightarrow \mathbb{R}^{\numsam} $ be a linear operator that satisfies the R-RIP with constant $ \delta_{\rank} $. Then, $ \forall \signal \in \mathbb{R}^{\dimension} $, the following holds true:
\begin{align}
(1-\delta_{\rank})\vectornormbig{\mathcal{P}_{\mathcal{S}}\signal}_F &\leq \vectornormbig{\mathcal{P}_{\mathcal{S}}\sensing^\ast \sensing \mathcal{P}_{\mathcal{S}}\signal}_F \nonumber \\ &\leq (1+\delta_{\rank})\vectornormbig{\mathcal{P}_{\mathcal{S}}\signal}_F, \label{eq:prel:01}
\end{align} where $ \mathcal{S} $ is a set of orthonormal, rank-1 matrices in $\mathbb{R}^{\dimension}$ such that $ \text{rank}(\mathcal{P}_{\mathcal{S}}\signal) \leq \rank, ~\forall \signal \in \mathbb{R}^{\dimension} $.
\end{lemma}

\begin{lemma}{\label{lemma:3}}\cite{Goldfarb:2011}
Let $ \sensing: \mathbb{R}^{\dimension} \rightarrow \mathbb{R}^{\numsam} $ be a linear operator that satisfies the R-RIP with constant $ \delta_{\rank} $ and $ \mathcal{S} $ be a set of orthonormal, rank-1 matrices in $\mathbb{R}^{\dimension}$ such that $ \text{rank}(\mathcal{P}_{\mathcal{S}}\signal) \leq \rank, ~\forall \signal \in \mathbb{R}^{\dimension} $. Then, for $ \mu > 0 $, $ \sensing $ satisfies:
\begin{align}
\lambda(\mu\mathcal{P}_{\mathcal{S}}\sensing^\ast \sensing \mathcal{P}_{\mathcal{S}}) \in [\mu(1-\delta_{\rank}), \mu(1+\delta_{\rank})].
\end{align} where $ \lambda(\boldsymbol{\mathcal{B}}) $ represents the range of eigenvalues of the linear operator $ \boldsymbol{\mathcal{B}}: \mathbb{R}^{\numsam} \rightarrow \mathbb{R}^{\dimension} $.
Moreover, $ \forall \signal \in \mathbb{R}^{\dimension} $, it follows that:
\begin{align}
&\vectornormbig{(\id - \mu\mathcal{P}_{\mathcal{S}}\sensing^\ast \sensing \mathcal{P}_{\mathcal{S}})\mathcal{P}_{\mathcal{S}}\signal}_F \nonumber \\ &\leq \max \left\{ \mu(1 + \delta_{\rank}) - 1, 1- \mu(1 - \delta_{\rank}) \right\} \vectornormbig{\mathcal{P}_{\mathcal{S}}\signal}_F.
\end{align}
\end{lemma}

\begin{lemma}{\label{lemma:4}}\cite{Goldfarb:2011}
Let $ \sensing: \mathbb{R}^{\dimension} \rightarrow \mathbb{R}^{\numsam} $ be a linear operator that satisfies the R-RIP with constant $ \delta_{\rank} $ and $ \mathcal{S}_1, \mathcal{S}_2 $ be two sets of orthonormal, rank-1 matrices in $\mathbb{R}^{\dimension}$ such that 
\begin{align}
\text{rank}(\mathcal{P}_{\mathcal{S}_1 \cup \mathcal{S}_2}\signal) \leq \rank, ~\forall \signal \in \mathbb{R}^{\dimension}. 
\end{align} Then, the following inequality holds:
\begin{align}
\vectornormbig{\mathcal{P}_{\mathcal{S}_1}\sensing^\ast \sensing \mathcal{P}_{\mathcal{S}_1^{\bot}}\signal}_F \leq \delta_{\rank}\vectornormbig{\mathcal{P}_{\mathcal{S}_1^{\bot}}\signal}_F, \forall \boldsymbol{X} \in \text{span}(\mathcal{S}_2).
\end{align}
\end{lemma}

\section{Algrebraic Pursuits in a nutshell}{\label{sec:ALPS}}

\begin{algorithm*}[t!]
   \caption{\textsc{Matrix ALPS I}}\label{algo: class}
\begin{algorithmic}[1]
   \Statex {\bfseries Input:} $\obs$, $\sensing$, $\rank$, Tolerance $ \eta $, MaxIterations
   \Statex {\bfseries Initialize:} $ \signal(0) \leftarrow 0 $, $ \mathcal{X}_0 \leftarrow \lbrace \emptyset \rbrace $, $ i \leftarrow 0 $
   \Statex {\bfseries repeat} 
   \State \hspace{0.16cm} $ \mathcal{D}_i \leftarrow \mathcal{P}_{\rank}\big( \mathcal{P}_{\mathcal{X}_i^{\bot}} \nabla f(\signal(i)) \big) $ \hspace*{\fill}\textit{(Best rank-$ \rank $ subspace orthogonal to $ \mathcal{X}_i $)~~~~~~~~~}
   \State \hspace{0.16cm} $ \mathcal{S}_i \leftarrow \mathcal{D}_i \cup \mathcal{X}_i$ \hspace*{\fill}\textit{(Active subspace expansion)~~~~~~~~~}
   \State \hspace{0.16cm} $ \mu_i \leftarrow \argmin_{\mu} \vectornormbig{\obs - \sensing\big( \signal(i) - \frac{\mu}{2} \mathcal{P}_{\mathcal{S}_i} \nabla f(\signal(i)) \big)}_2^2 = \frac{\vectornorm{\mathcal{P}_{\mathcal{S}_i} \nabla f(\signal(i))}_F^2}{\vectornorm{\sensing \mathcal{P}_{\mathcal{S}_i} \nabla f(\signal(i))}_2^2} $ \hspace*{\fill} \textit{(Step size selection)~~~~~~~~~}
   \State \hspace{0.16cm} $ \boldsymbol{V}(i) \leftarrow \signal(i) - \frac{\mu_i}{2} \mathcal{P}_{\mathcal{S}_i}\nabla f(\signal(i)) $ \hspace*{\fill} \textit{(Error norm reduction via gradient descent)~~~~~~~~~}
   \State \hspace{0.16cm} $ \lbrace \mathcal{W}_i,~\boldsymbol{W}(i)\rbrace \leftarrow \mathcal{P}_{\rank}(\boldsymbol{V}(i))  $ \hspace*{\fill}{\textit{(Best rank-$ \rank $ subspace selection)~~~~~~~~~}}
   \State \hspace{0.16cm} $ \xi_i \leftarrow \argmin_{\xi} \vectornormbig{\obs - \sensing\big( \boldsymbol{W}(i) - \frac{\xi}{2} \mathcal{P}_{\mathcal{W}_i} \nabla f(\boldsymbol{W}(i)) \big)}_2^2 = \frac{\vectornorm{\mathcal{P}_{\mathcal{W}_i} \nabla f(\boldsymbol{W}(i))}_F^2}{\vectornorm{\sensing \mathcal{P}_{\mathcal{W}_i} \nabla f(\boldsymbol{W}(i))}_2^2} $ \hspace*{\fill} \textit{(Step size selection)~~~~~~~~~}
   \State \hspace{0.16cm} $ \signal(i+1) \leftarrow \boldsymbol{W}(i) - \frac{\xi_i}{2} \mathcal{P}_{\mathcal{W}_i} \nabla f(\boldsymbol{W}(i)) ~~\text{with}~ \mathcal{X}_{i+1} \leftarrow \mathcal{P}_k(\signal(i+1)) $ \hspace*{\fill}\textit{(De-bias using gradient descent)~~~~~~~~~}
   \Statex \hspace{0.16cm} $ i \leftarrow i + 1 $
   \Statex {\bfseries until} $\vectornorm{\signal(i) - \signal(i-1)}_2 \leq \eta \vectornorm{\signal(i)}_2 $ or MaxIterations.
\end{algorithmic}
\end{algorithm*}

\begin{algorithm*}[t!]
   \caption{ADMiRA Instance}\label{algo: class}
\begin{algorithmic}[1]
   \Statex {\bfseries Input:} $\obs$, $\sensing$, $\rank$, Tolerance $ \eta $, MaxIterations
   \Statex {\bfseries Initialize:} $ \signal(0) \leftarrow 0 $, $ \mathcal{X}_0 \leftarrow \lbrace \emptyset \rbrace $, $ i \leftarrow 0 $
   \Statex {\bfseries repeat} 
   \State \hspace{0.16cm} $ \mathcal{D}_i \leftarrow \mathcal{P}_{\rank}\big(\mathcal{P}_{\mathcal{X}_i^{\bot}}\nabla f(\signal(i)) \big) $ \hspace*{\fill}\textit{(Best rank-$ \rank $ subspace orthogonal to $ \mathcal{X}_i $)~~~~~~~~~}
   \State \hspace{0.16cm} $ \mathcal{S}_i \leftarrow \mathcal{D}_i \cup \mathcal{X}_i$ \hspace*{\fill}\textit{(Active subspace expansion)~~~~~~~~~}
   \State \hspace{0.16cm} $ \boldsymbol{V}(i) \leftarrow \argmin_{\boldsymbol{V}: \boldsymbol{V} \in \text{span}(\mathcal{S}_i)} \vectornormbig{\obs - \sensing \boldsymbol{V}}_2^2 $ \hspace*{\fill} \textit{(Error norm reduction via least-squares optimization)~~~~~~~~~}
   \State \hspace{0.16cm} $ \lbrace \mathcal{X}_{i+1},~\signal(i+1)\rbrace \leftarrow \mathcal{P}_{\rank}(\boldsymbol{V}(i))  $ \hspace*{\fill}{\textit{(Best rank-$ \rank $ subspace selection)~~~~~~~~~}}
   \Statex \hspace{0.16cm} $ i \leftarrow i + 1 $
   \Statex {\bfseries until} $\vectornorm{\signal(i) - \signal(i-1)}_2 \leq \eta \vectornorm{\signal(i)}_2 $ or MaxIterations.
\end{algorithmic}
\end{algorithm*}

Explicit descriptions of the proposed algorithms are provided in Algorithms 1 and 2. Algorithm 1 follows from the ALgrebraic PursuitS (ALPS) scheme for the vector case \cite{cevher2011alps}. \textsc{Matrix ALPS I} provides efficient strategies for adaptive step size selection and additional signal estimate updates at each iteration (these motions are explained in detail in the next subsection). Algorithm 2 (ADMiRA) \cite{admira2010} further improves the performance of Algorithm 1 by introducing least squares optimization steps on restricted subspaces---this technique borrows from a series of vector reconstruction algorithms such as CoSaMP \cite{cosamp}, Subspace Pursuit (SP) \cite{SP} and Hard Thresholding Pursuit (HTP) \cite{HTP}. 

In a nutshell, both algorithms simply seek to improve the subspace selection by iteratively collecting an extended subspace $ \mathcal{S}_i $ with $ \text{rank}(\text{span}(\mathcal{S}_i)) \leq 2\rank $ and then finding the rank-$ \rank $ matrix that fits the measurements in this restricted subspace using least squares or gradient descent motions. 

At each iteration, the Algorithms 1 and 2 perform motions from the following list:

\begin{framed}
1) \textit{Best rank-$ \rank $ subspace orthogonal to $ \mathcal{X}_i $ and active subspace expansion:} We identify the best rank-$ \rank $ subspace of the current gradient $ \nabla f(\signal(i)) $, orthogonal to $ \mathcal{X}_i $ and then merge this low-rank subspace with $ \mathcal{X}_i $. This motion guarantees that, at each iteration, we expand the current rank-$ \rank $ subspace estimate with $ \rank $ new, rank-1 orthogonal subspaces to explore. 

2a) \textit{Error norm reduction via greedy descent with adaptive step size selection (Algorithm 1):} We decrease the data error by performing a single gradient descent step. This scheme is based on a one-shot step size selection procedure (Step size selection step)---detailed description of this approach is given in Section \ref{section:ingredients}.

2b) \textit{Error norm reduction via least squares optimization (Algorithm 2):} We decrease the data error $f(\signal)$ on the active $ O(\rank) $-low rank subspace. Assuming $ \sensing $ is well-conditioned over low-rank subspaces, the main complexity of this operation is dominated by the solution of a symmetric linear system of equations. 

3) \textit{Best rank-$ \rank $ subspace selection:} We project the constrained solution onto the set of rank-$\rank$ matrices $ \mathcal{C}_\rank := \lbrace \boldsymbol{A} \in \mathbb{R}^{m \times n}: \text{rank}(\boldsymbol{A}) \leq k \rbrace $ to arbitrate the active support set. 
This step is calculated in polynomial time complexity as a function of $ \dimension $ using SVD or other matrix rank-revealing decomposition algorithms---further discussions about this step and its approximations can be found in Sections \ref{section:approximate} and \ref{section:QR}.

4) \textit{De-bias using gradient descent (Algorithm 1):} We de-bias the current estimate $ \boldsymbol{W}(i) $ by performing an additional gradient descent step, decreasing the data error. The step size selection procedure follows the same motions as in 2a).

\end{framed}

\section{Ingredients for hard thresholding methods}{\label{section:ingredients}}

\subsection{Step size selection}

For the sparse vector approximation problem, recent works on the performance of the IHT algorithm provide strong convergence rate guarantees in terms of RIP constants \cite{Blumensath_iterativehard}. However, as a prerequisite to achieve these strong isometry constant bounds, the step size is set $ \mu_i = 1, \forall i, $ given that the sensing matrix satisfies $ \vectornorm{\boldsymbol{\Phi}}_2^2 < 1 $ where $\vectornorm{\cdot}_2$ denotes the spectral norm \cite{HTP}; similar analysis can be found in \cite{SVP} for the matrix case. From a different perspective, \cite{garg2009gradient} proposes a constant step size $ \mu_i = 1/(1+\delta_{2K}), ~\forall i $, based on a simple but intuitive convergence analysis of the gradient descent method. 

Unfortunately, most of the above problem assumptions are not naturally met; the authors in \cite{NIHT} provide an intuitive example where IHT algorithm behaves differently under various scalings of the sensing matrix; similar counterexamples can be devised for the matrix case. Violating these assumptions usually leads to unpredictable signal recovery performance of the class of hard thresholding methods. 
Therefore, more sophisticated step size selection procedures should be devised to tackle these issues during actual recovery. On the other hand, the computation of R-RIP constants has exponential time complexity for the strategy of \cite{SVP}. 

To this end, existing approaches broadly fall into two categories: constant and adaptive step size selection. In this work, we present efficient strategies to adaptively select the step size $ \mu_i $ that implies fast convergence rate, for mild R-RIP assumptions on $ \sensing $. Constant step size strategies easily follow from \cite{KyrillidisCevherRecipes} and are not listed in this work. 

\textbf{Adaptive step size selection.} 
There is limited work on the adaptive step size selection for hard thresholding methods. To the best of our knowledge, apart from \cite{KyrillidisCevherRecipes}, \cite{NIHT}-\cite{AIHT} are the only studies that attempt this via line searching for the vector case. At the time of review process, we become aware of \cite{tannernormalized} which implements ideas presented in \cite{NIHT} for the matrix case.

\begin{figure*}[!t]
\hspace{-0.2cm}\centering
\begin{tabular}{ccc}
\centerline{\subfigure[]{\includegraphics[width = 0.34\textwidth]{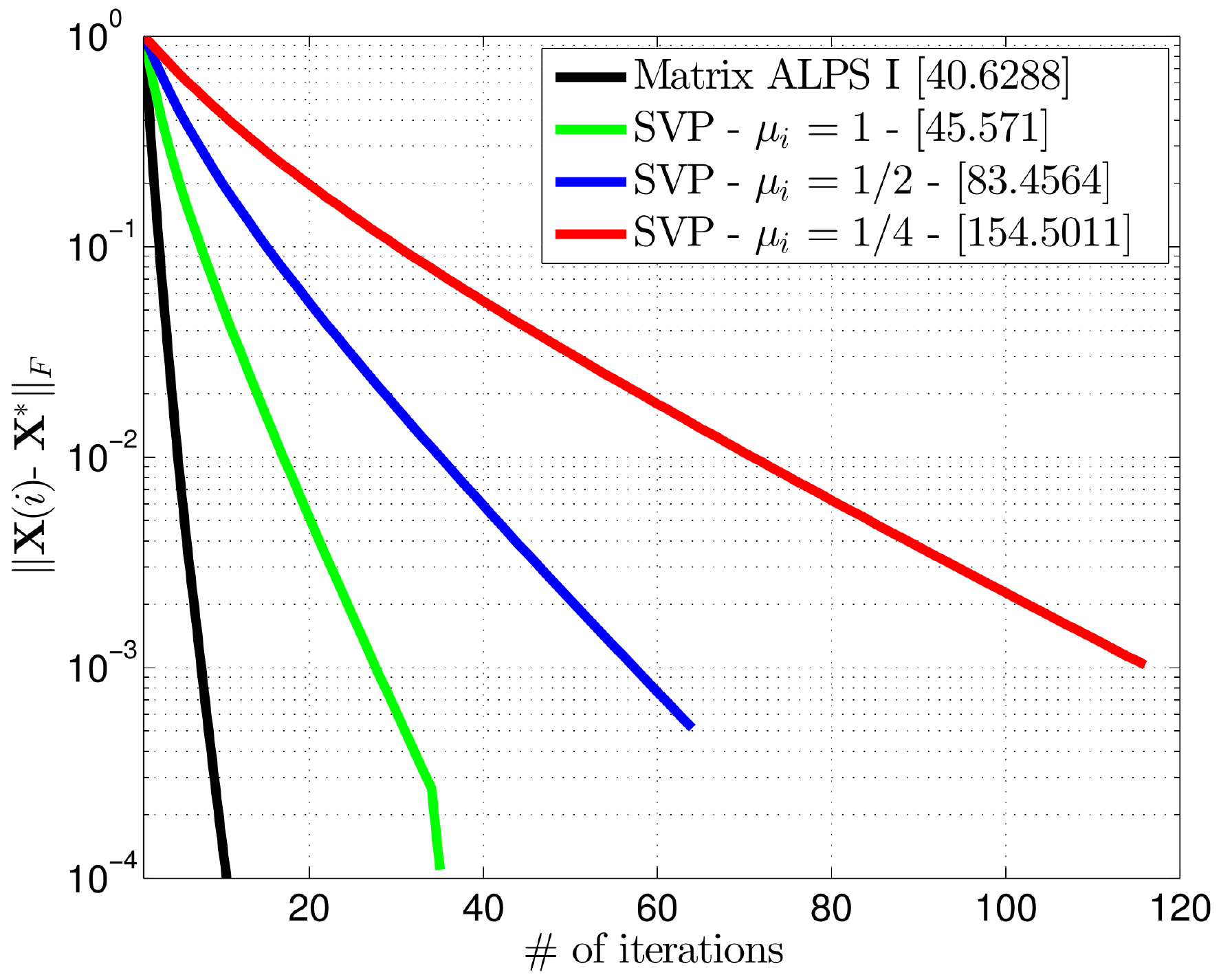}\label{fig:1a}} 
\hfill
\subfigure[]{\includegraphics[width = 0.34\textwidth]{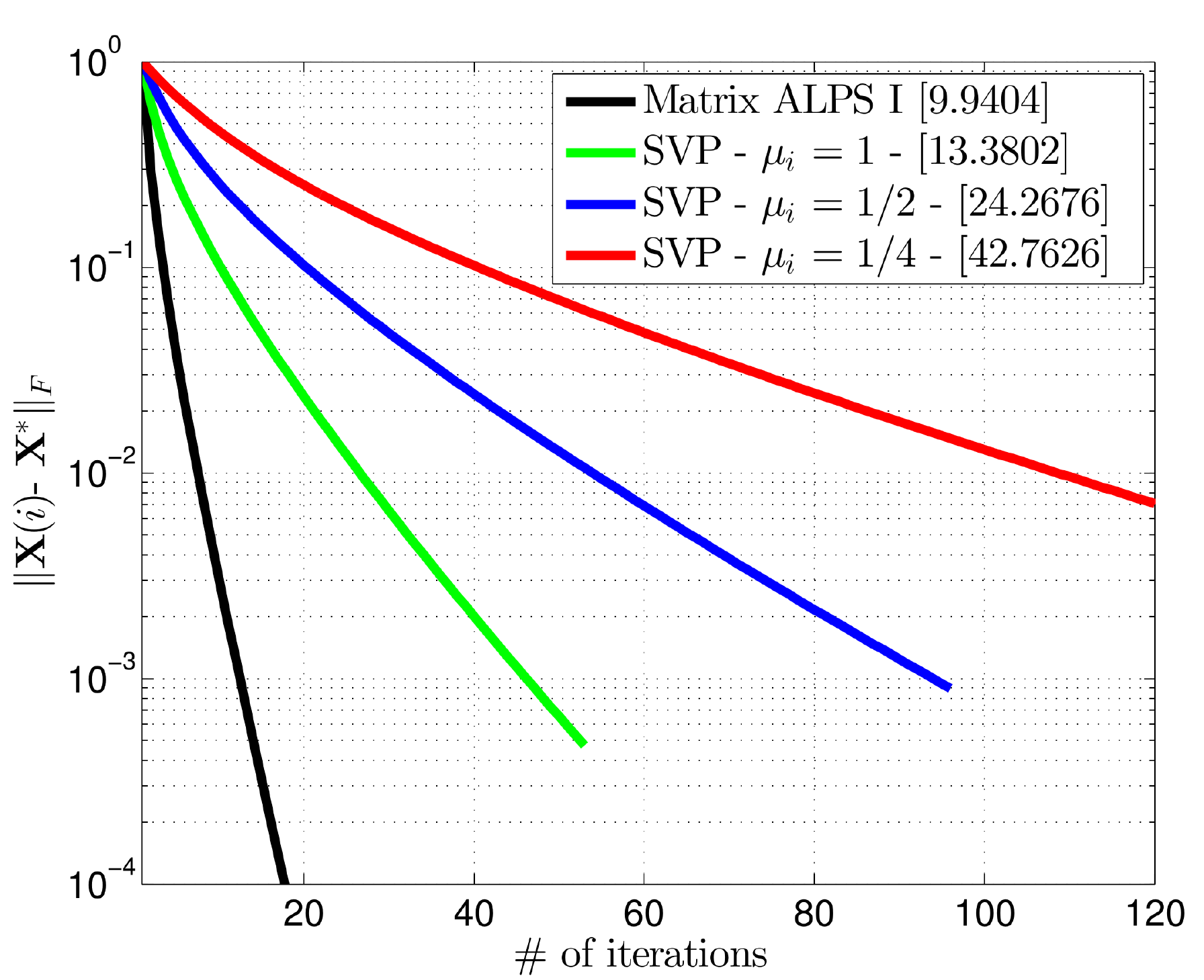}\label{fig:1b}}
\hfill
\subfigure[]{\includegraphics[width = 0.34\textwidth]{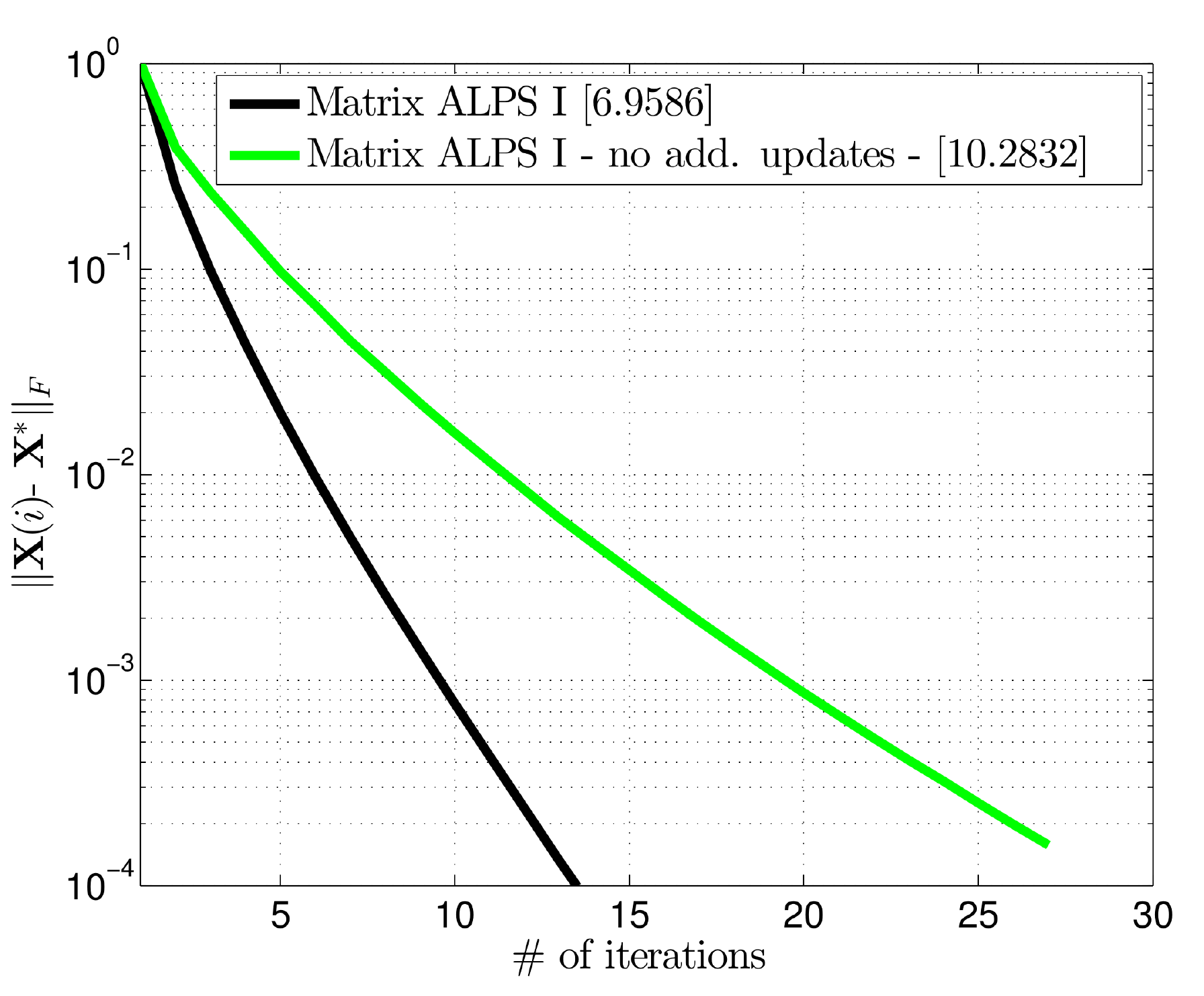}\label{fig:1c}}}
\end{tabular}
\caption{\small\sl Median error per iteration for various step size policies and 20 Monte-Carlo repetitions. In brackets, we present the median time consumed for convergene in seconds. (a) $m = n = 2048 $, $p = 0.4n^2, $ and rank $ \rank = 70 $---$\sensing $ is formed by permuted and subsampled noiselets \cite{coifman2001noiselets}. (b) $n = 2048 $ , $ m = 512 $, $p = 0.4n^2$, and rank $\rank = 50$---we use underdetermined linear map $\sensing$ according to the MC problem (c) $ n = 2048 $, $ m = 512 $, $p = 0.4n^2, $ and rank $ \rank = 40 $---we use underdetermined linear map $\sensing$ according to the MC problem.}
\end{figure*}

According to Algorithm 1, let $ \signal(i) $ be the current rank-$ \rank $ matrix estimate spanned by the set of orthonormal, rank-1 matrices in $ \mathcal{X}_i $. Using regular gradient descent motions, the new rank-$ \rank $ estimate $ \boldsymbol{W}(i) $ can be calculated through: 
\begin{align}
\boldsymbol{V}_i = \signal(i) - \frac{\mu}{2} \nabla f(\signal(i)), ~~~\lbrace \mathcal{W}_i,~\boldsymbol{W}(i)\rbrace \leftarrow \mathcal{P}_{\rank}(\boldsymbol{V}(i)). \nonumber
\end{align} We highlight that the rank-$\rank$ approximate matrix may not be unique. It then holds that the subspace spanned by $ \mathcal{W}_i $ originates: $ i) $ either from the subspace of $ \mathcal{X}_i $, $ ii) $ or from the best subspace (in terms of the Frobenius norm metric) of the current gradient $ \nabla f(\signal(i)) $, {\it orthogonal to $ \mathcal{X}_i $}, $ iii) $ or from the combination of orthonormal, rank-1 matrices lying on the union of the above two subspaces. The statements above can be summarized in the following expression:
\begin{align}
\text{span}(\mathcal{W}_i) \in \text{span}\left( \mathcal{D}_i \cup \mathcal{X}_i\right)
\end{align} for any step size $ \mu_i $ and $\mathcal{D}_i \leftarrow \mathcal{P}_{\rank}\big( \mathcal{P}_{\mathcal{X}_i^{\bot}} \nabla f(\signal(i)) \big)$. Since $ \text{rank}(\text{span}(\mathcal{W}_i)) \leq \rank $, we easily deduce the following key observation: let $ \mathcal{S}_i \leftarrow \mathcal{D}_i \cup \mathcal{X}_i $ be a set of rank-1, orthonormal matrices where $ \text{rank}(\text{span}(\mathcal{S}_i)) \leq 2\rank $. 
Given $ \mathcal{W}_{i} $ is unknown before the $ i $-th iteration, $ \mathcal{S}_i $ spans the smallest subspace that contains $ \mathcal{W}_i $ such that the following equality
\begin{align}
\mathcal{P}_{\rank}&\left(\signal(i) - \frac{\mu_i}{2}\nabla f(\signal(i)) \right) \nonumber \\ &= \mathcal{P}_{\rank}\left(\signal(i) - \frac{\mu_i}{2}\mathcal{P}_{\mathcal{S}_i}\nabla f(\signal(i)) \right)
\end{align} necessarily holds.\footnote{In the case of multiple identical singular values, any ties are lexicographically dissolved.}

To compute step-size $ \mu_i $, we use:
\begin{align}
\mu_i &= \argmin_\mu \vectornormbig{\obs - \sensing \left(\signal(i)  - \frac{\mu}{2} \mathcal{P}_{\mathcal{S}_i}\nabla f(\signal(i))\right)}_2^2 \nonumber \\ &= \frac{\| \mathcal{P}_{\mathcal{S}_i}\nabla f(\signal(i))\|_F^2}{\|\sensing\mathcal{P}_{\mathcal{S}_i}\nabla f(\signal(i))\|_2^2}, \label{eq:step}
\end{align} i.e., $ \mu_i $ is the minimizer of the objective function, given the current gradient $\nabla f(\signal(i))$. 
Note that:
\begin{align}{\label{eq:additional:1}}
1-\delta_{2\rank}(\sensing)\le \frac{1}{\mu_i} \le 1+\delta_{2\rank}(\sensing),
\end{align} due to R-RIP---i.e., we select $ 2\rank $ subspaces such that $ \mu_i $ satisfies (\ref{eq:additional:1}). We can derive similar arguments for the additional step size selection $ \xi_i $ in Step 6 of Algorithm 1. 

Adaptive $ \mu_i $ scheme results in more restrictive worst-case isometry constants compared to \cite{HTP, foucart2010sparse, SVP}, but faster convergence and better stability are empirically observed in general. In \cite{SVP}, the authors present the Singular Value Projection (SVP) algorithm, an iterative hard thresholding algorithm for the ARM problem. According to \cite{SVP}, both constant and iteration dependent (but user-defined) step sizes are considered. Adaptive strategies presented in \cite{SVP} require the computation of R-RIP constants which has exponential time complexity. Figures 1(a)-(b) illustrate some characteristic examples. The performance varies for different problem configurations. For $\mu > 1$, SVP {\it diverges} for various test cases. We note that, for large fixed matrix dimensions $m, n$, adaptive step size selection becomes computationally expensive compared to constant step size selection strategies, as the rank of $\bestsignal$ increases.

\subsection{Updates on restricted subspaces}

In Algorithm 1, at each iteration, the new estimate $ \boldsymbol{W}(i) \leftarrow \mathcal{P}_{\rank}\left(\boldsymbol{V}(i)\right) $ can be further refined by applying a single or multiple gradient descent updates with line search restricted on $ \mathcal{W}_{i} $ \cite{HTP} (Step 7 in Algorithm 1): 
\begin{align}
\signal(i+1) \leftarrow \boldsymbol{W}(i) - \frac{\xi_i}{2} \mathcal{P}_{\mathcal{W}_i} \nabla f(\boldsymbol{W}(i)),  \nonumber
\end{align} $\text{where}~ \xi_i = \frac{\vectornorm{\mathcal{P}_{\mathcal{W}_i} \nabla f(\boldsymbol{W}(i))}_F^2}{\vectornorm{\sensing \mathcal{P}_{\mathcal{W}_i} \nabla f(\boldsymbol{W}(i))}_2^2}. $ In spirit, the gradient step above is the same as block coordinate descent in convex optimization where we find the subspaces adaptively. Figure 1(c) depicts the acceleration achieved by using additional gradient updates over restricted low-rank subspaces for a test case.

\subsection{Acceleration via memory-based schemes and low-rank matrix approximations}
Memory-based techniques can be used to improve convergence speed. Furthermore, low-rank matrix approximation tools overcome the computational overhead of computing the best low-rank projection by inexactly solving (\ref{eq:svd_proj}). We keep the discussion on memory utilization for Section \ref{sec:memory} and low-rank matrix approximations for Sections \ref{section:approximate} and \ref{section:QR} where we present new algorithmic frameworks for low-rank matrix recovery.

\begin{figure*}[!t]
\hspace{-0.2cm}\centering
\begin{tabular}{ccc}
\centerline{\subfigure[]{\includegraphics[width = 0.3\textwidth]{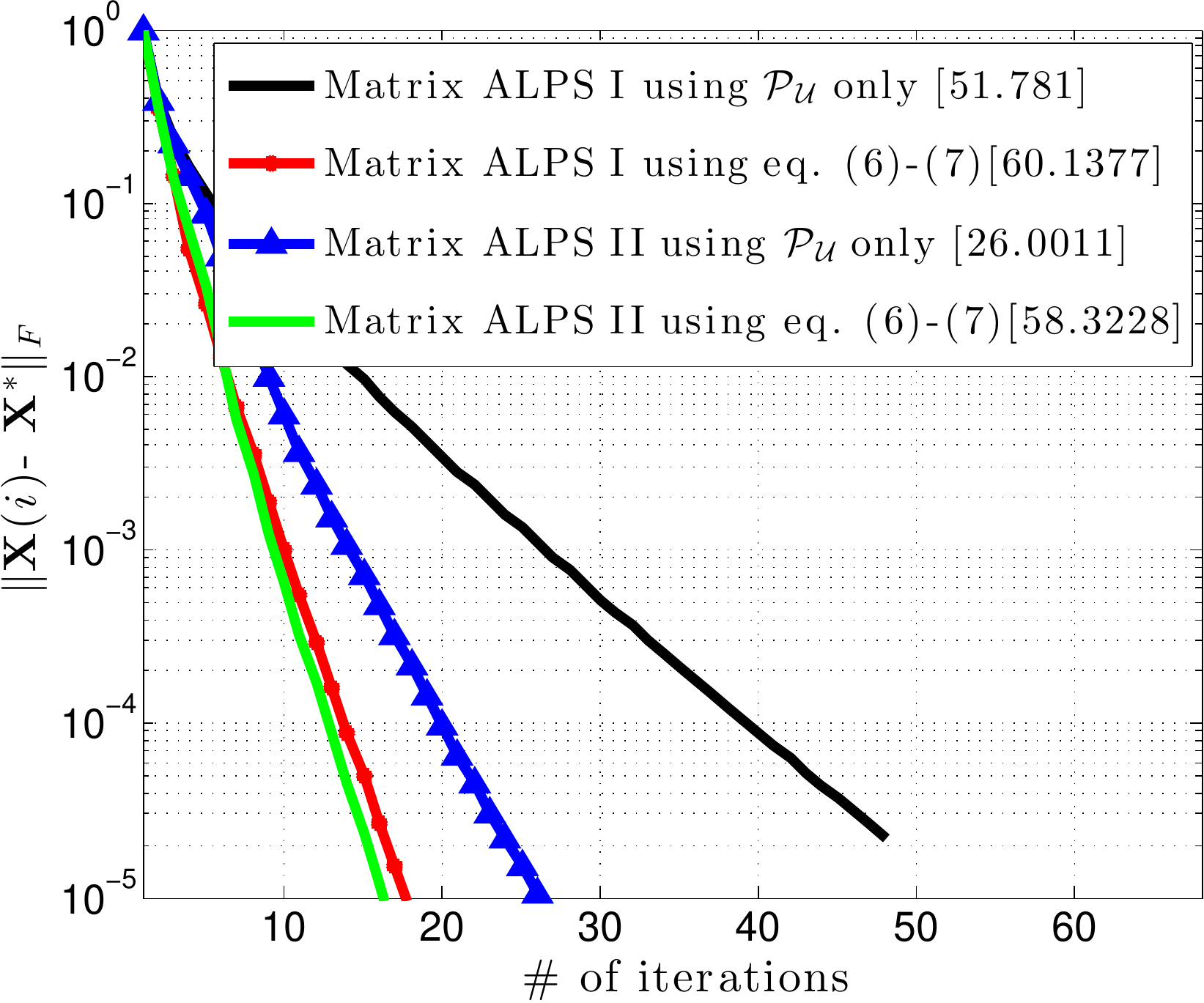}} 
\hfill
\subfigure[]{\includegraphics[width = 0.3\textwidth]{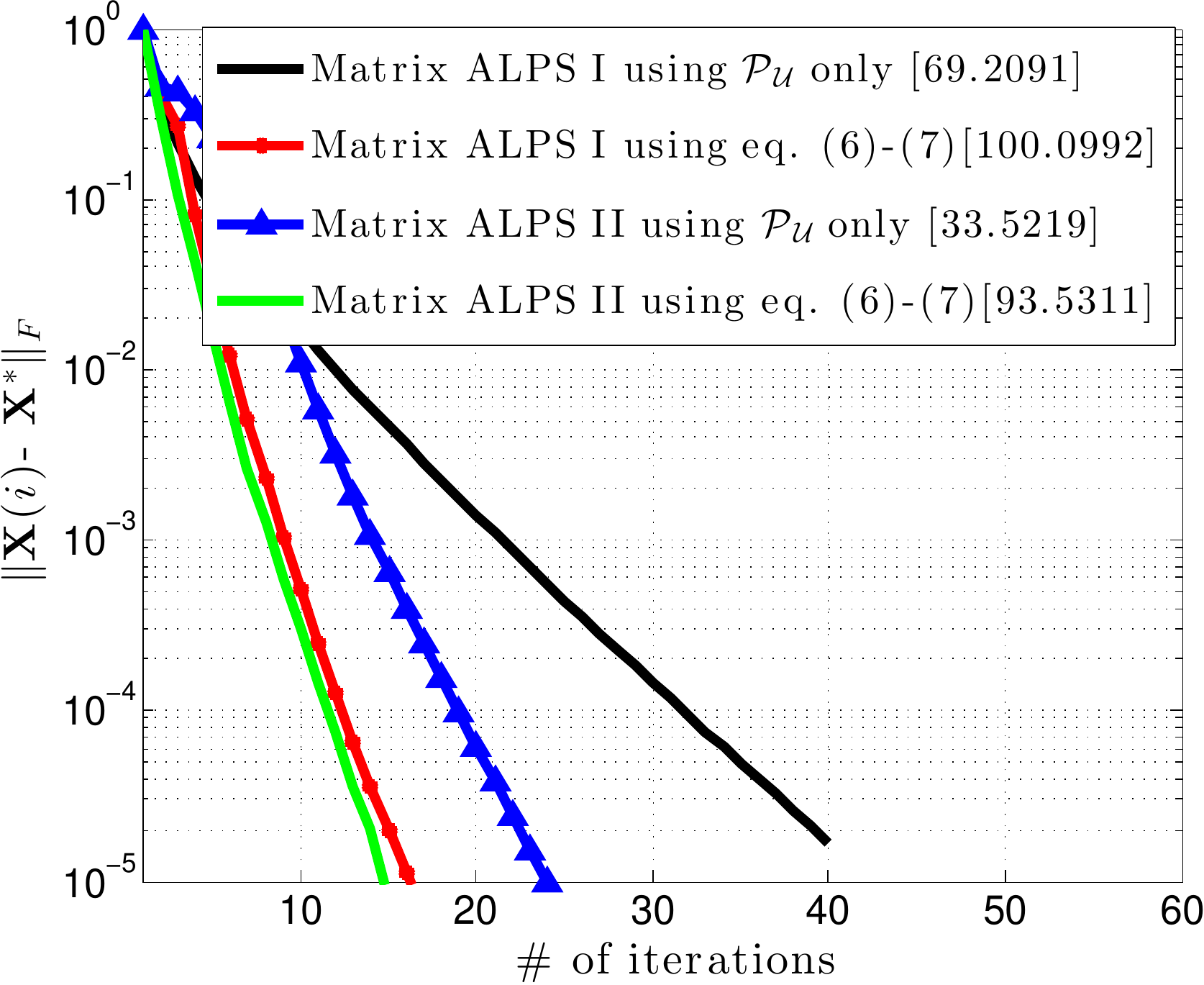}}
\hfill
\subfigure[]{\includegraphics[width = 0.3\textwidth]{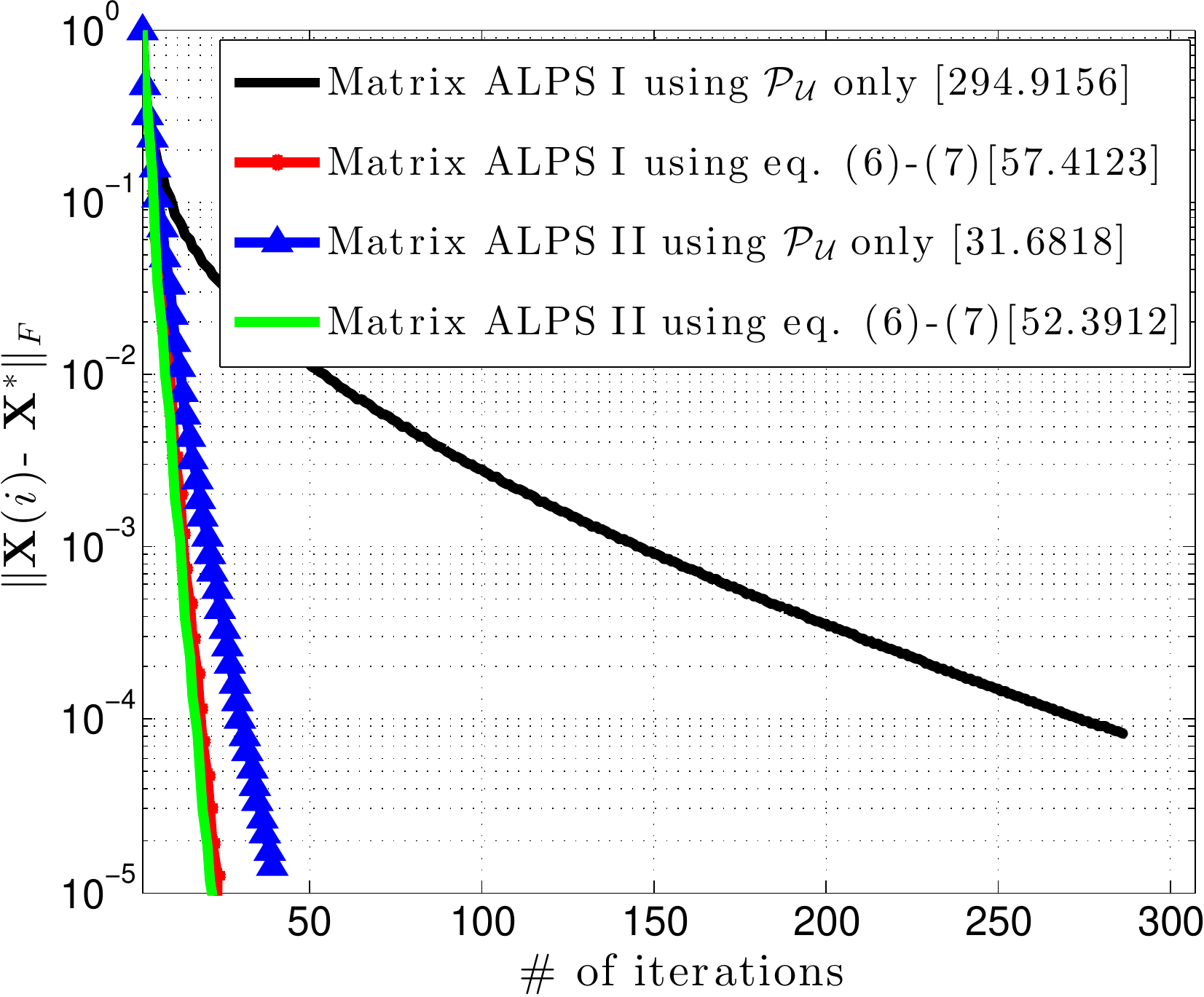}}}
\end{tabular}
\caption{\small\sl Median error per iteration for \textsc{Matrix ALPS I} and \textsc{Matrix ALPS II} variants over 10 Monte-Carlo repetitions. In brackets, we present the median time consumed for convergene in seconds. (a) $n = 2048, m = 512$, $p = 0.25n^2, $ and rank $ \rank = 40 $. (b) $n = 2000, m = 1000$, $p = 0.25n^2, $ and rank $ \rank = 50 $. (c) $ n = m = 1000 $, $p = 0.25n^2, $ and rank $ \rank = 50 $. }\label{proj_figure}
\end{figure*}

\subsection{Active low-rank subspace tracking}
Per iteration of Algorithms 1 and 2, we perform projection operations $\mathcal{P}_{\mathcal{S}}\signal$ and $\mathcal{P}_{\mathcal{S}^{\bot}}\signal$ where $\signal \in \mathbb{R}^{\dimension}$, as described by (\ref{eq:newproj}) and (\ref{eq:neworthoproj}), respectively. Since $\mathcal{S} $ is constituted by outer products of left and right singular vectors as in Definition \ref{def:svd_proj}, $\mathcal{P}_{\mathcal{S}}\signal $ (resp. $\mathcal{P}_{\mathcal{S}^{\bot}}\signal$) projects onto the (resp. complement of the) best low-rank subspace in $R(\signal)$ and $R(\signal^T)$. These operations are highly connected with the adaptive step size selection and the updates on restricted subspaces. Unfortunately, the time-complexity to compute $\mathcal{P}_{\mathcal{S}}\signal$ is dominated by three matrix-matrix multiplications which decelerates the convergence of the proposed schemes in high-dimensional settings. To accelerate the convergence in many test cases, it turns out that we do not have to use the best projection $\mathcal{P}_{\mathcal{S}}$ in practice.\footnote{From a different perspective and for a different problem case, similar ideas have been used in \cite{ALM}.} Rather, employing {\it inexact} projections is sufficient to converge to the optimal solution: either $i)$ $\mathcal{P}_{\mathcal{U}}\signal$ onto the best low-rank subspace in $R(\signal)$ only (if $m \ll n$) or $ii)$ $\signal\mathcal{P}_{\mathcal{V}}$ onto the best low-rank subspace in $R(\signal^T)$ only (if $m \gg n$)\footnote{We can move between these two cases by a simple transpose of the problem.}; $\mathcal{P}_{\mathcal{U}}$ and $\mathcal{P}_{\mathcal{V}}$ are defined in Definition \ref{def:svd_proj} and require only one matrix-matrix multiplication. 

Figure 2 shows the time overhead due to the exact projection application $\mathcal{P}_{\mathcal{S}}$ compared to $\mathcal{P}_{\mathcal{U}}$ for $m \leq n$. In Figure 2(a), we use subsampled and permuted noiselets for linear map $\sensing$ and in Figures 2(b)-(c), we test the MC problem. While in the case $m = n$ the use of (\ref{eq:newproj})-(\ref{eq:neworthoproj}) has a clear advantage over inexact projections using only $\mathcal{P}_{\mathcal{U}}$, the latter case converges faster to the desired accuracy $5\cdot 10^{-4}$ when $m \ll n$ as shown in Figures 2(a)-(b).
In our derivations, we assume $\mathcal{P}_{\mathcal{S}}$ and $\mathcal{P}_{\mathcal{S}^{\bot}}$ as defined in  (\ref{eq:newproj}) and (\ref{eq:neworthoproj}). 

\section{Convergence guarantees}{\label{section:convergence}}

In this section, we present the theoretical convergence guarantees of Algorithms 1 and 2 as functions of R-RIP constants. 
To characterize the performance of the proposed algorithms, both in terms of convergence rate and noise resilience, we use the following recursive expression:
\begin{align}
\vectornorm{\signal(i+1) - \xtrue}_F \leq \rho \vectornorm{\signal(i) - \xtrue}_F + \gamma \vectornorm{\noise}_2. \label{it:00}
\end{align} In (\ref{it:00}), $ \gamma $ denotes the approximation guarantee and provides insights into algorithm's reconstruction capabilities when additive noise is present; $ \rho < 1 $ expresses the convergence rate towards a region around $ \xtrue $, whose radius is determined by $ \frac{\gamma}{1-\rho} \vectornorm{\noise}_2 $. In short, (\ref{it:00}) characterizes how the distance to the true signal $ \bestsignal $ is decreased and how the noise level affects the accuracy of the solution, at each iteration.

\subsection{\textsc{Matrix ALPS I}}
An important lemma for our derivations below is given next:

\begin{lemma}\label{lemma:act_subspace_exp}[Active subspace expansion] Let $ \signal(i) $ be the matrix estimate at the $ i $-th iteration and let $ \mathcal{X}_i $ be a set of orthonormal, rank-1 matrices such that $ \mathcal{X}_i \leftarrow \mathcal{P}_{\rank}(\signal(i)) $.  Then, at each iteration, the Active Subspace Expansion step in Algorithms 1 and 2 identifies information in $ \bestsignal $, such that:
\begin{align}
\vectornormbig{\mathcal{P}_{\mathcal{X}^\ast} \mathcal{P}_{\mathcal{S}_i^{\bot}}\bestsignal}_F &\leq (2\delta_{2\rank} + 2\delta_{3\rank})\vectornormbig{\signal(i) - \bestsignal}_F \nonumber \\ &+ \sqrt{2(1+\delta_{2\rank})}\vectornormbig{\noise}_2, \label{eq:lemma6}
\end{align} where $ \mathcal{S}_i \leftarrow \mathcal{X}_i \cup \mathcal{D}_i $ and $ \mathcal{X}^\ast \leftarrow \mathcal{P}_{\rank}(\bestsignal) $.
\end{lemma}

Lemma \ref{lemma:act_subspace_exp} states that, at each iteration, the active subspace expansion step identifies a 2$ \rank $ rank subspace such that the amount of unrecovered energy of $ \bestsignal $---i.e., the projection of $ \bestsignal $ onto the orthogonal subspace of $ \text{span}(\mathcal{S}_i) $---is bounded by (\ref{eq:lemma6}).

Then, Theorem 1 characterizes the iteration invariant of Algorithm 1 for the matrix case:
\begin{theorem}\label{thm:mALPS0}[Iteration invariant for \textsc{Matrix ALPS I}] The $(i+1)$-th matrix estimate $\signal(i+1)$ of \textsc{Matrix ALPS I} satisfies the following recursion:
\begin{align}
\vectornormbig{\signal(i+1) - \bestsignal}_F &\leq \rho \vectornormbig{\signal(i) - \bestsignal}_F + \gamma \vectornormbig{\noise}_2, \label{eq:mALPS0:thm}
\end{align} where $ \rho:= \Big( \frac{1 + 2\delta_{2\rank}}{1-\delta_{2\rank}}\Big)\Big(\frac{4\delta_{2\rank}}{1-\delta_{2\rank}} + (2\delta_{2\rank} + 2\delta_{3\rank})\frac{2\delta_{3\rank}}{1-\delta_{2\rank}}\Big) $ and $\gamma := \Big( \frac{1 + 2\delta_{2\rank}}{1-\delta_{2\rank}}\Big)\Big(\frac{2\sqrt{1+\delta_{2\rank}}}{1 - \delta_{2\rank}} + \frac{2\delta_{3\rank}}{1-\delta_{2\rank}}  \sqrt{2(1+\delta_{2\rank})}\Big) + \frac{\sqrt{1+\delta_{\rank}}}{1-\delta_{\rank}}.$ Moreover, when $\delta_{3\rank} < 0.1235$, the iterations are contractive.
\end{theorem}

To provide some intuition behind this result, assume that $ \bestsignal $ is a rank-$ \rank $ matrix. Then, according to Theorem \ref{thm:mALPS0}, for $ \rho < 1 $, the approximation parameter $ \gamma $ in (\ref{eq:mALPS0:thm}) satisfies:
\begin{align}
\gamma < 5.7624, ~~\text{for}~~ \delta_{3\rank} < 0.1235. \nonumber 
\end{align} Moreover, we derive the following:
\begin{align}
\rho < \frac{1+2\delta_{3\rank}}{(1-\delta_{3\rank})^2} \big( 4\delta_{3\rank} + 8 \delta_{3\rank}^2\big) < \frac{1}{2} \Rightarrow \delta_{3\rank} < 0.079, \nonumber
\end{align} which is {\it a stronger} R-RIP condition assumption compared to state-of-the-art approaches \cite{admira2010}. In the next section, we further improve this guarantee using Algorithm 2.

Unfolding the recursive formula (\ref{eq:mALPS0:thm}), we obtain the following upper bound for $ \vectornormbig{\signal(i) - \bestsignal}_F $ at the $ i $-th iteration:
\begin{align}
\vectornormbig{\signal(i) - \bestsignal}_F \leq \rho^{i}\vectornormbig{\signal(0) - \bestsignal}_F + \frac{\gamma}{1-\rho}\vectornormbig{\noise}_2.
\end{align} Then, given $ \signal(0) = \mathbf{0} $, \textsc{Matrix ALPS I} finds a rank-$ \rank $ solution $ \widehat{\signal} \in \mathbb{R}^{\dimension} $ such that $ \vectornormbig{\widehat{\signal} - \bestsignal}_F \leq \frac{\gamma + 1 -\rho}{1-\rho}\vectornormbig{\noise}_2 $ after $ i:= \Big\lceil \frac{\log (\vectornorm{\bestsignal}_F/\vectornorm{\noise}_2)}{\log(1/\rho)} \Big\rceil $ iterations.

If we ignore steps 5 and 6 in Algorithm 1, we obtain another projected gradient descent variant for the affine rank minimization problem, for which we obtain the following performance guarantees---the proof follows from the proof of Theorem \ref{thm:mALPS0}.

\begin{corollary}{\label{cor:1}}[\textsc{Matrix ALPS I} Instance] 
In Algorithm 1, we ignore steps 5 and 6 and let $ \lbrace \mathcal{X}_{i+1},~\signal(i+1) \rbrace \leftarrow \mathcal{P}_\rank (\boldsymbol{V}_i) $. Then, by the same analysis, we observe that the following recursion is satisfied:
\begin{align}
\vectornormbig{\signal(i+1) - \bestsignal}_F &\leq \rho \vectornormbig{\signal(i) - \bestsignal}_F + \gamma \vectornormbig{\noise}_2,
\end{align} for $ \rho:= \Big(\frac{4\delta_{2\rank}}{1-\delta_{2\rank}} + (2\delta_{2\rank} + 2\delta_{3\rank})\frac{2\delta_{3\rank}}{1-\delta_{2\rank}}\Big) $ and $ \gamma := \Big(\frac{2\sqrt{1+\delta_{2\rank}}}{1 - \delta_{2\rank}} + \frac{2\delta_{3\rank}}{1-\delta_{2\rank}}  \sqrt{2(1+\delta_{2\rank})}\Big) $. Moreover, $ \rho < 1 $ when $ \delta_{3\rank} < 0.1594 $.
\end{corollary} 

We observe that the absence of the additional estimate update over restricted support sets results in less restrictive isometry constants compared to Theorem \ref{thm:mALPS0}. In practice, additional updates result in faster convergence, as shown in Figure 1(c).

\subsection{ADMiRA Instance}

In \textsc{Matrix ALPS I}, the gradient descent steps constitute a first-order approximation to least-squares minimization problems. Replacing Step 4 in Algorithm 1 with the following optimization problem:
\begin{align}
\boldsymbol{V}(i) \leftarrow \argmin_{\boldsymbol{V}: \boldsymbol{V} \in \text{span}(\mathcal{S}_i)} \vectornormbig{\obs - \sensing \boldsymbol{V}}_2^2, \label{eq:opt:1}
\end{align}
we obtain ADMiRA (furthermore, we remove the de-bias step in Algorithm 1). Assuming that the linear operator $ \sensing $, restricted on sufficiently low-rank subspaces, is well conditioned in terms of the R-RIP assumption, the optimization problem (\ref{eq:opt:1}) has a unique optimal minimizer. By exploiting the optimality condition in Lemma \ref{lemma:5}, ADMiRA instance in Algorithm 2 features the following guarantee:

\begin{theorem}\label{thm:mALPS5}[Iteration invariant for ADMiRA instance] The $(i+1)$-th matrix estimate $\signal(i+1)$ of ADMiRA answers the following recursive expression:
\begin{equation}\nonumber
\vectornormbig{\signal(i+1) - \bestsignal}_F \leq \rho \vectornormbig{\signal(i) - \bestsignal}_F + \gamma \vectornormbig{\noise}_F,
\end{equation} $\rho := \big(2\delta_{2\rank} + 2\delta_{3\rank}\big)\sqrt{\frac{1+3\delta_{3\rank}^2}{1-\delta_{3\rank}^2}}, $ and $\gamma := \sqrt{\frac{1+3\delta_{3\rank}^2}{1-\delta_{3\rank}^2}} \sqrt{2(1+\delta_{3\rank})} $ $+ \Big(\frac{\sqrt{1+3\delta_{3\rank}^2}}{1-\delta_{3\rank}} + \sqrt{3}\Big)\sqrt{1+\delta_{2\rank}}. $ Moreover, when $\delta_{3\rank} < 0.2267 $, the iterations are contractive.
\end{theorem}

Similarly to \textsc{Matrix ALPS I} analysis, the parameter $ \gamma $ in Theorem \ref{thm:mALPS5} satisfies:
\begin{align}
\gamma < 5.1848, ~\text{for}~ \delta_{3\rank} < 0.2267. \nonumber
\end{align} Furthermore, to compare the approximation guarantees of Theorem \ref{thm:mALPS5} with \cite{admira2010}, we further observe:
\begin{align}
\delta_{3\rank} < 0.1214, ~\text{for}~ \rho < 1/2. \nonumber
\end{align} We remind that \cite{admira2010} provides convergence guarantees for ADMiRA with $ \delta_{4\rank} < 0.04 $ for $ \rho = 1/2 $.

\section{Complexity Analysis}{\label{section:complexity}}

In each iteration, computational requirements of the proposed hard thresholding methods mainly depend on the total number of linear mapping operations $\sensing$, gradient descent steps, least-squares optimizations, projection operations and matrix decompositions for low rank approximation. Different algorithmic configurations (e.g. removing steps 6 and 7 in Algorithm 1) lead to hard thresholding variants with less computational complexity per iteration and better R-RIP conditions for convergence but a degraded performance in terms of stability and convergence speed is observed in practice. On the other hand, these additional processing steps increase the required time-complexity per iteration; hence, low iteration counts are desired to tradeoff these operations. 

A non-exhaustive list of linear map examples includes the identity operator (Principal component analysis (PCA) problem), Fourier/Wavelets/Noiselets tranformations and the famous Matrix Completion problem where $\sensing $ is a mask operator such that only a fraction of elements in $\signal $ is observed. Assuming the most demanding case where $ \sensing $ and $ \sensing^\ast $ are dense linear maps with no structure, the computation of the gradient $  \nabla f(\signal(i)) $ at each iteration requires $ O(\numsam \rank m n) $ arithmetic operations. 

Given a set $ \mathcal{S} $ of orthonormal, rank-1 matrices, the projection $ \mathcal{P}_{\mathcal{S}}\signal $ for any matrix $ \signal \in \mathbb{R}^{\dimension} $ requires time complexity $ O(\max\lbrace m^2 n,$ $ m n^2 \rbrace) $ as a sequence of matrix-matrix multiplication operations.\footnote{While such operation has $O(\max\lbrace m^2 n,$ $ m n^2 \rbrace)$ complexity, each application of $\mathcal{P}_{\mathcal{S}} \signal$ requires three matrix-matrix multiplications. To reduce such computational cost, we {\it relax} this operation in Section \ref{section:experiments} where in practice we use only $\mathcal{P}_{\mathcal{U}}$ that needs one matrix-matrix multiplication. } In \textsc{Matrix ALPS I}, the adaptive step size selection steps require $ O(\max\lbrace \numsam \rank m n, m^2 n\rbrace) $ time complexity for the calculation of $ \mu_i $ and $ \xi_i $ quantities. In ADMiRA solving a least-squares system restricted on rank-2$ \rank $ and rank-$ \rank  $ subspaces requires $ O(\numsam k^2) $ complexity; according to \cite{cosamp}, \cite{admira2010}, the complexity of this step can be further reduced using iterative techniques such as the Richardson method or conjugate gradients algorithm. 

Using the Lanczos method, we require $ O(\rank m n) $ arithmetic operations to compute a rank-$ \rank $ matrix approximation for a given constant accuracy; a prohibitive time-complexity that does not scale well for many practical applications. Sections \ref{section:approximate} and \ref{section:QR} describe approximate low rank matrix projections and how they affect the convergence guarantees of the proposed algorithms.

Overall, the operation that dominates per iteration requires $ O(\max\lbrace \numsam \rank m n, m^2 n ,m n^2\rbrace) $ time complexity in the proposed schemes. 

\section{Memory-based Acceleration}{\label{sec:memory}}
Iterative algorithms can use memory to gain momentum in convergence. Based on Nesterov's optimal gradient methods \cite{nesterov2007gradient}, we propose a hard thresholding variant, described in Algorithm 3 where an additional update on $\signal(i+1)$ with momentum step size $\tau_i$ is performed using previous matrix estimates.

\begin{algorithm*}[th!]
   \caption{\textsc{Matrix ALPS II}}\label{algo: class}
\begin{algorithmic}[1]
   \Statex {\bfseries Input:} $\obs$, $\sensing$, $\rank$, Tolerance $ \eta $, MaxIterations
   \Statex {\bfseries Initialize:} $ \signal(0) \leftarrow 0 $, $ \mathcal{X}_0 \leftarrow \lbrace \emptyset \rbrace $, $ \boldsymbol{Q}(0) \leftarrow 0 $, $ \mathcal{Q}_0 \leftarrow \lbrace \emptyset \rbrace $, $ \tau_i ~\forall i $, $ i \leftarrow 0 $
   \Statex {\bfseries repeat} 
   \State \hspace{0.16cm} $ \mathcal{D}_i \leftarrow \mathcal{P}_{\rank}\big( \mathcal{P}_{\mathcal{Q}_i^{\bot}} \nabla f(\boldsymbol{Q}(i)) \big) $ \hspace*{\fill}\textit{(Best rank-$ \rank $ subspace orthogonal to $ \mathcal{Q}_i $)~~~~~~~~~}
   \State \hspace{0.16cm} $ \mathcal{S}_i \leftarrow \mathcal{D}_i \cup \mathcal{Q}_i$ \hspace*{\fill}\textit{(Active subspace expansion)~~~~~~~~~}
   \State \hspace{0.16cm} $ \mu_i \leftarrow \argmin_{\mu} \vectornormbig{\obs - \sensing\big( \boldsymbol{Q}(i) - \frac{\mu}{2} \mathcal{P}_{\mathcal{S}_i} \nabla f(\boldsymbol{Q}(i)) \big)}_2^2 = \frac{\vectornorm{\mathcal{P}_{\mathcal{S}_i} \nabla f(\boldsymbol{Q}(i))}_F^2}{\vectornorm{\sensing \mathcal{P}_{\mathcal{S}_i} \nabla f(\boldsymbol{Q}(i))}_2^2} $ \hspace*{\fill} \textit{(Step size selection)~~~~~~~~~}
   \State \hspace{0.16cm} $ \boldsymbol{V}(i) \leftarrow \boldsymbol{Q}(i) - \frac{\mu_i}{2} \mathcal{P}_{\mathcal{S}_i}\nabla f(\boldsymbol{Q}(i)) $ \hspace*{\fill} \textit{(Error norm reduction via gradient descent)~~~~~~~~~}
   \State \hspace{0.16cm} $ \lbrace \mathcal{X}_{i+1},~\signal(i+1) \rbrace \leftarrow \mathcal{P}_{\rank}(\boldsymbol{V}(i)) $ \hspace*{\fill}{\textit{(Best rank-$ \rank $ subspace selection)~~~~~~~~~}}
   \State \hspace{0.16cm} $ \boldsymbol{Q}(i+1) \leftarrow \signal(i+1) + \tau_i(\signal(i+1) - \signal(i)) $ \hspace*{\fill}\textit{(Momentum update)~~~~~~~~~}
   \State \hspace{0.16cm} $\mathcal{Q}_{i+1} \leftarrow \text{ortho}(\mathcal{X}_i \cup \mathcal{X}_{i+1})$
   \Statex \hspace{0.16cm} $ i \leftarrow i + 1 $
   \Statex {\bfseries until} $\vectornorm{\signal(i) - \signal(i-1)}_2 \leq \eta \vectornorm{\signal(i)}_2 $ or MaxIterations.
\end{algorithmic}
\end{algorithm*}

Similar to $ \mu_i $ strategies, $ \tau_i $ can be preset as constant or adaptively computed at each iteration. Constant momentum step size selection has no additional computational cost but convergence rate acceleration is not guaranteed for some problem formulations in practice. On the other hand, empirical evidence has shown that adaptive $ \tau_i $ selection strategies result in faster convergence compared to zero-memory methods with {\it similar complexity}. 

For the case of strongly convex objective functions, Nesterov \cite{nesterov} proposed the following constant momentum step size selection scheme:
$\tau_i = \frac{\alpha_i(1-\alpha_i)}{\alpha_i^2 + \alpha_{i+1}} $, 
where $ \alpha_0 \in (0,1) $ and $ \alpha_{i+1} $ is computed as the root $ \in (0,1) $ of
\begin{align}
\alpha_{i+1}^2 = (1 - \alpha_{i+1}) \alpha_i^2 + q\alpha_{i+1},  ~\text{for}~~q \triangleq \frac{1}{\kappa^2(\sensing)}, 
\end{align} where $ \kappa(\sensing) $ denotes the condition number of $ \sensing $. In this scheme, exact calculation of $ q $ parameter is computationally expensive for large-scale data problems and approximation schemes are leveraged to compensate this complexity bottleneck.

Based upon adaptive $ \mu_i $ selection, we propose to select $ \tau_i $ as the minimizer of the objective function:
\begin{align}
\tau_i &= \argmin_{\tau} \vectornorm{\obs - \sensing \boldsymbol{Q}(i+1)}_2^2  \nonumber \\ 
       &= \frac{\langle \obs - \sensing \signal(i), \sensing \signal(i) - \sensing \signal(i-1)\rangle }{\vectornorm{\sensing \signal(i) - \sensing \signal(i-1)}_2^2}, \label{tau_optimized:00}
\end{align} where $ \sensing \signal(i), \sensing \signal(i-1) $ are already {\it pre-computed} at each iteration. According to (\ref{tau_optimized:00}), $ \tau_i $ is dominated by the calculation of a vector inner product, a computationally cheaper process than $ q $ calculation. 

Theorem \ref{thm:mALPS0:memory} characterizes Algorithm 3 for {\it constant} momentum step size selection. To keep the main ideas simple, we ignore the additional gradient updates in Algorithm 3. In addition, we only consider the noiseless case for clarity. The convergence rate proof for these cases is provided in the appendix.

\begin{theorem}\label{thm:mALPS0:memory}[Iteration invariant for \textsc{Matrix ALPS II}] Let $ \obs = \sensing \bestsignal $ be a noiseless set of observations. To recover $ \bestsignal $ from $ \obs $ and $ \sensing $, the $(i+1)$-th matrix estimate $\signal(i+1)$ of \textsc{Matrix ALPS II} satisfies the following recursion:
\begin{align}
\vectornormbig{\signal(i+1) - \bestsignal}_F &\leq \alpha(1+\tau_i) \vectornormbig{\signal(i) - \bestsignal}_F \nonumber \\ &+ \alpha \tau_i \vectornormbig{\signal(i-1) - \bestsignal}_F, \label{eq:mALPS0_memory:thm}
\end{align} where $ \alpha:= \frac{4\delta_{3\rank}}{1-\delta_{3\rank}} + (2\delta_{3\rank} + 2\delta_{4\rank})\frac{2\delta_{3\rank}}{1-\delta_{3\rank}} $. Moreover, solving the above second-order recurrence, the following inequality holds true:
\begin{align}
\vectornormbig{\signal(i+1) - \bestsignal}_F \leq \rho^{i+1} \vectornormbig{\signal(0) - \bestsignal}_F, \label{eq:mALPS0_memory_result}
\end{align} for $\rho:= \frac{\alpha(1+\tau_i) + \sqrt{\alpha^2(1+\tau_i)^2 + 4\alpha\tau_i}}{2}$.
\end{theorem}

\begin{figure}[!t]
\centering
\begin{tabular}{cc}
\centerline{\subfigure{\includegraphics[width = 0.38\textwidth]{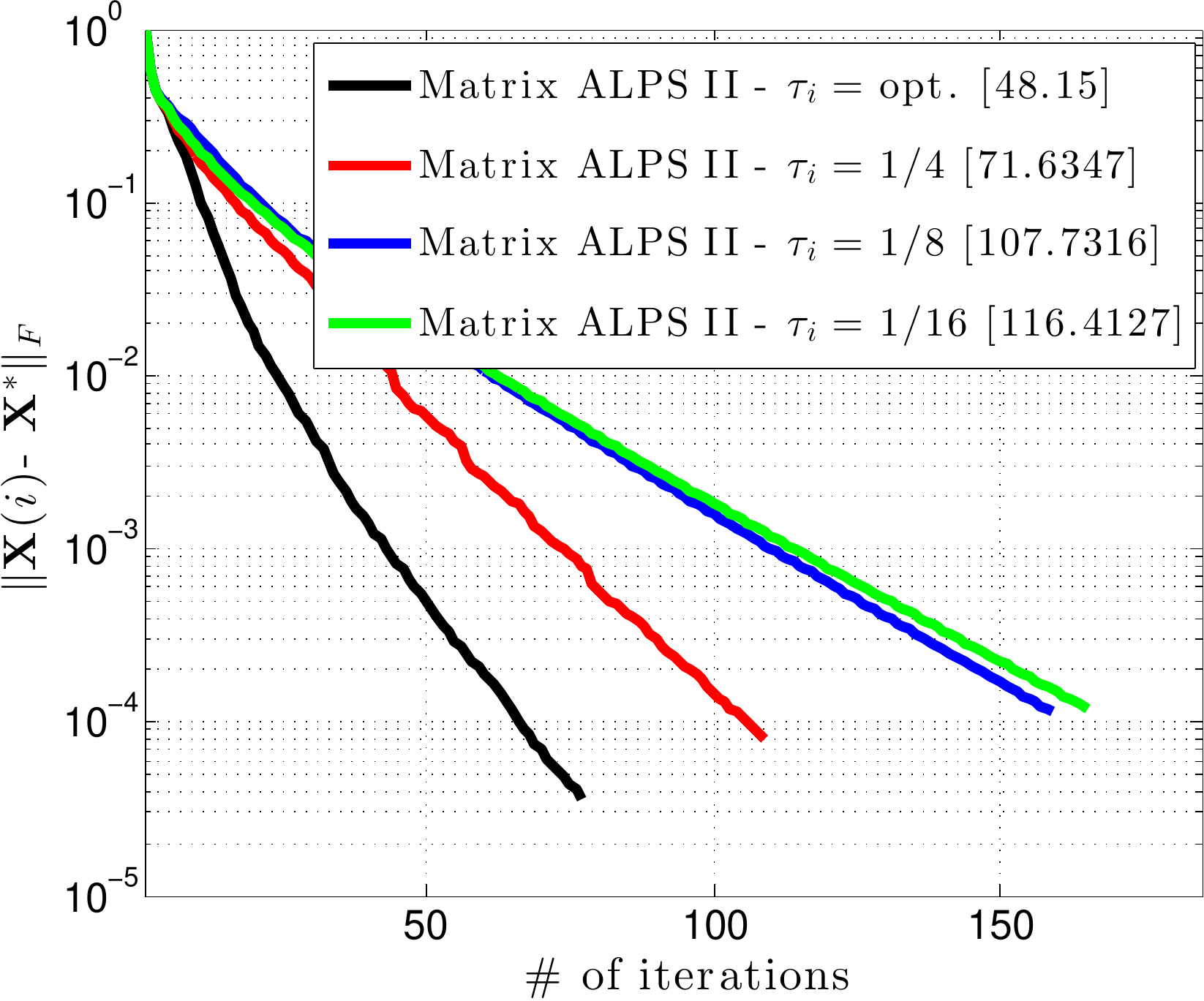}\label{fig:3b}}}
\end{tabular}
\caption{\small\sl Median error per iteration for various momentum step size policies and 10 Monte-Carlo repetitions. Here, $n = 1024$, $ m = 256 $, $p = 0.25n^2, $ and rank $ \rank = 40 $. We use permuted and subsampled noiselets for the linear map $\sensing $. In brackets, we present the median time for convergence in seconds. 
}\label{fig: lambda}
\end{figure}

Theorem \ref{thm:mALPS0:memory} provides convergence rate behaviour proof for the case where $ \tau_i $ is constant $\forall i $. The more elaborate case where $\tau_i$ follows the policy described in (\ref{tau_optimized:00}) is left as an open question for future work. To provide some insight for (\ref{eq:mALPS0_memory_result}), for $\tau_i = 1/4, ~\forall i$ and $\tau_i = 1/2, ~\forall i$, $\delta_{4\rank} < 0.1187 $ and $\delta_{4\rank} < 0.095 $ guarantee convergence in Algorithm 3, respectively. While the RIP requirements for memory-based \textsc{Matrix ALPS II} are more stringent than the schemes proposed in the previous section, it outperforms Algorithms 1 and 2. Figure 2 shows the acceleration achieved in \textsc{Matrix ALPS II} by using inexact projections $\mathcal{P}_{\mathcal{U}}$. Using the proper projections (\ref{eq:newproj})-(\ref{eq:neworthoproj}), Figure \ref{fig: lambda} shows acceleration in practice when using the adaptive momentum step size strategy: while a wide range of constant momentum step sizes leads to convergence, providing flexibility to select an appropriate $\tau_i$, adaptive $\tau_i$ avoids this arbitrary $\tau_i$ selection while further decreases the number of iterations needed for convergence in most cases.  

\section{Accelerating \textsc{Matrix ALPS}: $ \epsilon $-Approximation of SVD via Column Subset Selection}{\label{section:approximate}}

A time-complexity bottleneck in the proposed schemes is the computation of the singular value decomposition to find subspaces that  describe the unexplored information in matrix $\bestsignal $. Unfortunately, the computational cost of regular SVD for best subspace tracking is prohibitive for many applications. 

Based on \cite{drineas1, drineas2},  we can obtain randomized SVD approximations of a matrix $\signal $ using {\it column subset selection} ideas: we compute a leverage score for each column that represents its ``significance''. In particular, we define a probability distribution that weights each column depending on the amount of information they contain; usually, the distribution is related to the $\ell_2$-norm of the columns. The main idea of this approach is to compute a surrogate rank-$\rank$ matrix $\mathcal{P}_{\rank}^{\epsilon}(\signal) $ by subsampling the columns according to this distribution. It turns out that the total number of sampled columns is a function of the parameter $\epsilon$.  Moreover, \cite{deshpande1, deshpande2} proved that, given a target rank $\rank$ and an approximation parameter $\epsilon$, we can compute an $\epsilon$-approximate rank-$\rank$ matrix $\mathcal{P}_{\rank}^\epsilon(\signal) $ according to the following defintion.

\begin{definition}{\label{def:appr_svd}}[$ \epsilon $-approximate low-rank projection]
Let $ \signal $ be an arbitrary matrix. Then, $ \mathcal{P}_{\rank}^{\epsilon}(\signal) $ projection provides a rank-$ \rank $ matrix approximation to $ \signal $ such that:
\begin{align}
\vectornormbig{\mathcal{P}_{\rank}^{\epsilon}(\signal) - \boldsymbol{X}}_F^2 \leq (1 + \epsilon) \vectornormbig{\mathcal{P}_{\rank}(\signal) - \boldsymbol{X}}_F^2, \label{eq:appr_svd:00}
\end{align} where $ \mathcal{P}_{\rank}(\signal) \in \argmin_{\boldsymbol{Y}: \text{rank}(\boldsymbol{Y}) \leq \rank} \vectornorm{\signal - \boldsymbol{Y}}_F $.
\end{definition}

For the following theoretical results, we assume the following condition on the sensing operator $\sensing:$ $\vectornormbig{\sensing^{\ast} \boldsymbol{\beta}}_F \leq \lambda, ~\forall \boldsymbol{\beta} \in \mathbb{R}^{p}$ where $\lambda > 0$. 
Using $ \epsilon $-approximation schemes to perform the Active subspace selection step, the following upper bound holds. The proof is provided in the Appendix:

\begin{lemma}\label{lemma:appr_act_subspace_exp}[$ \epsilon $-approximate active subspace expansion] Let $ \signal(i) $ be the matrix estimate at the $ i $-th iteration and let $ \mathcal{X}_i $ be a set of orthonormal, rank-1 matrices in $\mathbb{R}^{\dimension}$ such that $ \mathcal{X}_i \leftarrow \mathcal{P}_{\rank}(\signal(i)) $. Furthermore, let 
\begin{align}
\mathcal{D}_i^{\epsilon} \leftarrow \mathcal{P}_{\rank}^{\epsilon}\big( \mathcal{P}_{\mathcal{X}_i^{\bot}} \nabla f(\signal(i)) \big), \nonumber 
\end{align} be a set of orthonormal, rank-1 matrices that span rank-$ \rank $ subspace such that (\ref{eq:appr_svd:00}) is satisfied for $ \boldsymbol{X} := \mathcal{P}_{\mathcal{X}_i^{\bot}} \nabla f(\signal(i)) $. Then, at each iteration, the Active Subspace Expansion step in Algorithms 1 and 2 captures information contained in the true matrix $ \bestsignal $, such that:
\begin{align}
\vectornormbig{&\mathcal{P}_{\mathcal{X}^\ast} \mathcal{P}_{\mathcal{S}_i^{\bot}}\bestsignal}_F \nonumber \\
&\leq \big(2\delta_{2\rank} + 2\delta_{3\rank}\big)\vectornormbig{\signal(i) - \bestsignal}_F + \sqrt{2(1+\delta_{2\rank})}\vectornormbig{\noise}_2 \nonumber \\ &+ 2\lambda \sqrt{\epsilon}, \label{eq:lemma6_appr}
\end{align} where $ \mathcal{S}_i \leftarrow \mathcal{X}_i \cup \mathcal{D}_i^{\epsilon} $ and $ \mathcal{X}^\ast \leftarrow \mathcal{P}_{\rank}(\bestsignal) $. 
\end{lemma}

Furthermore, to prove the following theorems, we extend Lemma \ref{lemma:comb_selection}, provided in the Appendix, as follows. The proof easily follows from the proof of Lemma \ref{lemma:comb_selection}, using Definition \ref{def:appr_svd}:
\begin{lemma}{\label{lemma:appr_comb_selection}}[$ \epsilon $-approximation rank-$ \rank $ subspace selection] Let $ \boldsymbol{V}(i) $ be a rank-$ 2\rank $ proxy matrix in the subspace spanned by $ \mathcal{S}_i $ and let $ \widehat{\boldsymbol{W}}(i) \leftarrow \mathcal{P}_{\rank}^{\epsilon}(\boldsymbol{V}(i)) $ denote the rank-$ \rank $ $ \epsilon $-approxi- mation to $ \boldsymbol{V}(i) $, according to (\ref{eq:svd_proj}). Then:
\begin{align}
\vectornormbig{\widehat{\boldsymbol{W}}(i)  - \boldsymbol{V}(i)}_F^2 &\leq (1 + \epsilon) \vectornormbig{\boldsymbol{W}(i)  - \boldsymbol{V}(i)}_F \nonumber \\ &\leq (1+\epsilon) \vectornormbig{\mathcal{P}_{\mathcal{S}_i}(\boldsymbol{V}(i) - \bestsignal)}_F \nonumber \\ &\leq (1+\epsilon)\vectornormbig{\boldsymbol{V}(i) - \bestsignal}_F \label{eq:mALPS5:13_appr}
\end{align} where $ \boldsymbol{W}(i) \leftarrow \mathcal{P}_{\rank}(\boldsymbol{V}(i)) $.
\end{lemma}

\subsection{\textsc{Matrix ALPS I} using $ \epsilon $-approximate low-rank projection via column subset selection}

Using $ \epsilon $-approximate SVD in \textsc{Matrix ALPS I}, the following iteration invariant theorem holds:

\begin{theorem}\label{thm:mALPS0_appr}[Iteration invariant with $ \epsilon $-approximate projections for \textsc{Matrix ALPS I}] The $(i+1)$-th matrix estimate $\signal(i+1)$ of \textsc{Matrix ALPS I} with $ \epsilon $-approximate projections $ \mathcal{D}_i^\epsilon \leftarrow \mathcal{P}_{\rank}^{\epsilon}\big( \mathcal{P}_{\mathcal{X}_i^{\bot}} \nabla f(\signal(i)) \big) $ and $ \widehat{\boldsymbol{W}}(i) \leftarrow \mathcal{P}_{\rank}^{\epsilon}(\boldsymbol{V}(i)) $ in Algorithm 1 satisfies the following recursion:
\begin{align}
\vectornormbig{\signal(i+1) - \bestsignal}_F \leq \rho\vectornormbig{\signal(i) - \bestsignal}_F 
+\gamma \vectornorm{\noise}_2 + \beta \lambda, \label{eq:mALPS0:thm_appr}
\end{align}
where $\rho:= \left(1 + \frac{3\delta_{\rank}}{1-\delta_{\rank}}\right)\left(2 + \epsilon\right)\big [ ( 1+ \frac{\delta_{3\rank}}{1-\delta_{2\rank}})4\delta_{3\rank} + \frac{2\delta_{2\rank}}{1-\delta_{2\rank}}\big], $ $\beta := \left(1 + \frac{3\delta_{\rank}}{1-\delta_{\rank}}\right)\left(2 + \epsilon\right)\left(1+ \frac{\delta_{3\rank}}{1-\delta_{2\rank}} \right)2\sqrt{\epsilon}, $ and \\ $
\gamma := \left(1 + \frac{3\delta_{\rank}}{1-\delta_{\rank}}\right)\big(2 + \epsilon\big)\Big[ \big(1+ \frac{\delta_{3\rank}}{1-\delta_{2\rank}}\big)\sqrt{2(1+\delta_{2\rank})} + $ \\ $ 2\frac{\sqrt{1+\delta_{2\rank}}}{1-\delta_{2\rank}}\Big]. $
\end{theorem}

\begin{figure}[!ht]
\centering
\subfigure[]{\includegraphics[width = 0.40\textwidth]{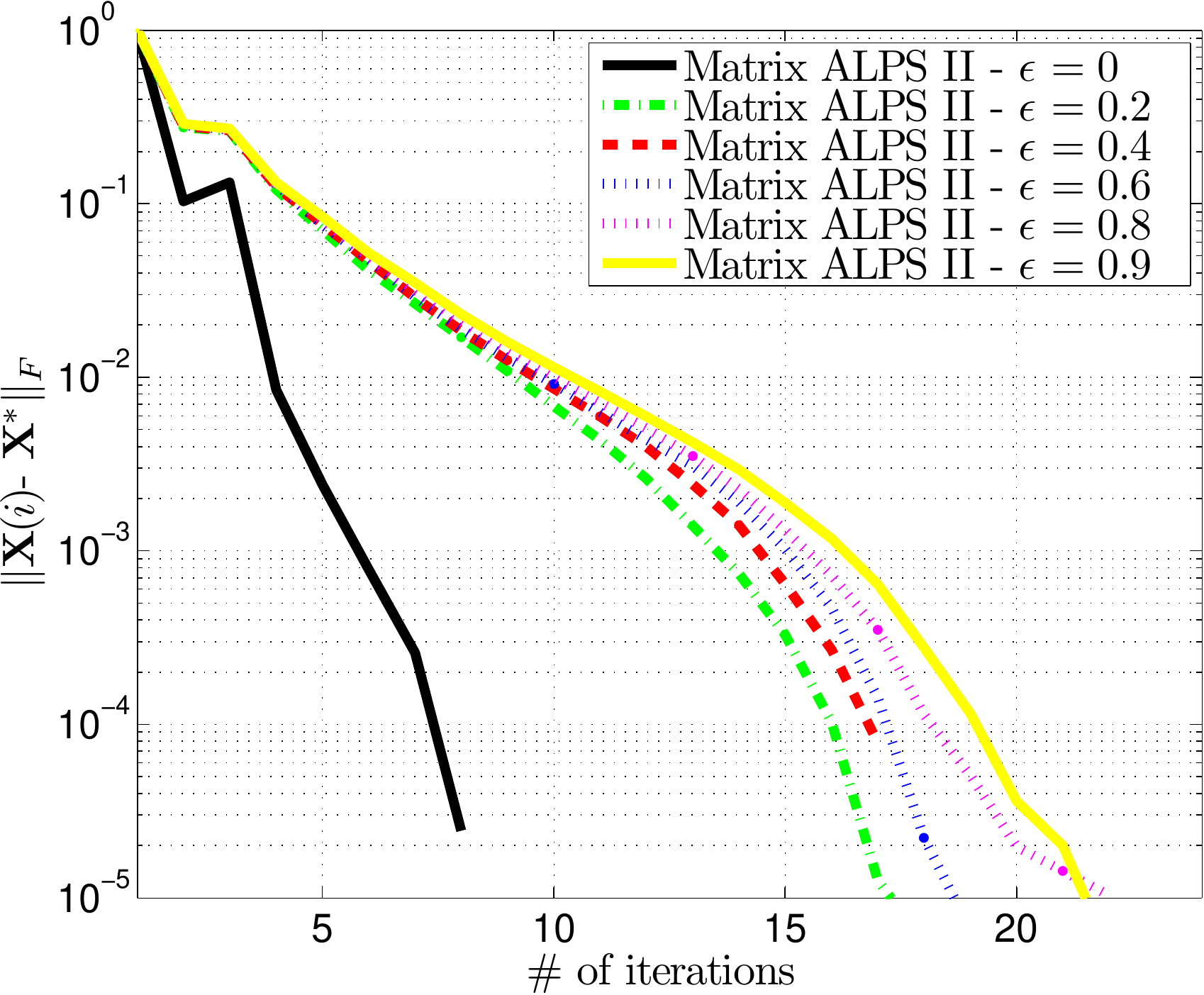}\label{fig:3cc}}
\caption{\small\sl Performance comparison using $\epsilon$-approximation SVD \cite{deshpande2} in \textsc{Matrix ALPS II}. $m = n = 256$, $p = 0.4n^2$, rank of $\bestsignal$ equals $2$ and $\sensing $ constituted by permuted noiselets. The non-smoothness in the error curves is due to the extreme low rankness of $\bestsignal$ for this setting. } \label{fig: lambda1}
\end{figure}

Similar analysis can be conducted for the ADMiRA algorithm. To illustrate the impact of SVD $\epsilon$-approximation on the signal reconstruction performance of the proposed methods, we replace the {\it best} rank-$\rank$ projections in steps 1 and 5 of Algorithm 1 by the $\epsilon$-approximation SVD algorithm, presented in \cite{deshpande2}. In this paper, the column subset selection algorithm satisfies the following theorem:

\begin{theorem}{\label{thm:adaptiveVolume}} 
Let $\signal \in \mathbb{R}^{\dimension} $ be a signal of interest with arbitrary $ \text{rank} < \min\lbrace m, n \rbrace $ and let $\signal_{\rank}$ represent the {\it best} rank-$\rank$ approximation of $\signal $. After $2 (\rank + 1)(\log(\rank + 1) + 1) $ passes over the data, the Linear Time Low-Rank Matrix Approximation algorithm in \cite{deshpande2} computes a rank-$\rank$ approximation $\mathcal{P}_{\rank}^{\epsilon}(\signal) \in \mathbb{R}^{\dimension} $ such that Definition \ref{def:appr_svd} is satisfied with probability at least 3/4.
\end{theorem} 

The proof is provided in \cite{deshpande2}. In total, Linear Time Low-Rank Matrix Approximation algorithm \cite{deshpande2} requires $O(mn $ $(\rank/ \epsilon + \rank^2 \log \rank) + (m + n)(\rank^2/\epsilon^2 + \rank^3 \log \rank/\epsilon $ $+ \rank^4 \log^2 \rank))$ and $O(\min\lbrace m, n \rbrace (\rank/\epsilon + \rank^2 \log \rank)) $ time and space complexity, respectively. However, while column subset selection methods such as \cite{deshpande2} reduce the overall complexity of low-rank projections in theory, in practice this applies only in very high-dimensional settings. To strengthen this argument, in Figure \ref{fig: lambda1} we compare SVD-based \textsc{Matrix ALPS II}	with \textsc{Matrix ALPS II} using the $\epsilon$-approximate column subset selection method in \cite{deshpande2}. We observe that the total number of iterations for convergence increases due to $\epsilon$-approximate low-rank projections, as expected. Nevertheless, we observe that, on average, the column subset selection process \cite{deshpande2} is computationally prohibitive compared to regular SVD due to the time overhead in the column selection procedure---fewer passes over the data are desirable in practice to tradeoff the increased number of iterations for convergence. In the next section, we present alternatives based on recent trends in randomized matrix decompositions and how we can use them in low-rank recovery.

\section{Accelerating \textsc{Matrix ALPS}: SVD Approximation using Randomized Matrix Decompositions }{\label{section:QR}}

\begin{algorithm*}[Htp!]
   \caption{Randomized \textsc{Matrix ALPS II} with QR Factorization}\label{algo: class}
\begin{algorithmic}[1]
   \Statex {\bfseries Input:} $\obs$, $\sensing$, $\rank$, $q$, Tolerance $ \eta $, MaxIterations
   \Statex {\bfseries Initialize:} $ \signal(0) \leftarrow 0 $, $ \mathcal{X}_0 \leftarrow \lbrace \emptyset \rbrace $, $ \boldsymbol{Q}(0) \leftarrow 0 $, $ \mathcal{Q}_0 \leftarrow \lbrace \emptyset \rbrace $, $ \tau_i ~\forall i $, $ i \leftarrow 0 $
   \Statex {\bfseries repeat} 
   \State \hspace{0.16cm} $ \mathcal{D}_i \leftarrow $ \textsc{RandomizedPowerIteration}$ \big(\mathcal{P}_{\mathcal{Q}_i^{\bot}} \nabla f(\mathbf{Q}(i)), ~\rank, ~q\big) $ \hspace*{\fill}\textit{(Rank-$ \rank $ subspace via Randomized Power Iteration)~~~~~~~~~}
   \State \hspace{0.16cm} $ \mathcal{S}_i \leftarrow \mathcal{D}_i \cup \mathcal{Q}_i$ \hspace*{\fill}\textit{(Active subspace expansion)~~~~~~~~~}
   \State \hspace{0.16cm} $ \mu_i \leftarrow \argmin_{\mu} \vectornormbig{\obs - \sensing\big( \boldsymbol{Q}(i) - \frac{\mu}{2} \mathcal{P}_{\mathcal{S}_i} \nabla f(\boldsymbol{Q}(i)) \big)}_2^2 = \frac{\vectornorm{\mathcal{P}_{\mathcal{S}_i} \nabla f(\boldsymbol{Q}(i))}_F^2}{\vectornorm{\sensing \mathcal{P}_{\mathcal{S}_i} \nabla f(\boldsymbol{Q}(i))}_2^2} $ \hspace*{\fill} \textit{(Step size selection)~~~~~~~~~}
   \State \hspace{0.16cm} $ \boldsymbol{V}(i) \leftarrow \boldsymbol{Q}(i) - \frac{\mu_i}{2} \mathcal{P}_{\mathcal{S}_i}\nabla f(\boldsymbol{Q}(i)) $ \hspace*{\fill} \textit{(Error norm reduction via gradient descent)~~~~~~~~~}
   \State \hspace{0.16cm} $\mathcal{W} \leftarrow $ \textsc{RandomizedPowerIteration}$ \big(\mathbf{V}(i), ~\rank, ~q\big) $  \hspace*{\fill}\textit{(Rank-$ \rank $ subspace via Randomized Power Iteration)~~~~~~~~~}
   \State \hspace{0.16cm} $ \signal(i+1) \leftarrow \mathcal{P}_{\mathcal{W}} \mathbf{V}(i)  $ \hspace*{\fill}{\textit{(Best rank-$ \rank $ subspace selection)~~~~~~~~~}}
   \State \hspace{0.16cm} $ \boldsymbol{Q}(i+1) \leftarrow \signal(i+1) + \tau_i(\signal(i+1) - \signal(i)) $ \hspace*{\fill}\textit{(Momentum update)~~~~~~~~~}
   \State \hspace{0.16cm} $\mathcal{Q}_{i+1} \leftarrow \text{ortho}(\mathcal{X}_i \cup \mathcal{X}_{i+1})$
   \Statex \hspace{0.16cm} $ i \leftarrow i + 1 $
   \Statex {\bfseries until} $\vectornorm{\signal(i) - \signal(i-1)}_2 \leq \eta \vectornorm{\signal(i)}_2 $ or MaxIterations.
\end{algorithmic}
\end{algorithm*} 

Finding low-cost SVD approximations to tackle the above complexity issues is a challenging task. Recent works on probabilistic methods for matrix approximation \cite{findingstructure} provide a family of efficient approximate projections on the set of rank-deficient matrices with clear computational advantages over regular SVD computation in practice and attractive theoretical guarantees. In this work, we build on the low-cost, power-iteration {\it subspace tracking} scheme, described in Algorithms 4.3 and 4.4 in \cite{findingstructure}. Our proposed algorithm is described in Algorithm 4.

The convergence guarantees of Algorithm 4 follow the same motions described in Section \ref{section:approximate}, where $\epsilon $ is a function of $m, ~n, ~\rank $ and $ q $. 

\section{Experiments}{\label{section:experiments}}

\subsection{List of algorithms}
In the following experiments, we compare the following algorithms: $(i)$ the Singular Value Projection (SVP) algorithm \cite{SVP}, a non-convex first-order projected gradient descent algorithm with {\it constant} step size selection (we study the case where $\mu = 1 $), $(ii)$ the inexact ALM algorithm \cite{ALM} based on augmented Langrance multiplier method, $(iii)$ the OptSpace algorithm \cite{OptSpace}, a gradient descent algorithm on the Grassmann manifold, $(iv)$ the Grassmannian Rank-One Update Subspace Estimation (GROUSE) and the Grassmannian Robust Adaptive Subspace Tracking methods (GRASTA) \cite{GROUSE, GRASTA}, two stochastic gradient descent algorithms that operate on the Grassmannian---moreover, to allay the impact of outliers in the subspace selection step, GRASTA incorporates the augmented Lagrangian of $\ell_1$-norm loss function into the Grassmannian optimization framework, $(v)$ the Riemannian Trust Region Matrix Completion algorithm (RTRMC) \cite{RTRMC}, a matrix completion method using first- and second-order Riemannian trust-region approaches, $(vi)$ the Low rank Matrix Fitting algorithm (LMatFit) \cite{LMatFit}, a nonlinear successive over-relaxation algorithm and $(vii)$ the algorithms \textsc{Matrix ALPS I}, ADMiRA \cite{admira2010}, \textsc{Matrix ALPS II} and Randomized \textsc{Matrix ALPS II} with QR Factorization (referred shortly as \textsc{Matrix ALPS II} with QR) presented in this paper.

\begin{figure*}[!t]
\hspace{-0.2cm}\centering
\begin{tabular}{ccc}
\centerline{\subfigure[]{\includegraphics[width = 0.3\textwidth]{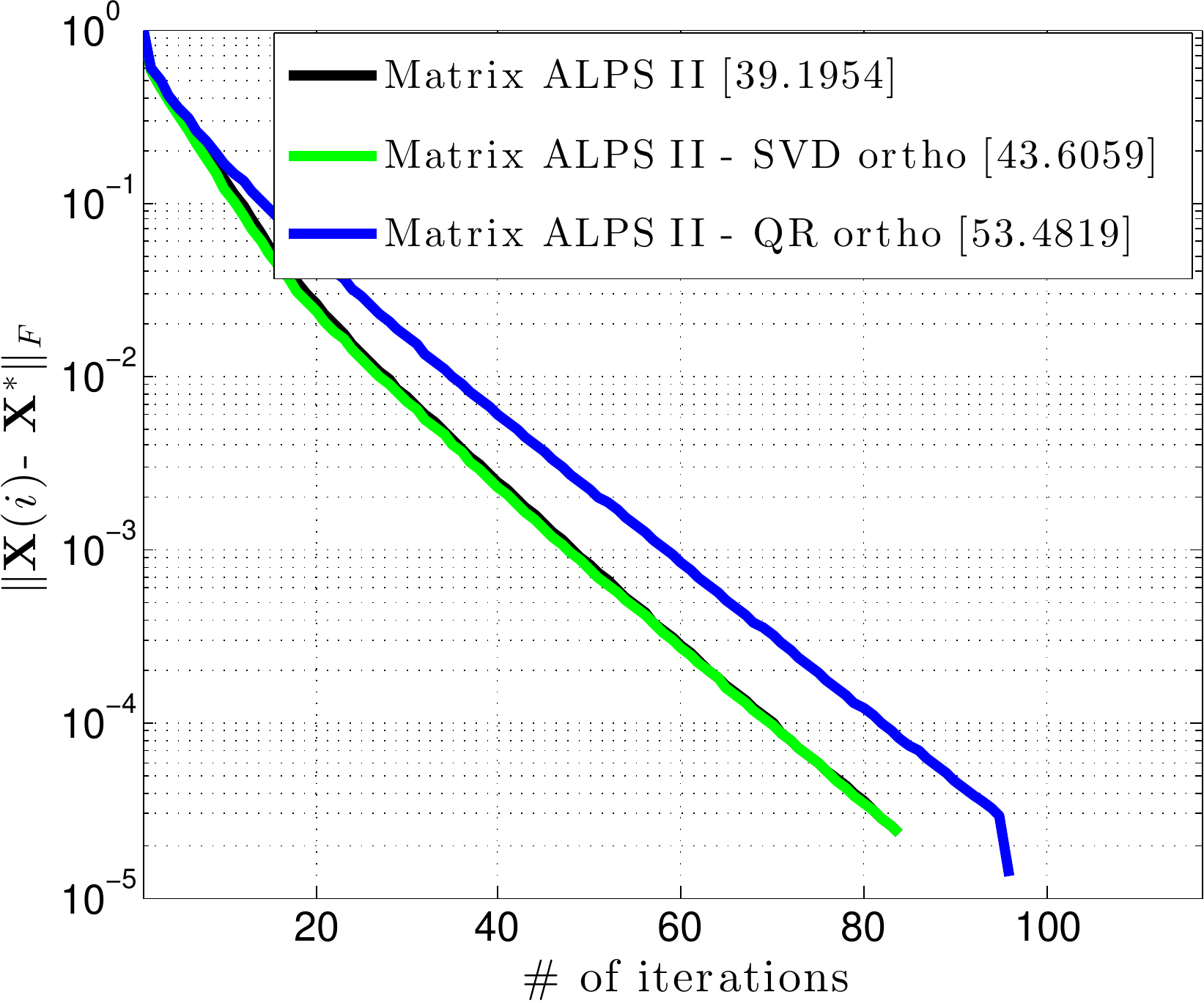}} 
\hfill
\subfigure[]{\includegraphics[width = 0.3\textwidth]{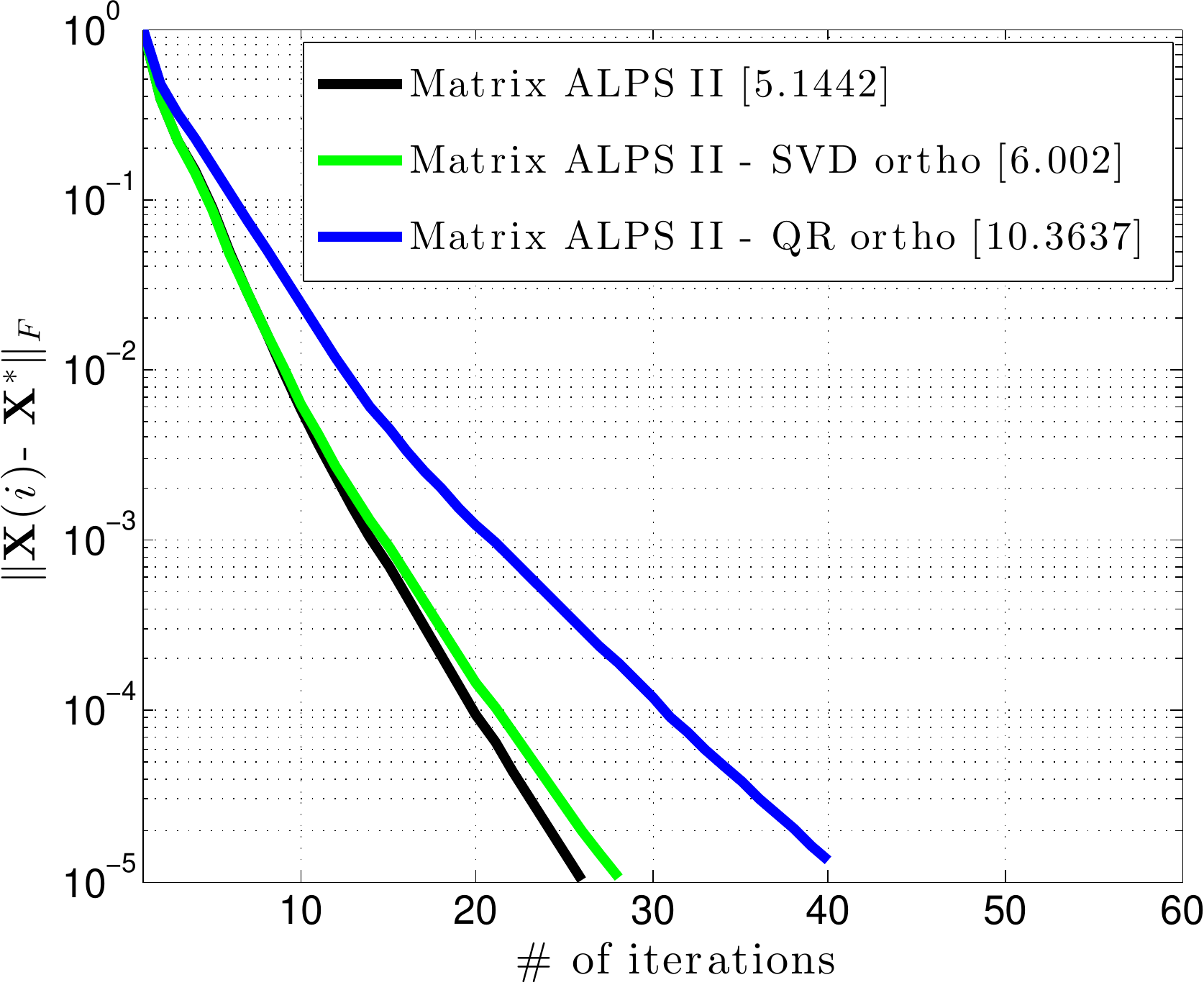}}
\hfill
\subfigure[]{\includegraphics[width = 0.3\textwidth]{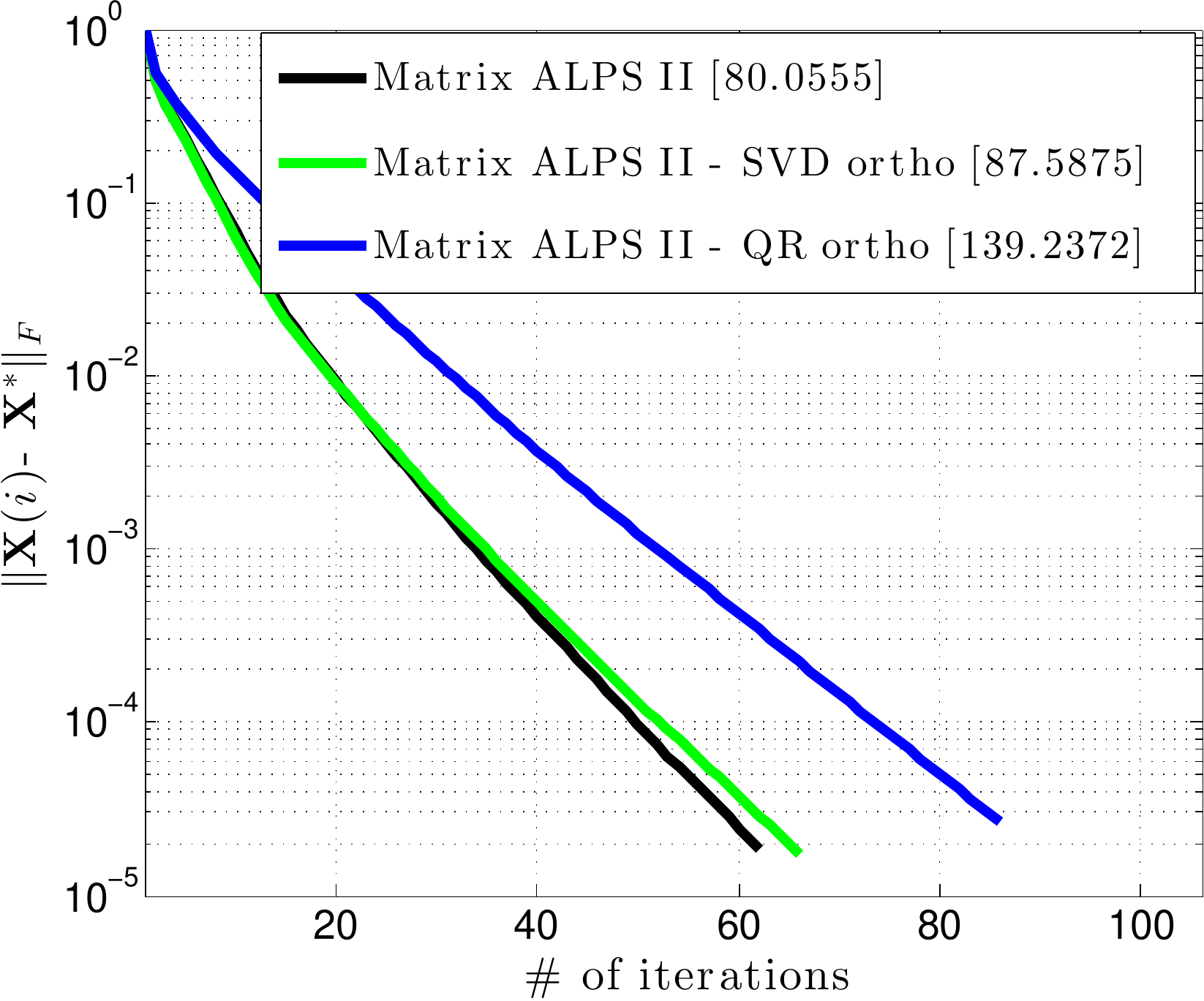}}}
\end{tabular}
\caption{\small\sl Median error per iteration for \textsc{Matrix ALPS II} variants over 10 Monte-Carlo repetitions. In brackets, we present the mean time consumed for convergene in seconds. (a) $n = 1024, m = 256$, $p = 0.25n^2, $ and rank $ \rank = 20 $. (b) $n = 2048, m = 512$, $p = 0.25n^2, $ and rank $ \rank = 60 $. (c) $ n = 1000 $, $ m = 500 $, $p = 0.25n^2, $ and rank $ \rank = 50 $. }\label{ortho_figure}
\end{figure*}

\subsection{Implementation details}
To properly compare the algorithms in the above list, we preset a set of parameters that are common. We denote the ratio between the number of observed samples and the number of variables in $\bestsignal$ as $\text{SR}:=\numsam/(m \cdot n)$ (sampling ratio). Furthemore, we reserve $\text{FR}$ to represent the degree of freedom in a rank-$\rank$ matrix to the number of observations---this corresponds to the following definition $\text{FR}:=(\rank(m + n -\rank))/\numsam$. In most of the experiments, we fix the number of observable data $\numsam = 0.3mn$ and vary the dimensions and the rank $\rank$ of the matrix $\bestsignal$. This way, we create a wide range of different problem configurations with variable $\text{FR}$. 

Most of the algorithms in comparison as well as the proposed schemes are implemented in \textsc{Matlab}. We note that the LMaFit software package contains parts implemented in C that reduce the per iteration computational time. This provides insights for further time savings in our schemes; we leave a fully optimized implementation of our algorithms as future work. In this paper, we mostly test cases where $m \ll n$. Such settings can be easily found in real-world problems such as recommender systems (e.g. Netflix, Amazon, etc.) where the number of products, movies, etc. is much greater than the number of active users.

In all algorithms, we fix the maximum number of iterations to 500, unless otherwise stated. To solve a least squares problem over a restricted low-rank subspace, we use conjugate gradients with maximum number of iterations given by $\rm{cg\_maxiter}:= 500$ and tolerance parameter $\rm{cg\_tol} := 10^{-10}$. We use the same stopping criteria for the majority of algorithms under consideration:
\begin{align}
\frac{\vectornormbig{\signal(i) - \signal(i-1)}_F}{\vectornormbig{\signal(i)}_F} \leq \rm{tol},
\end{align} where $\signal(i), ~\signal(i-1) $ denote the current and the previous estimate of $\bestsignal$ and $\rm{tol} := 5\cdot 10^{-5}$. If this is not the case, we tweak the algorithms to minimize the total execution time and achieve similar reconstruction performance as the rest of the algorithms. For SVD calculations, we use the $\rm{lansvd}$ implementation in PROPACK package \cite{propack}---moreover, all the algorithms in comparison use the same linear operators $\sensing$ and $\sensing^{\ast}$ for gradient and SVD calculations and conjugate-gradient least-squares minimizations. For fairness, we modified all the algorithms so that they {\it exploit the true rank}. Small deviations from the true rank result in relatively small degradation in terms of the reconstruction performance. In case the rank of $\bestsignal$ is unknown, one has to predict the dimension of the principal singular space. The authors in \cite{SVP}, based on ideas in \cite{OptSpace}, propose to compute singular values incrementally until a significant gap between singular values is found. Similar strategies can be found in \cite{ALM} for the convex case.

In \textsc{Matrix ALPS II} and \textsc{Matrix ALPS II} with QR, we perform $\mathcal{Q}_i \leftarrow \text{ortho}(\mathcal{X}_i \cup \mathcal{X}_{i+1})$ to construct a set of orthonormal rank-1 matrices that span the subspace, spanned by $\mathcal{X}_i \cup \mathcal{X}_{i+1}$. 
While such operation can be implemented using factorization procedures (such as SVD or QR decompositions), in practice this degrades the time complexity of the algorithm substantially as the rank $\rank$ and the problem dimensionality increase. In our implementations, we simply {\it union} the set of orthonormal rank-1 matrices, without further orthogonalization. Thus, we employ {\it inexact} projections for computational efficiency which results in faster convergence. Figure 5 shows the time overhead due to the additional orthogonalization process. We compare three algorithms: \textsc{Matrix ALPS II} (no orthogonalization step), \textsc{Matrix ALPS II} using SVD for orthogonalization and, \textsc{Matrix ALPS II} using QR for orthogonalization. In Figures 5(a)-(b), we use subsampled and permuted noiselets for linear map $\sensing$ and in Figure 5(c), we test the MC problem. In all the experimental cases considered in this work, we observed identical performace in terms of reconstruction accuracy for the three variants, as can be also seen in Figure 5. To this end, for the rest of the paper, we use \textsc{Matrix ALPS II} where $\mathcal{Q}_i \leftarrow \mathcal{X}_i \cup \mathcal{X}_{i+1}$.

\subsection{Limitations of $\vectornormbig{\cdot}_{\ast}$-based algorithms: a toy example}

While nucluear norm heuristic is widely used in solving the low-rank minimization problem, \cite{nonuclear} presents simple problem cases where convex, nuclear norm-based, algorithms {\it fail} in practice. Using the  $\vectornormbig{\cdot}_{\ast}$-norm in the objective function as the convex surrogate of the $\text{rank}(\cdot)$ metric might lead to a candidate set with multiple solutions, introducing ambiguity in the selection process. Borrowing the example in \cite{nonuclear}, we test the list of algorithms above on a toy problem setting that does not satisfy the rank-RIP. To this end, we design the following problem: let $\bestsignal \in \mathbb{R}^{5 \times 4}$ be the matrix of interest with $\text{rank}(\bestsignal) = 2$, as shown in Figure \ref{fig:toy}(a). We consider the case where we have access to $\bestsignal $ only through a subset of its entries, as shown in Figure \ref{fig:toy}(b).

\begin{figure}[ht]
\begin{center}
\begin{minipage}[c]{0.25\linewidth}
\begin{align}
\left( \begin{array}{cccc}
2 & 2 & 1 & 1 \\
2 & 2 & 1 & 1 \\
2 & 2 & 1 & 1 \\
2 & 2 & 1 & 1 \\
1 & 1 & 2 & 1 \end{array} \right) \nonumber 
\end{align} \hspace{0.6cm} (a)
\end{minipage}
\hspace{-0.5cm}
\begin{minipage}[c]{0.25\linewidth}
\begin{align}
\left( \begin{array}{cccc}
2 & 2 & 1 & 1 \\
2 & 2 & 1 & 1 \\
? & ? & ? & 1 \\
2 & ? & ? & 1 \\
1 & 1 & 2 & 1 \end{array} \right) \nonumber 
\end{align} \hspace{0.6cm}  (b)
\end{minipage}
\end{center}
\caption{\small\sl Matrix Completion toy example for $ \bestsignal \in \mathbb{R}^{5 \times 4}$. We use `?' to denote the unobserved entried. } \label{fig:toy}
\end{figure}

In Figure \ref{fig:toy2}, we present the reconstruction performance of various matrix completion solvers after 300 iterations. Although there are multiple solutions that induce the recovered matrix and have the same rank as $\bestsignal$, most of the algorithms in comparison reconstruct $\bestsignal$ successfully. We note that, in some cases, the inadequancy of an algorithm to reconstruct $\bestsignal$ is not because of the (relaxed) problem formulation but due to its fast---but inaccurate---implementation (fast convergence versus reconstruction accuracy tradeoff).

\begin{figure*}[ht]
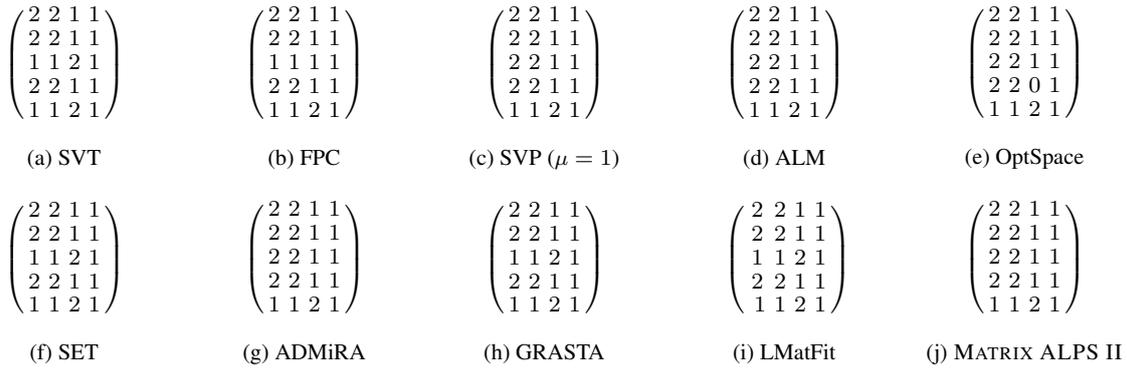

\hspace{1.5cm}\begin{minipage}[c]{0.15\linewidth}
\centering
\begin{align}
\left( \begin{array}{cccc}
2 & 2 & 1 & 1 \\
2 & 2 & 1 & 1 \\
1 & 1 & 2 & 1 \\
2 & 2 & 1 & 1 \\
1 & 1 & 2 & 1 \end{array} \right) \nonumber 
\end{align} \hspace{-1.1cm} (a) SVT
\end{minipage}
\hspace{0.4cm}
\begin{minipage}[c]{0.15\linewidth}
\centering
\begin{align}
\left( \begin{array}{cccc}
2 & 2 & 1 & 1 \\
2 & 2 & 1 & 1 \\
1 & 1 & 1 & 1 \\
2 & 2 & 1 & 1 \\
1 & 1 & 2 & 1 \end{array} \right) \nonumber 
\end{align} \hspace{-1.1cm} (b) FPC
\end{minipage}
\hspace{0.4cm}
\begin{minipage}[c]{0.15\linewidth}
\centering
\begin{align}
\left( \begin{array}{cccc}
2 & 2 & 1 & 1 \\
2 & 2 & 1 & 1 \\
2 & 2 & 1 & 1 \\
2 & 2 & 1 & 1 \\
1 & 1 & 2 & 1 \end{array} \right) \nonumber 
\end{align} \hspace{-1.1cm} (c) SVP ($\mu = 1$)
\end{minipage}
\hspace{0.4cm}
\begin{minipage}[c]{0.15\linewidth}
\centering
\begin{align}
\left( \begin{array}{cccc}
2 & 2 & 1 & 1 \\
2 & 2 & 1 & 1 \\
2 & 2 & 1 & 1 \\
2 & 2 & 1 & 1 \\
1 & 1 & 2 & 1 \end{array} \right) \nonumber 
\end{align} \hspace{-1.1cm} (d) ALM
\end{minipage}
\hspace{0.4cm}
\begin{minipage}[c]{0.15\linewidth}
\centering
\begin{align}
\left( \begin{array}{cccc}
2 & 2 & 1 & 1 \\
2 & 2 & 1 & 1 \\
2 & 2 & 1 & 1 \\
2 & 2 & 0 & 1 \\
1 & 1 & 2 & 1 \end{array} \right) \nonumber 
\end{align} \hspace{-1.1cm} (e) OptSpace
\end{minipage} 
\\

\hspace{1.5cm}\begin{minipage}[c]{0.15\linewidth}
\centering
\begin{align}
\left( \begin{array}{cccc}
2 & 2 & 1 & 1 \\
2 & 2 & 1 & 1 \\
1 & 1 & 2 & 1 \\
2 & 2 & 1 & 1 \\
1 & 1 & 2 & 1 \end{array} \right) \nonumber 
\end{align} \hspace{-1.1cm} (f) SET
\end{minipage}
\hspace{0.4cm}
\begin{minipage}[c]{0.15\linewidth}
\centering
\begin{align}
\left( \begin{array}{cccc}
2 & 2 & 1 & 1 \\
2 & 2 & 1 & 1 \\
2 & 2 & 1 & 1 \\
2 & 2 & 1 & 1 \\
1 & 1 & 2 & 1 \end{array} \right) \nonumber 
\end{align} \hspace{-1.1cm} (g) ADMiRA
\end{minipage}
\hspace{0.4cm}
\begin{minipage}[c]{0.15\linewidth}
\centering
\begin{align}
\left( \begin{array}{cccc}
2 & 2 & 1 & 1 \\
2 & 2 & 1 & 1 \\
1 & 1 & 2 & 1 \\
2 & 2 & 1 & 1 \\
1 & 1 & 2 & 1 \end{array} \right) \nonumber 
\end{align} \hspace{-1.1cm} (h) GRASTA
\end{minipage}
\hspace{0.4cm}
\begin{minipage}[c]{0.15\linewidth}
\centering
\begin{align}
\left( \begin{array}{cccc}
2 & 2 & 1 & 1 \\
2 & 2 & 1 & 1 \\
1 & 1 & 2 & 1 \\
2 & 2 & 1 & 1 \\\
1 & 1 & 2 & 1 \end{array} \right) \nonumber 
\end{align} \hspace{-1.1cm} (i) LMatFit
\end{minipage}
\hspace{0.4cm}
\begin{minipage}[c]{0.15\linewidth}
\centering
\begin{align}
\left( \begin{array}{cccc}
2 & 2 & 1 & 1 \\
2 & 2 & 1 & 1 \\
2 & 2 & 1 & 1 \\
2 & 2 & 1 & 1 \\
1 & 1 & 2 & 1 \end{array} \right) \nonumber 
\end{align} \hspace{-1.1cm} (j) \textsc{Matrix ALPS II}
\end{minipage}
\caption{\small\sl Toy example reconstruction performance for various algorithms. We observe that $\bestsignal$ is an integer matrix---since the algorithms under consideration return real matrices as solutions, we round the solution elementwise. } \label{fig:toy2}
\end{figure*}

\subsection{Synthetic data}

\textbf{General affine rank minimization using noiselets:} In this experiment, the set of observations $\obs \in \mathbb{R}^{\numsam}$ satisfy:
\begin{align}
\obs = \sensing \bestsignal + \noise
\end{align} Here, we use permuted and subsampled noiselets for the linear operator $\sensing $ \cite{sparcs}. The signal $\bestsignal$ is generated as the multiplication of two low-rank matrices, $\mathbf{L} \in \mathbb{R}^{m \times \rank} $ and $\mathbf{R} \in \mathbb{R}^{n \times \rank} $, such that $\bestsignal = \mathbf{L} \mathbf{R}^T$ and $\vectornormbig{\bestsignal}_F = 1 $. Both $\mathbf{L} $ and $\mathbf{R} $ have random independent and identically distributed (iid) Gaussian entries with zero mean and unit variance. In the noisy case, the additive noise term $\noise \in \mathbb{R}^{\numsam}$ contains entries drawn from a zero mean Gaussian distribution with $\vectornormbig{\noise}_2 \in \lbrace 10^{-3}, 10^{-4} \rbrace $. 

We compare the following algorithms: SVP, ADMiRA, \textsc{Matrix ALPS I}, \textsc{Matrix ALPS II} and \textsc{Matrix ALPS II} with QR for various problem configurations, as depicted in Table \ref{table:1} (there is no available code with arbitrary sensing operators for the rest algorithms). In Table \ref{table:1}, we show the median values of reconstruction error, number of iterations and execution time over 50 Monte Carlo iterations. For all cases, we assume $\text{SR} = 0.3 $ and we set the maximum number of iterations to 500. Bold font denotes the fastest execution time. Furthermore, Figure \ref{fig: TableI_fig} illustrates the effectiveness of the algorithms for some representative problem configurations.

In Table \ref{table:1}, \textsc{Matrix ALPS II} and \textsc{Matrix ALPS II} with QR obtain accurate low-rank solutions much faster than the rest of the algorithms in comparison. In high dimensional settings, \textsc{Matrix ALPS II} with QR scales better as the problem dimensions increase, leading to faster convergence. Moreover, its execution time is at least a few orders of magnitude smaller compared to SVP, ADMiRA and \textsc{Matrix ALPS I} implementations.

\begin{table*}
\caption{General ARM using Noiselets.} {\label{table:1}}
\begin{center}
\begin{tabular}{|c|c|c|c|c|c|c|c|c|c|c|c|c|c}
\multicolumn{4}{c|}{Configuration} & FR & \multicolumn{3}{|c|}{SVP} & \multicolumn{3}{|c|}{ADMiRA} & \multicolumn{3}{|c}{\textsc{Matrix ALPS I}} \\
\hline \hline
\multicolumn{1}{c}{$m$}  & \multicolumn{1}{c}{$n$} & \multicolumn{1}{c}{$\rank$} & \multicolumn{1}{c|}{$\vectornormbig{\noise}_2$} & & 
\multicolumn{1}{|c}{\rm{iter.}} & \multicolumn{1}{c}{\rm{err.}} & \multicolumn{1}{c|}{\rm{time}} &
\multicolumn{1}{|c}{\rm{iter.}} & \multicolumn{1}{c}{\rm{err.}} & \multicolumn{1}{c|}{\rm{time}} &
\multicolumn{1}{|c}{\rm{iter.}} & \multicolumn{1}{c}{\rm{err.}} & \multicolumn{1}{c}{\rm{time}} \\
\hline\hline
\multicolumn{1}{c}{$256$} & \multicolumn{1}{c}{$512$} & \multicolumn{1}{c}{$5$} & \multicolumn{1}{c|}{$0$} & $ 0.097 $ & 
\multicolumn{1}{|c}{$38$} & \multicolumn{1}{c}{$2.2\cdot 10^{-4}$} & \multicolumn{1}{c|}{$0.78$} &
\multicolumn{1}{|c}{$27$} & \multicolumn{1}{c}{$4.4\cdot 10^{-5}$} & \multicolumn{1}{c|}{$2.26$} &
\multicolumn{1}{|c}{$13.5$} & \multicolumn{1}{c}{$1\cdot 10^{-5}$} & \multicolumn{1}{c}{$0.7$} \\
\hline
\multicolumn{1}{c}{$256$} & \multicolumn{1}{c}{$512$} & \multicolumn{1}{c}{$5$} & \multicolumn{1}{c|}{$10^{-3}$} & $ 0.097 $ & 
\multicolumn{1}{|c}{$38$} & \multicolumn{1}{c}{$6\cdot 10^{-4}$} & \multicolumn{1}{c|}{$0.91$} &
\multicolumn{1}{|c}{$700$} & \multicolumn{1}{c}{$2\cdot 10^{-3}$} & \multicolumn{1}{c|}{$65.94$} &
\multicolumn{1}{|c}{$16$} & \multicolumn{1}{c}{$7\cdot 10^{-4}$} & \multicolumn{1}{c}{$0.92$} \\
\hline
\multicolumn{1}{c}{$256$} & \multicolumn{1}{c}{$512$} & \multicolumn{1}{c}{$5$} & \multicolumn{1}{c|}{$10^{-4}$} & $ 0.097 $ & 
\multicolumn{1}{|c}{$38$} & \multicolumn{1}{c}{$2.1\cdot 10^{-4}$} & \multicolumn{1}{c|}{$0.94$} &
\multicolumn{1}{|c}{$700$} & \multicolumn{1}{c}{$4.1\cdot 10^{-4}$} & \multicolumn{1}{c|}{$69.03$} &
\multicolumn{1}{|c}{$11.5$} & \multicolumn{1}{c}{$7.9\cdot 10^{-5}$} & \multicolumn{1}{c}{$0.72$} \\
\hline
\multicolumn{1}{c}{$256$} & \multicolumn{1}{c}{$512$} & \multicolumn{1}{c}{$10$} & \multicolumn{1}{c|}{$0$} & $ 0.193 $ & 
\multicolumn{1}{|c}{$50$} & \multicolumn{1}{c}{$3.4\cdot 10^{-4}$} & \multicolumn{1}{c|}{$1.44$} &
\multicolumn{1}{|c}{$38$} & \multicolumn{1}{c}{$5\cdot 10^{-5}$} & \multicolumn{1}{c|}{$4.42$} &
\multicolumn{1}{|c}{$13$} & \multicolumn{1}{c}{$3.9\cdot 10^{-5}$} & \multicolumn{1}{c}{$0.92$} \\
\hline
\multicolumn{1}{c}{$256$} & \multicolumn{1}{c}{$512$} & \multicolumn{1}{c}{$10$} & \multicolumn{1}{c|}{$10^{-3}$} & $ 0.193 $ & 
\multicolumn{1}{|c}{$50$} & \multicolumn{1}{c}{$9\cdot 10^{-4}$} & \multicolumn{1}{c|}{$1.39$} &
\multicolumn{1}{|c}{$700$} & \multicolumn{1}{c}{$1.7\cdot 10^{-3}$} & \multicolumn{1}{c|}{$56.94$} &
\multicolumn{1}{|c}{$29$} & \multicolumn{1}{c}{$1.2\cdot 10^{-3}$} & \multicolumn{1}{c}{$1.78$} \\
\hline
\multicolumn{1}{c}{$256$} & \multicolumn{1}{c}{$512$} & \multicolumn{1}{c}{$10$} & \multicolumn{1}{c|}{$10^{-4}$} & $ 0.193 $ & 
\multicolumn{1}{|c}{$50$} & \multicolumn{1}{c}{$3.5\cdot 10^{-4}$} & \multicolumn{1}{c|}{$1.38$} &
\multicolumn{1}{|c}{$700$} & \multicolumn{1}{c}{$9.3\cdot 10^{-5}$} & \multicolumn{1}{c|}{$64.69$} &
\multicolumn{1}{|c}{$14$} & \multicolumn{1}{c}{$1.4\cdot 10^{-4}$} & \multicolumn{1}{c}{$0.93$} \\
\hline
\multicolumn{1}{c}{$256$} & \multicolumn{1}{c}{$512$} & \multicolumn{1}{c}{$20$} & \multicolumn{1}{c|}{$0$} & $ 0.38 $ & 
\multicolumn{1}{|c}{$86$} & \multicolumn{1}{c}{$7\cdot 10^{-4}$} & \multicolumn{1}{c|}{$3.32$} &
\multicolumn{1}{|c}{$700$} & \multicolumn{1}{c}{$4.1\cdot 10^{-5}$} & \multicolumn{1}{c|}{$81.93$} &
\multicolumn{1}{|c}{$45$} & \multicolumn{1}{c}{$2\cdot 10^{-4}$} & \multicolumn{1}{c}{$4.09$} \\
\hline
\multicolumn{1}{c}{$256$} & \multicolumn{1}{c}{$512$} & \multicolumn{1}{c}{$20$} & \multicolumn{1}{c|}{$10^{-3}$} & $ 0.38 $ & 
\multicolumn{1}{|c}{$86$} & \multicolumn{1}{c}{$1.5\cdot 10^{-3}$} & \multicolumn{1}{c|}{$3.45$} &
\multicolumn{1}{|c}{$700$} & \multicolumn{1}{c}{$4.2\cdot 10^{-2}$} & \multicolumn{1}{c|}{$77.35$} &
\multicolumn{1}{|c}{$69$} & \multicolumn{1}{c}{$2.3\cdot 10^{-3}$} & \multicolumn{1}{c}{$5.05$} \\
\hline
\multicolumn{1}{c}{$256$} & \multicolumn{1}{c}{$512$} & \multicolumn{1}{c}{$20$} & \multicolumn{1}{c|}{$10^{-4}$} & $ 0.38 $ & 
\multicolumn{1}{|c}{$86$} & \multicolumn{1}{c}{$7\cdot 10^{-4}$} & \multicolumn{1}{c|}{$3.26$} &
\multicolumn{1}{|c}{$700$} & \multicolumn{1}{c}{$4\cdot 10^{-2}$} & \multicolumn{1}{c|}{$79.47$} &
\multicolumn{1}{|c}{$46$} & \multicolumn{1}{c}{$4\cdot 10^{-4}$} & \multicolumn{1}{c}{$4.1$} \\
\hline
\multicolumn{1}{c}{$512$} & \multicolumn{1}{c}{$1024$} & \multicolumn{1}{c}{$30$} & \multicolumn{1}{c|}{$0$} & $ 0.287 $ & 
\multicolumn{1}{|c}{$66$} & \multicolumn{1}{c}{$4.9\cdot 10^{-4}$} & \multicolumn{1}{c|}{$8.79$} &
\multicolumn{1}{|c}{$295$} & \multicolumn{1}{c}{$5.4\cdot 10^{-5}$} & \multicolumn{1}{c|}{$143.53$} &
\multicolumn{1}{|c}{$24$} & \multicolumn{1}{c}{$1\cdot 10^{-4}$} & \multicolumn{1}{c}{$8.01$} \\
\hline
\multicolumn{1}{c}{$512$} & \multicolumn{1}{c}{$1024$} & \multicolumn{1}{c}{$40$} & \multicolumn{1}{c|}{$0$} & $ 0.38 $ & 
\multicolumn{1}{|c}{$86$} & \multicolumn{1}{c}{$7\cdot 10^{-4}$} & \multicolumn{1}{c|}{$10.09$} &
\multicolumn{1}{|c}{$700$} & \multicolumn{1}{c}{$4.3\cdot 10^{-2}$} & \multicolumn{1}{c|}{$251.27$} &
\multicolumn{1}{|c}{$45$} & \multicolumn{1}{c}{$2\cdot 10^{-4}$} & \multicolumn{1}{c}{$11.08$} \\
\hline
\multicolumn{1}{c}{$1024$} & \multicolumn{1}{c}{$2048$} & \multicolumn{1}{c}{$50$} & \multicolumn{1}{c|}{$0$} & $ 0.24 $ & 
\multicolumn{1}{|c}{$57$} & \multicolumn{1}{c}{$4.3\cdot 10^{-4}$} & \multicolumn{1}{c|}{$42.88$} &
\multicolumn{1}{|c}{$103$} & \multicolumn{1}{c}{$5.2\cdot 10^{-5}$} & \multicolumn{1}{c|}{$312.62$} &
\multicolumn{1}{|c}{$18$} & \multicolumn{1}{c}{$5.7\cdot 10^{-5}$} & \multicolumn{1}{c}{$35.86$} \\
\hline \hline \hline

\multicolumn{4}{c|}{} & & \multicolumn{3}{|c|}{\textsc{Matrix ALPS II}} & \multicolumn{6}{|c}{\textsc{Matrix ALPS II} with QR}  \\
\hline \hline
\multicolumn{1}{c}{$m$}  & \multicolumn{1}{c}{$n$} & \multicolumn{1}{c}{$\rank$} & \multicolumn{1}{c|}{$\vectornormbig{\noise}_2$} & & 
\multicolumn{1}{|c}{\rm{iter.}} & \multicolumn{1}{c}{\rm{err.}} & \multicolumn{1}{c|}{\rm{time}} &
\multicolumn{2}{|c}{\rm{iter.}} & \multicolumn{2}{c}{\rm{err.}} & \multicolumn{2}{c}{\rm{time}}  \\
\hline\hline
\multicolumn{1}{c}{$256$} & \multicolumn{1}{c}{$512$} & \multicolumn{1}{c}{$5$}  & \multicolumn{1}{c|}{$0$} & $ 0.097 $ & 
\multicolumn{1}{|c}{$8$} & \multicolumn{1}{c}{$7.1 \cdot 10^{-6}$} & \multicolumn{1}{c|}{$0.42$} &
\multicolumn{2}{|c}{$10$} & \multicolumn{2}{c}{$9.1 \cdot 10^{-6}$} & \multicolumn{2}{c}{$\mathbf{0.39}$} \\
\hline
\multicolumn{1}{c}{$256$} & \multicolumn{1}{c}{$512$} & \multicolumn{1}{c}{$5$}  & \multicolumn{1}{c|}{$10^{-3}$} & $ 0.097 $ & 
\multicolumn{1}{|c}{$9$} & \multicolumn{1}{c}{$7 \cdot 10^{-4}$} & \multicolumn{1}{c|}{$\mathbf{0.56}$} &
\multicolumn{2}{|c}{$20$} & \multicolumn{2}{c}{$7 \cdot 10^{-4}$} & \multicolumn{2}{c}{$0.93$} \\
\hline
\multicolumn{1}{c}{$256$} & \multicolumn{1}{c}{$512$} & \multicolumn{1}{c}{$5$}  & \multicolumn{1}{c|}{$10^{-4}$} & $ 0.097 $ & 
\multicolumn{1}{|c}{$8$} & \multicolumn{1}{c}{$7 \cdot 10^{-5}$} & \multicolumn{1}{c|}{$0.5$} &
\multicolumn{2}{|c}{$10$} & \multicolumn{2}{c}{$7.8 \cdot 10^{-5}$} & \multicolumn{2}{c}{$\mathbf{0.46}$} \\
\hline
\multicolumn{1}{c}{$256$} & \multicolumn{1}{c}{$512$} & \multicolumn{1}{c}{$10$}  & \multicolumn{1}{c|}{$0$} & $ 0.193 $ & 
\multicolumn{1}{|c}{$10$} & \multicolumn{1}{c}{$2.3 \cdot 10^{-5}$} & \multicolumn{1}{c|}{$0.68$} &
\multicolumn{2}{|c}{$13$} & \multicolumn{2}{c}{$2.4 \cdot 10^{-5}$} & \multicolumn{2}{c}{$\mathbf{0.64}$} \\
\hline
\multicolumn{1}{c}{$256$} & \multicolumn{1}{c}{$512$} & \multicolumn{1}{c}{$10$}  & \multicolumn{1}{c|}{$10^{-3}$} & $ 0.193 $ & 
\multicolumn{1}{|c}{$19$} & \multicolumn{1}{c}{$1 \cdot 10^{-3}$} & \multicolumn{1}{c|}{$\mathbf{1.29}$} &
\multicolumn{2}{|c}{$27$} & \multicolumn{2}{c}{$1 \cdot 10^{-3}$} & \multicolumn{2}{c}{$1.35$} \\
\hline
\multicolumn{1}{c}{$256$} & \multicolumn{1}{c}{$512$} & \multicolumn{1}{c}{$10$}  & \multicolumn{1}{c|}{$10^{-4}$} & $ 0.193 $ & 
\multicolumn{1}{|c}{$10$} & \multicolumn{1}{c}{$1.1 \cdot 10^{-4}$} & \multicolumn{1}{c|}{$0.68$} &
\multicolumn{2}{|c}{$13$} & \multicolumn{2}{c}{$1.1 \cdot 10^{-4}$} & \multicolumn{2}{c}{$\mathbf{0.62}$} \\
\hline
\multicolumn{1}{c}{$256$} & \multicolumn{1}{c}{$512$} & \multicolumn{1}{c}{$20$}  & \multicolumn{1}{c|}{$0$} & $ 0.38 $ & 
\multicolumn{1}{|c}{$21$} & \multicolumn{1}{c}{$1 \cdot 10^{-4}$} & \multicolumn{1}{c|}{$1.92$} &
\multicolumn{2}{|c}{$24$} & \multicolumn{2}{c}{$1 \cdot 10^{-4}$} & \multicolumn{2}{c}{$\mathbf{1.26}$} \\
\hline
\multicolumn{1}{c}{$256$} & \multicolumn{1}{c}{$512$} & \multicolumn{1}{c}{$20$}  & \multicolumn{1}{c|}{$10^{-3}$} & $ 0.38 $ & 
\multicolumn{1}{|c}{$36$} & \multicolumn{1}{c}{$1.5 \cdot 10^{-3}$} & \multicolumn{1}{c|}{$2.67$} &
\multicolumn{2}{|c}{$39$} & \multicolumn{2}{c}{$1.5 \cdot 10^{-3}$} & \multicolumn{2}{c}{$\mathbf{1.69}$} \\
\hline
\multicolumn{1}{c}{$256$} & \multicolumn{1}{c}{$512$} & \multicolumn{1}{c}{$20$}  & \multicolumn{1}{c|}{$10^{-4}$} & $ 0.38 $ & 
\multicolumn{1}{|c}{$21$} & \multicolumn{1}{c}{$2 \cdot 10^{-4}$} & \multicolumn{1}{c|}{$1.87$} &
\multicolumn{2}{|c}{$24$} & \multicolumn{2}{c}{$2 \cdot 10^{-4}$} & \multicolumn{2}{c}{$\mathbf{1.22}$} \\
\hline
\multicolumn{1}{c}{$512$} & \multicolumn{1}{c}{$1024$} & \multicolumn{1}{c}{$30$}  & \multicolumn{1}{c|}{$0$} & $ 0.287 $ & 
\multicolumn{1}{|c}{$14$} & \multicolumn{1}{c}{$4.5 \cdot 10^{-5}$} & \multicolumn{1}{c|}{$4.7$} &
\multicolumn{2}{|c}{$18$} & \multicolumn{2}{c}{$3.3 \cdot 10^{-5}$} & \multicolumn{2}{c}{$\mathbf{4.15}$} \\
\hline
\multicolumn{1}{c}{$512$} & \multicolumn{1}{c}{$1024$} & \multicolumn{1}{c}{$40$}  & \multicolumn{1}{c|}{$0$} & $ 0.38 $ & 
\multicolumn{1}{|c}{$21$} & \multicolumn{1}{c}{$1\cdot 10^{-4}$} & \multicolumn{1}{c|}{$6.01$} &
\multicolumn{2}{|c}{$24$} & \multicolumn{2}{c}{$1 \cdot 10^{-4}$} & \multicolumn{2}{c}{$\mathbf{4.53}$} \\
\hline
\multicolumn{1}{c}{$1024$} & \multicolumn{1}{c}{$2048$} & \multicolumn{1}{c}{$50$}  & \multicolumn{1}{c|}{$0$} & $ 0.24 $ & 
\multicolumn{1}{|c}{$12$} & \multicolumn{1}{c}{$2.5\cdot 10^{-5}$} & \multicolumn{1}{c|}{$22.76$} &
\multicolumn{2}{|c}{$15$} & \multicolumn{2}{c}{$3.3 \cdot 10^{-5}$} & \multicolumn{2}{c}{$\mathbf{17.94}$} \\
\hline
\end{tabular}
\end{center}
\end{table*}

\begin{figure*}[!htp]
\centering
\begin{tabular}{ccc}
\centerline{\subfigure[]{\includegraphics[width = 0.33\textwidth]{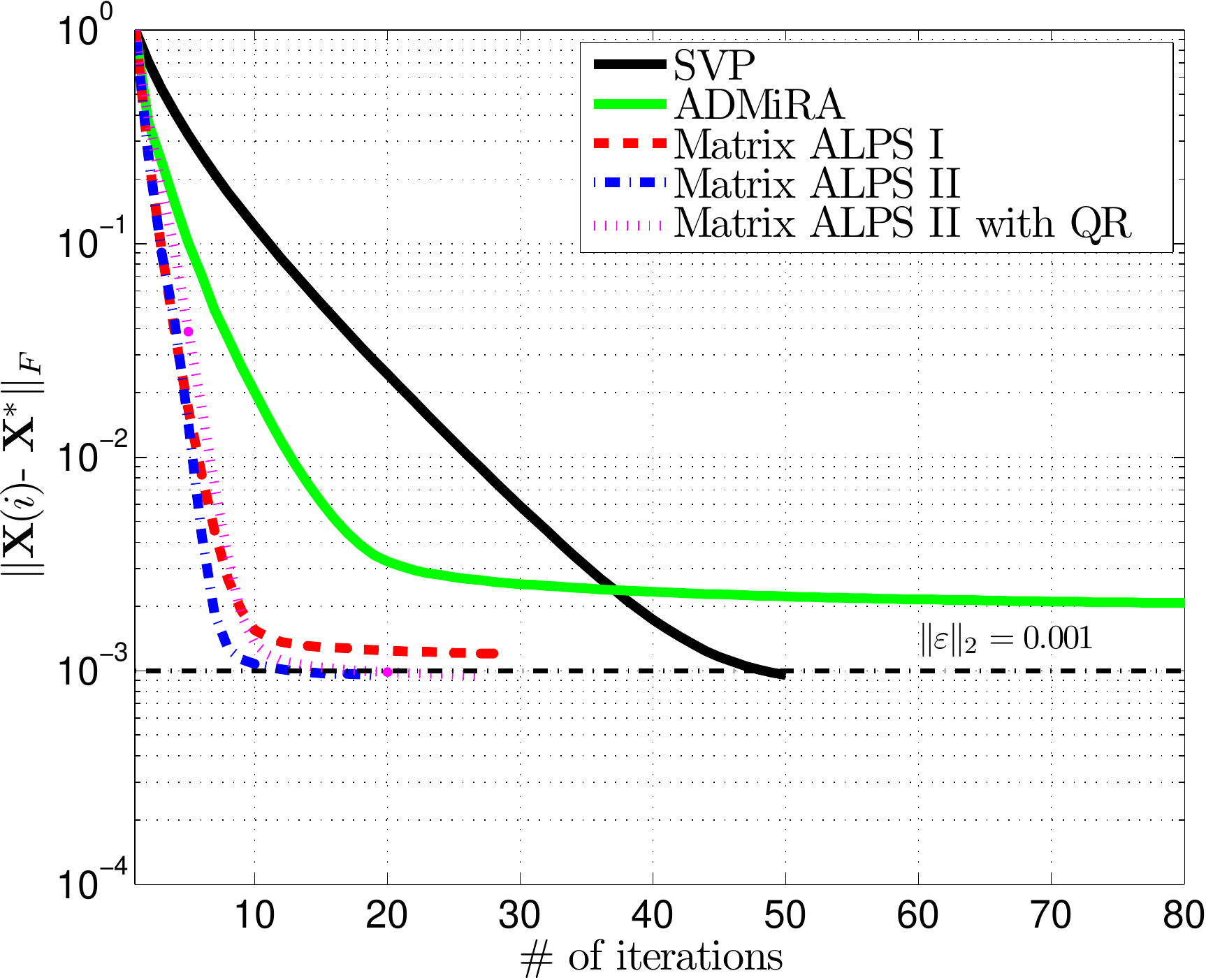}\label{fig:2a}} 
\hfill
\subfigure[]{\includegraphics[width = 0.33\textwidth]{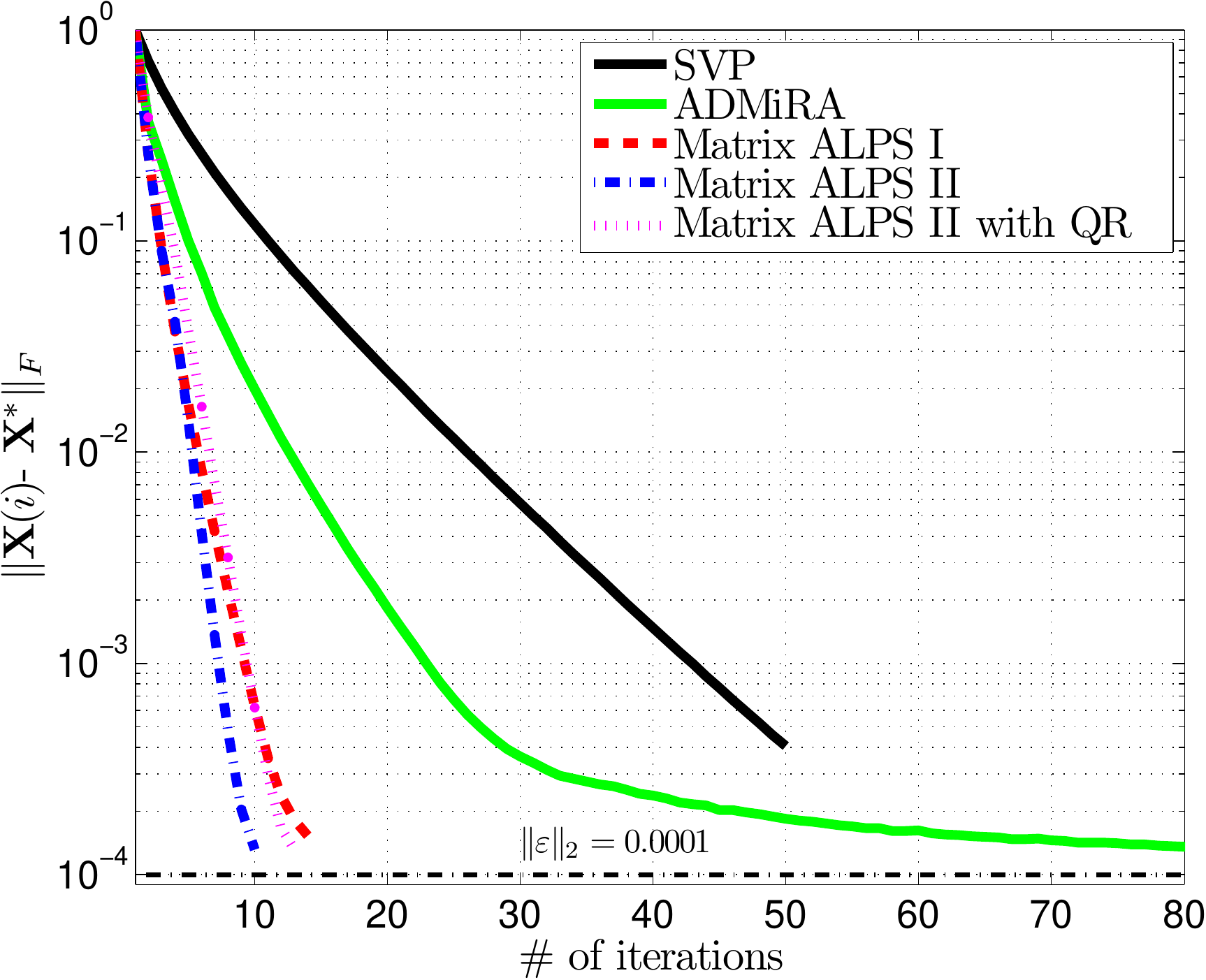}\label{fig:2b}}
\hfill 
\subfigure[]{\includegraphics[width = 0.33\textwidth]{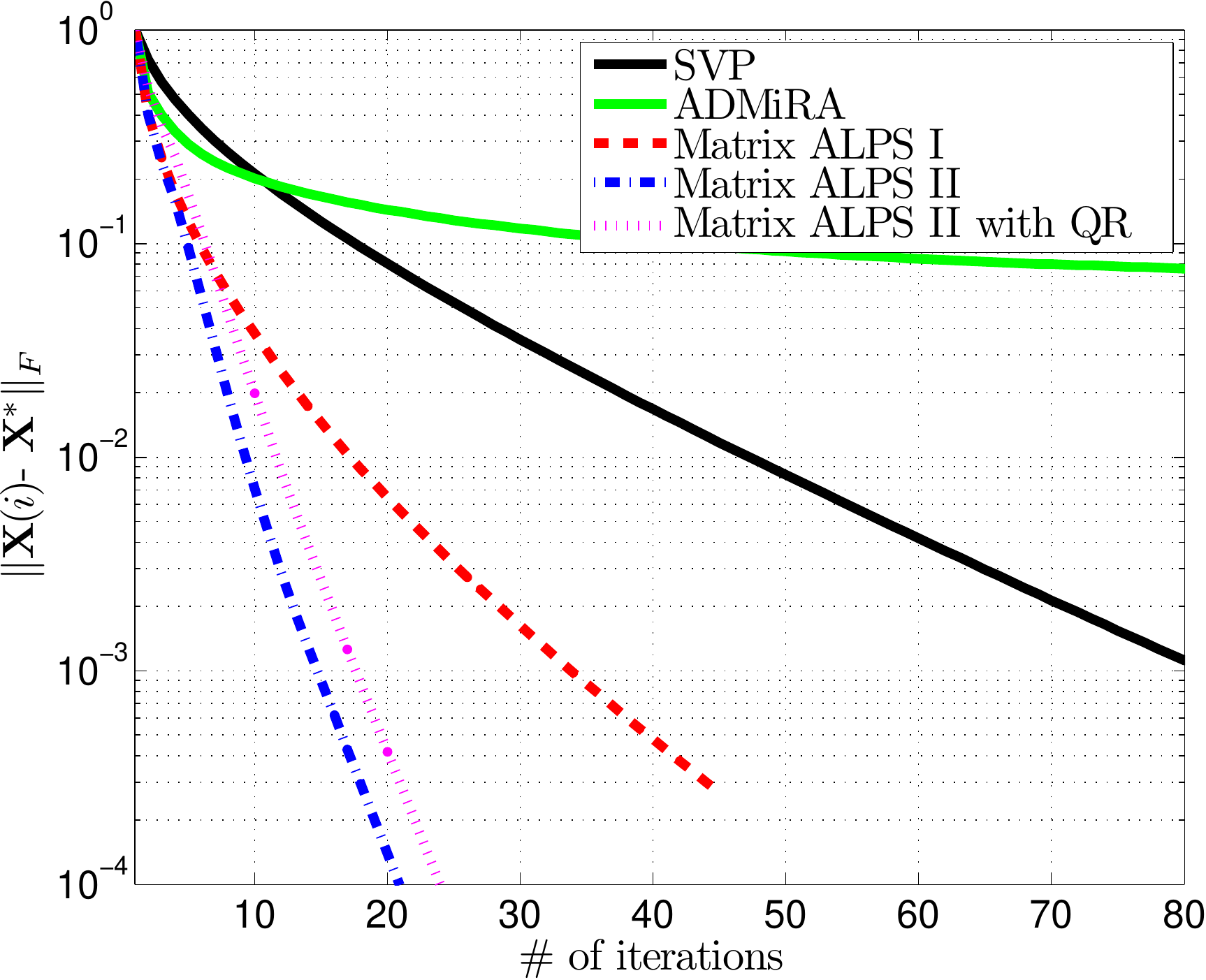}\label{fig:2c}}} \\
\centerline{\subfigure[]{\includegraphics[width = 0.33\textwidth]{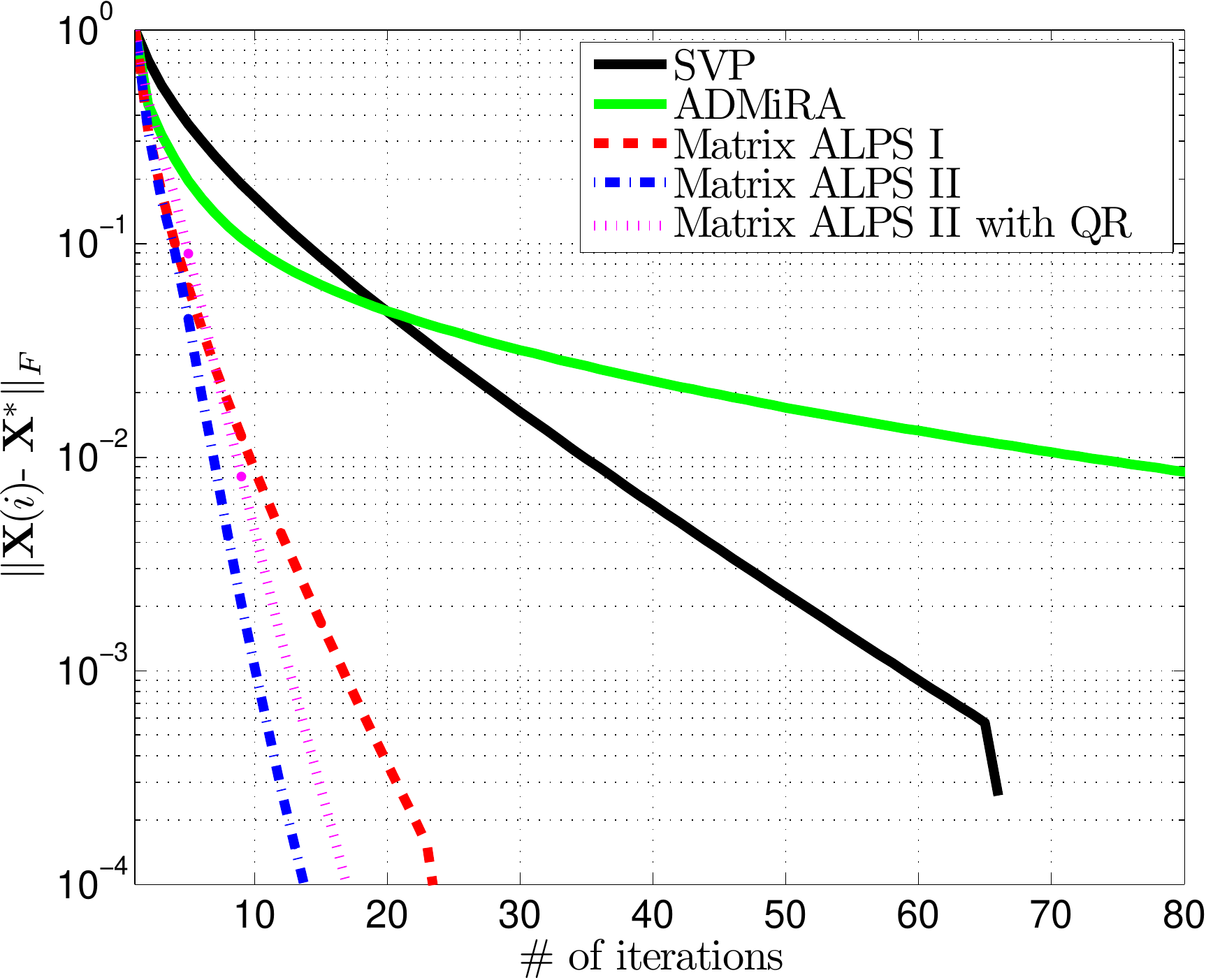}\label{fig:2a}} 
\hfill
\subfigure[]{\includegraphics[width = 0.33\textwidth]{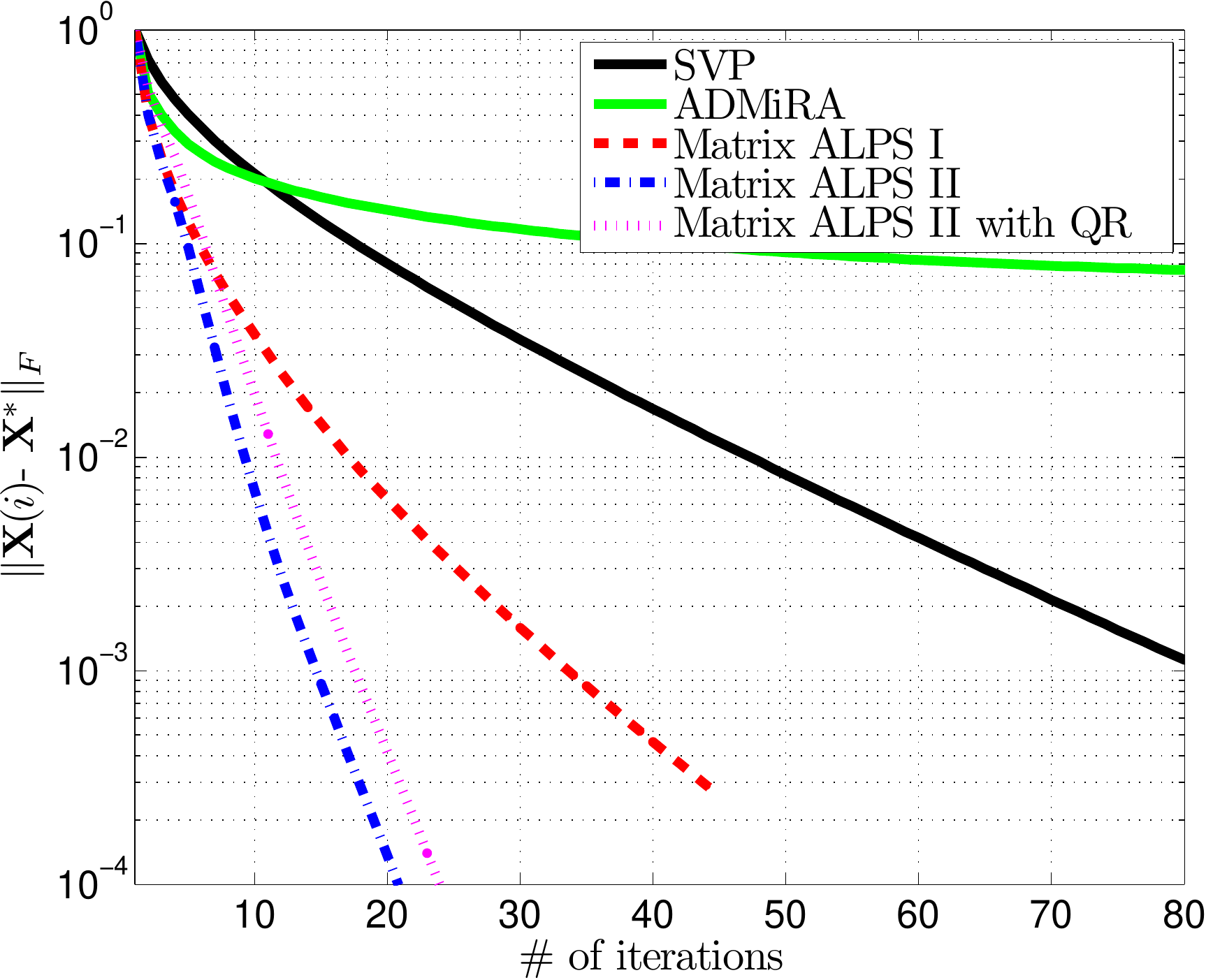}\label{fig:2b}}
\hfill 
\subfigure[]{\includegraphics[width = 0.33\textwidth]{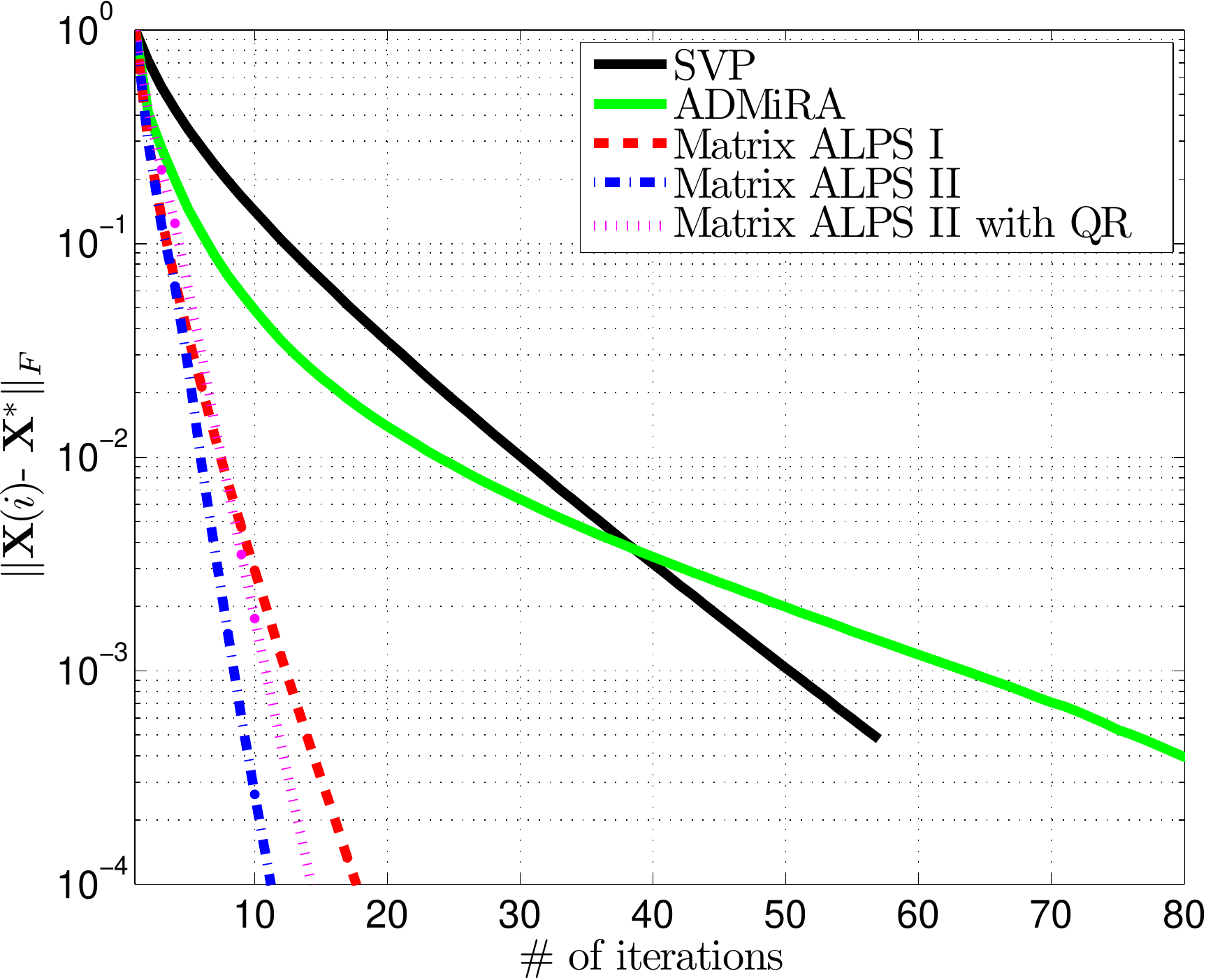}\label{fig:2c}}}
\end{tabular}
\caption{\small\sl Low rank signal reconstruction using noiselet linear operator. The error curves are the median values across 50 Monte-Carlo realizations over each iteration. For all cases, we assume $p = 0.3mn$. (a) $m = 256$, $n = 512 $, $\rank = 10$ and $\vectornormbig{\noise}_2 = 10^{-3} $. (b) $m = 256$, $n = 512 $, $\rank = 10$ and $\vectornormbig{\noise}_2 = 10^{-4} $.  (c) $m = 256$, $n = 512 $, $\rank = 20$ and $\vectornormbig{\noise}_2 = 0 $. (d) $m = 512$, $n = 1024 $, $\rank = 30$ and $\vectornormbig{\noise}_2 = 0 $. (e) $m = 512$, $n = 1024 $, $\rank = 40$ and $\vectornormbig{\noise}_2 = 0 $.  (f) $m = 1024$, $n = 2048 $, $\rank = 50$ and $\vectornormbig{\noise}_2 = 0 $.   } \label{fig: TableI_fig}
\end{figure*}

\textbf{Robust matrix completion:} We design matrix completion problems in the following way. The signal of interest $\bestsignal \in \mathbb{R}^{\dimension}$ is synthesized as a rank-$\rank$ matrix, factorized as $\bestsignal:=\mathbf{L}\mathbf{R}^T$ with $\vectornormbig{\bestsignal}_F = 1$ where $\mathbf{L} \in \mathbb{R}^{m \times \rank} $ and $\mathbf{R} \in \mathbb{R}^{n \times \rank} $ as defined above. In sequence, we subsample $\bestsignal$ by observing $\numsam = 0.3 mn $ entries, drawn uniformly at random. We denote the set of ordered pairs that represent the coordinates of the observable entries as $ \Omega = \lbrace (i,j): [\bestsignal]_{ij} \text{ is known} \rbrace \subseteq \lbrace 1, \dots, m \rbrace \times \lbrace 1, \dots, n \rbrace $ and let $\boldsymbol{\mathcal{A}}_{\Omega}$ denote the linear operator (mask) that samples a matrix according to $\Omega$. Then, the set of observations satisfies:
\begin{align}
\obs = \boldsymbol{\mathcal{A}}_{\Omega}\bestsignal + \noise, \label{eq:RMC}
\end{align} i.e., the known entries of $\bestsignal$ are structured as a vector $\obs \in \mathbb{R}^{\numsam}$, disturbed by a dense noise vector $\noise \in \mathbb{R}^{\numsam}$ with fixed-energy, which is populated by iid zero-mean Gaussians. 

To demonstrate the reconstruction accuracy and the convergence speeds, we generate various problem configurations (both noisy and noiseless settings), according to (\ref{eq:RMC}). The energy of the additive noise takes values $\vectornormbig{\noise}_2 \in \lbrace 10^{-3}, $ $10^{-4} \rbrace $. All the algorithms are tested for the same signal-matrix-noise realizations. A summary of the results can be found in Tables \ref{table:2}, \ref{table:3} and, \ref{table:4} where we present the median values of reconstruction error, number of iterations and execution time over 50 Monte Carlo iterations. For all cases, we assume $\text{SR} = 0.3 $ and set the maximum number of iterations to 700. Bold font denotes the fastest execution time. Some convergence error curves for specific cases are illustrated in Figures \ref{fig: TableII_fig} and \ref{fig: TableIII_IV_fig}.

In Table \ref{table:2}, LMaFit \cite{LMatFit} implementation has the fastest convergence for small scale problem configuration where $m = 300 $ and $n = 600$. We note that part of LMaFit implementation uses C code for acceleration. GROUSE \cite{GROUSE} is a competitive low-rank recovery method with attractive execution times for the {\it extreme low rank} problem settings due to stochastic gradient descent techniques. Nevertheless, its execution time performance degrades significantly as we increase the rank of $\bestsignal$. Moreover, we observe how randomized low rank projections accelerate the convergence speed where \textsc{Matrix ALPS II} with QR converges faster than \textsc{Matrix ALPS II}. In Tables \ref{table:3} and \ref{table:4}, we increase the problem dimensions. Here, \textsc{Matrix ALPS II} with QR has faster convergence for most of the cases and scales well as the problem size increases. We note that we do not exploit stochastic gradient descent techniques in the recovery process to accelerate convergence which is left for future work.

\begin{table*}[!htp]
\caption{Matrix Completion problem for $m = 300$ and $n = 600$. ``$-$'' depicts no information or not applicable due to time overhead.} {\label{table:2}}
\begin{center}
\begin{tabular}{|c|c|c|c|c|c|c|c|c|c|c|c|c|c|}
\multicolumn{4}{c|}{Configuration} & FR & \multicolumn{3}{|c|}{SVP} & \multicolumn{3}{|c|}{GROUSE} & \multicolumn{3}{|c}{TFOCS} \\
\hline \hline
\multicolumn{1}{c}{$m$} & \multicolumn{1}{c}{$n$} & \multicolumn{1}{c}{$\rank$} & \multicolumn{1}{c|}{$\vectornormbig{\noise}_2$} & & 
\multicolumn{1}{|c}{\rm{iter.}} & \multicolumn{1}{c}{\rm{err.}} & \multicolumn{1}{c|}{\rm{time}} &
\multicolumn{1}{|c}{\rm{iter.}} & \multicolumn{1}{c}{\rm{err.}} & \multicolumn{1}{c|}{\rm{time}} &
\multicolumn{1}{|c}{\rm{iter.}} & \multicolumn{1}{c}{\rm{err.}} & \multicolumn{1}{c}{\rm{time}} \\
\hline\hline
\multicolumn{1}{c}{$300$} & \multicolumn{1}{c}{$600$} & \multicolumn{1}{c}{$5$} & \multicolumn{1}{c|}{$0$} & $0.083$ & 
\multicolumn{1}{|c}{$43$} & \multicolumn{1}{c}{$2.9\cdot 10^{-4}$} & \multicolumn{1}{c|}{$0.59$} &
\multicolumn{1}{|c}{$-$} & \multicolumn{1}{c}{$1.52\cdot10^{-4}$} & \multicolumn{1}{c|}{$0.08$} &
\multicolumn{1}{|c}{$-$} & \multicolumn{1}{c}{$8.69\cdot10^{-5}$} & \multicolumn{1}{c}{$3.36$} \\
\hline
\multicolumn{1}{c}{$300$} & \multicolumn{1}{c}{$600$} & \multicolumn{1}{c}{$5$} & \multicolumn{1}{c|}{$10^{-3}$} & $0.083$ & 
\multicolumn{1}{|c}{$42$} & \multicolumn{1}{c}{$6\cdot 10^{-4}$} & \multicolumn{1}{c|}{$0.65$} &
\multicolumn{1}{|c}{$-$} & \multicolumn{1}{c}{$2\cdot10^{-4}$} & \multicolumn{1}{c|}{$0.082$} &
\multicolumn{1}{|c}{$-$} & \multicolumn{1}{c}{$5\cdot10^{-4}$} & \multicolumn{1}{c}{$3.85$} \\
\hline
\multicolumn{1}{c}{$300$} & \multicolumn{1}{c}{$600$} & \multicolumn{1}{c}{$5$} & \multicolumn{1}{c|}{$10^{-4}$} & $0.083$ & 
\multicolumn{1}{|c}{$43$} & \multicolumn{1}{c}{$3\cdot 10^{-4}$} & \multicolumn{1}{c|}{$0.64$} &
\multicolumn{1}{|c}{$-$} & \multicolumn{1}{c}{$2\cdot10^{-4}$} & \multicolumn{1}{c|}{$0.079$} &
\multicolumn{1}{|c}{$-$} & \multicolumn{1}{c}{$1\cdot10^{-4}$} & \multicolumn{1}{c}{$3.5$} \\
\hline
\multicolumn{1}{c}{$300$} & \multicolumn{1}{c}{$600$} & \multicolumn{1}{c}{$10$} & \multicolumn{1}{c|}{$0$} & $0.165$ & 
\multicolumn{1}{|c}{$54$} & \multicolumn{1}{c}{$4\cdot 10^{-4}$} & \multicolumn{1}{c|}{$0.9$} &
\multicolumn{1}{|c}{$-$} & \multicolumn{1}{c}{$4.5\cdot10^{-6}$} & \multicolumn{1}{c|}{$0.22$} &
\multicolumn{1}{|c}{$-$} & \multicolumn{1}{c}{$2\cdot10^{-4}$} & \multicolumn{1}{c}{$6.43$} \\
\hline
\multicolumn{1}{c}{$300$} & \multicolumn{1}{c}{$600$} & \multicolumn{1}{c}{$10$} & \multicolumn{1}{c|}{$10^{-3}$} & $0.165$ & 
\multicolumn{1}{|c}{$54$} & \multicolumn{1}{c}{$9\cdot 10^{-4}$} & \multicolumn{1}{c|}{$0.89$} &
\multicolumn{1}{|c}{$-$} & \multicolumn{1}{c}{$2\cdot10^{-4}$} & \multicolumn{1}{c|}{$0.16$} &
\multicolumn{1}{|c}{$-$} & \multicolumn{1}{c}{$8\cdot10^{-4}$} & \multicolumn{1}{c}{$7.83$} \\
\hline
\multicolumn{1}{c}{$300$} & \multicolumn{1}{c}{$600$} & \multicolumn{1}{c}{$10$} & \multicolumn{1}{c|}{$10^{-4}$} & $0.165$ & 
\multicolumn{1}{|c}{$54$} & \multicolumn{1}{c}{$4\cdot 10^{-4}$} & \multicolumn{1}{c|}{$0.91$} &
\multicolumn{1}{|c}{$-$} & \multicolumn{1}{c}{$2\cdot10^{-4}$} & \multicolumn{1}{c|}{$0.16$} &
\multicolumn{1}{|c}{$-$} & \multicolumn{1}{c}{$1\cdot10^{-4}$} & \multicolumn{1}{c}{$6.75$} \\
\hline
\multicolumn{1}{c}{$300$} & \multicolumn{1}{c}{$600$} & \multicolumn{1}{c}{$20$} & \multicolumn{1}{c|}{$0$} & $0.326$ & 
\multicolumn{1}{|c}{$85$} & \multicolumn{1}{c}{$8\cdot 10^{-4}$} & \multicolumn{1}{c|}{$2.04$} &
\multicolumn{1}{|c}{$-$} & \multicolumn{1}{c}{$1\cdot10^{-4}$} & \multicolumn{1}{c|}{$0.81$} &
\multicolumn{1}{|c}{$-$} & \multicolumn{1}{c}{$2\cdot10^{-4}$} & \multicolumn{1}{c}{$30.04$} \\
\hline
\multicolumn{1}{c}{$300$} & \multicolumn{1}{c}{$600$} & \multicolumn{1}{c}{$40$} & \multicolumn{1}{c|}{$0$} & $0.637$ & 
\multicolumn{1}{|c}{$241$} & \multicolumn{1}{c}{$3.4\cdot 10^{-3}$} & \multicolumn{1}{c|}{$11.1$} &
\multicolumn{1}{|c}{$-$} & \multicolumn{1}{c}{$3.1 \cdot 10^{-3}$} & \multicolumn{1}{c|}{$13.94$} &
\multicolumn{1}{|c}{$-$} & \multicolumn{1}{c}{$-$} & \multicolumn{1}{c}{$-$} \\

\hline \hline \hline
\multicolumn{4}{c|}{} & & \multicolumn{3}{|c|}{Inexact ALM} & \multicolumn{3}{|c|}{OptSpace} & \multicolumn{3}{|c}{GRASTA} \\
\hline \hline
\multicolumn{1}{c}{$m$} & \multicolumn{1}{c}{$n$} & \multicolumn{1}{c}{$\rank$} & \multicolumn{1}{c|}{$\vectornormbig{\noise}_2$} & & 
\multicolumn{1}{|c}{\rm{iter.}} & \multicolumn{1}{c}{\rm{err.}} & \multicolumn{1}{c|}{\rm{time}} &
\multicolumn{1}{|c}{\rm{iter.}} & \multicolumn{1}{c}{\rm{err.}} & \multicolumn{1}{c|}{\rm{time}} &
\multicolumn{1}{|c}{\rm{iter.}} & \multicolumn{1}{c}{\rm{err.}} & \multicolumn{1}{c}{\rm{time}} \\
\hline\hline
\multicolumn{1}{c}{$300$} & \multicolumn{1}{c}{$600$} & \multicolumn{1}{c}{$5$} & \multicolumn{1}{c|}{$0$} & $0.083$ & 
\multicolumn{1}{|c}{$24$} & \multicolumn{1}{c}{$6.7\cdot 10^{-5}$} & \multicolumn{1}{c|}{$0.47$} &
\multicolumn{1}{|c}{$31$} & \multicolumn{1}{c}{$2.8\cdot 10^{-6}$} & \multicolumn{1}{c|}{$2.41$} &
\multicolumn{1}{|c}{$-$} & \multicolumn{1}{c}{$2.2\cdot 10^{-4}$} & \multicolumn{1}{c}{$2.07$} \\
\hline
\multicolumn{1}{c}{$300$} & \multicolumn{1}{c}{$600$} & \multicolumn{1}{c}{$5$} & \multicolumn{1}{c|}{$10^{-3}$} & $0.083$ & 
\multicolumn{1}{|c}{$24$} & \multicolumn{1}{c}{$6\cdot 10^{-4}$} & \multicolumn{1}{c|}{$0.49$} &
\multicolumn{1}{|c}{$297$} & \multicolumn{1}{c}{$5\cdot 10^{-4}$} & \multicolumn{1}{c|}{$22.82$} &
\multicolumn{1}{|c}{$-$} & \multicolumn{1}{c}{$1\cdot 10^{-4}$} & \multicolumn{1}{c}{$2.07$} \\
\hline
\multicolumn{1}{c}{$300$} & \multicolumn{1}{c}{$600$} & \multicolumn{1}{c}{$5$} & \multicolumn{1}{c|}{$10^{-4}$} & $0.083$ & 
\multicolumn{1}{|c}{$24$} & \multicolumn{1}{c}{$1\cdot 10^{-4}$} & \multicolumn{1}{c|}{$0.49$} &
\multicolumn{1}{|c}{$267$} & \multicolumn{1}{c}{$1\cdot 10^{-4}$} & \multicolumn{1}{c|}{$21.56$} &
\multicolumn{1}{|c}{$-$} & \multicolumn{1}{c}{$8\cdot 10^{-5}$} & \multicolumn{1}{c}{$2.1$} \\
\hline
\multicolumn{1}{c}{$300$} & \multicolumn{1}{c}{$600$} & \multicolumn{1}{c}{$10$} & \multicolumn{1}{c|}{$0$} & $0.165$ & 
\multicolumn{1}{|c}{$26$} & \multicolumn{1}{c}{$1\cdot 10^{-4}$} & \multicolumn{1}{c|}{$0.6$} &
\multicolumn{1}{|c}{$37$} & \multicolumn{1}{c}{$2.3\cdot 10^{-6}$} & \multicolumn{1}{c|}{$8.42$} &
\multicolumn{1}{|c}{$-$} & \multicolumn{1}{c}{$8.6\cdot 10^{-6}$} & \multicolumn{1}{c}{$4.5$} \\
\hline
\multicolumn{1}{c}{$300$} & \multicolumn{1}{c}{$600$} & \multicolumn{1}{c}{$10$} & \multicolumn{1}{c|}{$10^{-3}$} & $0.165$ & 
\multicolumn{1}{|c}{$26$} & \multicolumn{1}{c}{$8\cdot 10^{-4}$} & \multicolumn{1}{c|}{$0.59$} &
\multicolumn{1}{|c}{$304$} & \multicolumn{1}{c}{$8\cdot 10^{-4}$} & \multicolumn{1}{c|}{$66.02$} &
\multicolumn{1}{|c}{$-$} & \multicolumn{1}{c}{$5.5\cdot 10^{-3}$} & \multicolumn{1}{c}{$3.43$} \\
\hline
\multicolumn{1}{c}{$300$} & \multicolumn{1}{c}{$600$} & \multicolumn{1}{c}{$10$} & \multicolumn{1}{c|}{$10^{-4}$} & $0.165$ & 
\multicolumn{1}{|c}{$26$} & \multicolumn{1}{c}{$1\cdot 10^{-4}$} & \multicolumn{1}{c|}{$0.61$} &
\multicolumn{1}{|c}{$304$} & \multicolumn{1}{c}{$1\cdot 10^{-4}$} & \multicolumn{1}{c|}{$65.56$} &
\multicolumn{1}{|c}{$-$} & \multicolumn{1}{c}{$5.3\cdot 10^{-3}$} & \multicolumn{1}{c}{$3.44$} \\
\hline
\multicolumn{1}{c}{$300$} & \multicolumn{1}{c}{$600$} & \multicolumn{1}{c}{$20$} & \multicolumn{1}{c|}{$0$} & $0.326$ & 
\multicolumn{1}{|c}{$44$} & \multicolumn{1}{c}{$3\cdot 10^{-4}$} & \multicolumn{1}{c|}{$1.37$} &
\multicolumn{1}{|c}{$-$} & \multicolumn{1}{c}{$-$} & \multicolumn{1}{c|}{$-$} &
\multicolumn{1}{|c}{$-$} & \multicolumn{1}{c}{$5\cdot 10^{-4}$} & \multicolumn{1}{c}{$10.51$} \\
\hline
\multicolumn{1}{c}{$300$} & \multicolumn{1}{c}{$600$} & \multicolumn{1}{c}{$40$} & \multicolumn{1}{c|}{$0$} & $0.637$ & 
\multicolumn{1}{|c}{$134$} & \multicolumn{1}{c}{$1.6\cdot 10^{-3}$} & \multicolumn{1}{c|}{$7.08$} &
\multicolumn{1}{|c}{$-$} & \multicolumn{1}{c}{$-$} & \multicolumn{1}{c|}{$-$} &
\multicolumn{1}{|c}{$-$} & \multicolumn{1}{c}{$5.2\cdot 10^{-3}$} & \multicolumn{1}{c}{$251.34$} \\

\hline \hline \hline
\multicolumn{4}{c|}{} & & \multicolumn{3}{|c|}{RTRMC} & \multicolumn{3}{|c|}{LMaFit} & \multicolumn{3}{|c}{\textsc{Matrix ALPS I}} \\
\hline \hline
\multicolumn{1}{c}{$m$} & \multicolumn{1}{c}{$n$} & \multicolumn{1}{c}{$\rank$} & \multicolumn{1}{c|}{$\vectornormbig{\noise}_2$} & & 
\multicolumn{1}{|c}{\rm{iter.}} & \multicolumn{1}{c}{\rm{err.}} & \multicolumn{1}{c|}{\rm{time}} &
\multicolumn{1}{|c}{\rm{iter.}} & \multicolumn{1}{c}{\rm{err.}} & \multicolumn{1}{c|}{\rm{time}} &
\multicolumn{1}{|c}{\rm{iter.}} & \multicolumn{1}{c}{\rm{err.}} & \multicolumn{1}{c}{\rm{time}} \\
\hline\hline
\multicolumn{1}{c}{$300$} & \multicolumn{1}{c}{$600$} & \multicolumn{1}{c}{$5$} & \multicolumn{1}{c|}{$0$} & $0.083$ & 
\multicolumn{1}{|c}{$13$} & \multicolumn{1}{c}{$1.2\cdot 10^{-4}$} & \multicolumn{1}{c|}{$0.59$} &
\multicolumn{1}{|c}{$20$} & \multicolumn{1}{c}{$2.2\cdot 10^{-4}$} & \multicolumn{1}{c|}{$\mathbf{0.054}$} &
\multicolumn{1}{|c}{$22$} & \multicolumn{1}{c}{$1.8\cdot 10^{-5}$} & \multicolumn{1}{c}{$0.76$} \\
\hline
\multicolumn{1}{c}{$300$} & \multicolumn{1}{c}{$600$} & \multicolumn{1}{c}{$5$} & \multicolumn{1}{c|}{$10^{-3}$} & $0.083$ & 
\multicolumn{1}{|c}{$13$} & \multicolumn{1}{c}{$1\cdot 10^{-4}$} & \multicolumn{1}{c|}{$0.59$} &
\multicolumn{1}{|c}{$19$} & \multicolumn{1}{c}{$5\cdot 10^{-4}$} & \multicolumn{1}{c|}{$\mathbf{0.049}$} &
\multicolumn{1}{|c}{$37$} & \multicolumn{1}{c}{$7\cdot 10^{-4}$} & \multicolumn{1}{c}{$1.34$} \\
\hline
\multicolumn{1}{c}{$300$} & \multicolumn{1}{c}{$600$} & \multicolumn{1}{c}{$5$} & \multicolumn{1}{c|}{$10^{-4}$} & $0.083$ & 
\multicolumn{1}{|c}{$13$} & \multicolumn{1}{c}{$2\cdot 10^{-4}$} & \multicolumn{1}{c|}{$0.59$} &
\multicolumn{1}{|c}{$21$} & \multicolumn{1}{c}{$1\cdot 10^{-4}$} & \multicolumn{1}{c|}{$\mathbf{0.052}$} &
\multicolumn{1}{|c}{$18$} & \multicolumn{1}{c}{$1\cdot 10^{-4}$} & \multicolumn{1}{c}{$0.61$} \\
\hline
\multicolumn{1}{c}{$300$} & \multicolumn{1}{c}{$600$} & \multicolumn{1}{c}{$10$} & \multicolumn{1}{c|}{$0$} & $0.165$ & 
\multicolumn{1}{|c}{$16$} & \multicolumn{1}{c}{$1.1\cdot 10^{-3}$} & \multicolumn{1}{c|}{$1.03$} &
\multicolumn{1}{|c}{$23$} & \multicolumn{1}{c}{$1\cdot 10^{-4}$} & \multicolumn{1}{c|}{$\mathbf{0.064}$} &
\multicolumn{1}{|c}{$16$} & \multicolumn{1}{c}{$1\cdot 10^{-4}$} & \multicolumn{1}{c}{$0.65$} \\
\hline
\multicolumn{1}{c}{$300$} & \multicolumn{1}{c}{$600$} & \multicolumn{1}{c}{$10$} & \multicolumn{1}{c|}{$10^{-3}$} & $0.165$ & 
\multicolumn{1}{|c}{$17$} & \multicolumn{1}{c}{$1\cdot 10^{-4}$} & \multicolumn{1}{c|}{$1.09$} &
\multicolumn{1}{|c}{$26$} & \multicolumn{1}{c}{$8\cdot 10^{-4}$} & \multicolumn{1}{c|}{$\mathbf{0.077}$} &
\multicolumn{1}{|c}{$30$} & \multicolumn{1}{c}{$1.1\cdot 10^{-3}$} & \multicolumn{1}{c}{$1.16$} \\
\hline
\multicolumn{1}{c}{$300$} & \multicolumn{1}{c}{$600$} & \multicolumn{1}{c}{$10$} & \multicolumn{1}{c|}{$10^{-4}$} & $0.165$ & 
\multicolumn{1}{|c}{$17$} & \multicolumn{1}{c}{$2\cdot 10^{-4}$} & \multicolumn{1}{c|}{$1.09$} &
\multicolumn{1}{|c}{$32$} & \multicolumn{1}{c}{$1\cdot 10^{-4}$} & \multicolumn{1}{c|}{$\mathbf{0.097}$} &
\multicolumn{1}{|c}{$16$} & \multicolumn{1}{c}{$1\cdot 10^{-4}$} & \multicolumn{1}{c}{$0.63$} \\
\hline
\multicolumn{1}{c}{$300$} & \multicolumn{1}{c}{$600$} & \multicolumn{1}{c}{$20$} & \multicolumn{1}{c|}{$0$} & $0.326$ & 
\multicolumn{1}{|c}{$22$} & \multicolumn{1}{c}{$4\cdot 10^{-4}$} & \multicolumn{1}{c|}{$2.99$} &
\multicolumn{1}{|c}{$37$} & \multicolumn{1}{c}{$2\cdot 10^{-4}$} & \multicolumn{1}{c|}{$\mathbf{0.12}$} &
\multicolumn{1}{|c}{$37$} & \multicolumn{1}{c}{$2\cdot 10^{-4}$} & \multicolumn{1}{c}{$2.05$} \\
\hline
\multicolumn{1}{c}{$300$} & \multicolumn{1}{c}{$600$} & \multicolumn{1}{c}{$40$} & \multicolumn{1}{c|}{$0$} & $0.637$ & 
\multicolumn{1}{|c}{$35$} & \multicolumn{1}{c}{$3 \cdot 10^{-5}$} & \multicolumn{1}{c|}{$11.83$} &
\multicolumn{1}{|c}{$233$} & \multicolumn{1}{c}{$4.9\cdot 10^{-4}$} & \multicolumn{1}{c|}{$\mathbf{2.52}$} &
\multicolumn{1}{|c}{$500$} & \multicolumn{1}{c}{$6.5\cdot 10^{-2}$} & \multicolumn{1}{c}{$45.67$} \\

\hline \hline \hline
\multicolumn{4}{c|}{} & & \multicolumn{3}{|c|}{ADMiRA} & \multicolumn{3}{|c|}{\textsc{Matrix ALPS II}} & \multicolumn{3}{|c}{\textsc{Matrix ALPS II} with QR} \\
\hline \hline
\multicolumn{1}{c}{$m$} & \multicolumn{1}{c}{$n$} & \multicolumn{1}{c}{$\rank$} & \multicolumn{1}{c|}{$\vectornormbig{\noise}_2$} & & 
\multicolumn{1}{|c}{\rm{iter.}} & \multicolumn{1}{c}{\rm{err.}} & \multicolumn{1}{c|}{\rm{time}} &
\multicolumn{1}{|c}{\rm{iter.}} & \multicolumn{1}{c}{\rm{err.}} & \multicolumn{1}{c|}{\rm{time}} &
\multicolumn{1}{|c}{\rm{iter.}} & \multicolumn{1}{c}{\rm{err.}} & \multicolumn{1}{c}{\rm{time}} \\
\hline\hline
\multicolumn{1}{c}{$300$} & \multicolumn{1}{c}{$600$} & \multicolumn{1}{c}{$5$} & \multicolumn{1}{c|}{$0$} & $0.083$ & 
\multicolumn{1}{|c}{$59$} & \multicolumn{1}{c}{$5.2\cdot 10^{-5}$} & \multicolumn{1}{c|}{$2.86$} &
\multicolumn{1}{|c}{$10$} & \multicolumn{1}{c}{$1.7\cdot 10^{-5}$} & \multicolumn{1}{c|}{$0.34$} &
\multicolumn{1}{|c}{$14$} & \multicolumn{1}{c}{$3.2\cdot 10^{-5}$} & \multicolumn{1}{c}{$0.45$} \\
\hline
\multicolumn{1}{c}{$300$} & \multicolumn{1}{c}{$600$} & \multicolumn{1}{c}{$5$} & \multicolumn{1}{c|}{$10^{-3}$} & $0.083$ & 
\multicolumn{1}{|c}{$700$} & \multicolumn{1}{c}{$4\cdot 10^{-3}$} & \multicolumn{1}{c|}{$30.96$} &
\multicolumn{1}{|c}{$12$} & \multicolumn{1}{c}{$6\cdot 10^{-4}$} & \multicolumn{1}{c|}{$0.44$} &
\multicolumn{1}{|c}{$24$} & \multicolumn{1}{c}{$6\cdot 10^{-4}$} & \multicolumn{1}{c}{$0.81$} \\
\hline
\multicolumn{1}{c}{$300$} & \multicolumn{1}{c}{$600$} & \multicolumn{1}{c}{$5$} & \multicolumn{1}{c|}{$10^{-4}$} & $0.083$ & 
\multicolumn{1}{|c}{$700$} & \multicolumn{1}{c}{$4.5\cdot 10^{-3}$} & \multicolumn{1}{c|}{$31.45$} &
\multicolumn{1}{|c}{$10$} & \multicolumn{1}{c}{$1\cdot 10^{-4}$} & \multicolumn{1}{c|}{$0.36$} &
\multicolumn{1}{|c}{$14$} & \multicolumn{1}{c}{$1\cdot 10^{-4}$} & \multicolumn{1}{c}{$0.47$} \\
\hline
\multicolumn{1}{c}{$300$} & \multicolumn{1}{c}{$600$} & \multicolumn{1}{c}{$10$} & \multicolumn{1}{c|}{$0$} & $0.165$ & 
\multicolumn{1}{|c}{$47$} & \multicolumn{1}{c}{$1\cdot 10^{-3}$} & \multicolumn{1}{c|}{$2.56$} &
\multicolumn{1}{|c}{$12$} & \multicolumn{1}{c}{$3\cdot 10^{-5}$} & \multicolumn{1}{c|}{$0.48$} &
\multicolumn{1}{|c}{$16$} & \multicolumn{1}{c}{$3.4\cdot 10^{-5}$} & \multicolumn{1}{c}{$0.49$} \\
\hline
\multicolumn{1}{c}{$300$} & \multicolumn{1}{c}{$600$} & \multicolumn{1}{c}{$10$} & \multicolumn{1}{c|}{$10^{-3}$} & $0.165$ & 
\multicolumn{1}{|c}{$700$} & \multicolumn{1}{c}{$1.5\cdot 10^{-3}$} & \multicolumn{1}{c|}{$28.49$} &
\multicolumn{1}{|c}{$19$} & \multicolumn{1}{c}{$9\cdot 10^{-4}$} & \multicolumn{1}{c|}{$0.74$} &
\multicolumn{1}{|c}{$29$} & \multicolumn{1}{c}{$9\cdot 10^{-4}$} & \multicolumn{1}{c}{$0.95$} \\
\hline
\multicolumn{1}{c}{$300$} & \multicolumn{1}{c}{$600$} & \multicolumn{1}{c}{$10$} & \multicolumn{1}{c|}{$10^{-4}$} & $0.165$ & 
\multicolumn{1}{|c}{$700$} & \multicolumn{1}{c}{$1\cdot 10^{-4}$} & \multicolumn{1}{c|}{$31.99$} &
\multicolumn{1}{|c}{$12$} & \multicolumn{1}{c}{$1\cdot 10^{-4}$} & \multicolumn{1}{c|}{$0.49$} &
\multicolumn{1}{|c}{$16$} & \multicolumn{1}{c}{$1\cdot 10^{-4}$} & \multicolumn{1}{c}{$0.54$} \\
\hline
\multicolumn{1}{c}{$300$} & \multicolumn{1}{c}{$600$} & \multicolumn{1}{c}{$20$} & \multicolumn{1}{c|}{$0$} & $0.326$ & 
\multicolumn{1}{|c}{$700$} & \multicolumn{1}{c}{$1.2\cdot 10^{-3}$} & \multicolumn{1}{c|}{$41.86$} &
\multicolumn{1}{|c}{$20$} & \multicolumn{1}{c}{$1\cdot 10^{-4}$} & \multicolumn{1}{c|}{$1.16$} &
\multicolumn{1}{|c}{$23$} & \multicolumn{1}{c}{$1\cdot 10^{-4}$} & \multicolumn{1}{c}{$0.79$} \\
\hline
\multicolumn{1}{c}{$300$} & \multicolumn{1}{c}{$600$} & \multicolumn{1}{c}{$20$} & \multicolumn{1}{c|}{$0$} & $0.326$ & 
\multicolumn{1}{|c}{$-$} & \multicolumn{1}{c}{$-$} & \multicolumn{1}{c|}{$-$} &
\multicolumn{1}{|c}{$72$} & \multicolumn{1}{c}{$2\cdot 10^{-4}$} & \multicolumn{1}{c|}{$7.21$} &
\multicolumn{1}{|c}{$68$} & \multicolumn{1}{c}{$2\cdot 10^{-4}$} & \multicolumn{1}{c}{$2.6$} \\

\hline
\end{tabular}
\end{center}
\end{table*}

\begin{figure*}[!htp]
\centering
\begin{tabular}{ccc}
\centerline{\subfigure[]{\includegraphics[width = 0.4\textwidth]{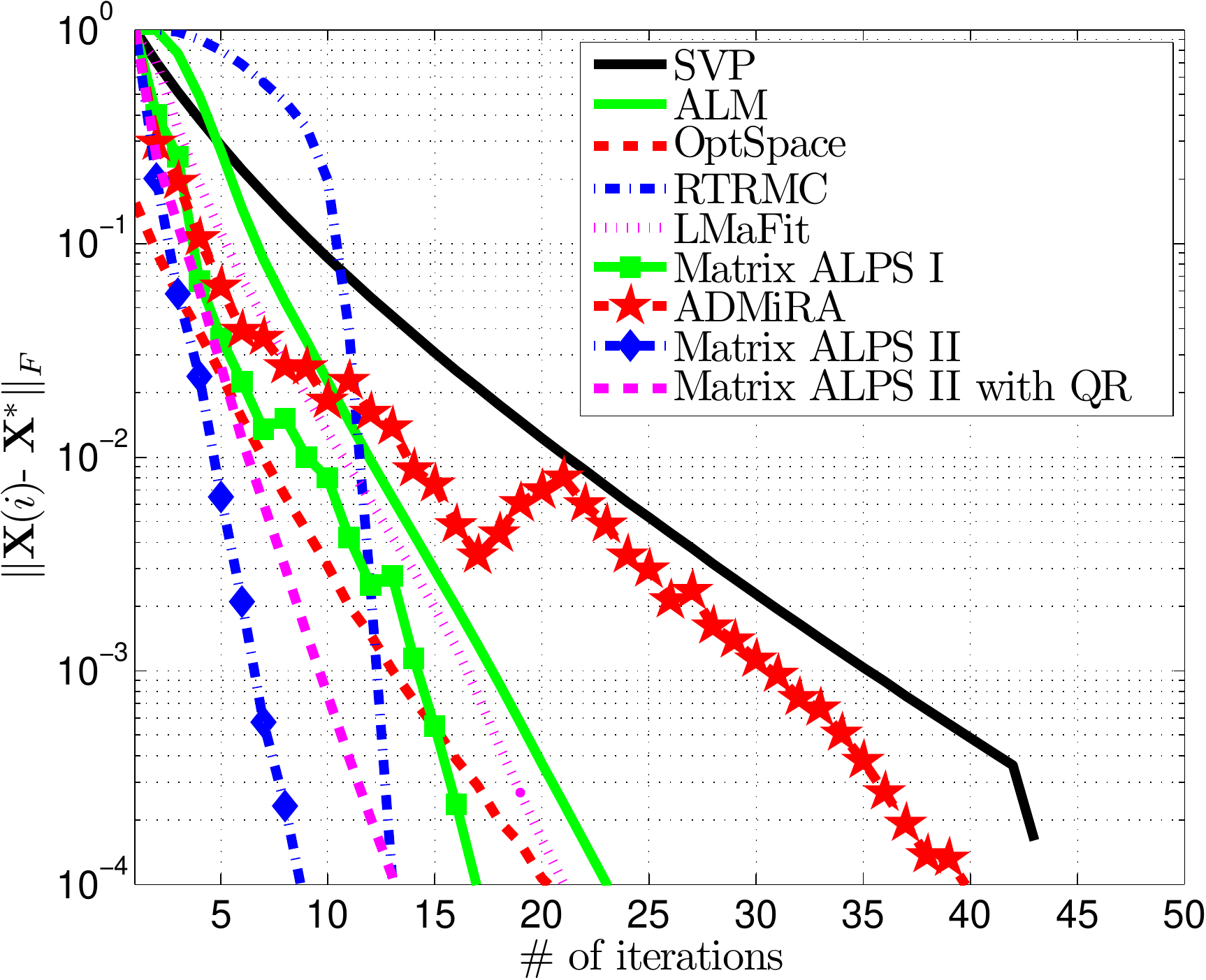}\label{fig:2a}} 
\hspace{1cm}
\subfigure[]{\includegraphics[width = 0.4\textwidth]{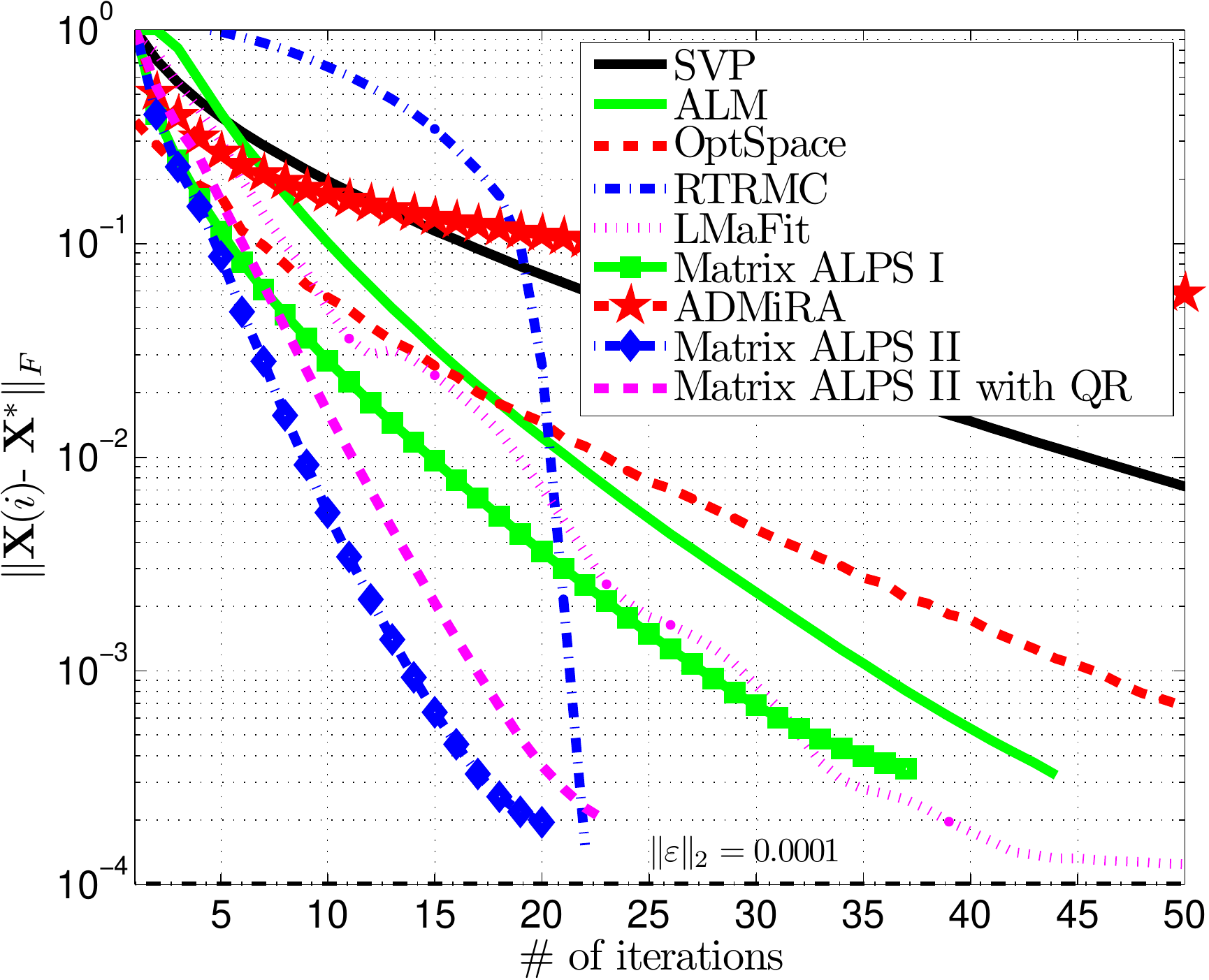}\label{fig:2b}}}
\end{tabular}
\caption{\small\sl Low rank matrix recovery for the matrix completion problem. The error curves are the median values across 50 Monte-Carlo realizations over each iteration. For all cases, we assume $p = 0.3mn$. (a) $m = 300$, $n = 600 $, $\rank = 5$ and $\vectornormbig{\noise}_2 = 0 $. (b) $m = 300$, $n = 600 $, $\rank = 20$ and $\vectornormbig{\noise}_2 = 10^{-4} $. } \label{fig: TableII_fig}
\end{figure*}

\begin{table*} [!htp]
\caption{Matrix Completion problem for $m = 700$ and $n = 1000$. ``$-$'' depicts no information or not applicable due to time overhead.} {\label{table:3}}
\begin{center}
\begin{tabular}{|c|c|c|c|c|c|c|c|c|c|c|c|c|c}
\multicolumn{4}{c|}{Configuration} & FR & \multicolumn{3}{|c|}{SVP} & \multicolumn{3}{|c|}{Inexact ALM} & \multicolumn{3}{|c}{GROUSE} \\
\hline \hline
\multicolumn{1}{c}{$m$}  & \multicolumn{1}{c}{$n$} & \multicolumn{1}{c}{$\rank$} & \multicolumn{1}{c|}{$\vectornormbig{\noise}_2$} & & 
\multicolumn{1}{|c}{\rm{iter.}} & \multicolumn{1}{c}{\rm{err.}} & \multicolumn{1}{c|}{\rm{time}} &
\multicolumn{1}{|c}{\rm{iter.}} & \multicolumn{1}{c}{\rm{err.}} & \multicolumn{1}{c|}{\rm{time}} &
\multicolumn{1}{|c}{\rm{iter.}} & \multicolumn{1}{c}{\rm{err.}} & \multicolumn{1}{c}{\rm{time}} \\
\hline\hline
\multicolumn{1}{c}{$700$} & \multicolumn{1}{c}{$1000$} & \multicolumn{1}{c}{$5$} & \multicolumn{1}{c|}{$0$} & $ 0.04 $ & 
\multicolumn{1}{|c}{$34$} & \multicolumn{1}{c}{$1.9\cdot 10^{-4}$} & \multicolumn{1}{c|}{$1.77$} &
\multicolumn{1}{|c}{$23$} & \multicolumn{1}{c}{$6.5\cdot 10^{-5}$} & \multicolumn{1}{c|}{$1.69$} &
\multicolumn{1}{|c}{$-$} & \multicolumn{1}{c}{$3.5\cdot 10^{-5}$} & \multicolumn{1}{c}{$\mathbf{0.23}$} \\
\hline
\multicolumn{1}{c}{$700$} & \multicolumn{1}{c}{$1000$} & \multicolumn{1}{c}{$5$} & \multicolumn{1}{c|}{$10{-3}$} & $ 0.04 $ & 
\multicolumn{1}{|c}{$34$} & \multicolumn{1}{c}{$4.2\cdot 10^{-4}$} & \multicolumn{1}{c|}{$1.92$} &
\multicolumn{1}{|c}{$23$} & \multicolumn{1}{c}{$3.7\cdot 10^{-4}$} & \multicolumn{1}{c|}{$1.87$} &
\multicolumn{1}{|c}{$-$} & \multicolumn{1}{c}{$3.1\cdot 10^{-4}$} & \multicolumn{1}{c}{$\mathbf{0.24}$} \\
\hline
\multicolumn{1}{c}{$700$} & \multicolumn{1}{c}{$1000$} & \multicolumn{1}{c}{$30$} & \multicolumn{1}{c|}{$0$} & $0.239 $ & 
\multicolumn{1}{|c}{$61$} & \multicolumn{1}{c}{$4.6\cdot 10^{-4}$} & \multicolumn{1}{c|}{$6.39$} &
\multicolumn{1}{|c}{$29$} & \multicolumn{1}{c}{$1.2\cdot 10^{-4}$} & \multicolumn{1}{c|}{$3.91$} &
\multicolumn{1}{|c}{$-$} & \multicolumn{1}{c}{$3.2\cdot 10^{-5}$} & \multicolumn{1}{c}{$3.15$} \\
\hline
\multicolumn{1}{c}{$700$} & \multicolumn{1}{c}{$1000$} & \multicolumn{1}{c}{$30$} & \multicolumn{1}{c|}{$10^{-3}$} & $0.239 $ & 
\multicolumn{1}{|c}{$61$} & \multicolumn{1}{c}{$1.1\cdot 10^{-3}$} & \multicolumn{1}{c|}{$6.33$} &
\multicolumn{1}{|c}{$29$} & \multicolumn{1}{c}{$1\cdot 10^{-3}$} & \multicolumn{1}{c|}{$3.87$} &
\multicolumn{1}{|c}{$-$} & \multicolumn{1}{c}{$8\cdot 10^{-4}$} & \multicolumn{1}{c}{$3.14$} \\
\hline
\multicolumn{1}{c}{$700$} & \multicolumn{1}{c}{$1000$} & \multicolumn{1}{c}{$50$} & \multicolumn{1}{c|}{$0$} & $0.393$ & 
\multicolumn{1}{|c}{$95$} & \multicolumn{1}{c}{$8.5\cdot 10^{-4}$} & \multicolumn{1}{c|}{$14.47$} &
\multicolumn{1}{|c}{$49$} & \multicolumn{1}{c}{$3.2\cdot 10^{-4}$} & \multicolumn{1}{c|}{$9.02$} &
\multicolumn{1}{|c}{$-$} & \multicolumn{1}{c}{$1.3\cdot 10^{-5}$} & \multicolumn{1}{c}{$10.31$} \\
\hline
\multicolumn{1}{c}{$700$} & \multicolumn{1}{c}{$1000$} & \multicolumn{1}{c}{$50$} & \multicolumn{1}{c|}{$10^{-3}$} & $0.393$ & 
\multicolumn{1}{|c}{$95$} & \multicolumn{1}{c}{$1.6\cdot 10^{-3}$} & \multicolumn{1}{c|}{$15.15$} &
\multicolumn{1}{|c}{$49$} & \multicolumn{1}{c}{$1.4\cdot 10^{-3}$} & \multicolumn{1}{c|}{$9.11$} &
\multicolumn{1}{|c}{$-$} & \multicolumn{1}{c}{$8\cdot 10^{-4}$} & \multicolumn{1}{c}{$10.34$} \\
\hline
\multicolumn{1}{c}{$700$} & \multicolumn{1}{c}{$1000$} & \multicolumn{1}{c}{$110$} & \multicolumn{1}{c|}{$0$} & $0.833$ & 
\multicolumn{1}{|c}{$683$} & \multicolumn{1}{c}{$1.2\cdot 10^{-2}$} & \multicolumn{1}{c|}{$253.1$} &
\multicolumn{1}{|c}{$374$} & \multicolumn{1}{c}{$5.8\cdot 10^{-3}$} & \multicolumn{1}{c|}{$152.61$} &
\multicolumn{1}{|c}{$-$} & \multicolumn{1}{c}{$1.2 \cdot 10^{-1}$} & \multicolumn{1}{c}{$110.93$} \\
\hline
\multicolumn{1}{c}{$700$} & \multicolumn{1}{c}{$1000$} & \multicolumn{1}{c}{$110$} & \multicolumn{1}{c|}{$10^{-3}$} & $0.833$ & 
\multicolumn{1}{|c}{$682$} & \multicolumn{1}{c}{$1.3\cdot 10^{-2}$} & \multicolumn{1}{c|}{$256.21$} &
\multicolumn{1}{|c}{$374$} & \multicolumn{1}{c}{$6.8\cdot 10^{-3}$} & \multicolumn{1}{c|}{$154.34$} &
\multicolumn{1}{|c}{$-$} & \multicolumn{1}{c}{$1.05 \cdot 10^{-1}$} & \multicolumn{1}{c}{$111.05$} \\
\hline  \hline \hline

\multicolumn{4}{c|}{} & & \multicolumn{3}{|c|}{LMaFit} & \multicolumn{3}{|c}{\textsc{Matrix ALPS II}} & \multicolumn{3}{|c}{\textsc{Matrix ALPS II} with QR}  \\
\hline \hline
\multicolumn{1}{c}{$m$}  & \multicolumn{1}{c}{$n$} & \multicolumn{1}{c}{$\rank$} & \multicolumn{1}{c|}{$\vectornormbig{\noise}_2$} & & 
\multicolumn{1}{|c}{\rm{iter.}} & \multicolumn{1}{c}{\rm{err.}} & \multicolumn{1}{c|}{\rm{time}} & 
\multicolumn{1}{|c}{\rm{iter.}} & \multicolumn{1}{c}{\rm{err.}} & \multicolumn{1}{c|}{\rm{time}} &
\multicolumn{1}{|c}{\rm{iter.}} & \multicolumn{1}{c}{\rm{err.}} & \multicolumn{1}{c}{\rm{time}}  \\
\hline
\multicolumn{1}{c}{$700$} & \multicolumn{1}{c}{$1000$} & \multicolumn{1}{c}{$5$}  & \multicolumn{1}{c|}{$0$} & $ 0.04 $ & 
\multicolumn{1}{|c}{$24$} & \multicolumn{1}{c}{$7.2 \cdot 10^{-6}$} & \multicolumn{1}{c|}{$0.67$} &
\multicolumn{1}{|c}{$8$} & \multicolumn{1}{c}{$1.5 \cdot 10^{-5}$} & \multicolumn{1}{c}{$1.15$} &
\multicolumn{1}{|c}{$15$} & \multicolumn{1}{c}{$8.3 \cdot 10^{-5}$} & \multicolumn{1}{c}{$1.05$} \\
\hline
\multicolumn{1}{c}{$700$} & \multicolumn{1}{c}{$1000$} & \multicolumn{1}{c}{$5$}  & \multicolumn{1}{c|}{$10^{-3}$} & $ 0.04 $ & 
\multicolumn{1}{|c}{$17$} & \multicolumn{1}{c}{$3.7 \cdot 10^{-4}$} & \multicolumn{1}{c|}{$0.5$} &
\multicolumn{1}{|c}{$10$} & \multicolumn{1}{c}{$4.5 \cdot 10^{-4}$} & \multicolumn{1}{c}{$1.38$} &
\multicolumn{1}{|c}{$15$} & \multicolumn{1}{c}{$3.8 \cdot 10^{-4}$} & \multicolumn{1}{c}{$1.1$} \\
\hline
\multicolumn{1}{c}{$700$} & \multicolumn{1}{c}{$1000$} & \multicolumn{1}{c}{$30$}  & \multicolumn{1}{c|}{$0$} & $ 0.239 $ & 
\multicolumn{1}{|c}{$34$} & \multicolumn{1}{c}{$9.2 \cdot 10^{-6}$} & \multicolumn{1}{c|}{$\mathbf{1.95}$} &
\multicolumn{1}{|c}{$14$} & \multicolumn{1}{c}{$4.5 \cdot 10^{-5}$} & \multicolumn{1}{c}{$3.69$} &
\multicolumn{1}{|c}{$35$} & \multicolumn{1}{c}{$1.1 \cdot 10^{-4}$} & \multicolumn{1}{c}{$2.6$} \\
\hline
\multicolumn{1}{c}{$700$} & \multicolumn{1}{c}{$1000$} & \multicolumn{1}{c}{$30$}  & \multicolumn{1}{c|}{$10^{-3}$} & $ 0.239 $ & 
\multicolumn{1}{|c}{$30$} & \multicolumn{1}{c}{$1 \cdot 10^{-3}$} & \multicolumn{1}{c|}{$\mathbf{1.71}$} &
\multicolumn{1}{|c}{$25$} & \multicolumn{1}{c}{$1.1 \cdot 10^{-3}$} & \multicolumn{1}{c}{$6.1$} &
\multicolumn{1}{|c}{$35$} & \multicolumn{1}{c}{$1 \cdot 10^{-3}$} & \multicolumn{1}{c}{$2.61$} \\
\hline
\multicolumn{1}{c}{$700$} & \multicolumn{1}{c}{$1000$} & \multicolumn{1}{c}{$50$}  & \multicolumn{1}{c|}{$0$} & $ 0.393 $ & 
\multicolumn{1}{|c}{$53$} & \multicolumn{1}{c}{$2.7 \cdot 10^{-5}$} & \multicolumn{1}{c|}{$4.59$} &
\multicolumn{1}{|c}{$25$} & \multicolumn{1}{c}{$8.6 \cdot 10^{-5}$} & \multicolumn{1}{c}{$8.87$} &
\multicolumn{1}{|c}{$57$} & \multicolumn{1}{c}{$1.6 \cdot 10^{-5}$} & \multicolumn{1}{c}{$\mathbf{4.47}$} \\
\hline
\multicolumn{1}{c}{$700$} & \multicolumn{1}{c}{$1000$} & \multicolumn{1}{c}{$50$}  & \multicolumn{1}{c|}{$10^{-3}$} & $ 0.393 $ & 
\multicolumn{1}{|c}{$52$} & \multicolumn{1}{c}{$1.4 \cdot 10^{-3}$} & \multicolumn{1}{c|}{$4.53$} &
\multicolumn{1}{|c}{$40$} & \multicolumn{1}{c}{$1.6 \cdot 10^{-3}$} & \multicolumn{1}{c}{$14.38$} &
\multicolumn{1}{|c}{$57$} & \multicolumn{1}{c}{$1.4 \cdot 10^{-3}$} & \multicolumn{1}{c}{$\mathbf{4.49}$} \\
\hline
\multicolumn{1}{c}{$700$} & \multicolumn{1}{c}{$1000$} & \multicolumn{1}{c}{$110$}  & \multicolumn{1}{c|}{$0$} & $ 0.833 $ & 
\multicolumn{1}{|c}{$584$} & \multicolumn{1}{c}{$9 \cdot 10^{-4}$} & \multicolumn{1}{c|}{$101.95$} &
\multicolumn{1}{|c}{$280$} & \multicolumn{1}{c}{$8 \cdot 10^{-4}$} & \multicolumn{1}{c}{$214.93$} &
\multicolumn{1}{|c}{$553$} & \multicolumn{1}{c}{$7 \cdot 10^{-4}$} & \multicolumn{1}{c}{$\mathbf{51.72}$} \\
\hline
\multicolumn{1}{c}{$700$} & \multicolumn{1}{c}{$1000$} & \multicolumn{1}{c}{$110$}  & \multicolumn{1}{c|}{$10^{-3}$} & $ 0.833 $ & 
\multicolumn{1}{|c}{$584$} & \multicolumn{1}{c}{$3.7 \cdot 10^{-3}$} & \multicolumn{1}{c|}{$102.15$} &
\multicolumn{1}{|c}{$336$} & \multicolumn{1}{c}{$4.7 \cdot 10^{-3}$} & \multicolumn{1}{c}{$261.98$} &
\multicolumn{1}{|c}{$551$} & \multicolumn{1}{c}{$3.7 \cdot 10^{-3}$} & \multicolumn{1}{c}{$\mathbf{51.62}$} \\
\hline
\end{tabular}
\end{center}
\end{table*}

\begin{table*} [!htp]
\caption{Matrix Completion problem for $m = 500$ and $n = 2000$. ``$-$'' depicts no information or not applicable due to time overhead.} {\label{table:4}}
\begin{center}
\begin{tabular}{|c|c|c|c|c|c|c|c|c|c|c|c|c|c}
\multicolumn{4}{c|}{Configuration} & FR & \multicolumn{3}{|c|}{SVP} & \multicolumn{3}{|c|}{Inexact ALM} & \multicolumn{3}{|c}{GROUSE} \\
\hline \hline
\multicolumn{1}{c}{$m$}  & \multicolumn{1}{c}{$n$} & \multicolumn{1}{c}{$\rank$} & \multicolumn{1}{c|}{$\vectornormbig{\noise}_2$} & & 
\multicolumn{1}{|c}{\rm{iter.}} & \multicolumn{1}{c}{\rm{err.}} & \multicolumn{1}{c|}{\rm{time}} &
\multicolumn{1}{|c}{\rm{iter.}} & \multicolumn{1}{c}{\rm{err.}} & \multicolumn{1}{c|}{\rm{time}} &
\multicolumn{1}{|c}{\rm{iter.}} & \multicolumn{1}{c}{\rm{err.}} & \multicolumn{1}{c}{\rm{time}} \\
\hline\hline
\multicolumn{1}{c}{$500$} & \multicolumn{1}{c}{$2000$} & \multicolumn{1}{c}{$30$} & \multicolumn{1}{c|}{$0$} & $ 0.083 $ & 
\multicolumn{1}{|c}{$64$} & \multicolumn{1}{c}{$5.3\cdot 10^{-4}$} & \multicolumn{1}{c|}{$10.18$} &
\multicolumn{1}{|c}{$32$} & \multicolumn{1}{c}{$1.9\cdot 10^{-4}$} & \multicolumn{1}{c|}{$6.47$} &
\multicolumn{1}{|c}{$-$} & \multicolumn{1}{c}{$1.6\cdot 10^{-4}$} & \multicolumn{1}{c}{$\mathbf{2.46}$} \\
\hline
\multicolumn{1}{c}{$500$} & \multicolumn{1}{c}{$2000$} & \multicolumn{1}{c}{$30$} & \multicolumn{1}{c|}{$10^{-3}$} & $ 0.083 $ & 
\multicolumn{1}{|c}{$64$} & \multicolumn{1}{c}{$1.1\cdot 10^{-3}$} & \multicolumn{1}{c|}{$6.69$} &
\multicolumn{1}{|c}{$32$} & \multicolumn{1}{c}{$1\cdot 10^{-3}$} & \multicolumn{1}{c|}{$4.51$} &
\multicolumn{1}{|c}{$-$} & \multicolumn{1}{c}{$6\cdot 10^{-4}$} & \multicolumn{1}{c}{$\mathbf{1.94}$} \\
\hline
\multicolumn{1}{c}{$500$} & \multicolumn{1}{c}{$2000$} & \multicolumn{1}{c}{$30$} & \multicolumn{1}{c|}{$10^{-4}$} & $ 0.083 $ & 
\multicolumn{1}{|c}{$64$} & \multicolumn{1}{c}{$5.4\cdot 10^{-4}$} & \multicolumn{1}{c|}{$10.14$} &
\multicolumn{1}{|c}{$32$} & \multicolumn{1}{c}{$2.2\cdot 10^{-4}$} & \multicolumn{1}{c|}{$6.51$} &
\multicolumn{1}{|c}{$-$} & \multicolumn{1}{c}{$1.6\cdot 10^{-4}$} & \multicolumn{1}{c}{$\mathbf{2.46}$} \\
\hline
\multicolumn{1}{c}{$500$} & \multicolumn{1}{c}{$2000$} & \multicolumn{1}{c}{$50$} & \multicolumn{1}{c|}{$0$} & $ 0.408 $ & 
\multicolumn{1}{|c}{$103$} & \multicolumn{1}{c}{$1.1\cdot 10^{-4}$} & \multicolumn{1}{c|}{$15.74$} &
\multicolumn{1}{|c}{$54$} & \multicolumn{1}{c}{$5\cdot 10^{-4}$} & \multicolumn{1}{c|}{$10.8$} &
\multicolumn{1}{|c}{$-$} & \multicolumn{1}{c}{$8\cdot 10^{-5}$} & \multicolumn{1}{c}{$7.32$} \\
\hline
\multicolumn{1}{c}{$500$} & \multicolumn{1}{c}{$2000$} & \multicolumn{1}{c}{$50$} & \multicolumn{1}{c|}{$10^{-3}$} & $ 0.408 $ & 
\multicolumn{1}{|c}{$103$} & \multicolumn{1}{c}{$1.8\cdot 10^{-3}$} & \multicolumn{1}{c|}{$24.97$} &
\multicolumn{1}{|c}{$54$} & \multicolumn{1}{c}{$1.55\cdot 10^{-3}$} & \multicolumn{1}{c|}{$16.14$} &
\multicolumn{1}{|c}{$-$} & \multicolumn{1}{c}{$9\cdot 10^{-4}$} & \multicolumn{1}{c}{$8.6$} \\
\hline
\multicolumn{1}{c}{$500$} & \multicolumn{1}{c}{$2000$} & \multicolumn{1}{c}{$50$} & \multicolumn{1}{c|}{$10^{-4}$} & $ 0.408 $ & 
\multicolumn{1}{|c}{$102$} & \multicolumn{1}{c}{$1.1\cdot 10^{-3}$} & \multicolumn{1}{c|}{$24.85$} &
\multicolumn{1}{|c}{$54$} & \multicolumn{1}{c}{$5\cdot 10^{-4}$} & \multicolumn{1}{c|}{$16.17$} &
\multicolumn{1}{|c}{$-$} & \multicolumn{1}{c}{$7\cdot 10^{-5}$} & \multicolumn{1}{c}{$8.59$} \\
\hline
\multicolumn{1}{c}{$500$} & \multicolumn{1}{c}{$2000$} & \multicolumn{1}{c}{$80$} & \multicolumn{1}{c|}{$0$} & $ 0.645 $ & 
\multicolumn{1}{|c}{$239$} & \multicolumn{1}{c}{$3.5\cdot 10^{-3}$} & \multicolumn{1}{c|}{$92.91$} &
\multicolumn{1}{|c}{$134$} & \multicolumn{1}{c}{$1.7\cdot 10^{-3}$} & \multicolumn{1}{c|}{$59.33$} &
\multicolumn{1}{|c}{$-$} & \multicolumn{1}{c}{$1\cdot 10^{-4}$} & \multicolumn{1}{c}{$79.64$} \\
\hline
\multicolumn{1}{c}{$500$} & \multicolumn{1}{c}{$2000$} & \multicolumn{1}{c}{$80$} & \multicolumn{1}{c|}{$10^{-3}$} & $ 0.645 $ & 
\multicolumn{1}{|c}{$239$} & \multicolumn{1}{c}{$4.2\cdot 10^{-3}$} & \multicolumn{1}{c|}{$94.86$} &
\multicolumn{1}{|c}{$134$} & \multicolumn{1}{c}{$2.8\cdot 10^{-3}$} & \multicolumn{1}{c|}{$60.68$} &
\multicolumn{1}{|c}{$-$} & \multicolumn{1}{c}{$1\cdot 10^{-4}$} & \multicolumn{1}{c}{$79.98$} \\
\hline
\multicolumn{1}{c}{$500$} & \multicolumn{1}{c}{$2000$} & \multicolumn{1}{c}{$80$} & \multicolumn{1}{c|}{$10^{-4}$} & $ 0.645 $ & 
\multicolumn{1}{|c}{$239$} & \multicolumn{1}{c}{$3.6\cdot 10^{-3}$} & \multicolumn{1}{c|}{$93.95$} &
\multicolumn{1}{|c}{$134$} & \multicolumn{1}{c}{$1.8\cdot 10^{-3}$} & \multicolumn{1}{c|}{$60.76$} &
\multicolumn{1}{|c}{$-$} & \multicolumn{1}{c}{$1\cdot 10^{-4}$} & \multicolumn{1}{c}{$79.48$} \\
\hline
\multicolumn{1}{c}{$500$} & \multicolumn{1}{c}{$2000$} & \multicolumn{1}{c}{$100$} & \multicolumn{1}{c|}{$0$} & $ 0.8 $ & 
\multicolumn{1}{|c}{$523$} & \multicolumn{1}{c}{$1.1\cdot 10^{-2}$} & \multicolumn{1}{c|}{$259.13$} &
\multicolumn{1}{|c}{$307$} & \multicolumn{1}{c}{$6\cdot 10^{-3}$} & \multicolumn{1}{c|}{$173.14$} &
\multicolumn{1}{|c}{$-$} & \multicolumn{1}{c}{$4.5\cdot 10^{-2}$} & \multicolumn{1}{c}{$143.41$} \\
\hline
\multicolumn{1}{c}{$500$} & \multicolumn{1}{c}{$2000$} & \multicolumn{1}{c}{$100$} & \multicolumn{1}{c|}{$10^{-3}$} & $ 0.8 $ & 
\multicolumn{1}{|c}{$525$} & \multicolumn{1}{c}{$1.2\cdot 10^{-2}$} & \multicolumn{1}{c|}{$262.19$} &
\multicolumn{1}{|c}{$308$} & \multicolumn{1}{c}{$7\cdot 10^{-3}$} & \multicolumn{1}{c|}{$176.04$} &
\multicolumn{1}{|c}{$-$} & \multicolumn{1}{c}{$5.2\cdot 10^{-2}$} & \multicolumn{1}{c}{$142.85$} \\
\hline
\multicolumn{1}{c}{$500$} & \multicolumn{1}{c}{$2000$} & \multicolumn{1}{c}{$100$} & \multicolumn{1}{c|}{$10^{-4}$} & $ 0.8 $ & 
\multicolumn{1}{|c}{$523$} & \multicolumn{1}{c}{$1.1\cdot 10^{-2}$} & \multicolumn{1}{c|}{$262.11$} &
\multicolumn{1}{|c}{$307$} & \multicolumn{1}{c}{$6\cdot 10^{-3}$} & \multicolumn{1}{c|}{$170.47$} &
\multicolumn{1}{|c}{$-$} & \multicolumn{1}{c}{$5.1\cdot 10^{-2}$} & \multicolumn{1}{c}{$144.78$} \\
\hline \hline \hline

\multicolumn{4}{c|}{} & & \multicolumn{3}{|c|}{LMaFit} & \multicolumn{3}{|c}{\textsc{Matrix ALPS II}} & \multicolumn{3}{|c}{\textsc{Matrix ALPS II} with QR}  \\
\hline \hline
\multicolumn{1}{c}{$m$}  & \multicolumn{1}{c}{$n$} & \multicolumn{1}{c}{$\rank$} & \multicolumn{1}{c|}{$\vectornormbig{\noise}_2$} & & 
\multicolumn{1}{|c}{\rm{iter.}} & \multicolumn{1}{c}{\rm{err.}} & \multicolumn{1}{c|}{\rm{time}} &
\multicolumn{1}{|c}{\rm{iter.}} & \multicolumn{1}{c}{\rm{err.}} & \multicolumn{1}{c|}{\rm{time}} &
\multicolumn{1}{|c}{\rm{iter.}} & \multicolumn{1}{c}{\rm{err.}} & \multicolumn{1}{c}{\rm{time}}  \\
\hline\hline
\multicolumn{1}{c}{$500$} & \multicolumn{1}{c}{$2000$} & \multicolumn{1}{c}{$30$}  & \multicolumn{1}{c|}{$0$} & $ 0.083 $ & 
\multicolumn{1}{|c}{$37$} & \multicolumn{1}{c}{$1.3 \cdot 10^{-5}$} & \multicolumn{1}{c|}{$3.05$} &
\multicolumn{1}{|c}{$13$} & \multicolumn{1}{c}{$3.1 \cdot 10^{-5}$} & \multicolumn{1}{c}{$4.84$} &
\multicolumn{1}{|c}{$37$} & \multicolumn{1}{c}{$1.2 \cdot 10^{-5}$} & \multicolumn{1}{c}{$4.04$} \\
\hline
\multicolumn{1}{c}{$500$} & \multicolumn{1}{c}{$2000$} & \multicolumn{1}{c}{$30$}  & \multicolumn{1}{c|}{$10^{-3}$} & $ 0.083 $ & 
\multicolumn{1}{|c}{$37$} & \multicolumn{1}{c}{$1 \cdot 10^{-3}$} & \multicolumn{1}{c|}{$2.52$} &
\multicolumn{1}{|c}{$22$} & \multicolumn{1}{c}{$1.1 \cdot 10^{-3}$} & \multicolumn{1}{c}{$5.35$} &
\multicolumn{1}{|c}{$37$} & \multicolumn{1}{c}{$1 \cdot 10^{-3}$} & \multicolumn{1}{c}{$3.32$} \\
\hline
\multicolumn{1}{c}{$500$} & \multicolumn{1}{c}{$2000$} & \multicolumn{1}{c}{$30$}  & \multicolumn{1}{c|}{$10^{-4}$} & $ 0.083 $ & 
\multicolumn{1}{|c}{$35$} & \multicolumn{1}{c}{$1 \cdot 10^{-4}$} & \multicolumn{1}{c|}{$2.86$} &
\multicolumn{1}{|c}{$13$} & \multicolumn{1}{c}{$1.3 \cdot 10^{-4}$} & \multicolumn{1}{c}{$4.85$} &
\multicolumn{1}{|c}{$37$} & \multicolumn{1}{c}{$1.6 \cdot 10^{-4}$} & \multicolumn{1}{c}{$4.05$} \\
\hline
\multicolumn{1}{c}{$500$} & \multicolumn{1}{c}{$2000$} & \multicolumn{1}{c}{$50$}  & \multicolumn{1}{c|}{$0$} & $0.408 $ & 
\multicolumn{1}{|c}{$60$} & \multicolumn{1}{c}{$6 \cdot 10^{-5}$} & \multicolumn{1}{c|}{$6.06$} &
\multicolumn{1}{|c}{$22$} & \multicolumn{1}{c}{$1 \cdot 10^{-4}$} & \multicolumn{1}{c}{$7.6$} &
\multicolumn{1}{|c}{$60$} & \multicolumn{1}{c}{$2 \cdot 10^{-4}$} & \multicolumn{1}{c}{$\mathbf{5.67}$} \\
\hline
\multicolumn{1}{c}{$500$} & \multicolumn{1}{c}{$2000$} & \multicolumn{1}{c}{$50$}  & \multicolumn{1}{c|}{$10^{-3}$} & $ 0.408 $ & 
\multicolumn{1}{|c}{$60$} & \multicolumn{1}{c}{$1.4 \cdot 10^{-3}$} & \multicolumn{1}{c|}{$7.26$} &
\multicolumn{1}{|c}{$36$} & \multicolumn{1}{c}{$1.6 \cdot 10^{-3}$} & \multicolumn{1}{c}{$19.64$} &
\multicolumn{1}{|c}{$59$} & \multicolumn{1}{c}{$1.6 \cdot 10^{-3}$} & \multicolumn{1}{c}{$\mathbf{6.91}$} \\
\hline
\multicolumn{1}{c}{$500$} & \multicolumn{1}{c}{$2000$} & \multicolumn{1}{c}{$50$}  & \multicolumn{1}{c|}{$10^{-4}$} & $ 0.408 $ & 
\multicolumn{1}{|c}{$60$} & \multicolumn{1}{c}{$2 \cdot 10^{-4}$} & \multicolumn{1}{c|}{$7.29$} &
\multicolumn{1}{|c}{$22$} & \multicolumn{1}{c}{$2 \cdot 10^{-4}$} & \multicolumn{1}{c}{$11.87$} &
\multicolumn{1}{|c}{$59$} & \multicolumn{1}{c}{$2 \cdot 10^{-4}$} & \multicolumn{1}{c}{$\mathbf{6.75}$} \\
\hline
\multicolumn{1}{c}{$500$} & \multicolumn{1}{c}{$2000$} & \multicolumn{1}{c}{$80$}  & \multicolumn{1}{c|}{$0$} & $ 0.645 $ & 
\multicolumn{1}{|c}{$183$} & \multicolumn{1}{c}{$3 \cdot 10^{-4}$} & \multicolumn{1}{c|}{$33.65$} &
\multicolumn{1}{|c}{$61$} & \multicolumn{1}{c}{$2 \cdot 10^{-4}$} & \multicolumn{1}{c}{$49.53$} &
\multicolumn{1}{|c}{$151$} & \multicolumn{1}{c}{$3 \cdot 10^{-4}$} & \multicolumn{1}{c}{$\mathbf{18.66}$} \\
\hline
\multicolumn{1}{c}{$500$} & \multicolumn{1}{c}{$2000$} & \multicolumn{1}{c}{$80$}  & \multicolumn{1}{c|}{$10^{-3}$} & $ 0.645 $ & 
\multicolumn{1}{|c}{$183$} & \multicolumn{1}{c}{$2.3 \cdot 10^{-3}$} & \multicolumn{1}{c|}{$33.48$} &
\multicolumn{1}{|c}{$92$} & \multicolumn{1}{c}{$2.4 \cdot 10^{-3}$} & \multicolumn{1}{c}{$75.51$} &
\multicolumn{1}{|c}{$151$} & \multicolumn{1}{c}{$2.3 \cdot 10^{-3}$} & \multicolumn{1}{c}{$\mathbf{18.87}$} \\
\hline
\multicolumn{1}{c}{$500$} & \multicolumn{1}{c}{$2000$} & \multicolumn{1}{c}{$80$}  & \multicolumn{1}{c|}{$10^{-4}$} & $ 0.645 $ & 
\multicolumn{1}{|c}{$183$} & \multicolumn{1}{c}{$3 \cdot 10^{-4}$} & \multicolumn{1}{c|}{$33.47$} &
\multicolumn{1}{|c}{$61$} & \multicolumn{1}{c}{$4 \cdot 10^{-4}$} & \multicolumn{1}{c}{$49.52$} &
\multicolumn{1}{|c}{$151$} & \multicolumn{1}{c}{$3 \cdot 10^{-4}$} & \multicolumn{1}{c}{$\mathbf{18.92}$} \\
\hline
\multicolumn{1}{c}{$500$} & \multicolumn{1}{c}{$2000$} & \multicolumn{1}{c}{$100$}  & \multicolumn{1}{c|}{$0$} & $ 0.8 $ & 
\multicolumn{1}{|c}{$519$} & \multicolumn{1}{c}{$1.5 \cdot 10^{-3}$} & \multicolumn{1}{c|}{$115.11$} &
\multicolumn{1}{|c}{$148$} & \multicolumn{1}{c}{$4 \cdot 10^{-4}$} & \multicolumn{1}{c}{$153.74$} &
\multicolumn{1}{|c}{$429$} & \multicolumn{1}{c}{$7 \cdot 10^{-4}$} & \multicolumn{1}{c}{$\mathbf{55.1}$} \\
\hline
\multicolumn{1}{c}{$500$} & \multicolumn{1}{c}{$2000$} & \multicolumn{1}{c}{$100$}  & \multicolumn{1}{c|}{$10^{-3}$} & $ 0.8 $ & 
\multicolumn{1}{|c}{$529$} & \multicolumn{1}{c}{$3.6 \cdot 10^{-3}$} & \multicolumn{1}{c|}{$117.7$} &
\multicolumn{1}{|c}{$228$} & \multicolumn{1}{c}{$3.7 \cdot 10^{-3}$} & \multicolumn{1}{c}{$239.92$} &
\multicolumn{1}{|c}{$427$} & \multicolumn{1}{c}{$3.4 \cdot 10^{-3}$} & \multicolumn{1}{c}{$\mathbf{55.7}$} \\
\hline
\multicolumn{1}{c}{$500$} & \multicolumn{1}{c}{$2000$} & \multicolumn{1}{c}{$100$}  & \multicolumn{1}{c|}{$10^{-3}$} & $ 0.8 $ & 
\multicolumn{1}{|c}{$520$} & \multicolumn{1}{c}{$1.6 \cdot 10^{-3}$} & \multicolumn{1}{c|}{$116.66$} &
\multicolumn{1}{|c}{$148$} & \multicolumn{1}{c}{$6 \cdot 10^{-4}$} & \multicolumn{1}{c}{$154.46$} &
\multicolumn{1}{|c}{$428$} & \multicolumn{1}{c}{$8 \cdot 10^{-4}$} & \multicolumn{1}{c}{$\mathbf{55.07}$} \\
\hline
\end{tabular}
\end{center}
\end{table*}

\begin{figure*}[!htp]
\centering
\begin{tabular}{ccc}
\centerline{\subfigure[]{\includegraphics[width = 0.33\textwidth]{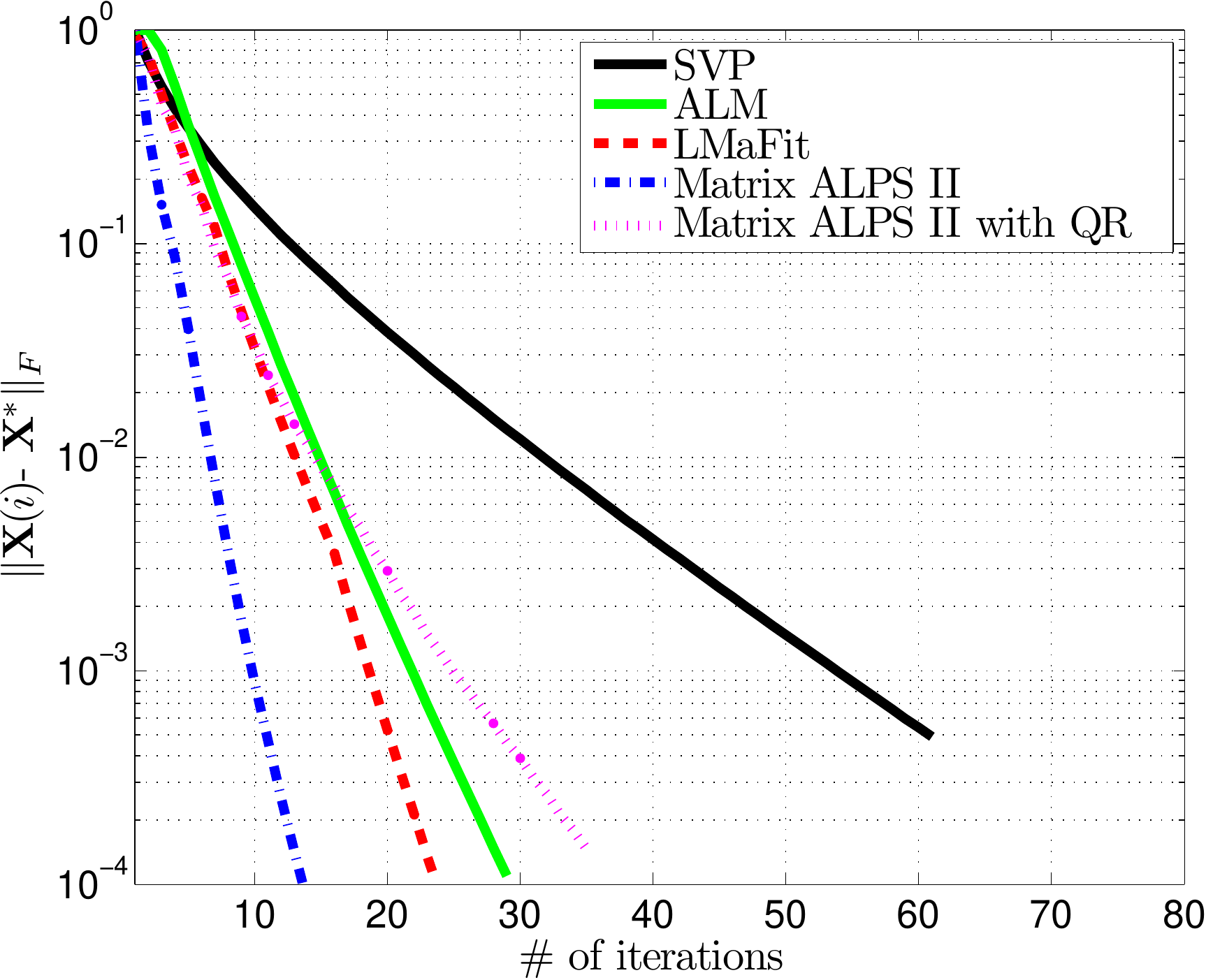}\label{fig:2a}} 
\hfill
\subfigure[]{\includegraphics[width = 0.33\textwidth]{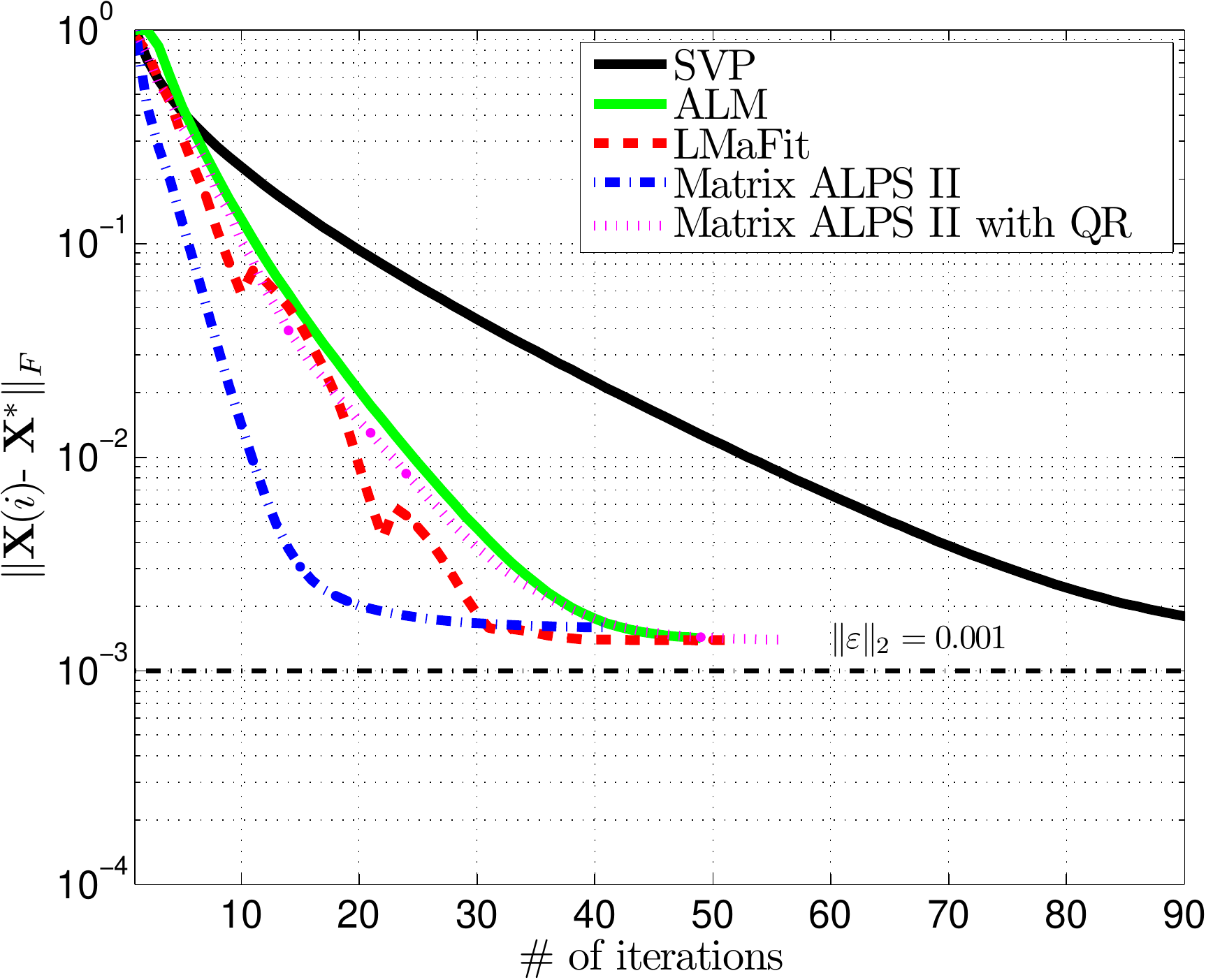}\label{fig:2b}}
\hfill 
\subfigure[]{\includegraphics[width = 0.33\textwidth]{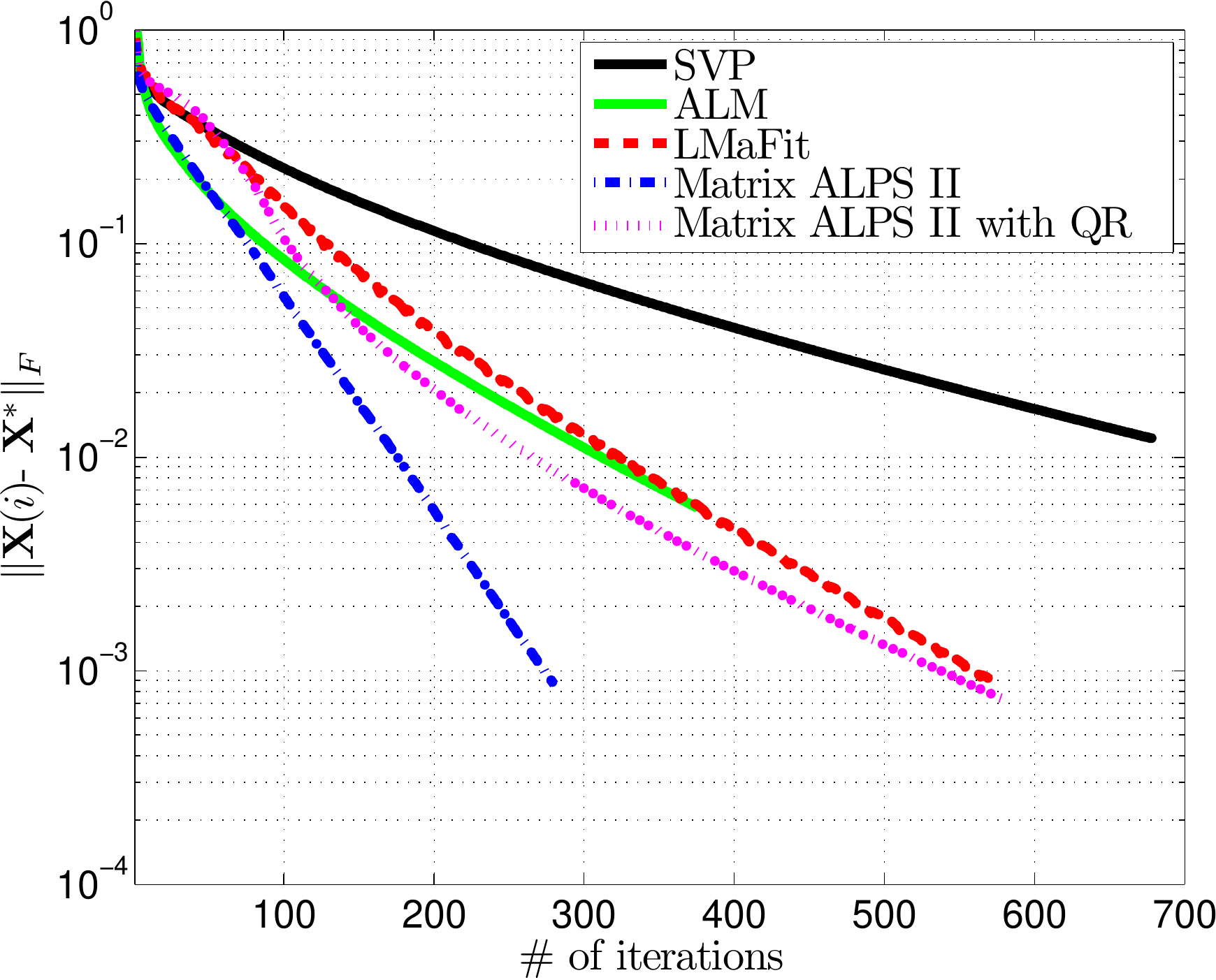}\label{fig:2c}}} \\
\centerline{\subfigure[]{\includegraphics[width = 0.33\textwidth]{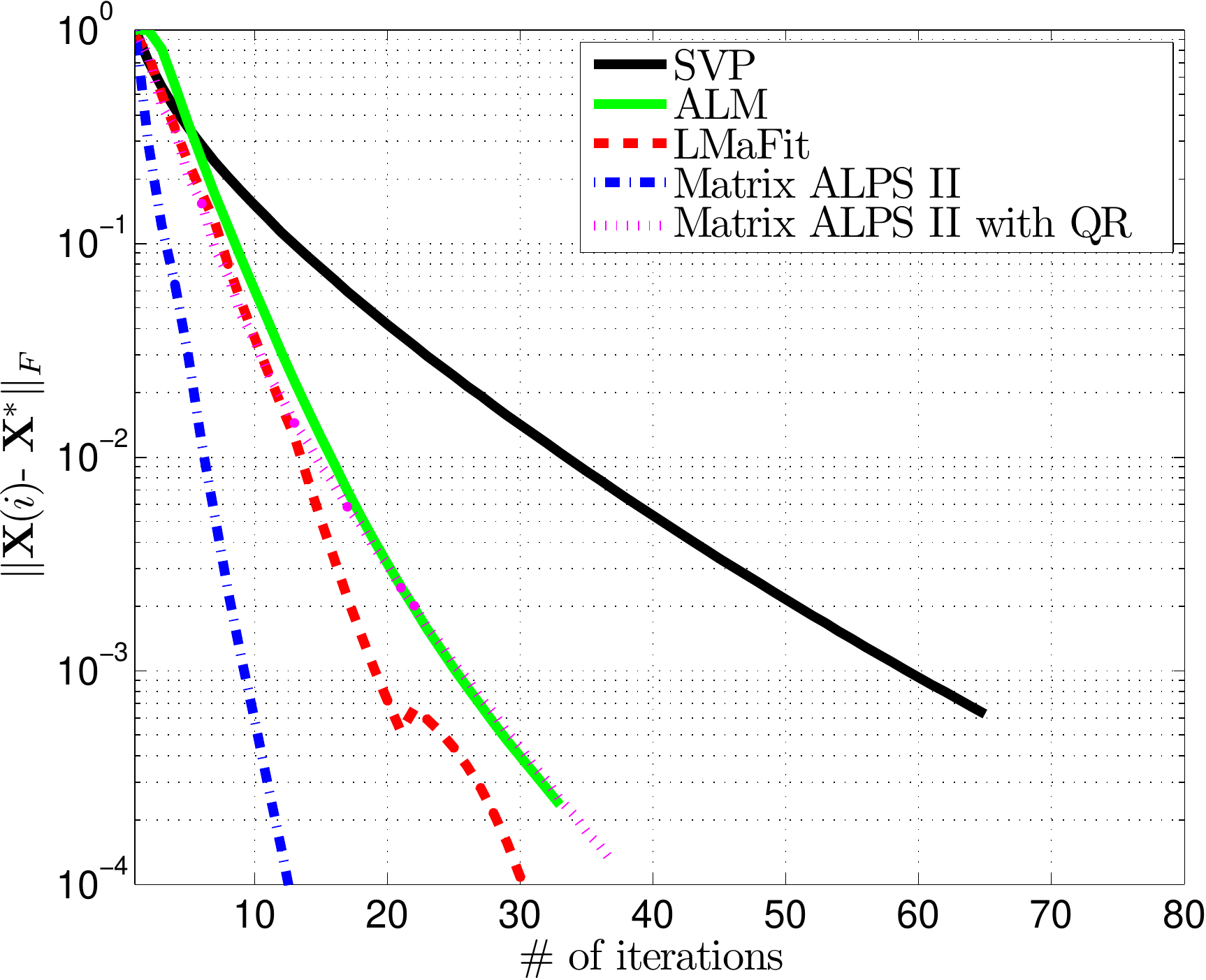}\label{fig:2a}} 
\hfill
\subfigure[]{\includegraphics[width = 0.33\textwidth]{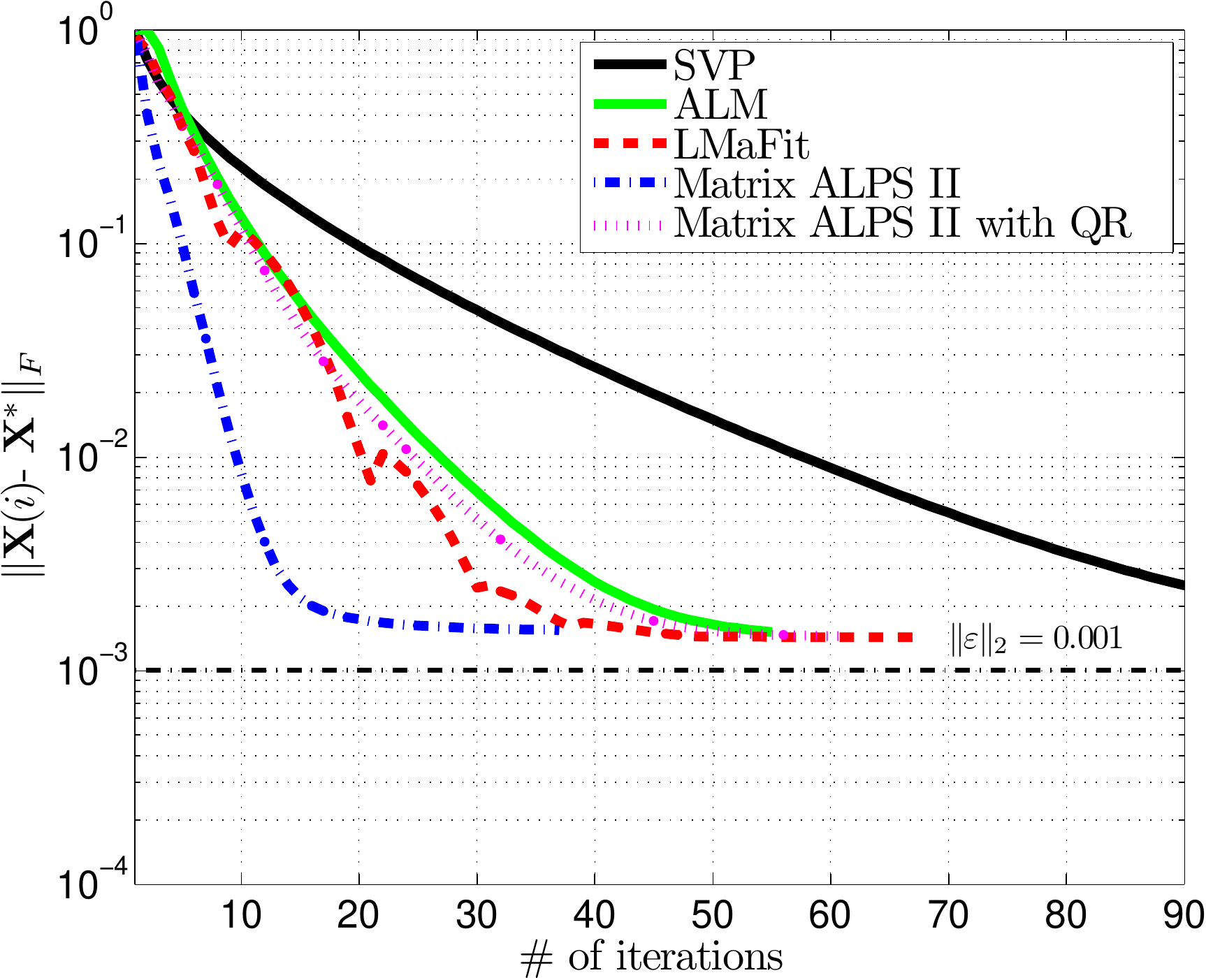}\label{fig:2b}}
\hfill 
\subfigure[]{\includegraphics[width = 0.33\textwidth]{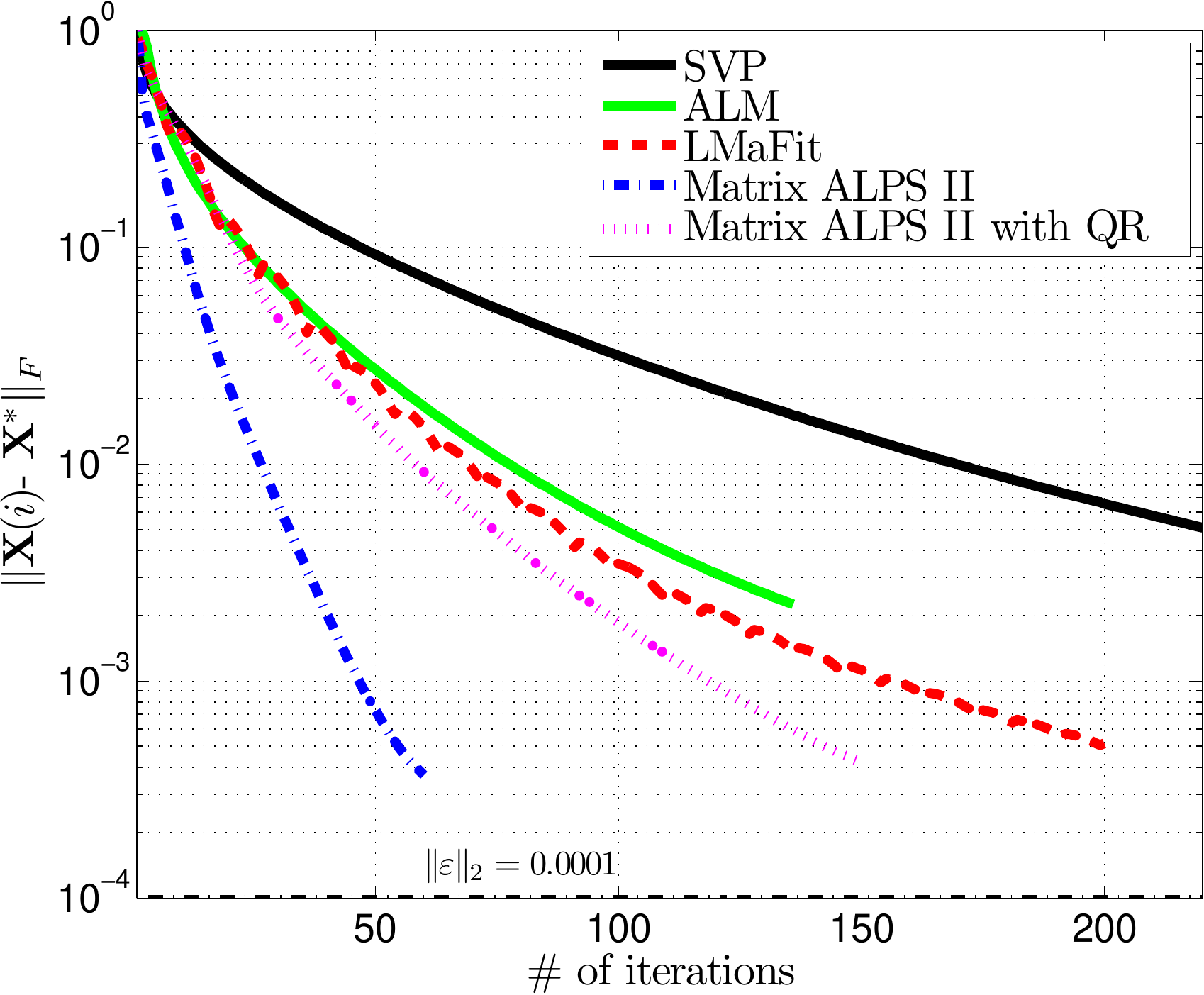}\label{fig:2c}}}
\end{tabular}
\caption{\small\sl Low rank matrix recovery for the matrix completion problem. The error curves are the median values across 50 Monte-Carlo realizations over each iteration. For all cases, we assume $p = 0.3mn$. (a) $m = 700$, $n = 1000 $, $\rank = 30$ and $\vectornormbig{\noise}_2 = 0 $. (b) $m = 700$, $n = 1000 $, $\rank = 50$ and $\vectornormbig{\noise}_2 = 10^{-3} $. (c) $m = 700$, $n = 1000 $, $\rank = 110$ and $\vectornormbig{\noise}_2 = 0 $. (d) $m = 500$, $n = 2000 $, $\rank = 10$ and $\vectornormbig{\noise}_2 = 0 $. (e) $m = 500$, $n = 2000 $, $\rank = 50$ and $\vectornormbig{\noise}_2 = 10^{-3} $. (f) $m = 500$, $n = 2000 $, $\rank = 80$ and $\vectornormbig{\noise}_2 = 10^{-4} $.  } \label{fig: TableIII_IV_fig}
\end{figure*}

\subsection{Real data}

We use real data images to highlight the reconstruction performance of the proposed schemes. To this end, we perform grayscale image denoising from an incomplete set of observed pixels---similar experiments can be found in \cite{LMatFit}. Based on the matrix completion setting, we observe a limited number of pixels from the original image and perform a low rank approximation based only on the set of measurements. While the true underlying image might not be low-rank, we apply our solvers to obtain low-rank approximations.

Figures \ref{fig:real1} and \ref{fig:real2} depict the reconstruction results. In the first test case, we use a $512 \times 512$ grayscale image as shown in the top left corner of Figure \ref{fig:real1}. For this case, we observe only the $35\%$ of the total number of pixels, randomly selected---a realization is depicted in the top right plot in Figure \ref{fig:real1}. In sequel, we fix the desired rank to $\rank = 40$. The best rank-$40$ approximation using SVD is shown in the top middle of Figure \ref{fig:real1} where the full set of pixels is observed. Given a fixed common tolerance and the same stopping criteria, Figure \ref{fig:real1} shows the recovery performance achieved by a range of algorithms under consideration for 10 Monte-Carlo realizations.  We repeat the same experiment for the second image in Figure \ref{fig:real2}. Here, the size of the image is $256 \times 256$, the desired rank is set to $\rank = 30$ and we observe the $33\%$ of the image pixels. In constrast to the image denoising procedure above, we measure the reconstruction error of the computed solutions with respect to the {\it best rank-$30$ approximation} of the true image. In both cases,  we note that \textsc{Matrix ALPS II} has a better phase transition performance as compared to the rest of the algorithms.
\begin{figure*}[!htp]
\centering
\begin{minipage}{0.28\linewidth}
\centering \normalsize{Original}
\end{minipage} 
\begin{minipage}{0.28\linewidth}
\centering \normalsize{Low Rank Approximation}
\end{minipage}
\begin{minipage}{0.28\linewidth}
\centering \normalsize{Observed Image}
\end{minipage} \vspace{0.1cm}\\

\includegraphics[width = 0.28\linewidth]{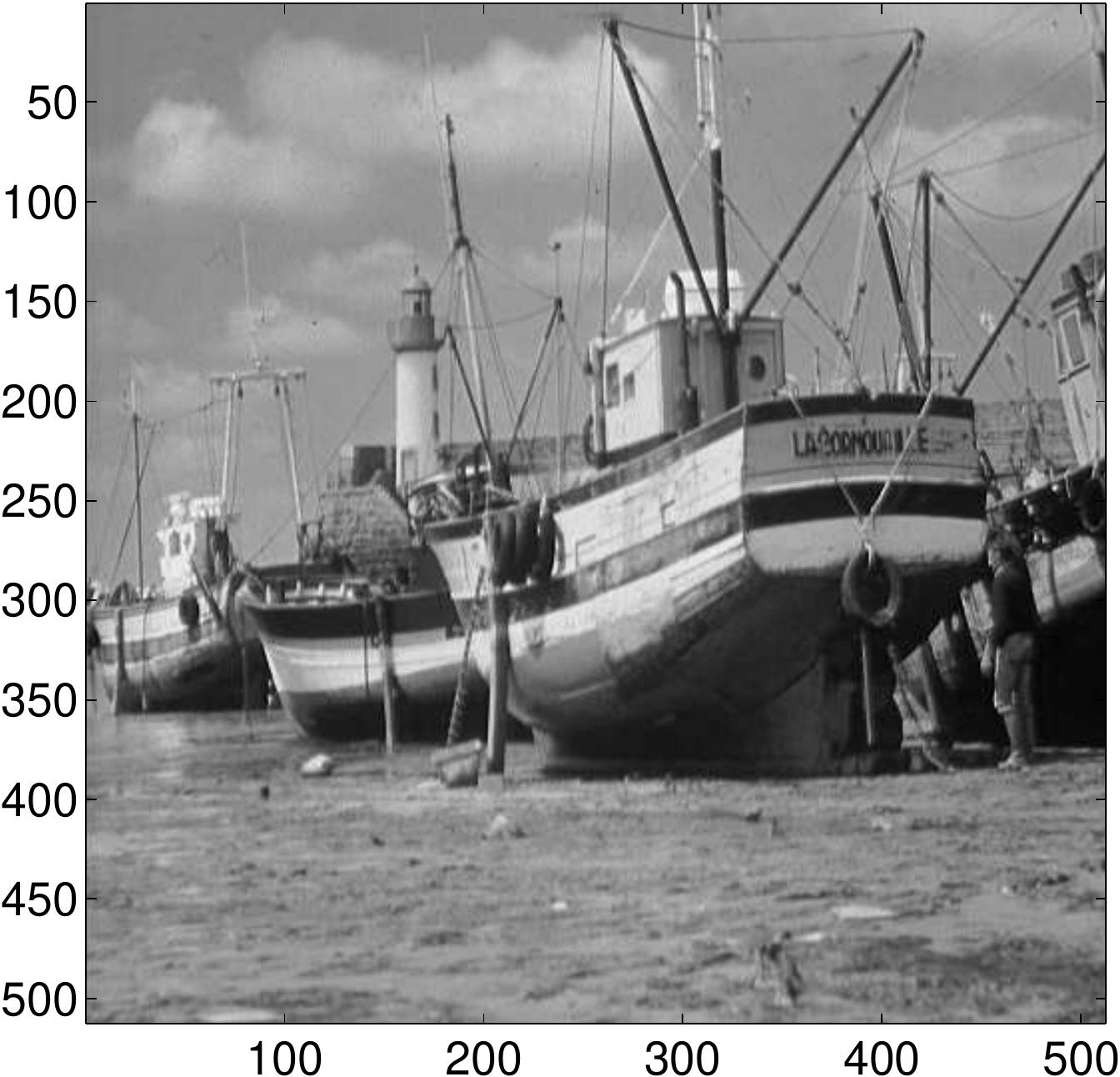} 
\includegraphics[width = 0.28\linewidth]{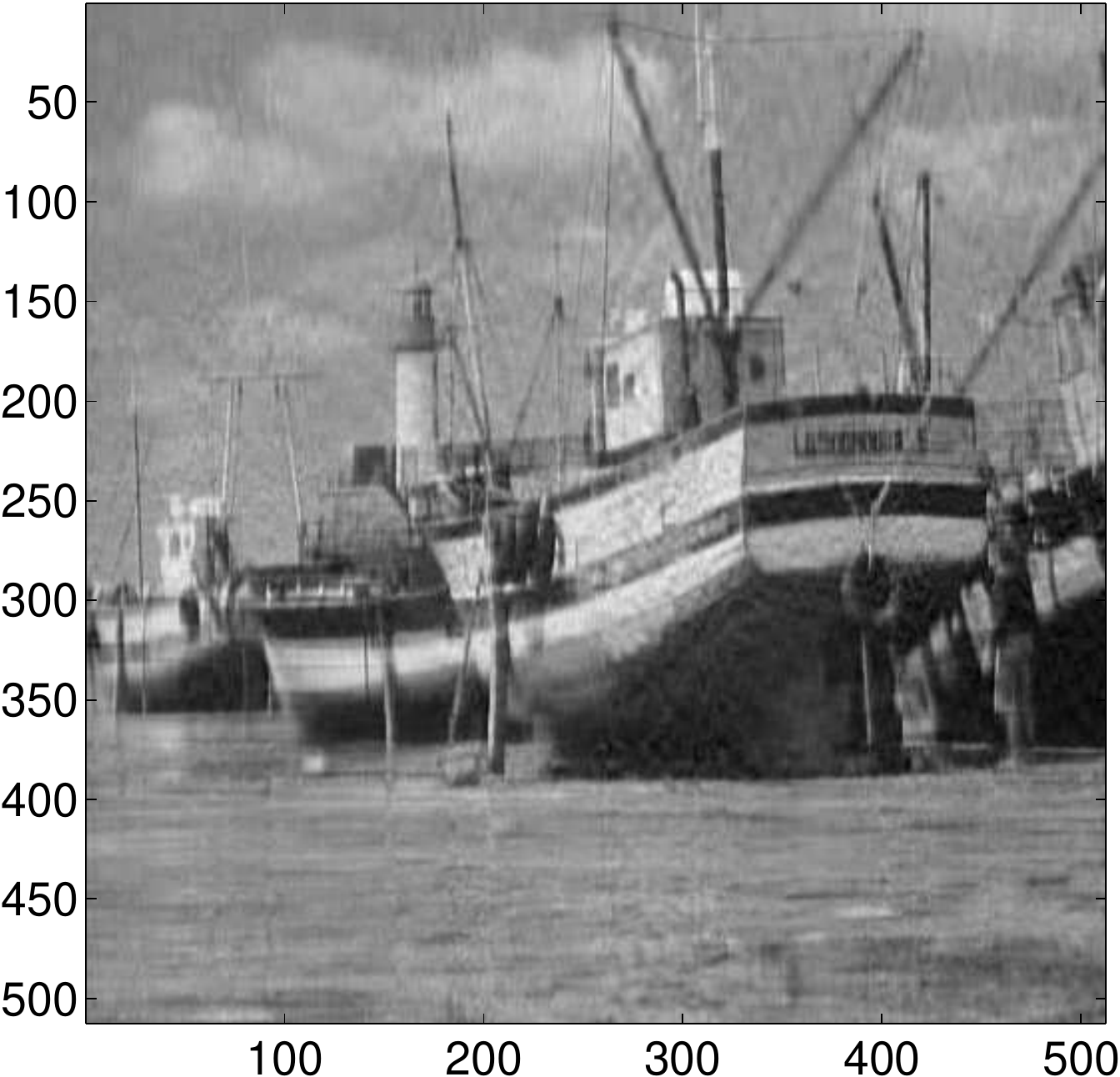} 
\includegraphics[width = 0.28\linewidth]{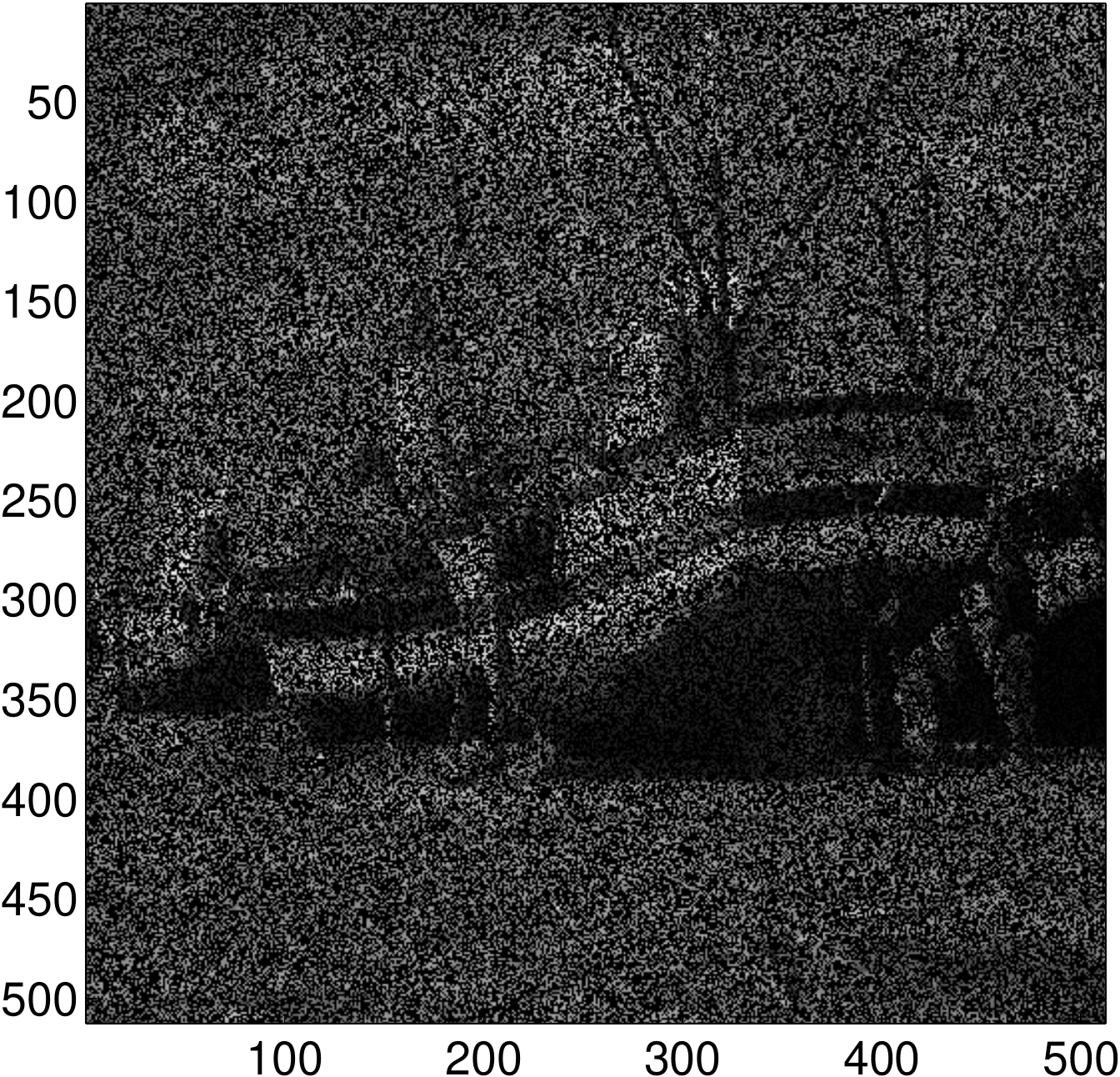}\\

\centering
\begin{minipage}{0.28\linewidth}
\centering \normalsize{SVP - $16.36$ dB}
\end{minipage} 
\begin{minipage}{0.28\linewidth}
\centering \normalsize{ALM - $16.37$ dB}
\end{minipage}
\begin{minipage}{0.28\linewidth}
\centering \normalsize{LMaFit - $16.43$ dB}
\end{minipage} \vspace{0.1cm}\\

\includegraphics[width = 0.28\linewidth]{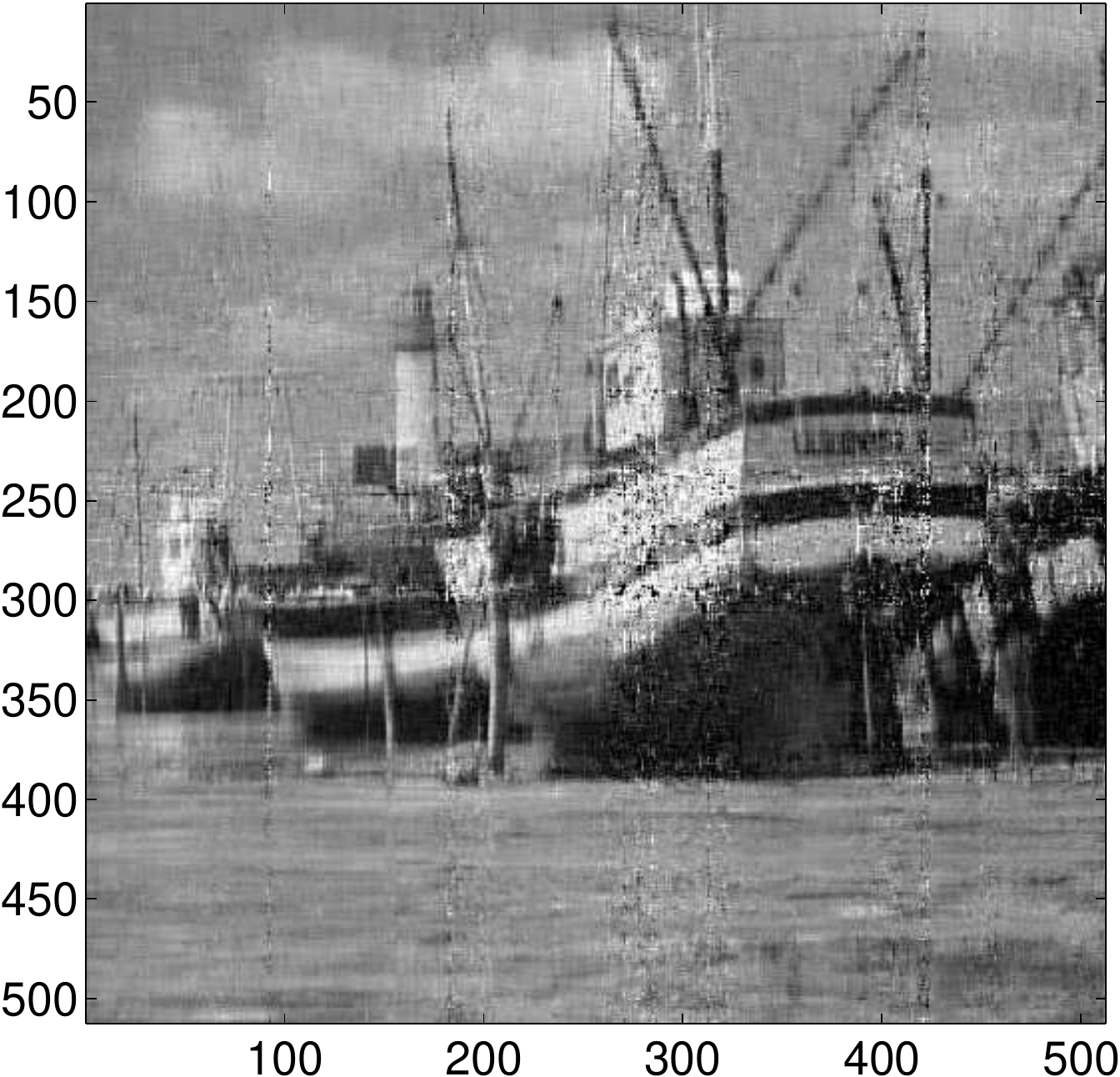} 
\includegraphics[width = 0.28\linewidth]{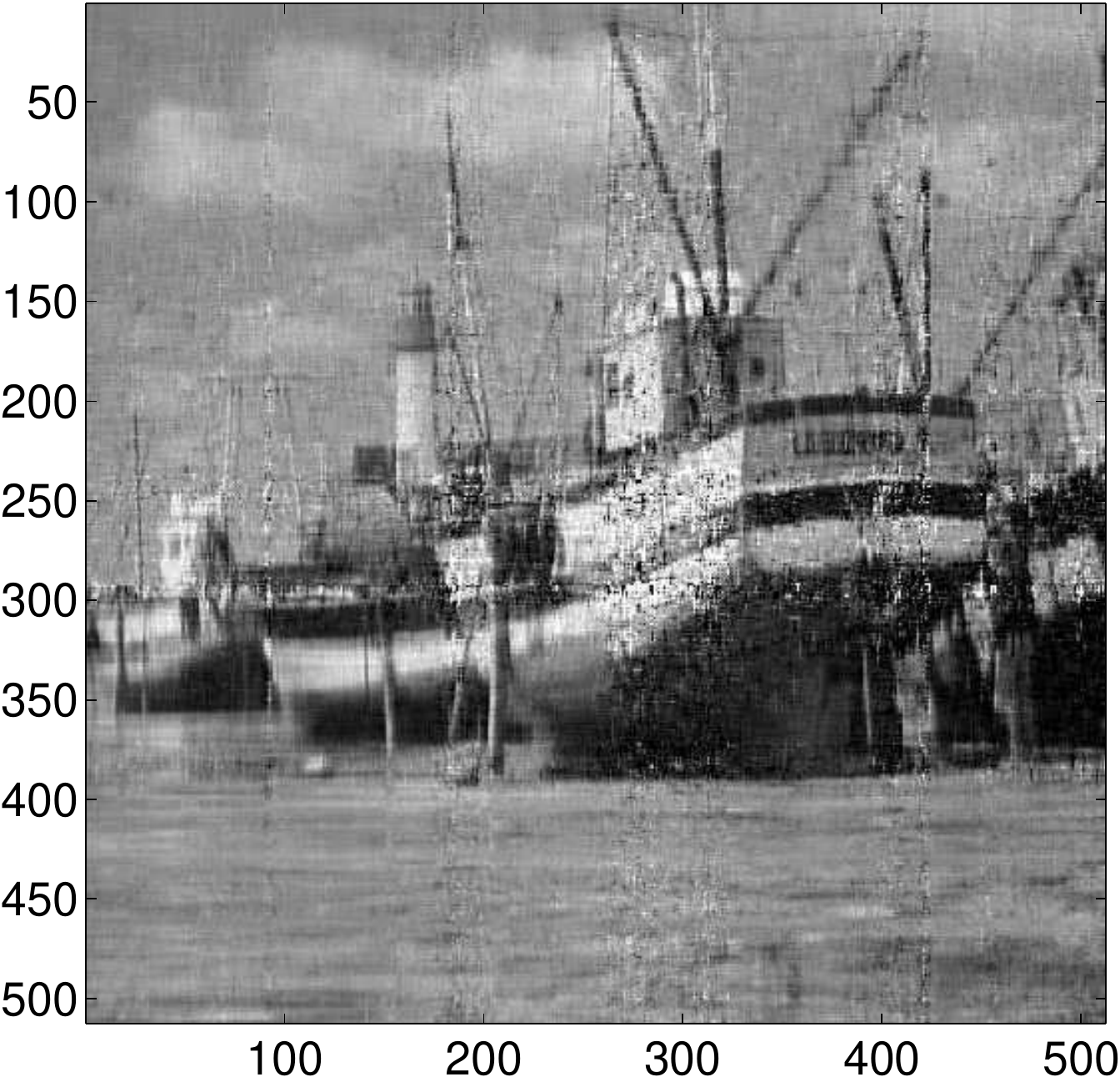} 
\includegraphics[width = 0.28\linewidth]{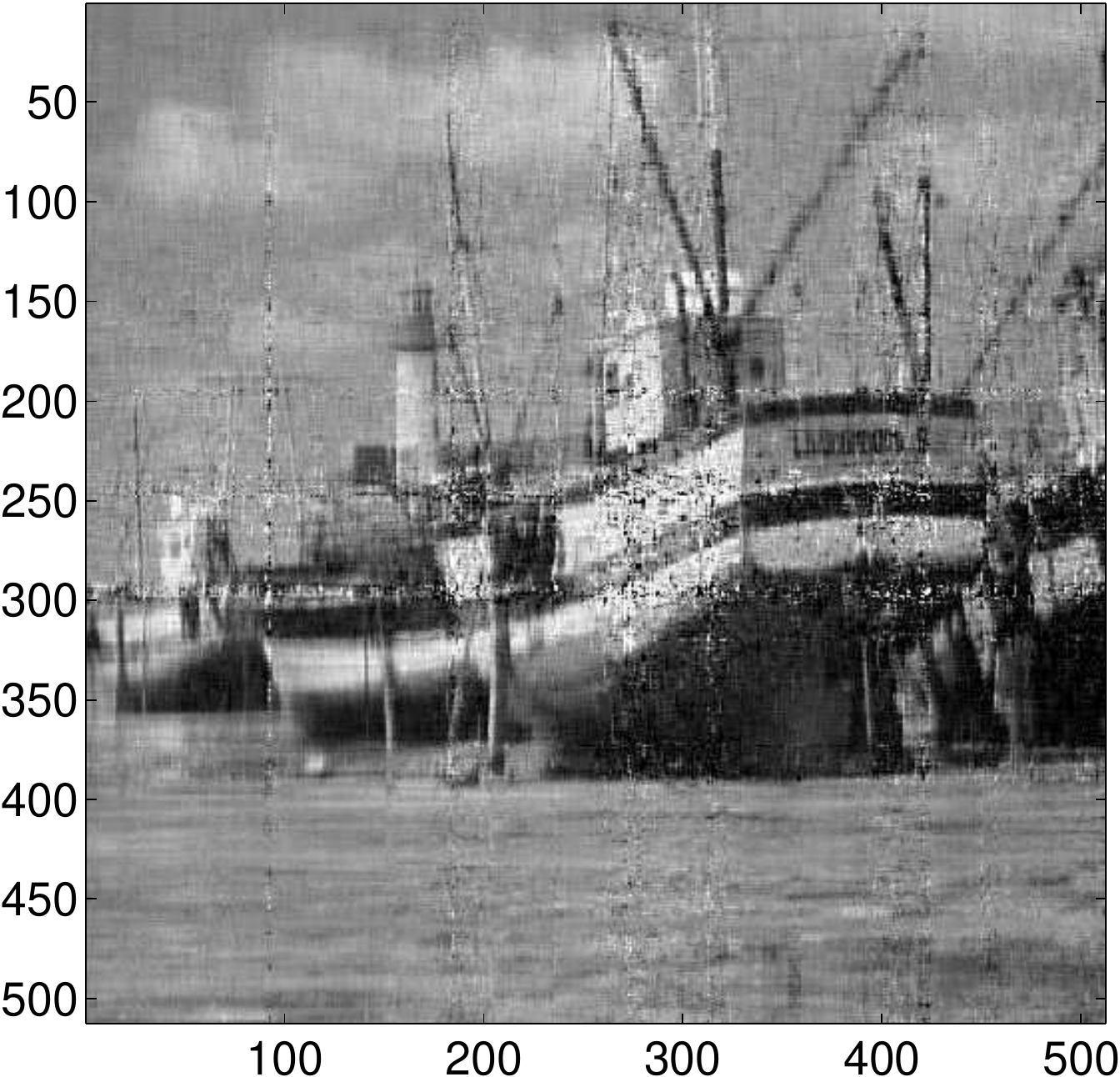} \\

\centering
\begin{minipage}{0.28\linewidth}
\centering \normalsize{\textsc{\textsc{Matrix ALPS I}} - $17.86$ dB}
\end{minipage} 
\begin{minipage}{0.28\linewidth}
\centering \normalsize{ADMiRA - $18.08$ dB}
\end{minipage}
\begin{minipage}{0.28\linewidth}
\centering \normalsize{\textsc{\textsc{Matrix ALPS II}} - $18.35$ dB}
\end{minipage} \vspace{0.1cm}\\

\includegraphics[width = 0.28\linewidth]{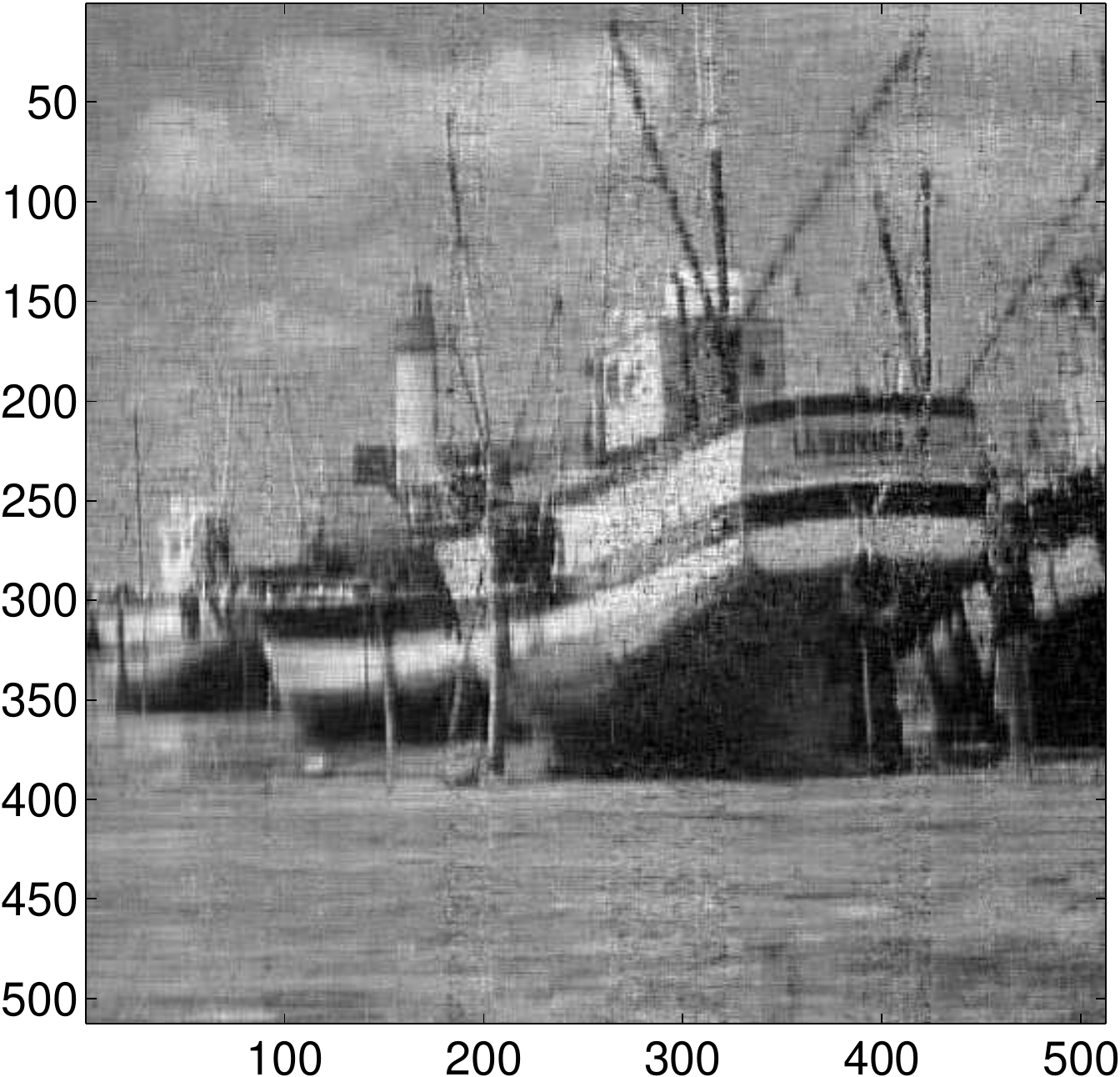} 
\includegraphics[width = 0.28\linewidth]{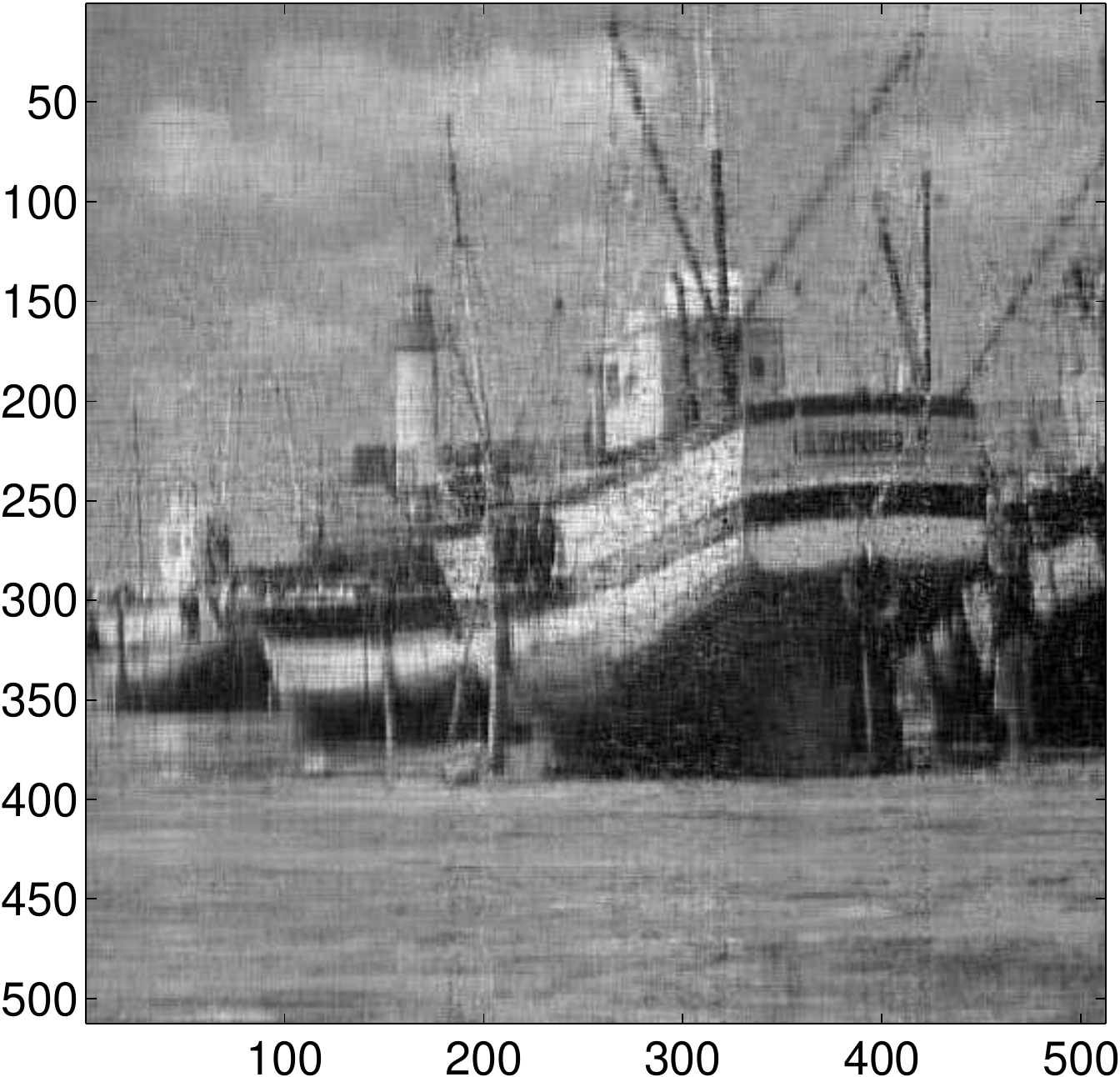} 
\includegraphics[width = 0.28\linewidth]{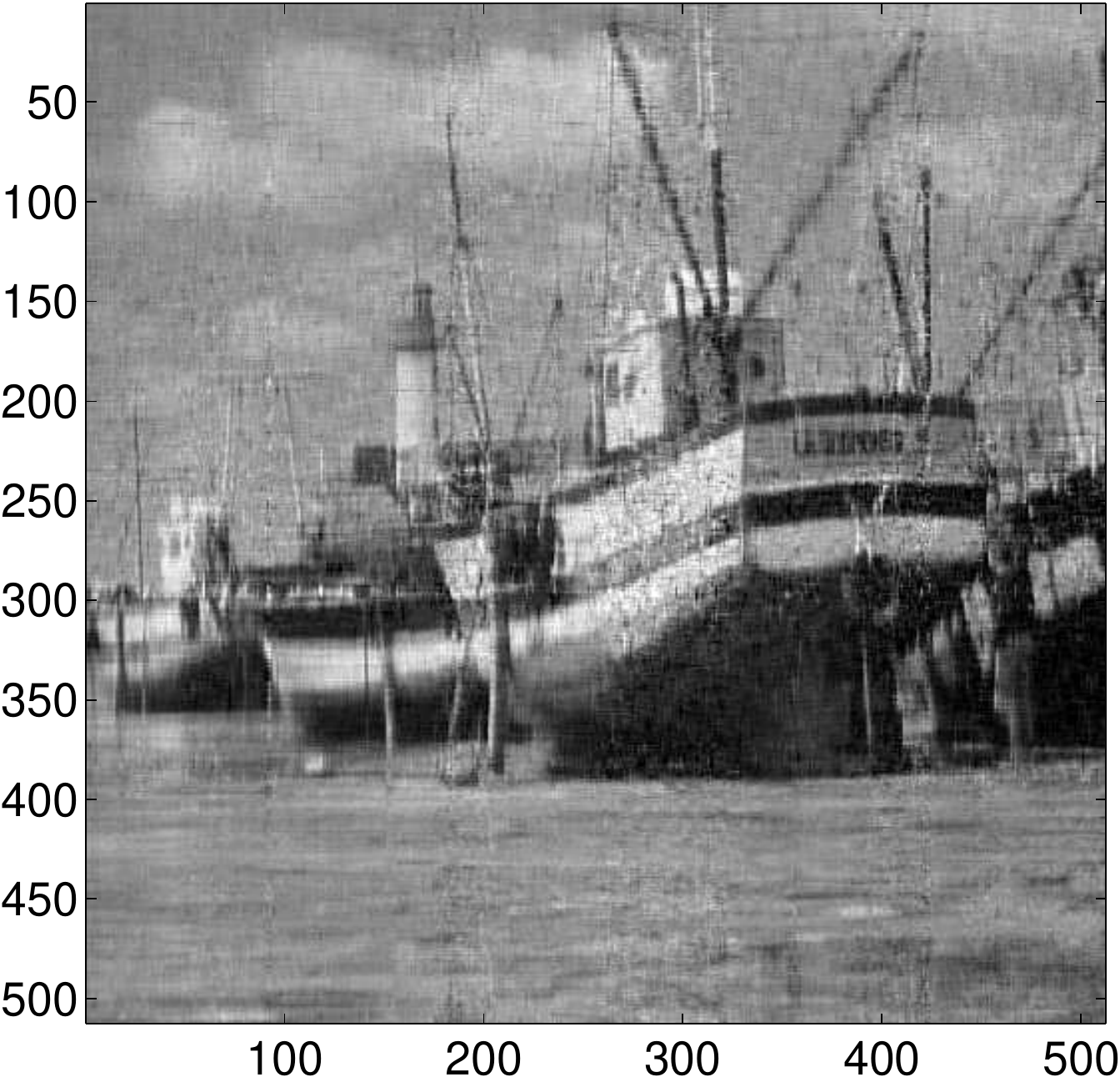}\\
\caption{\small{Reconstruction performance in image denoising settings. The image size is $512 \times 512$ and the desired rank is preset to $\rank = 40$. We observe $35\%$ of the pixels of the true image. We depict the median reconstruction error with respect to the true image in dB over $10$ Monte Carlo realizations.}} {\label{fig:real1}}
\end{figure*}

\begin{figure*}[!htp]
\centering
\begin{minipage}{0.28\linewidth}
\centering \normalsize{Original}
\end{minipage} 
\begin{minipage}{0.28\linewidth}
\centering \normalsize{Low Rank Approximation}
\end{minipage}
\begin{minipage}{0.28\linewidth}
\centering \normalsize{Observed Image}
\end{minipage} \vspace{0.1cm}\\

\includegraphics[width = 0.28\linewidth]{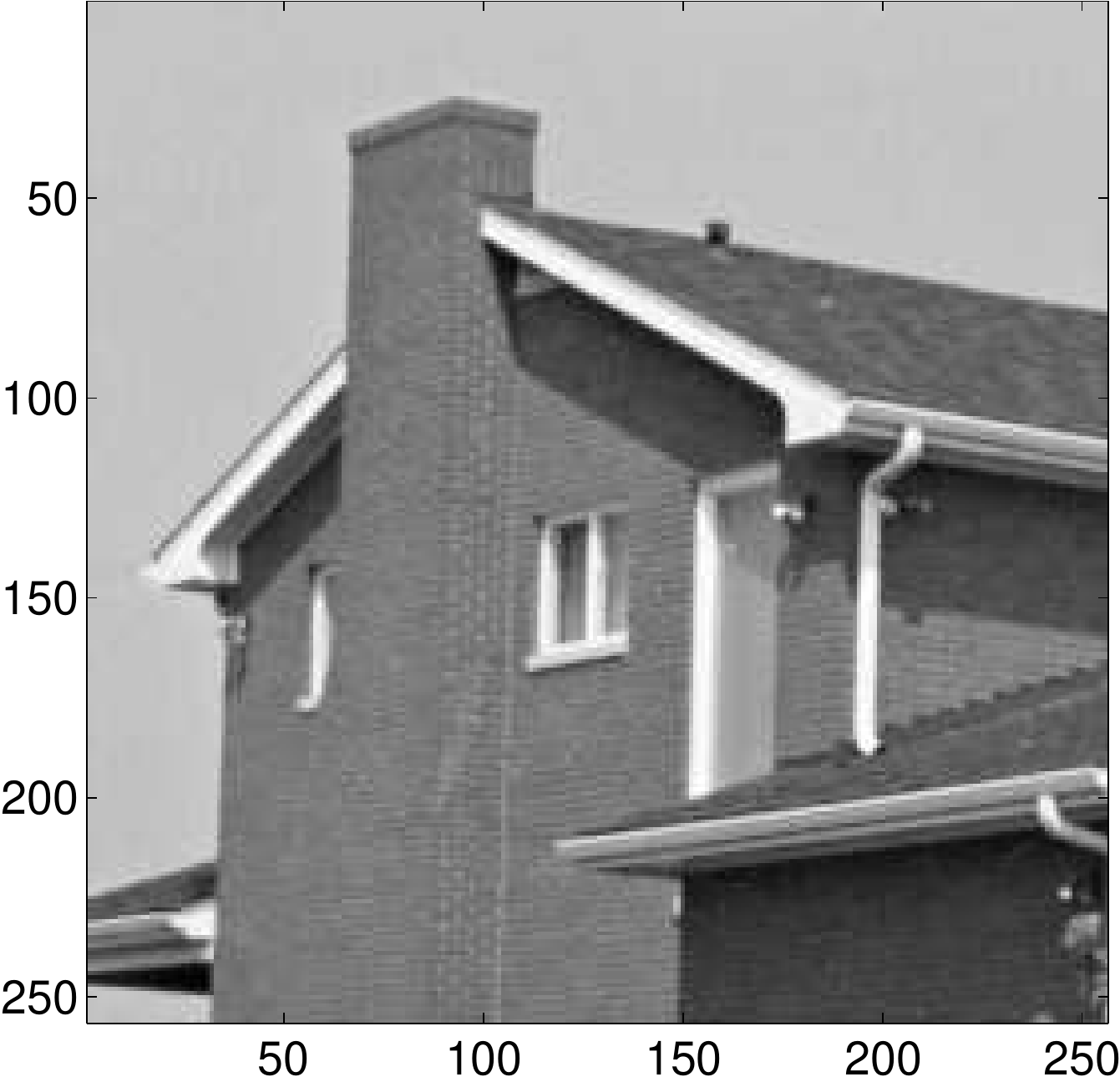} 
\includegraphics[width = 0.28\linewidth]{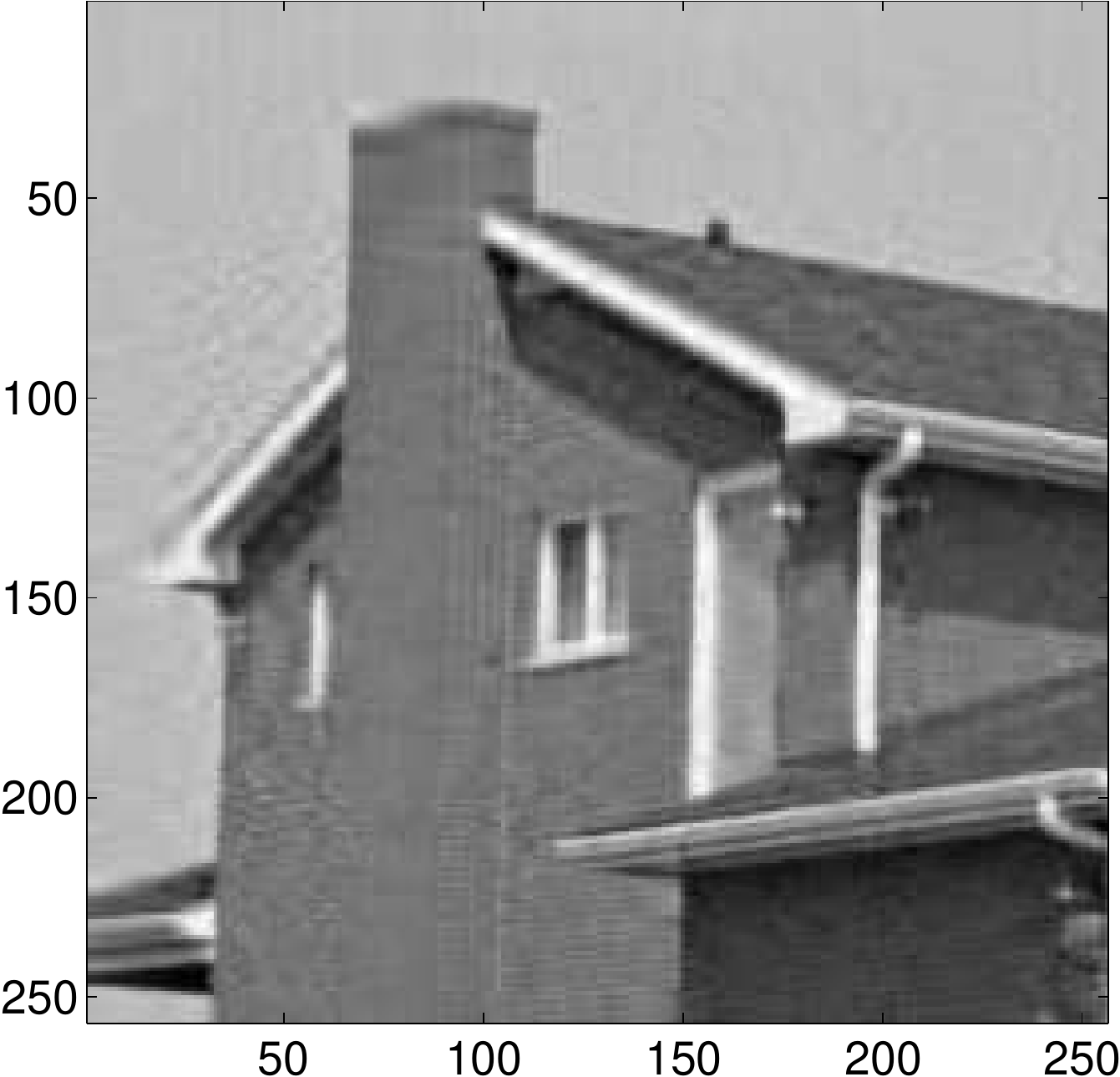} 
\includegraphics[width = 0.28\linewidth]{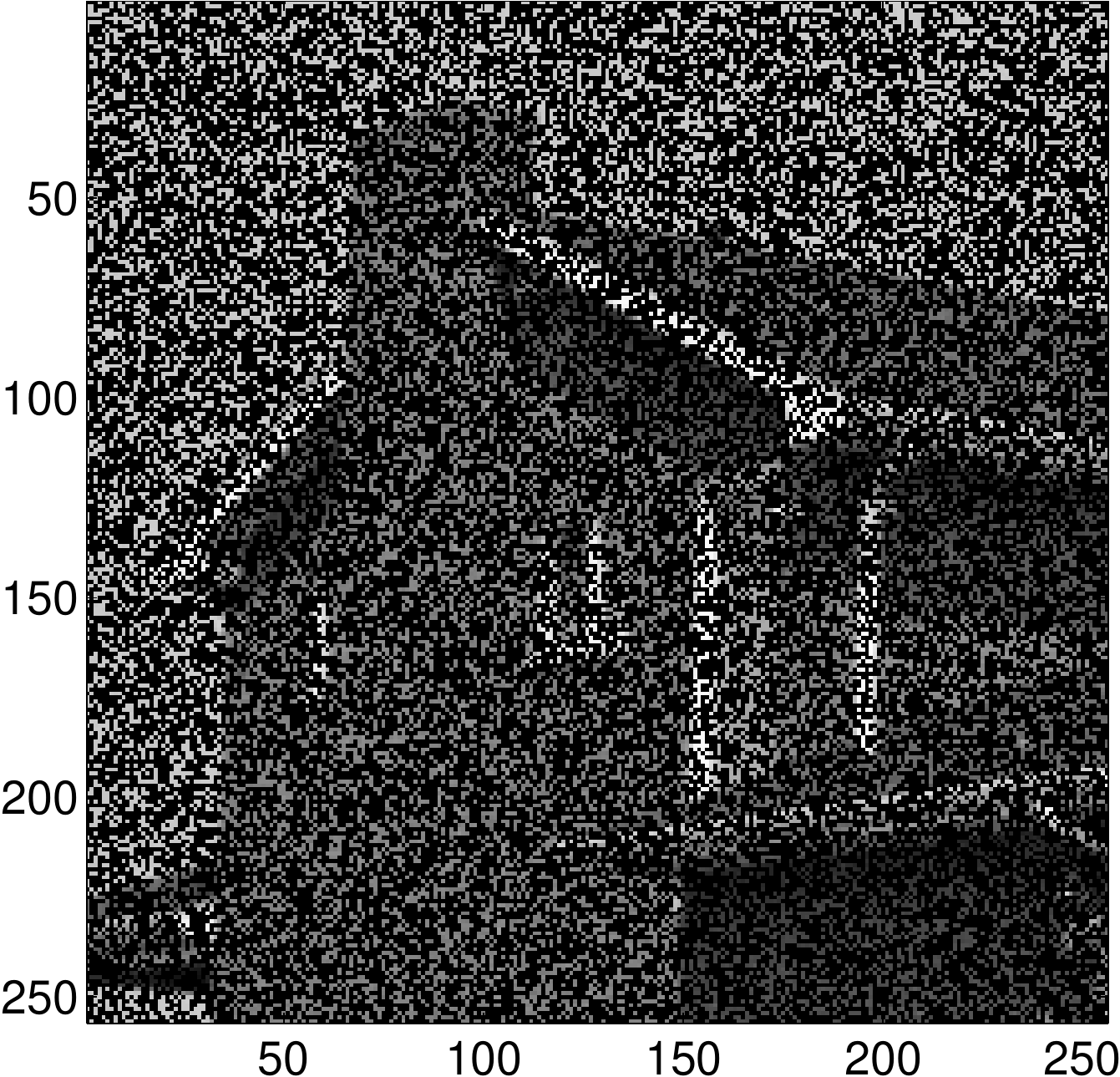}\\

\centering
\begin{minipage}{0.28\linewidth}
\centering \normalsize{SVP - $12.84$ dB}
\end{minipage} 
\begin{minipage}{0.28\linewidth}
\centering \normalsize{ALM - $57.79$ dB}
\end{minipage}
\begin{minipage}{0.28\linewidth}
\centering \normalsize{LMaFit - $12.57$ dB}
\end{minipage} \vspace{0.1cm}\\

\includegraphics[width = 0.28\linewidth]{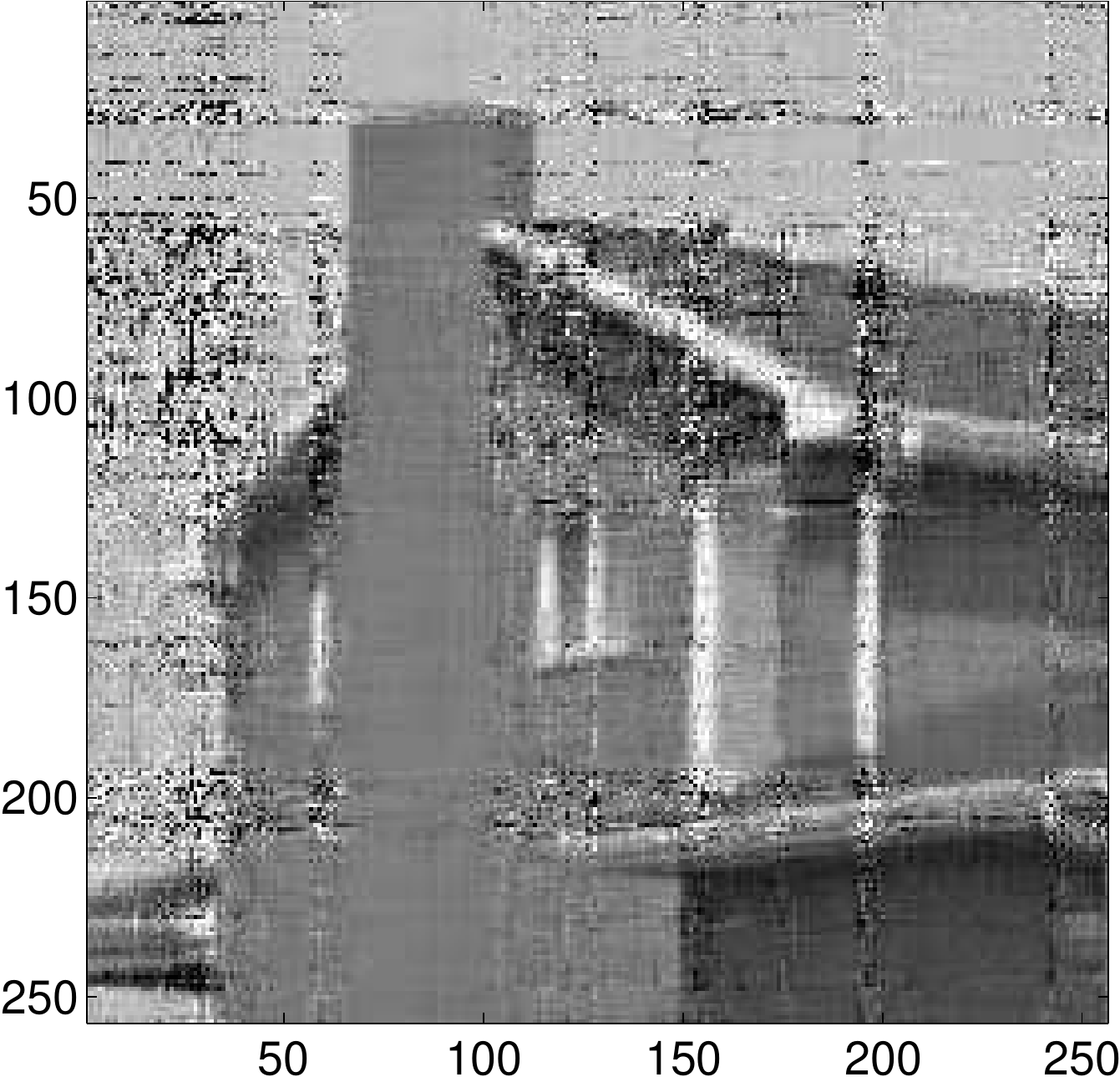} 
\includegraphics[width = 0.28\linewidth]{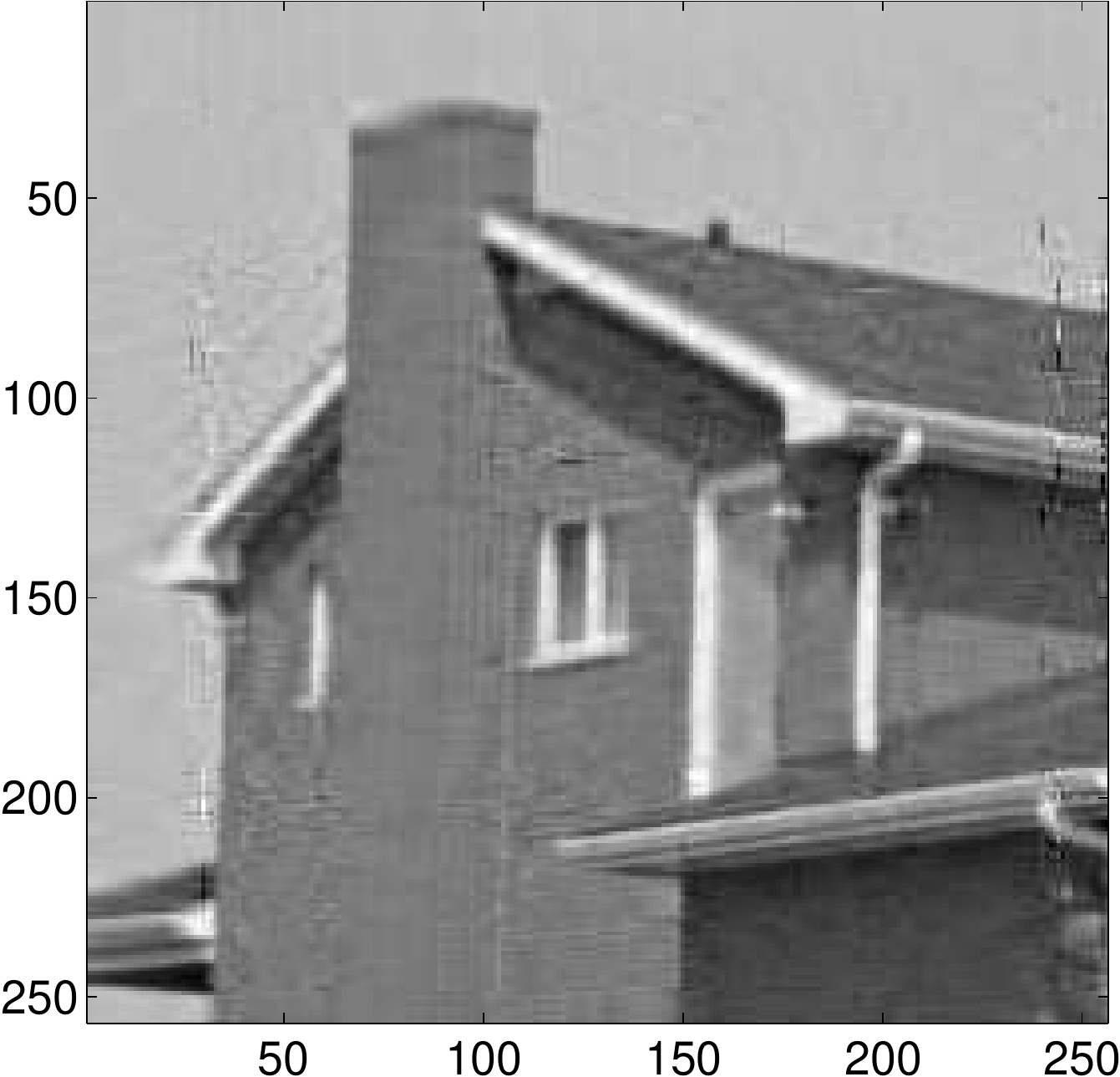} 
\includegraphics[width = 0.28\linewidth]{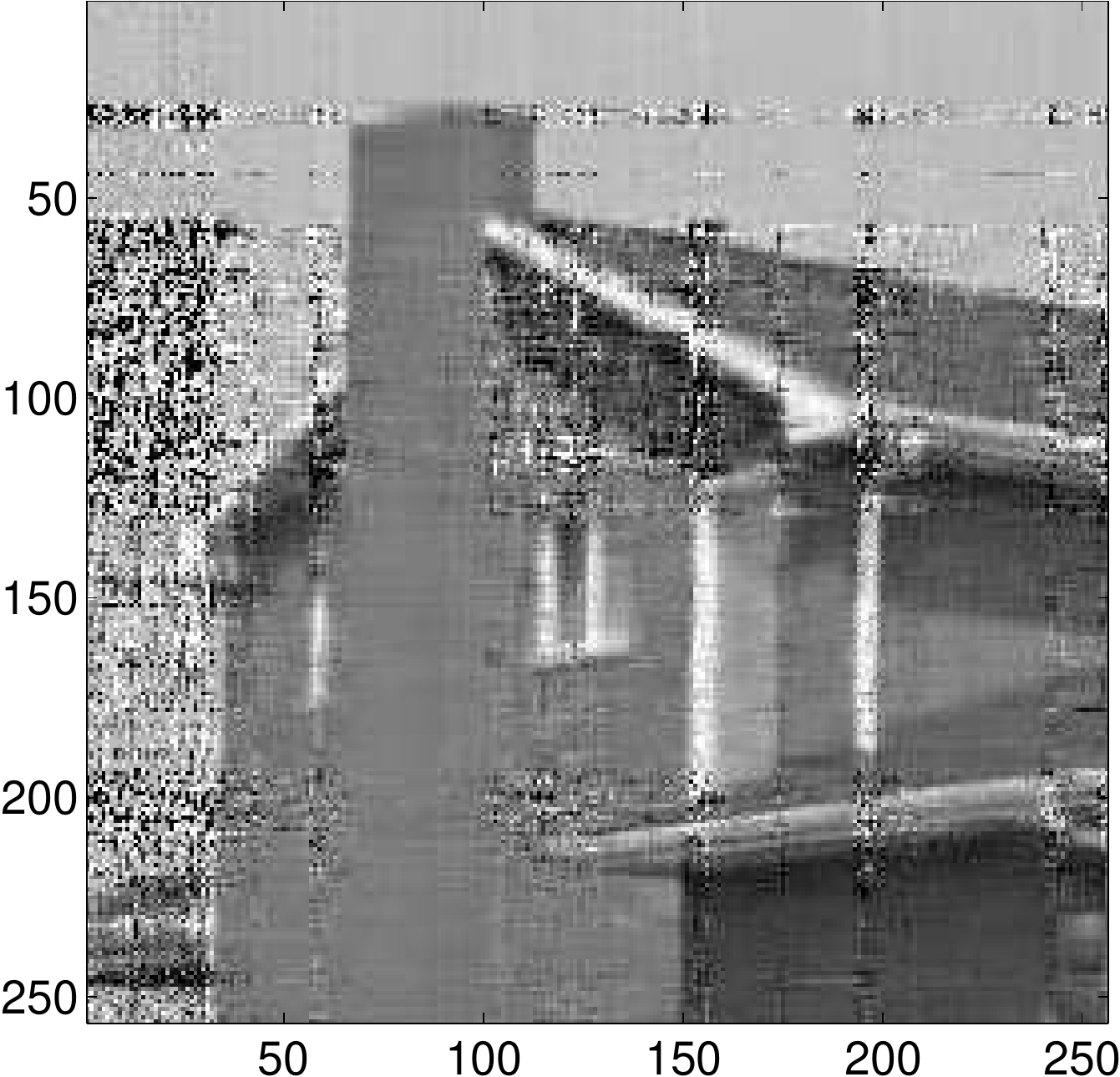} \\

\centering
\begin{minipage}{0.28\linewidth}
\centering \normalsize{\textsc{\textsc{Matrix ALPS I}} - $17.55$ dB}
\end{minipage} 
\begin{minipage}{0.28\linewidth}
\centering \normalsize{ADMiRA - $20.56$ dB}
\end{minipage}
\begin{minipage}{0.28\linewidth}
\centering \normalsize{\textsc{\textsc{Matrix ALPS II}} - $70.86$ dB}
\end{minipage} \vspace{0.1cm}\\

\includegraphics[width = 0.28\linewidth]{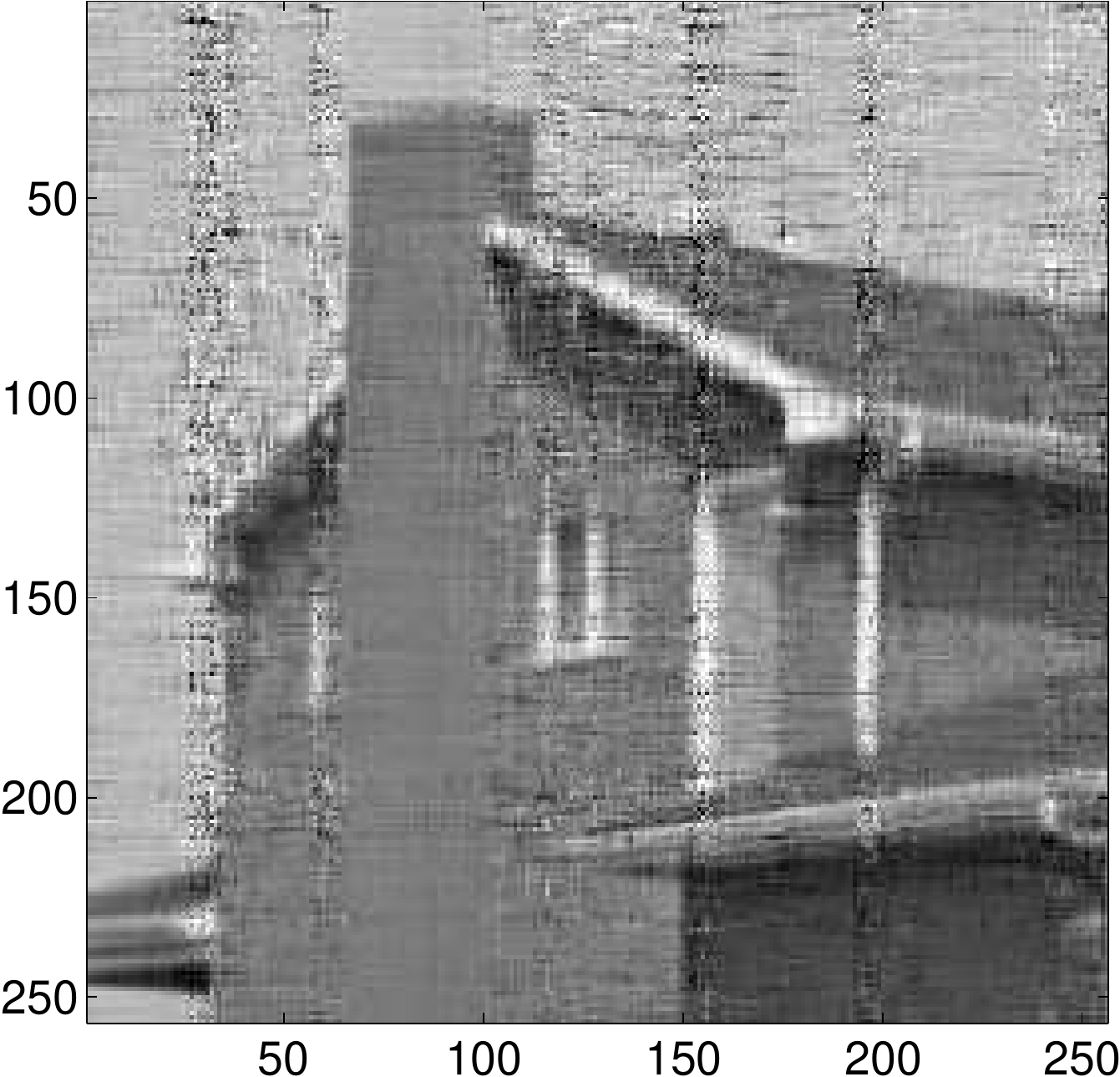} 
\includegraphics[width = 0.28\linewidth]{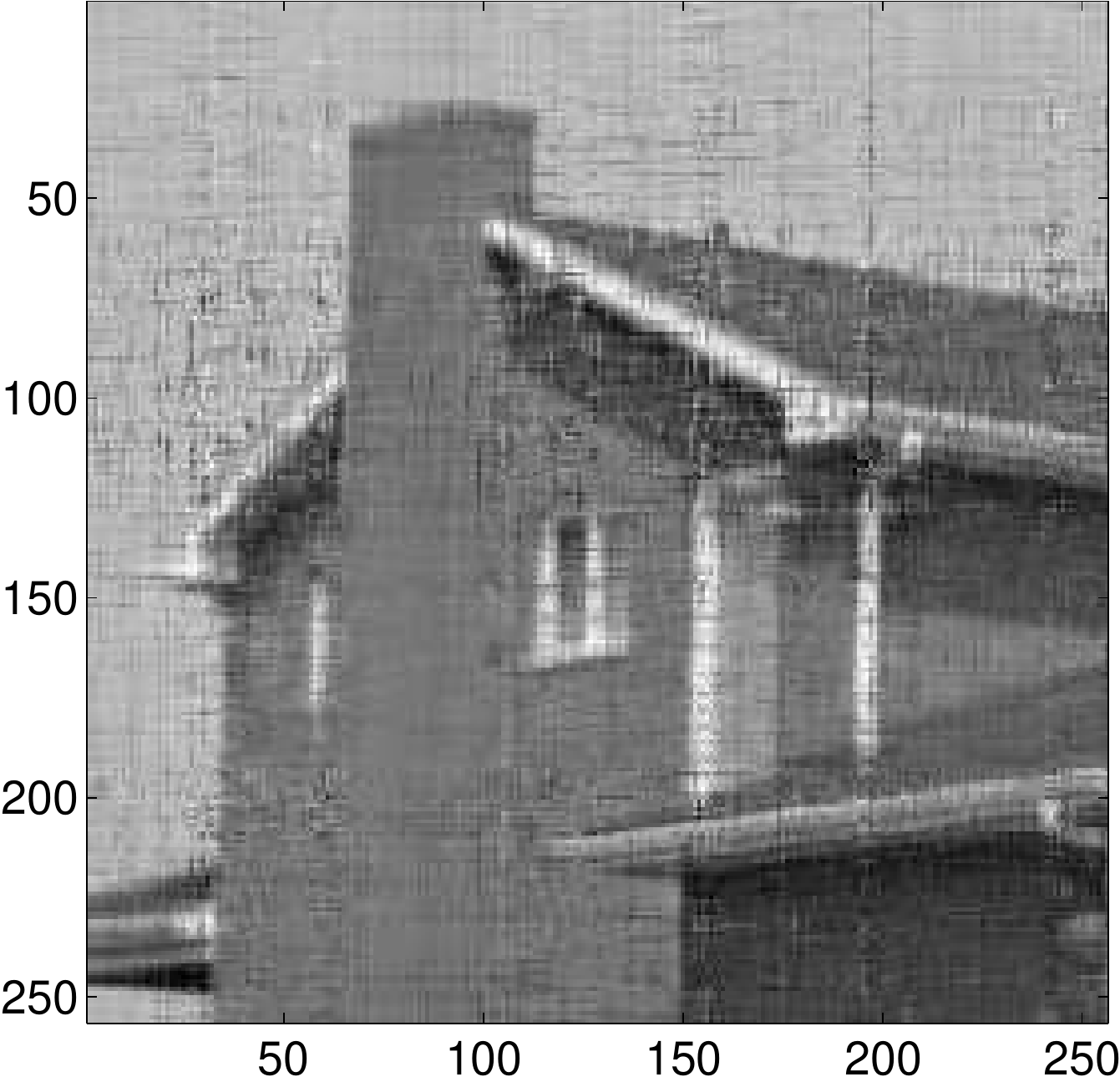} 
\includegraphics[width = 0.28\linewidth]{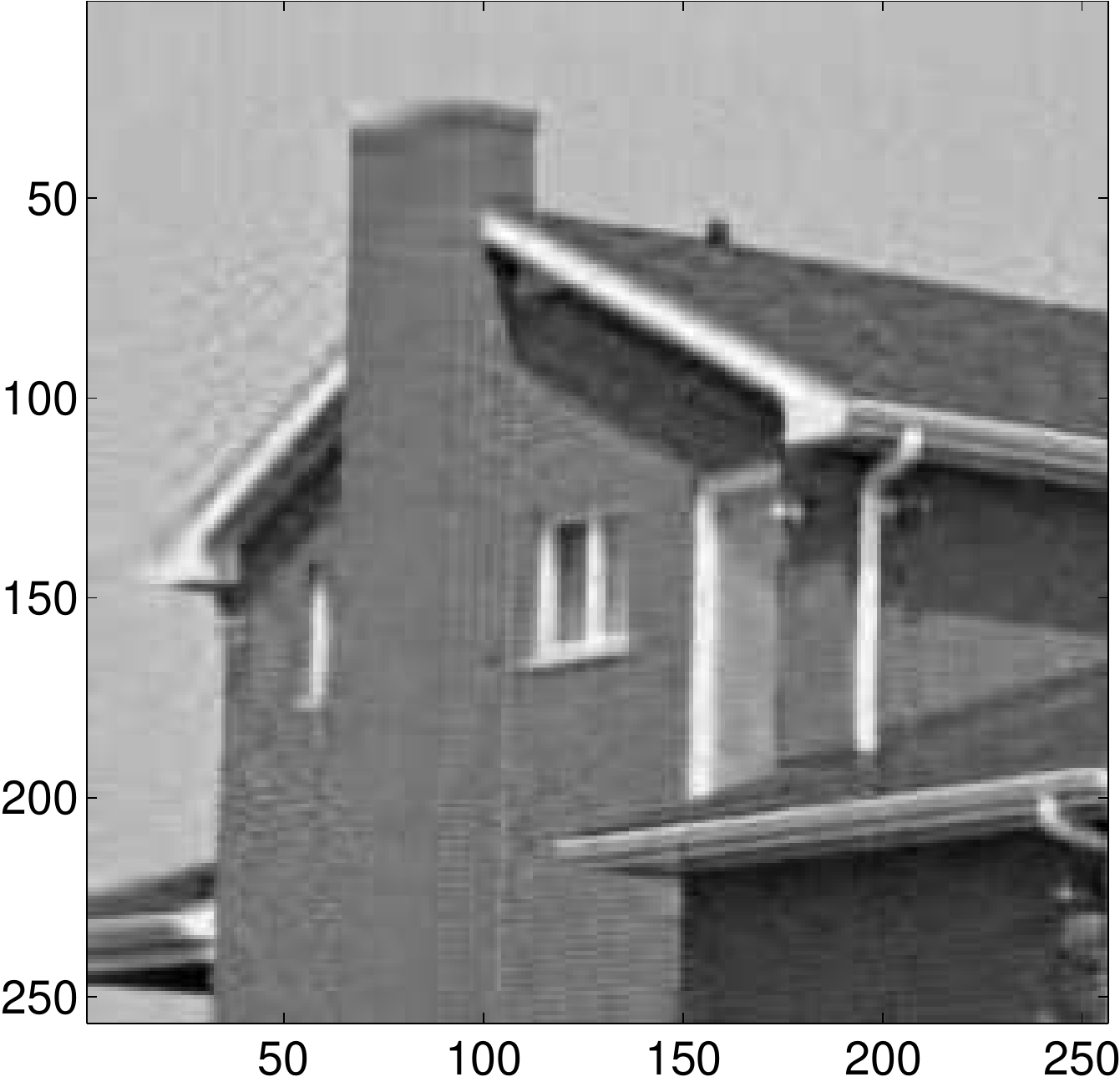}\\
\caption{\small{Reconstruction performance in image denoising settings. The image size is $256 \times 256$ and the desired rank is preset to $\rank = 30$. We observe $33\%$ of the pixels of the best rank-$ 30 $ approximation of the image. We depict the median reconstruction with respect to the best rank-$30$ approximation in dB over $10$ Monte Carlo realizations}} {\label{fig:real2}}
\end{figure*}

\section{Discussion}{\label{sec: conc}}

In this paper, we present new strategies and review existing ones for hard thresholding methods to recover low-rank matrices from dimensionality reducing, linear projections. Our discussion revolves around four basic building blocks that exploit the problem structure to reduce computational complexity without sacrificing stability.  

In theory, constant $ \mu_i $ selection schemes are accompanied with strong RIP constant conditions but empirical evidence reveal signal reconstruction vulnerabilities. While convergence derivations of adaptive schemes are characterized by weaker bounds, the performance gained by this choice in terms of convergence rate, is quite significant. Memory-based methods lead to convergence speed with (almost) no extra cost on the complexity of hard thresholding methods---we provide theoretical evidence for convergence for simple cases but more theoretical justification is needed to generalize this part as future work. Lastly, further estimate refinement over low rank subspaces using gradient update steps or pseudoinversion optimization techniques provides signal reconstruction efficacy, but more computational power is needed per iteration. 

We connect $\epsilon$-approximation low-rank revealing schemes with first-order gradient descent algorithms to solve general affine rank minimization problems; to the best of our knowledge, this is the first attempt to theoretically characterize the performance of iterative greedy algorithms with $\epsilon$-approximation schemes. In all cases, experimental results illustrate the effectiveness of the proposed schemes on different problem configurations.

\section*{Acknowledgments}
This work was supported in part by the European Commission under Grant MIRG-268398, ERC Future Proof, SNF 200021-132548 and DARPA KeCoM program \#11-DARPA-1055. VC also would like to acknowledge Rice University for his Faculty Fellowship.

\appendix

\section{Appendix}
\begin{remark}
Let $\signal \in \mathbb{R}^{\dimension}$ with SVD: $\signal = \boldsymbol{U} \boldsymbol{\Sigma} \boldsymbol{V}^T$, and $\boldsymbol{Y} \in \mathbb{R}^{\dimension}$ with SVD: $\boldsymbol{Y} = \widetilde{\boldsymbol{U}} \widetilde{\boldsymbol{\Sigma}} \widetilde{\boldsymbol{V}}^T$. Assume two sets: $i) $ $\mathcal{S}_1 = \lbrace \boldsymbol{u}_i\boldsymbol{u}_i^T: $ $~i \in \mathcal{I}_1 \rbrace $ where $\boldsymbol{u}_i $ is the $i$-th singular vector of $\signal$ and $\mathcal{I}_1 \subseteq \lbrace 1, \dots, $ $\text{rank}(\signal) \rbrace $ and, $ii)$ $\mathcal{S}_2 = \lbrace \boldsymbol{u}_i\boldsymbol{u}_i^T, \tilde{\boldsymbol{u}_j} \tilde{\boldsymbol{u}_j}^T $ $ :~i \in \mathcal{I}_2,~j \in \mathcal{I}_3 \rbrace $ where $\tilde{\boldsymbol{u}_i} $ is the $i$-th singular vector of $\boldsymbol{Y}$, $\mathcal{I}_1 \subseteq \mathcal{I}_2 \subseteq \lbrace 1, \dots, \text{rank}(\signal) \rbrace $ and, $\mathcal{I}_3 \subseteq \lbrace 1, \dots, \text{rank}(\boldsymbol{Y}) \rbrace $. We observe that the subspaces defined by $\boldsymbol{u}_i\boldsymbol{u}_i^T$ and $\tilde{\boldsymbol{u}_j} \tilde{\boldsymbol{u}_j}^T$ are not necessarily orthogonal. 

To this end, let $\widehat{\mathcal{S}}_2 = \text{ortho}(\mathcal{S}_2)$; this operation can be easily computed via SVD. Then, the following commutativity property holds true for any matrix $ \boldsymbol{W} \in \mathbb{R}^{\dimension} $:
\begin{align}
\mathcal{P}_{\mathcal{S}_1} \mathcal{P}_{\widehat{\mathcal{S}}_2} \boldsymbol{W} = \mathcal{P}_{\widehat{\mathcal{S}}_2} \mathcal{P}_{\mathcal{S}_1} \boldsymbol{W}.
\end{align} 
\end{remark}
\subsection{Proof of Lemma \ref{lemma:act_subspace_exp}}

Given $ \mathcal{X}^\ast \leftarrow \mathcal{P}_{\rank}(\bestsignal) $ using SVD factorization, we define the following quantities: $ \mathcal{S}_i \leftarrow \mathcal{X}_i \cup \mathcal{D}_i,~ \mathcal{S}^\ast_i \leftarrow \text{ortho}\left(\mathcal{X}_i \cup \mathcal{X}^\ast\right) $. Then, given the structure of the sets $\mathcal{S}_i$ and $\mathcal{S}_i^{\ast}$
\begin{align}
\mathcal{P}_{\mathcal{S}_i} \mathcal{P}_{(\mathcal{S}_i^\ast)^{\bot}} = \mathcal{P}_{\mathcal{D}_i} \mathcal{P}_{(\mathcal{X}^\ast \cup \mathcal{X}_i)^{\bot}},  \label{ser:eq:00}.
\end{align} and 
\begin{align}
\mathcal{P}_{\mathcal{S}^\ast_i} \mathcal{P}_{\mathcal{S}_i^{\bot}} = \mathcal{P}_{\mathcal{X}^\ast} \mathcal{P}_{(\mathcal{D}_i \cup \mathcal{X}_i)^{\bot}}
\end{align} Since the subspace defined in $ \mathcal{D}_i $ is the best rank-$ \rank $ subspace, orthogonal to the subspace spanned by $ \mathcal{X}_i $, the following holds true:
\begin{align}
\vectornormbig{\mathcal{P}_{\mathcal{D}_i} \mathcal{P}_{\mathcal{X}_i^{\bot}} \nabla f(\signal(i))}_F^2 &\geq \vectornormbig{\mathcal{P}_{\mathcal{X}^\ast} \mathcal{P}_{\mathcal{X}_i^{\bot}} \nabla f(\signal(i))}_F^2 \Rightarrow \nonumber \\
\vectornormbig{\mathcal{P}_{\mathcal{S}_i} \nabla f(\signal(i))}_F^2 &\geq \vectornormbig{\mathcal{P}_{\mathcal{S}_i^\ast} \nabla f(\signal(i))}_F^2 \nonumber
\end{align} Removing the common subspaces in $ \mathcal{S}_i $ and $ \mathcal{S}_i^\ast $ by the commutativity property of the projection operation and using the shortcut $\mathcal{P}_{\mathcal{A} \setminus \mathcal{B}} \equiv \mathcal{P}_{\mathcal{A}} \mathcal{P}_{\mathcal{B}^{\bot}} $ for sets $\mathcal{A}, ~\mathcal{B}$, we get:
\begin{align}
&\vectornormbig{\mathcal{P}_{\mathcal{S}_i \setminus \mathcal{S}_i^\ast} \nabla f(\signal(i))}_F^2 \geq \vectornormbig{\mathcal{P}_{\mathcal{S}_i^\ast \setminus \mathcal{S}_i} \nabla f(\signal(i))}_F^2 \Rightarrow \nonumber \\
&\vectornormbig{\mathcal{P}_{\mathcal{S}_i \setminus \mathcal{S}_i^\ast} \sensing^\ast \sensing (\bestsignal - \signal(i)) + \mathcal{P}_{\mathcal{S}_i \setminus \mathcal{S}_i^\ast} \sensing^\ast \noise}_F \geq \nonumber \\ &\vectornormbig{\mathcal{P}_{\mathcal{S}_i^\ast \setminus \mathcal{S}_i} \sensing^\ast \sensing (\bestsignal - \signal(i)) + \mathcal{P}_{\mathcal{S}_i^\ast \setminus \mathcal{S}_i} \sensing^\ast \noise}_F \label{ser:eq:01}
\end{align} Next, we assume that $\mathcal{P}_{(\mathcal{A} \setminus \mathcal{B})^{\bot}}$ denotes the orthogonal projection onto the subspace spanned by $\mathcal{P}_{\mathcal{A}} \mathcal{P}_{\mathcal{B}^{\bot}}$. Then, on the left hand side of (\ref{ser:eq:01}), we have:
\begin{align}
\vectornormbig{&\mathcal{P}_{\mathcal{S}_i \setminus \mathcal{S}_i^\ast} \sensing^\ast \sensing (\bestsignal - \signal(i)) + \mathcal{P}_{\mathcal{S}_i \setminus \mathcal{S}_i^\ast} \sensing^\ast \noise}_F \nonumber \\
&\stackrel{(i)}{\leq} \vectornormbig{\mathcal{P}_{\mathcal{S}_i \setminus \mathcal{S}_i^\ast} \sensing^\ast \sensing (\bestsignal - \signal(i))}_F + \vectornormbig{\mathcal{P}_{\mathcal{S}_i \setminus \mathcal{S}_i^\ast} \sensing^\ast \noise}_F \nonumber \\
&\stackrel{(ii)}{=} \vectornormbig{\mathcal{P}_{\mathcal{S}_i \setminus \mathcal{S}_i^\ast} (\bestsignal - \signal(i)) + \mathcal{P}_{\mathcal{S}_i \setminus \mathcal{S}_i^\ast} \sensing^\ast \sensing (\bestsignal - \signal(i))}_F \nonumber \\ &+ \vectornormbig{\mathcal{P}_{\mathcal{S}_i \setminus \mathcal{S}_i^\ast} \sensing^\ast \noise}_F \nonumber \\
&\stackrel{(iii)}{=} \big\|(\id - \mathcal{P}_{\mathcal{S}_i \setminus \mathcal{S}_i^\ast}\sensing^\ast \sensing \mathcal{P}_{\mathcal{S}_i \setminus \mathcal{S}_i^\ast}) (\bestsignal - \signal(i)) \nonumber \\ &+ \mathcal{P}_{\mathcal{S}_i \setminus \mathcal{S}_i^\ast} \sensing^\ast \sensing \mathcal{P}_{(\mathcal{S}_i \setminus \mathcal{S}_i^\ast)^\bot} (\bestsignal - \signal(i))\big\|_F + \vectornormbig{\mathcal{P}_{\mathcal{S}_i \setminus \mathcal{S}_i^\ast} \sensing^\ast \noise}_F \nonumber \\
&\leq \vectornormbig{(\id - \mathcal{P}_{\mathcal{S}_i \setminus \mathcal{S}_i^\ast}\sensing^\ast \sensing \mathcal{P}_{\mathcal{S}_i \setminus \mathcal{S}_i^\ast}) (\bestsignal - \signal(i))}_F \nonumber \\ &+ \vectornormbig{\mathcal{P}_{\mathcal{S}_i \setminus \mathcal{S}_i^\ast} \sensing^\ast \sensing \mathcal{P}_{(\mathcal{S}_i \setminus \mathcal{S}_i^\ast)^\bot} (\bestsignal - \signal(i))}_F + \vectornormbig{\mathcal{P}_{\mathcal{S}_i \setminus \mathcal{S}_i^\ast} \sensing^\ast \noise}_F \nonumber 
\end{align}
\begin{align}
&\stackrel{(iv)}{\leq} \delta_{3\rank}\vectornormbig{\bestsignal - \signal(i)}_F + \vectornormbig{\mathcal{P}_{\mathcal{S}_i \setminus \mathcal{S}_i^\ast} \sensing^\ast \noise}_F \nonumber \\ &+ \vectornormbig{\mathcal{P}_{\mathcal{S}_i \setminus \mathcal{S}_i^\ast} \sensing^\ast \sensing \mathcal{P}_{(\mathcal{S}_i \setminus \mathcal{S}_i^\ast)^\bot} (\bestsignal - \signal(i))}_F \nonumber \\
&\stackrel{(v)}{\leq} \delta_{3\rank}\vectornormbig{\bestsignal - \signal(i)}_F + \vectornormbig{\mathcal{P}_{\mathcal{S}_i \setminus \mathcal{S}_i^\ast} \sensing^\ast \noise}_F \nonumber \\ &+ \delta_{3\rank}\vectornormbig{\mathcal{P}_{(\mathcal{S}_i \setminus \mathcal{S}_i^\ast)^\bot} (\bestsignal - \signal(i))}_F \nonumber 
\end{align}
\begin{align}
&\stackrel{(vi)}{\leq} 2\delta_{3\rank}\vectornormbig{\bestsignal - \signal(i)}_F + \vectornormbig{\mathcal{P}_{\mathcal{S}_i \setminus \mathcal{S}_i^\ast} \sensing^\ast \noise}_F \label{ser:eq:02}
\end{align} where $ (i) $ due to triangle inequality over Frobenius metric norm, $ (ii) $ since $ \mathcal{P}_{\mathcal{S}_i \setminus \mathcal{S}_i^\ast} (\signal(i) - \bestsignal) = \mathbf{0} $, $ (iii) $ by using the fact that $ \signal(i) - \bestsignal := \mathcal{P}_{\mathcal{S}_i \setminus \mathcal{S}_i^\ast}(\signal(i) - \bestsignal) + \mathcal{P}_{(\mathcal{S}_i \setminus \mathcal{S}_i^\ast)^\bot}(\signal(i) - \bestsignal) $, $ (iv) $ due to Lemma \ref{lemma:3}, $ (v) $ due to Lemma \ref{lemma:4} and $ (vi) $ since $ \vectornormbig{\mathcal{P}_{(\mathcal{S}_i \setminus \mathcal{S}_i^\ast)^\bot} (\bestsignal - \signal(i))}_F \leq \vectornormbig{\signal(i) - \bestsignal}_F $.

For the right hand side of (\ref{ser:eq:01}), we calculate:
\begin{align}
&\vectornormbig{\mathcal{P}_{\mathcal{S}_i^\ast \setminus \mathcal{S}_i} \sensing^\ast \sensing (\bestsignal - \signal(i)) + \mathcal{P}_{\mathcal{S}_i^\ast \setminus \mathcal{S}_i} \sensing^\ast \noise}_F \nonumber \\
&\geq \vectornormbig{\mathcal{P}_{\mathcal{S}_i^\ast \setminus \mathcal{S}_i} (\bestsignal - \signal(i))}_F \nonumber \\ &- \vectornormbig{\mathcal{P}_{\mathcal{S}_i^\ast \setminus \mathcal{S}_i} \sensing^\ast \sensing \mathcal{P}_{(\mathcal{S}_i^\ast \setminus \mathcal{S}_i)^{\bot}} (\bestsignal - \signal(i))}_F \nonumber \\ &- \vectornormbig{(\mathcal{P}_{\mathcal{S}_i^\ast \setminus \mathcal{S}_i} \sensing^\ast \sensing \mathcal{P}_{\mathcal{S}_i^\ast \setminus \mathcal{S}_i} - \id)(\bestsignal - \signal(i))}_F - \vectornormbig{ \mathcal{P}_{\mathcal{S}_i^\ast \setminus \mathcal{S}_i}\sensing^\ast \noise}_F \nonumber \\
&\geq \vectornormbig{\mathcal{P}_{\mathcal{S}_i^\ast \setminus \mathcal{S}_i} (\bestsignal - \signal(i))}_F - 2\delta_{2\rank}\vectornormbig{\signal(i) - \bestsignal}_F \nonumber \\ &- \vectornormbig{ \mathcal{P}_{\mathcal{S}_i^\ast \setminus \mathcal{S}_i}\sensing^\ast \noise}_F \label{ser:eq:03}
\end{align} by using Lemmas \ref{lemma:3} and \ref{lemma:4}. Combining (\ref{ser:eq:02}) and (\ref{ser:eq:03}) in (\ref{ser:eq:01}), we get:
\begin{align}
\vectornormbig{\mathcal{P}_{\mathcal{X}^\ast \setminus \mathcal{S}_i}\bestsignal}_F &\leq (2\delta_{2\rank} + 2\delta_{3\rank})\vectornormbig{\signal(i) - \bestsignal}_F \nonumber \\ &+ \sqrt{2(1+\delta_{2\rank})}\vectornormbig{\noise}_2. \nonumber
\end{align}

\subsection{Proof of Theorem \ref{thm:mALPS0}}

Let $ \mathcal{X}^\ast \leftarrow \mathcal{P}_{\rank}(\bestsignal) $ be a set of orthonormal, rank-1 matrices that span the range of $ \bestsignal $. In Algorithm 1, $ \boldsymbol{W}(i) \leftarrow \mathcal{P}_{\rank}(\boldsymbol{V}(i)) $. Thus:
\begin{align}
\vectornormbig{\boldsymbol{W}(i) - \boldsymbol{V}(i)}_F^2 &\leq \vectornormbig{\bestsignal - \boldsymbol{V}(i)}_F^2 \Rightarrow \nonumber \\
\vectornormbig{\boldsymbol{W}(i) - \bestsignal + \bestsignal - \boldsymbol{V}(i)}_F^2 &\leq \vectornormbig{\bestsignal - \boldsymbol{V}(i)}_F^2 \Rightarrow \nonumber \\
\vectornormbig{\boldsymbol{W}(i) - \bestsignal}_F^2 &\leq 2\langle \boldsymbol{W}(i) - \bestsignal, \boldsymbol{V}(i) - \bestsignal \rangle \label{eq:mALPS0:00}
\end{align}

From Algorithm 1, $ i) ~\boldsymbol{V}(i) \in \text{span}(\mathcal{S}_i) $, $ ii) $ $ \signal(i) \in \text{span}( \mathcal{S}_i) $ and $ iii) $ $ \boldsymbol{W}(i) \in \text{span}(\mathcal{S}_i) $. We define $ \mathcal{E} \leftarrow \text{ortho}(\mathcal{S}_i \cup \mathcal{X}^{\ast}) $ where $ \text{rank}(\text{span}(\mathcal{E})) \leq 3\rank $ and let $ \mathcal{P}_{\mathcal{E}} $ be the orthogonal projection onto the subspace defined by $ \mathcal{E} $. 

Since $ \boldsymbol{W}(i) - \bestsignal \in \text{span}(\mathcal{E}) $ and $ \boldsymbol{V}(i) - \bestsignal \in \text{span}(\mathcal{E}) $, the following hold true:
\begin{align}
\boldsymbol{W}(i) &- \bestsignal = \mathcal{P}_{\mathcal{E}} (\boldsymbol{W}(i) - \bestsignal)~~\text{and}~~ \nonumber \\ \boldsymbol{V}(i) &- \bestsignal = \mathcal{P}_{\mathcal{E}} (\boldsymbol{V}(i) - \bestsignal). \nonumber
\end{align}

Then, (\ref{eq:mALPS0:00}) can be written as:
\begin{align}
&\vectornormbig{\boldsymbol{W}(i) - \bestsignal}_F^2 \leq 2\langle \mathcal{P}_{\mathcal{E}}(\boldsymbol{W}(i) - \bestsignal), \mathcal{P}_{\mathcal{E}}(\boldsymbol{V}(i) - \bestsignal) \rangle \Rightarrow \nonumber \\
&= \underbrace{2\langle \mathcal{P}_{\mathcal{E}}(\boldsymbol{W}(i) - \bestsignal), \mathcal{P}_{\mathcal{E}}(\signal(i) - \bestsignal - \mu_i \mathcal{P}_{\mathcal{S}_i}\sensing^\ast \sensing (\signal(i) - \bestsignal)) \rangle}_{\doteq A} \nonumber \\ &+ \underbrace{2\mu_i\langle \mathcal{P}_{\mathcal{E}}(\boldsymbol{W}(i) - \bestsignal), \mathcal{P}_{\mathcal{E}}\mathcal{P}_{\mathcal{S}_i}(\sensing^\ast \noise) \rangle}_{\doteq B}  \label{eq:mALPS0:01}
\end{align}

In B, we observe:
\begin{align}
B &:= 2\mu_i\langle \mathcal{P}_{\mathcal{E}}(\boldsymbol{W}(i) - \bestsignal), \mathcal{P}_{\mathcal{E}}\mathcal{P}_{\mathcal{S}_i}(\sensing^\ast \noise) \rangle \nonumber \\ &\stackrel{(i)}{=} 2\mu_i\langle \boldsymbol{W}(i) - \bestsignal, \mathcal{P}_{\mathcal{S}_i}(\sensing^\ast \noise) \rangle \nonumber \\
&\stackrel{(ii)}{\leq} 2\mu_i \vectornormbig{\boldsymbol{W}(i) - \bestsignal}_F \vectornormbig{\mathcal{P}_{\mathcal{S}_i}(\sensing^\ast \noise)}_F \nonumber \\
&\stackrel{(iii)}{\leq} 2\mu_i \sqrt{1+\delta_{2\rank}}\vectornormbig{\boldsymbol{W}(i) - \bestsignal}_F \vectornormbig{\noise}_2 \label{eq:mALPS0:02}
\end{align} where $ (i) $ holds since $ \mathcal{P}_{\mathcal{S}_i} \mathcal{P}_{\mathcal{E}} = \mathcal{P}_{\mathcal{E}}\mathcal{P}_{\mathcal{S}_i} = \mathcal{P}_{\mathcal{S}_i} $ for $ \text{span}(\mathcal{S}_i) \in \text{span}(\mathcal{E}) $, $ (ii) $ is due to Cauchy-Schwarz inequality and, $ (iii) $ is easily derived using Lemma \ref{lemma:1}.

In A, we perform the following motions:
\begin{align}
A &:= 2\langle \boldsymbol{W}(i) - \bestsignal, \mathcal{P}_{\mathcal{E}}(\signal(i) - \bestsignal) - \mu_i \mathcal{P}_{\mathcal{S}_i}\sensing^\ast \sensing \mathcal{P}_{\mathcal{E}} (\signal(i) - \bestsignal) \rangle \nonumber \\
&\stackrel{(i)}{=} 2\langle \boldsymbol{W}(i) - \bestsignal, \mathcal{P}_{\mathcal{E}}(\signal(i) - \bestsignal) \nonumber \\ &-\mu_i \mathcal{P}_{\mathcal{S}_i}\sensing^\ast \sensing \big[\mathcal{P}_{\mathcal{S}_i} + \mathcal{P}_{\mathcal{S}_i^{\bot}}\big]\mathcal{P}_{\mathcal{E}} (\signal(i) - \bestsignal) \rangle \nonumber \\
& = 2\langle \boldsymbol{W}(i) - \bestsignal, (\id - \mu_i\mathcal{P}_{\mathcal{S}_i}\sensing^\ast \sensing \mathcal{P}_{\mathcal{S}_i}) \mathcal{P}_{\mathcal{E}} (\signal(i) - \bestsignal) \rangle \nonumber \nonumber \\ &- 2\mu_i \langle \boldsymbol{W}(i) - \bestsignal, \mathcal{P}_{\mathcal{S}_i}\sensing^\ast \sensing \mathcal{P}_{\mathcal{S}_i^{\bot}} \mathcal{P}_{\mathcal{E}} (\signal(i) - \bestsignal) \rangle \nonumber \\
&\stackrel{(ii)}{\leq} 2\vectornormbig{\boldsymbol{W}(i) - \bestsignal}_F \vectornormbig{ (\id - \mu_i \mathcal{P}_{\mathcal{S}_i}\sensing^\ast \sensing \mathcal{P}_{\mathcal{S}_i}) \mathcal{P}_{\mathcal{E}} (\signal(i) - \bestsignal)}_F \nonumber \\ &+ 2\mu_i \vectornormbig{\boldsymbol{W}(i) - \bestsignal}_F \vectornormbig{\mathcal{P}_{\mathcal{S}_i}\sensing^\ast \sensing \mathcal{P}_{\mathcal{S}_i^{\bot}}\mathcal{P}_{\mathcal{E}}(\signal(i) - \bestsignal)}_F
\label{eq:mALPS:04}
\end{align} where $ (i) $ is due to $ \mathcal{P}_{\mathcal{E}}(\signal(i) - \bestsignal) := \mathcal{P}_{\mathcal{S}_i} \mathcal{P}_{\mathcal{E}}(\signal(i) - \bestsignal) + \mathcal{P}_{\mathcal{S}_i^{\bot}} \mathcal{P}_{\mathcal{E}}(\signal(i) - \bestsignal) $ and $ (ii) $ follows from Cauchy-Schwarz inequality. 
Since $ \frac{1}{1+\delta_{2\rank}} \leq \mu_i \leq \frac{1}{1-\delta_{2\rank}} $, Lemma \ref{lemma:3} implies:
\begin{align}
\lambda(\id - \mu_i\mathcal{P}_{\mathcal{S}_i}\sensing^\ast \sensing \mathcal{P}_{\mathcal{S}_i}) &\in \Bigg[1 - \frac{1-\delta_{2\rank}}{1+\delta_{2\rank}}, \frac{1+\delta_{2\rank}}{1-\delta_{2\rank}} - 1\Bigg] \nonumber \\ &\leq \frac{2\delta_{2\rank}}{1-\delta_{2\rank}}. \nonumber
\end{align} and thus:
\begin{align}
\vectornormbig{(\id - \mu_i\mathcal{P}_{\mathcal{S}_i}\sensing^\ast \sensing& \mathcal{P}_{\mathcal{S}_i})\mathcal{P}_{\mathcal{E}}(\signal(i) - \bestsignal)}_F \nonumber \\ &\leq \frac{2\delta_{2\rank}}{1-\delta_{2\rank}} \vectornormbig{\mathcal{P}_{\mathcal{E}}(\signal(i) - \bestsignal)}_F. \nonumber
\end{align} Furthermore, according to Lemma \ref{lemma:4}:
\begin{align}
\vectornormbig{\mathcal{P}_{\mathcal{S}_i}\sensing^\ast \sensing \mathcal{P}_{\mathcal{S}_i^{\bot}}\mathcal{P}_{\mathcal{E}}(\signal(i) - \bestsignal)}_F \leq  \delta_{3\rank} \vectornormbig{\mathcal{P}_{\mathcal{S}_i^{\bot}}\mathcal{P}_{\mathcal{E}}(\signal(i) - \bestsignal)}_F \nonumber
\end{align} since $ \text{rank}(\mathcal{P}_{\mathcal{K}}\signal) \leq 3\rank, ~\forall \signal \in \mathbb{R}^{\dimension} $ for $\mathcal{K} \leftarrow \text{ortho}(\mathcal{E} \cup \mathcal{S}_i)$. Since $ \mathcal{P}_{\mathcal{S}_i^{\bot}}\mathcal{P}_{\mathcal{E}}(\signal(i) - \bestsignal) = \mathcal{P}_{\mathcal{X}^\ast \setminus (\mathcal{D}_i \cup \mathcal{X}_i)}\bestsignal $ where 
\begin{align}
\mathcal{D}_i \leftarrow \mathcal{P}_{k}\left(\mathcal{P}_{\mathcal{X}_i^{\bot}}\nabla f(\signal(i))\right), \nonumber
\end{align} then:
\begin{align}
&\vectornormbig{\mathcal{P}_{\mathcal{S}_i^{\bot}}\mathcal{P}_{\mathcal{E}}(\signal(i) - \bestsignal)}_F = \vectornormbig{\mathcal{P}_{\mathcal{X}^\ast \setminus (\mathcal{D}_i \cup \mathcal{X}_i)}\bestsignal}_F \nonumber \\ &\leq (2\delta_{2\rank} + 2\delta_{3\rank})\vectornormbig{\signal(i) - \bestsignal}_F + \sqrt{2(1+\delta_{2\rank})}\vectornormbig{\noise}_2, \nonumber
\end{align} using Lemma \ref{lemma:act_subspace_exp}. Combining the above in (\ref{eq:mALPS:04}), we compute:
\begin{align}
A 
&\leq \Big(\frac{4\delta_{2\rank}}{1-\delta_{2\rank}} + (2\delta_{2\rank} + 2\delta_{3\rank})\frac{2\delta_{3\rank}}{1-\delta_{2\rank}}\Big)\vectornormbig{\boldsymbol{W}(i) - \bestsignal}_F \cdot \nonumber \\ & \vectornormbig{\signal(i) - \bestsignal}_F + \frac{2\delta_{3\rank}}{1-\delta_{2\rank}}  \vectornormbig{\boldsymbol{W}(i) - \bestsignal}_F \sqrt{2(1+\delta_{2\rank})}\vectornormbig{\noise}_2 \label{eq:mALPS0:03}
\end{align}

Combining (\ref{eq:mALPS0:02}) and (\ref{eq:mALPS0:03}) in (\ref{eq:mALPS0:01}), we get:
\begin{align}
&\vectornormbig{\boldsymbol{W}(i) - \bestsignal}_F \nonumber \\ &\leq \Big(\frac{4\delta_{2\rank}}{1-\delta_{2\rank}} + (2\delta_{2\rank} + 2\delta_{3\rank})\frac{2\delta_{3\rank}}{1-\delta_{2\rank}}\Big) \vectornormbig{\signal(i) - \bestsignal}_F \nonumber \\ &+ \Big(\frac{2\sqrt{1+\delta_{2\rank}}}{1 - \delta_{2\rank}} + \frac{2\delta_{3\rank}}{1-\delta_{2\rank}}  \sqrt{2(1+\delta_{2\rank})}\Big) \vectornormbig{\noise}_2 \label{eq:mALPS0:04}
\end{align}

Focusing on steps 5 and 6 of Algorithm 1, we perform similar motions to obtain:
\begin{align}
\vectornormbig{\signal(i+1) - \bestsignal}_F 
&\leq \Big( \frac{1 + 2\delta_{2\rank}}{1-\delta_{2\rank}}\Big)\vectornormbig{\boldsymbol{W}(i) - \bestsignal}_F \nonumber \\ &+ \frac{\sqrt{1+\delta_{\rank}}}{1-\delta_{\rank}} \vectornormbig{\noise}_2
\label{eq:mALPS0:05}
\end{align} 
Combining the recursions in (\ref{eq:mALPS0:04}) and (\ref{eq:mALPS0:05}), we finally compute:
\begin{align}
&\vectornormbig{\signal(i+1) - \bestsignal}_F \leq \rho\vectornormbig{\signal(i) - \bestsignal}_F + \gamma\vectornormbig{\noise}_2, \nonumber
\end{align} for $\rho:=\Big( \frac{1 + 2\delta_{2\rank}}{1-\delta_{2\rank}}\Big)\Big(\frac{4\delta_{2\rank}}{1-\delta_{2\rank}} + (2\delta_{2\rank} + 2\delta_{3\rank})\frac{2\delta_{3\rank}}{1-\delta_{2\rank}}\Big)$ and 
\begin{align}
\gamma &:= \Bigg(\Big( \frac{1 + 2\delta_{2\rank}}{1-\delta_{2\rank}}\Big)\Big(\frac{2\sqrt{1+\delta_{2\rank}}}{1 - \delta_{2\rank}} + \frac{2\delta_{3\rank}}{1-\delta_{2\rank}}  \sqrt{2(1+\delta_{2\rank})}\Big) \nonumber \\ &+ \frac{\sqrt{1+\delta_{\rank}}}{1-\delta_{\rank}} \Bigg) \nonumber
\end{align} For the convergence parameter $ \rho $, further compute:
\begin{align}
&\Big( \frac{1 + 2\delta_{2\rank}}{1-\delta_{2\rank}}\Big)\Big(\frac{4\delta_{2\rank}}{1-\delta_{2\rank}} + (2\delta_{2\rank} + 2\delta_{3\rank})\frac{2\delta_{3\rank}}{1-\delta_{2\rank}}\Big) \nonumber \\ &\leq \frac{1+2\delta_{3\rank}}{(1-\delta_{3\rank})^2} \big( 4\delta_{3\rank} + 8 \delta_{3\rank}^2\big) =: \hat{\rho}.
\end{align} for $ \delta_{\rank} \leq \delta_{2\rank} \leq \delta_{3\rank}$. Calculating the roots of this expression, we easily observe that $ \rho < \hat{\rho} < 1 $ for $ \delta_{3\rank} < 0.1235 $.


\subsection{Proof of Theorem \ref{thm:mALPS5}}

Before we present the proof of Theorem \ref{thm:mALPS5}, we list a series of lemmas that correspond to the motions Algorithm 2 performs.

\begin{lemma}{\label{lemma:greedy}}[Error norm reduction via least-squares optimization] Let $ \mathcal{S}_i $ be a set of orthonormal, rank-1 matrices that span a rank-2$ \rank $ subspace in $ \mathbb{R}^{\dimension} $. Then, the least squares solution $ \boldsymbol{V}(i) $ given by:
\begin{align}
\boldsymbol{V}(i) \leftarrow \argmin_{\boldsymbol{V}: \boldsymbol{V} \in \text{span}(\mathcal{S}_i)} \vectornormbig{\obs - \sensing \boldsymbol{V}}_2^2,  \label{eq:mALPS5:00}
\end{align} satisfies:
\begin{align}
\vectornormbig{\boldsymbol{V}(i) - \bestsignal}_F &\leq \frac{1}{\sqrt{1-\delta_{3\rank}^2(\sensing)}} \vectornormbig{\mathcal{P}_{\mathcal{S}_i^{\bot}}(\boldsymbol{V}(i) - \bestsignal)}_F \nonumber \\ &+ \frac{\sqrt{1+\delta_{2\rank}}}{1-\delta_{3\rank}} \vectornormbig{\noise}_2. \label{eq:mALPS5:01}
\end{align}
\end{lemma}

\begin{proof}
We observe that $ \vectornormbig{\boldsymbol{V}(i) - \bestsignal}_F^2 $ is decomposed as follows:
\begin{align}
\vectornormbig{\boldsymbol{V}(i) - \bestsignal}_F^2 = \vectornormbig{\mathcal{P}_{\mathcal{S}_i} (\boldsymbol{V}(i) - \bestsignal)}_F^2 + \vectornormbig{\mathcal{P}_{\mathcal{S}_i^{\bot}}(\boldsymbol{V}(i) - \bestsignal)}_F^2. \label{eq:mALPS5:02}
\end{align}
In (\ref{eq:mALPS5:00}), $ \boldsymbol{V}(i) $ is the minimizer over the low-rank subspace spanned by $ \mathcal{S}_i $ with $ \text{rank}(\text{span}(\mathcal{S}_i)) \leq 2\rank $. Using the optimality condition (Lemma \ref{lemma:5}) over the convex set $ \Theta = \lbrace \signal: \text{span}(\signal) \in \mathcal{S}_i \rbrace $, we have:
\begin{align}
&\langle \nabla f(\boldsymbol{V}(i)), \mathcal{P}_{\mathcal{S}_i}(\bestsignal - \boldsymbol{V}(i)) \rangle \geq 0 \Rightarrow \nonumber \\ &\langle \sensing \boldsymbol{V}(i) - \obs, \sensing \mathcal{P}_{\mathcal{S}_i}(\boldsymbol{V}(i) - \bestsignal) \rangle \leq 0. \label{eq:mALPS5:03a}
\end{align} for $ \mathcal{P}_{\mathcal{S}_i}\bestsignal \in \text{span}(\mathcal{S}_i) $.
Given condition (\ref{eq:mALPS5:03a}), the first term on the right hand side of (\ref{eq:mALPS5:02}) becomes: 
\begin{align}
&\vectornormbig{\mathcal{P}_{\mathcal{S}_i}(\boldsymbol{V}(i) - \bestsignal)}_F^2 \nonumber \\ &= \langle \boldsymbol{V}(i) - \bestsignal, \mathcal{P}_{\mathcal{S}_i}(\boldsymbol{V}(i) - \bestsignal) \rangle \nonumber \\ 
                            &\stackrel{(\ref{eq:mALPS5:03a})}{\leq} \langle \boldsymbol{V}(i) - \bestsignal, \mathcal{P}_{\mathcal{S}_i}(\boldsymbol{V}(i) - \bestsignal) \rangle \nonumber \\ &- \langle \sensing \boldsymbol{V}(i) - \obs, \sensing \mathcal{P}_{\mathcal{S}_i}(\boldsymbol{V}(i) - \bestsignal) \rangle \nonumber\\ 
                            &\leq | \langle \boldsymbol{V}(i) - \bestsignal, (\id - \sensing^\ast \sensing)\mathcal{P}_{\mathcal{S}_i}(\boldsymbol{V}(i) - \bestsignal) \rangle | \nonumber \\ &+ \langle \noise, \sensing \mathcal{P}_{\mathcal{S}_i}(\boldsymbol{V}(i) - \bestsignal) \rangle  \label{eq:mALPS5:08}
\end{align} Focusing on the term $ | \langle \boldsymbol{V}(i) - \bestsignal, (\id - \sensing^\ast \sensing)\mathcal{P}_{\mathcal{S}_i}(\boldsymbol{V}(i) - \bestsignal) \rangle | $, we derive the following:
\begin{align}
&| \langle \boldsymbol{V}(i) - \bestsignal, (\id - \sensing^\ast \sensing)\mathcal{P}_{\mathcal{S}_i}(\boldsymbol{V}(i) - \bestsignal) \rangle | \nonumber \\ 
&= | \langle \boldsymbol{V}(i) - \bestsignal, \mathcal{P}_{\mathcal{S}_i}(\boldsymbol{V}(i) - \bestsignal) \rangle  \nonumber \\ &- \langle \boldsymbol{V}(i) - \bestsignal, \sensing^\ast \sensing\mathcal{P}_{\mathcal{S}_i}(\boldsymbol{V}(i) - \bestsignal) \rangle | \nonumber\\
&\stackrel{(i)}{=} | \langle \mathcal{P}_{\mathcal{S}_i \cup \mathcal{X}^\ast}(\boldsymbol{V}(i) - \bestsignal), \mathcal{P}_{\mathcal{S}_i}(\boldsymbol{V}(i) - \bestsignal) \rangle  \nonumber \\ &- \langle \sensing \mathcal{P}_{\mathcal{S}_i \cup \mathcal{X}^\ast} (\boldsymbol{V}(i) - \bestsignal), \sensing \mathcal{P}_{\mathcal{S}_i}(\boldsymbol{V}(i) - \bestsignal) \rangle | \nonumber \\
&\stackrel{(ii)}{=} | \langle \mathcal{P}_{\mathcal{S}_i \cup \mathcal{X}^\ast}(\boldsymbol{V}(i) - \bestsignal), \mathcal{P}_{\mathcal{S}_i \cup \mathcal{X}^\ast}\mathcal{P}_{\mathcal{S}_i}(\boldsymbol{V}(i) - \bestsignal) \rangle  \nonumber \\ &- \langle \sensing \mathcal{P}_{\mathcal{S}_i \cup \mathcal{X}^\ast} (\boldsymbol{V}(i) - \bestsignal), \sensing \mathcal{P}_{\mathcal{S}_i \cup \mathcal{X}^\ast} \mathcal{P}_{\mathcal{S}_i}(\boldsymbol{V}(i) - \bestsignal) \rangle | \nonumber \\
&= | \langle \boldsymbol{V}(i) - \bestsignal, (\id - \mathcal{P}_{\mathcal{S}_i \cup \mathcal{X}^\ast}\sensing^\ast \sensing\mathcal{P}_{\mathcal{S}_i \cup \mathcal{X}^\ast})\mathcal{P}_{\mathcal{S}_i}(\boldsymbol{V}(i) - \bestsignal) \rangle |  \nonumber
\end{align} where $ (i) $ follows from the facts that $ \boldsymbol{V}(i) - \bestsignal \in \text{span}(\text{ortho}(\mathcal{S}_i \cup \mathcal{X}^\ast)) $ and thus $ \mathcal{P}_{\mathcal{S}_i \cup \mathcal{X}^{\ast}}(\boldsymbol{V}(i) - \bestsignal) = \boldsymbol{V}(i) - \bestsignal $ and $ (ii) $ is due to $ \mathcal{P}_{\mathcal{S}_i \cup \mathcal{X}^\ast} \mathcal{P}_{\mathcal{S}_i} = \mathcal{P}_{\mathcal{S}_i} $ since $ \text{span}(\mathcal{S}_i) \subseteq \text{span}(\text{ortho}(\mathcal{S}_i \cup \mathcal{X}^\ast)) $. Then, (\ref{eq:mALPS5:08}) becomes:
\begin{align}
&\vectornormbig{\mathcal{P}_{\mathcal{S}_i}(\boldsymbol{V}(i) - \bestsignal)}_F^2 \nonumber \\ &\leq | \langle \boldsymbol{V}(i) - \bestsignal, (\id - \mathcal{P}_{\mathcal{S}_i \cup \mathcal{X}^\ast}\sensing^\ast \sensing\mathcal{P}_{\mathcal{S}_i \cup \mathcal{X}^\ast})\mathcal{P}_{\mathcal{S}_i}(\boldsymbol{V}(i) - \bestsignal) \rangle |  \nonumber \\ &+ \langle \noise, \sensing \mathcal{P}_{\mathcal{S}_i}(\boldsymbol{V}(i) - \bestsignal) \rangle \nonumber \\
					&\stackrel{(i)}{\leq} \vectornormbig{\boldsymbol{V}(i) - \bestsignal}_F\vectornormbig{ (\id - \mathcal{P}_{\mathcal{S}_i \cup \mathcal{X}^\ast}\sensing^\ast \sensing\mathcal{P}_{\mathcal{S}_i \cup \mathcal{X}^\ast})\mathcal{P}_{\mathcal{S}_i}(\boldsymbol{V}(i) - \bestsignal)}_F \nonumber \\ &+ \vectornormbig{\mathcal{P}_{\mathcal{S}_i}\sensing^\ast \noise}_F \vectornormbig{\mathcal{P}_{\mathcal{S}_i}(\boldsymbol{V}(i) - \bestsignal)}_F \nonumber \\
                    &\stackrel{(ii)}{\leq} \delta_{3\rank} \vectornormbig{\mathcal{P}_{\mathcal{S}_i}(\boldsymbol{V}(i) - \bestsignal)}_F \vectornormbig{\boldsymbol{V}(i) - \bestsignal}_F \nonumber \\ &+ \sqrt{1+\delta_{2\rank}} \vectornormbig{\mathcal{P}_{\mathcal{S}_i}(\boldsymbol{V}(i) - \bestsignal)}_F \vectornormbig{\noise}_2,  \label{eq:mALPS5:09}                           
\end{align} where $ (i) $ comes from Cauchy-Swartz inequality and $ (ii) $ is due to Lemmas \ref{lemma:1} and \ref{lemma:3}. Simplifying the above quadratic expression, we obtain: 
\begin{align}
\vectornormbig{\mathcal{P}_{\mathcal{S}_i}(\boldsymbol{V}(i) - \bestsignal)}_F \leq \delta_{3\rank} \vectornormbig{\boldsymbol{V}(i) - \bestsignal}_F + \sqrt{1+\delta_{2\rank}} \vectornormbig{\noise}_2. \label{eq:mALPS5:10}                           
\end{align}

As a consequence, (\ref{eq:mALPS5:02}) can be upper bounded by:
\begin{align}
\vectornormbig{\boldsymbol{V}(i) - \bestsignal}_F^2 &\leq \big(\delta_{3\rank} \vectornormbig{\boldsymbol{V}(i) - \bestsignal}_F + \sqrt{1+\delta_{2\rank}} \vectornormbig{\noise}_2\big)^2 \nonumber \\ &+ \vectornormbig{\mathcal{P}_{\mathcal{S}_i^{\bot}}(\boldsymbol{V}(i) - \bestsignal)}_F^2. \label{eq:mALPS5:11}                           
\end{align}

We form the quadratic polynomial for this inequality assuming as unknown variable the quantity $ \vectornormbig{\boldsymbol{V}(i) - \bestsignal}_F $. Bounding by the largest root of the resulting polynomial, we get:
\begin{align}
\vectornormbig{\boldsymbol{V}(i) - \bestsignal}_F &\leq \frac{1}{\sqrt{1-\delta_{3\rank}^2(\sensing)}} \vectornormbig{\mathcal{P}_{\mathcal{S}_i^{\bot}}(\boldsymbol{V}(i) - \bestsignal)}_F \nonumber \\ &+ \frac{\sqrt{1+\delta_{2\rank}}}{1-\delta_{3\rank}} \vectornormbig{\noise}_2.\label{eq:mALPS5:12a}
\end{align} 
\end{proof}

The following Lemma characterizes how subspace {\it pruning} affects the recovered energy:

\begin{lemma}{\label{lemma:comb_selection}}[Best rank-$ \rank $ subspace selection] Let $ \boldsymbol{V}(i) \in \mathbb{R}^{\dimension} $ be a rank-$ 2\rank $ proxy matrix in the subspace spanned by $ \mathcal{S}_i $ and let $ \signal(i+1) \leftarrow \mathcal{P}_{\rank}(\boldsymbol{V}(i)) $ denote the best rank-$ \rank $ approximation to $ \boldsymbol{V}(i) $, according to (\ref{eq:svd_proj}). Then:
\begin{align}
\vectornormbig{\signal(i+1)  - \boldsymbol{V}(i)}_F &\leq \vectornormbig{\mathcal{P}_{\mathcal{S}_i}(\boldsymbol{V}(i) - \bestsignal)}_F \leq \vectornormbig{\boldsymbol{V}(i) - \bestsignal}_F. \label{eq:mALPS5:13}
\end{align}
\end{lemma}

\begin{proof}
Since $ \signal(i+1) $ denotes the best rank-$ \rank $ approximation to $ \boldsymbol{V}(i) $, the following inequality holds for any rank-$\rank$ matrix $ \signal \in \mathbb{R}^{\dimension} $ in the subspace spanned by $ \mathcal{S}_i $, i.e. $ \forall \signal \in \text{span}(\mathcal{S}_i) $:
\begin{align}
\vectornormbig{\signal(i+1)  - \boldsymbol{V}(i)}_F \leq \vectornormbig{\signal  - \boldsymbol{V}(i)}_F \label{eq:mALPS5:14}.
\end{align} Since $ \mathcal{P}_{\mathcal{S}_i} \boldsymbol{V}(i) =  \boldsymbol{V}(i) $, the left inequality in (\ref{eq:mALPS5:13}) is satisfied for $ \signal := \mathcal{P}_{\mathcal{S}_i} \bestsignal $ in (\ref{eq:mALPS5:14}). 
\end{proof}

\begin{lemma}{\label{lemma:noname}} Let $ \boldsymbol{V}(i) $ be the least squares solution in Step 2 of the ADMiRA algorithm
and let $ \signal(i+1) $ be a proxy, rank-$ \rank $ matrix to $ \boldsymbol{V}(i) $ according to: $ \signal(i+1) \leftarrow \mathcal{P}_k(\boldsymbol{V}(i)). $ Then, $ \vectornormbig{\signal(i+1)  - \bestsignal}_F $ can be expressed in terms of the distance from $ \boldsymbol{V}(i) $ to $ \bestsignal $ as follows:
\begin{align}
\vectornormbig{\signal(i+1) - \bestsignal}_F &\leq \sqrt{1 + 3\delta_{3\rank}^2} \vectornormbig{\boldsymbol{V}(i) - \bestsignal}_F \nonumber \\ &+ \sqrt{1 + 3\delta_{3\rank}^2} \sqrt{\frac{3(1+\delta_{2\rank})}{1 + 3\delta_{3\rank}^2}}\vectornormbig{\noise}_2. \label{eq:73}
\end{align}
\end{lemma}

\begin{proof}
We observe the following
\begin{align}
\vectornormbig{\signal(i+1)  - \bestsignal}_F^2 &= \vectornormbig{\signal(i+1)  - \boldsymbol{V}(i) + \boldsymbol{V}(i) - \bestsignal}_F^2 \nonumber \\ 
									   &= \vectornormbig{\boldsymbol{V}(i) - \bestsignal}_F^2 + \vectornormbig{\boldsymbol{V}(i) - \signal(i+1) }_F^2 \nonumber \\ &- 2\langle \boldsymbol{V}(i) - \bestsignal, \boldsymbol{V}(i) - \signal(i+1)  \rangle. \label{eq:18}
\end{align}
Focusing on the right hand side of expression (\ref{eq:18}), $ \langle \boldsymbol{V}(i) - \bestsignal, \boldsymbol{V}(i) - \signal(i+1)  \rangle = \langle \boldsymbol{V}(i) - \bestsignal, \mathcal{P}_{\mathcal{S}_i}(\boldsymbol{V}(i) - \signal(i+1) ) \rangle $ can be similarly analysed as in Lemma 10 where we obtain the following expression:
\begin{align}
&|\langle \boldsymbol{V}(i) - \bestsignal, \mathcal{P}_{\mathcal{S}_i}(\boldsymbol{V}(i) - \signal(i+1) ) \rangle | \nonumber \\ &\leq \delta_{3\rank} \vectornormbig{\boldsymbol{V}(i) - \bestsignal}_F \vectornormbig{\boldsymbol{V}(i) - \signal(i+1) }_F \nonumber \\ &+ \sqrt{1+\delta_{2\rank}} \vectornormbig{\boldsymbol{V}(i) - \signal(i+1) }_F \vectornormbig{\noise}_2. \label{eq:24}
\end{align}

Now, expression (\ref{eq:18}) can be further transformed as:
\begin{align}
\vectornormbig{\signal(i+1)  - \bestsignal}_F^2 
										&\stackrel{(i)}{\leq} \vectornormbig{\boldsymbol{V}(i) - \bestsignal}_F^2 + \vectornormbig{\boldsymbol{V}(i) - \signal(i+1) }_F^2 \nonumber \\ &+  2(\delta_{3\rank} \vectornormbig{\boldsymbol{V}(i) - \bestsignal}_F \vectornormbig{\boldsymbol{V}(i) - \signal(i+1) }_F \nonumber \\ &+ \sqrt{1+\delta_{2\rank}} \vectornormbig{\boldsymbol{V}(i) - \signal(i+1) }_F \vectornormbig{\noise}_2) \label{eq:25b}										                                       
\end{align} where $ (i) $ is due to (\ref{eq:24}). Using Lemma \ref{lemma:comb_selection}, we further have:
\begin{align}
\vectornormbig{\signal(i+1)  - \bestsignal}_F^2 &\leq  \vectornormbig{\boldsymbol{V}(i) - \bestsignal}_F^2 + \vectornormbig{\mathcal{P}_{\mathcal{S}_i}(\boldsymbol{V}(i) - \bestsignal)}_F^2 \nonumber \\ &+  2\Big(\delta_{3\rank} \vectornormbig{\boldsymbol{V}(i) - \bestsignal}_F \vectornormbig{\mathcal{P}_{\mathcal{S}_i}(\boldsymbol{V}(i) - \bestsignal)}_F \nonumber \\ &+ \sqrt{1+\delta_{2\rank}} \vectornormbig{\mathcal{P}_{\mathcal{S}_i}(\boldsymbol{V}(i) - \bestsignal)}_F \vectornormbig{\noise}_2\Big) \label{eq:mALPS5:20}
\end{align} Furthermore, replacing $ \vectornormbig{\mathcal{P}_{\mathcal{S}_i}(\bestsignal - \boldsymbol{V}(i))}_F $ with its upper bound defined in (\ref{eq:mALPS5:10}), we get:
\begin{align}
&\vectornormbig{\signal(i+1)  - \bestsignal}_2^2 \nonumber \\
															   &\stackrel{(i)}{\leq} \Big(1 + 3\delta_{3\rank}^2\Big)\Bigg(\vectornormbig{\boldsymbol{V}(i) - \bestsignal}_2 + \sqrt{\frac{3(1+\delta_{2\rank})}{1 + 3\delta_{3\rank}^2}}\vectornormbig{\noise}\Bigg)^2 \label{eq:73a}
\end{align} where $ (i) $ is obtained by completing the squares and eliminating negative terms. 
\end{proof}

Applying basic algebra tools in (\ref{eq:73}) and (\ref{eq:mALPS5:01}), we get:
\begin{align}
&\vectornormbig{\signal(i+1) - \bestsignal}_F \leq \sqrt{\frac{1+3\delta_{3\rank}^2}{1-\delta_{3\rank}^2}}\vectornormbig{\mathcal{P}_{\mathcal{S}_i^{\bot}}(\boldsymbol{V}(i) - \bestsignal)}_F \nonumber \\ &+ \Big(\frac{\sqrt{1+3\delta_{3\rank}^2}}{1-\delta_{3\rank}} + \sqrt{3}\Big)\sqrt{1+\delta_{2\rank}}\vectornormbig{\noise}_2. \nonumber
\end{align} 

Since $ \boldsymbol{V}(i) \in \text{span}(\mathcal{S}_i) $, we observe $ \mathcal{P}_{\mathcal{S}_i^{\bot}}(\boldsymbol{V}(i) - \bestsignal) = -\mathcal{P}_{\mathcal{S}_i^{\bot}} \bestsignal = -\mathcal{P}_{\mathcal{X}^\ast \setminus (\mathcal{D}_i \cup \mathcal{X}_i)} \bestsignal $. Then, using Lemma \ref{lemma:act_subspace_exp}, we obtain:
\begin{align}
&\vectornormbig{\signal(i+1) - \bestsignal}_F \nonumber \\ 
&\leq \big(2\delta_{2\rank} + 2\delta_{3\rank}\big)\sqrt{\frac{1+3\delta_{3\rank}^2}{1-\delta_{3\rank}^2}}  \vectornormbig{\bestsignal - \signal(i)}_F \nonumber \\ &+ \Bigg[\sqrt{\frac{1+3\delta_{3\rank}^2}{1-\delta_{3\rank}^2}} \sqrt{2(1+\delta_{3\rank})} \nonumber \\ &+ \Big(\frac{\sqrt{1+3\delta_{3\rank}^2}}{1-\delta_{3\rank}} + \sqrt{3}\Big)\sqrt{1+\delta_{2\rank}}\Bigg] \vectornormbig{\noise}_2
\end{align}

Given $ \delta_{2\rank} \leq \delta_{3\rank} $, $ \rho $ is upper bounded by $\rho < 4\delta_{3\rank}\sqrt{\frac{1+3\delta_{3\rank}}{1-\delta_{3\rank}^2}} $. Then, $4\delta_{3\rank}\sqrt{\frac{1+3\delta_{3\rank}}{1-\delta_{3\rank}^2}} < 1 \Leftrightarrow \delta_{3\rank} < 0.2267.$


\subsection{Proof of Theorem \ref{thm:mALPS0:memory}}

Let $ \mathcal{X}^\ast \leftarrow \mathcal{P}_{\rank}(\bestsignal) $ be a set of orthonormal, rank-1 matrices that span the range of $ \bestsignal $. In Algorithm 3, $ \signal(i+1) $ is the best rank-$ \rank $ approximation of $ \boldsymbol{V}(i) $. Thus:
\begin{align}
\vectornormbig{\signal(i+1) - \boldsymbol{V}(i)}_F^2 &\leq \vectornormbig{\bestsignal - \boldsymbol{V}(i)}_F^2 \Rightarrow \nonumber \\
\vectornormbig{\signal(i+1) - \bestsignal}_F^2 &\leq 2\langle \signal(i+1) - \bestsignal, \boldsymbol{V}(i) - \bestsignal \rangle \label{eq:mALPS0_memory:00}
\end{align}

From Algorithm 3, $ i) ~\boldsymbol{V}(i) \in \text{span}(\mathcal{S}_i) $, $ ii) $ $ \boldsymbol{Q}_i \in \text{span}( \mathcal{S}_i) $ and $ iii) $ $ \boldsymbol{W}(i) \in \text{span}(\mathcal{S}_i) $. We define $ \mathcal{E} \leftarrow \text{ortho}(\mathcal{S}_i \cup \mathcal{X}^{\ast}) $ where we observe $ \text{rank}(\text{span}(\mathcal{E})) \leq 4\rank $ and let $ \mathcal{P}_{\mathcal{E}} $ be the orthogonal projection onto the subspace defined by $ \mathcal{E} $. 

Since $ \signal(i+1) - \bestsignal \in \text{span}(\mathcal{E}) $ and $ \boldsymbol{V}(i) - \bestsignal \in \text{span}(\mathcal{E}) $, the following hold true:
\begin{align}
\signal(i+1) - \bestsignal = \mathcal{P}_{\mathcal{E}} (\signal(i+1) - \bestsignal), \nonumber
\end{align} and,
\begin{align}
\boldsymbol{V}(i) - \bestsignal = \mathcal{P}_{\mathcal{E}} (\boldsymbol{V}(i) - \bestsignal). \nonumber 
\end{align} 

\begin{figure*}[!htp]
\begin{align}
g(i+1) &\leq \left[b_1\Big(\frac{\alpha (1+\tau_i) + \sqrt{\Delta}}{2} \Big)^{i+1} + b_2\Big(\frac{\alpha (1+\tau_i) - \sqrt{\Delta}}{2} \Big)^{i+1}\right]\vectornormbig{\signal(0) - \bestsignal}_F \nonumber \\ 
&\leq \left[(b_1 + b_2)\Big(\frac{\alpha (1+\tau_i) + \sqrt{\Delta}}{2} \Big)^{i+1}\right]\vectornormbig{\signal(0) - \bestsignal}_F \label{eq:mALPS_memory:09}
\end{align} 
\hrulefill
\end{figure*}

Then, (\ref{eq:mALPS0_memory:00}) can be written as:
\begin{align}
&\vectornormbig{\signal(i+1) - \bestsignal}_F^2 \nonumber \\
&\leq 2\langle \mathcal{P}_{\mathcal{E}}(\signal(i+1) - \bestsignal), \mathcal{P}_{\mathcal{E}}(\boldsymbol{V}(i) - \bestsignal) \rangle \nonumber \\
 &= 2\langle \mathcal{P}_{\mathcal{E}}(\signal(i+1) - \bestsignal), \mathcal{P}_{\mathcal{E}}\left(\boldsymbol{Q}_i + \mu_i \mathcal{P}_{\mathcal{S}_i} \sensing^\ast \sensing (\bestsignal - \boldsymbol{Q}_i) - \bestsignal\right) \rangle \nonumber \\
&\stackrel{(i)}{=} 2\langle \signal(i+1) - \bestsignal, \mathcal{P}_{\mathcal{E}}(\boldsymbol{Q}_i - \bestsignal) \nonumber \\ &- \mu_i \mathcal{P}_{\mathcal{S}_i}\sensing^\ast \sensing \big[\mathcal{P}_{\mathcal{S}_i} + \mathcal{P}_{\mathcal{S}_i^{\bot}}\big]\mathcal{P}_{\mathcal{E}} (\boldsymbol{Q}_i - \bestsignal) \rangle \nonumber \\
& = 2\langle \signal(i+1) - \bestsignal, (\id - \mu_i\mathcal{P}_{\mathcal{S}_i}\sensing^\ast \sensing \mathcal{P}_{\mathcal{S}_i}) \mathcal{P}_{\mathcal{E}} (\boldsymbol{Q}_i - \bestsignal) \rangle \nonumber \\ &- 2\mu_i \langle \signal(i+1) - \bestsignal, \mathcal{P}_{\mathcal{S}_i}\sensing^\ast \sensing \mathcal{P}_{\mathcal{S}_i^{\bot}} \mathcal{P}_{\mathcal{E}} (\boldsymbol{Q}_i - \bestsignal) \rangle \nonumber \\
&\stackrel{(ii)}{\leq} 2\vectornormbig{\signal(i+1) - \bestsignal}_F \vectornormbig{ (\id - \mu_i \mathcal{P}_{\mathcal{S}_i}\sensing^\ast \sensing \mathcal{P}_{\mathcal{S}_i}) \mathcal{P}_{\mathcal{E}} (\boldsymbol{Q}_i - \bestsignal)}_F \nonumber \\ &+ 2\mu_i \vectornormbig{\signal(i+1) - \bestsignal}_F \vectornormbig{\mathcal{P}_{\mathcal{S}_i}\sensing^\ast \sensing \mathcal{P}_{\mathcal{S}_i^{\bot}}\mathcal{P}_{\mathcal{E}}(\boldsymbol{Q}_i - \bestsignal)}_F \label{eq:mALPS_memory:04}
\end{align} where $ (i) $ is due to $ \mathcal{P}_{\mathcal{E}}(\boldsymbol{Q}_i - \bestsignal) := \mathcal{P}_{\mathcal{S}_i} \mathcal{P}_{\mathcal{E}}(\boldsymbol{Q}_i - \bestsignal) + \mathcal{P}_{\mathcal{S}_i^{\bot}} \mathcal{P}_{\mathcal{E}}(\boldsymbol{Q}_i - \bestsignal) $ and $ (ii) $ follows from Cauchy-Schwarz inequality. 
Since $ \frac{1}{1+\delta_{3\rank}} \leq \mu_i \leq \frac{1}{1-\delta_{3\rank}} $, Lemma \ref{lemma:3} implies:
\begin{align}
\lambda(\id - \mu_i\mathcal{P}_{\mathcal{S}_i}\sensing^\ast \sensing \mathcal{P}_{\mathcal{S}_i}) \in \Bigg[1 - \frac{1-\delta_{3\rank}}{1+\delta_{3\rank}}, \frac{1+\delta_{3\rank}}{1-\delta_{3\rank}} - 1\Bigg] \leq \frac{2\delta_{3\rank}}{1-\delta_{3\rank}}. \nonumber
\end{align} and thus:
\begin{align}
\vectornormbig{(\id - \mu_i\mathcal{P}_{\mathcal{S}_i}\sensing^\ast \sensing \mathcal{P}_{\mathcal{S}_i})&\mathcal{P}_{\mathcal{E}}(\boldsymbol{Q}_i - \bestsignal)}_F \nonumber \\ &\leq \frac{2\delta_{3\rank}}{1-\delta_{3\rank}} \vectornormbig{\mathcal{P}_{\mathcal{E}}(\boldsymbol{Q}_i - \bestsignal)}_F. \nonumber
\end{align} Furthermore, according to Lemma \ref{lemma:4}:
\begin{align}
\vectornormbig{\mathcal{P}_{\mathcal{S}_i}\sensing^\ast \sensing \mathcal{P}_{\mathcal{S}_i^{\bot}}\mathcal{P}_{\mathcal{E}}(\boldsymbol{Q}_i - \bestsignal)}_F \leq  \delta_{4\rank} \vectornormbig{\mathcal{P}_{\mathcal{S}_i^{\bot}}\mathcal{P}_{\mathcal{E}}(\boldsymbol{Q}_i - \bestsignal)}_F \nonumber
\end{align} since $ \text{rank}(\mathcal{P}_{\mathcal{K}}\boldsymbol{Q}) \leq 4\rank, ~\forall \boldsymbol{Q} \in \mathbb{R}^{\dimension} $ where $\mathcal{K} \leftarrow \text{ortho}(\mathcal{E} \cup \mathcal{S}_i)$. Since $ \mathcal{P}_{\mathcal{S}_i^{\bot}}\mathcal{P}_{\mathcal{E}}(\boldsymbol{Q}_i - \bestsignal) = \mathcal{P}_{\mathcal{X}^\ast \setminus (\mathcal{D}_i \cup \mathcal{X}_i)}\bestsignal $ where 
\begin{align}
\mathcal{D}_i \leftarrow\mathcal{P}_{k}\left(\mathcal{P}_{\mathcal{Q}_i^{\bot}}\nabla f(\boldsymbol{Q}_i)\right), \nonumber
\end{align} then:
\begin{align}
\vectornormbig{\mathcal{P}_{\mathcal{S}_i^{\bot}}\mathcal{P}_{\mathcal{E}}(\boldsymbol{Q}_i - \bestsignal)}_F &= \vectornormbig{\mathcal{P}_{\mathcal{X}^\ast \setminus (\mathcal{D}_i \cup \mathcal{X}_i)}\bestsignal}_F \leq (2\delta_{3\rank} \nonumber \\ &+ 2\delta_{4\rank})\vectornormbig{\boldsymbol{Q}_i - \bestsignal}_F,
\end{align} using Lemma \ref{lemma:act_subspace_exp}. Using the above in (\ref{eq:mALPS_memory:04}), we compute:
\begin{align}
\vectornormbig{\signal(i+1) &- \bestsignal}_F \nonumber \\ &\leq \Big(\frac{4\delta_{3\rank}}{1-\delta_{3\rank}} + (2\delta_{3\rank} + 2\delta_{4\rank})\frac{2\delta_{3\rank}}{1-\delta_{3\rank}}\Big) \vectornormbig{\boldsymbol{Q}_i - \bestsignal}_F \label{eq:mALPS0_memory:04}
\end{align}

Furthermore:
\begin{align}
\vectornormbig{\boldsymbol{Q}_i - \bestsignal}_F &= \vectornormbig{\signal(i) + \tau_i(\signal(i) - \signal(i-1))}_F \nonumber \\ &= \vectornormbig{(1+\tau_i)(\signal(i) - \bestsignal) + \tau_i(\bestsignal - \signal(i-1))}_F \nonumber \\ &\leq (1+\tau_i)\vectornormbig{\signal(i) - \bestsignal}_F + \tau_i\vectornormbig{\signal(i-1) - \bestsignal}_F \label{eq:mALPS_memory:05}
\end{align} Combining (\ref{eq:mALPS0_memory:04}) and (\ref{eq:mALPS_memory:05}), we get:
\begin{align}
&\vectornormbig{\signal(i+1) - \bestsignal}_F \nonumber \\ &\leq (1+\tau_i)\Big(\frac{4\delta_{3\rank}}{1-\delta_{3\rank}} + (2\delta_{3\rank} + 2\delta_{4\rank})\frac{2\delta_{3\rank}}{1-\delta_{3\rank}}\Big) \vectornormbig{\signal(i) - \bestsignal}_F \nonumber \\ &+ \tau_i\Big(\frac{4\delta_{3\rank}}{1-\delta_{3\rank}} + (2\delta_{3\rank} + 2\delta_{4\rank})\frac{2\delta_{3\rank}}{1-\delta_{3\rank}}\Big) \vectornormbig{\signal(i-1) - \bestsignal}_F \label{eq:mALPS_memory:06}
\end{align} Let $ \alpha:= \frac{4\delta_{3\rank}}{1-\delta_{3\rank}} + (2\delta_{3\rank} + 2\delta_{4\rank})\frac{2\delta_{3\rank}}{1-\delta_{3\rank}} $ and $ g(i) := \vectornormbig{\signal(i+1) - \bestsignal}_F $. Then, (\ref{eq:mALPS_memory:06}) defines the following homogeneous recurrence:
\begin{align}
g(i+1) - \alpha(1+\tau_i)g(i) + \alpha \tau_i g(i-1) \leq 0 \label{eq:mALPS_memory:07}
\end{align} Using the {\it method of characteristic roots} to solve the above recurrence, we assume that the homogeneous linear recursion has solution of the form $ g(i) = r^i $ for $ r \in \mathbb{R} $. Thus, replacing $ g(i) = r^i $ in (\ref{eq:mALPS_memory:07}) and factoring out $ r^{(i-2)} $, we form the following characteristic polynomial:
\begin{align}
r^2 - \alpha (1 + \tau_i)r - \alpha \tau_i \leq 0 \label{eq:mALPS_memory:08}
\end{align} Focusing on the worst case where (\ref{eq:mALPS_memory:08}) is satisfied with equality, we compute the roots $ r_{1,2} $ of the quadratic characteristic polynomial as:
\begin{align}
r_{1,2} = \frac{\alpha (1+\tau_i) \pm \sqrt{\Delta}}{2}, ~\text{where}~ \Delta := \alpha^2(1+\tau_i)^2 + 4\alpha \tau_i. \nonumber
\end{align} Then, as a general solution, we combine the above roots with unknown coefficients $ b_1, b_2 $ to obtain (\ref{eq:mALPS_memory:09}).
Using the initial condition $ g(0) := \vectornormbig{\signal(0) - \bestsignal}_F \stackrel{\signal(0) = \mathbf{0}}{=} \vectornormbig{\bestsignal}_F = 1 $, we get $ b_1 + b_2 = 1 $. Thus, we conclude to the following recurrence:
\begin{align}
\vectornormbig{\signal(i+1) - \bestsignal}_F \leq \Big(\frac{\alpha (1+\tau_i) + \sqrt{\Delta}}{2} \Big)^{i+1}. \nonumber
\end{align}


\subsection{Proof of Lemma \ref{lemma:appr_act_subspace_exp}}

Let 
$\mathcal{D}_i^{\epsilon} \leftarrow \mathcal{P}_{\rank}^{\epsilon}(\mathcal{P}_{\mathcal{X}_i^{\bot}} \nabla f(\signal(i))) $ 
and 
$\mathcal{D}_i \leftarrow \mathcal{P}_{\rank}(\mathcal{P}_{\mathcal{X}_i^{\bot}} \nabla f(\signal(i))).$
Using Definition \ref{def:appr_svd}, the following holds true:
\begin{align}
\vectornormbig{\mathcal{P}_{\mathcal{D}_i^\epsilon} \nabla f(\signal(i))& - \nabla f(\signal(i))}_F^2 \nonumber \\ &\leq (1 + \epsilon) \vectornormbig{\mathcal{P}_{\mathcal{D}_i} \nabla f(\signal(i)) - \nabla f(\signal(i))}_F^2. \label{eq:appr_act:00}
\end{align}
Furthermore, we observe:
\begin{align}
&\vectornormbig{\nabla f(\signal(i))}_F^2 = \vectornormbig{\nabla f(\signal(i))}_F^2 \Leftrightarrow \nonumber \\ 
&\vectornormbig{\mathcal{P}_{\mathcal{D}_i^\epsilon} \nabla f(\signal(i))}_F^2 + \vectornormbig{\mathcal{P}_{(\mathcal{D}_i^\epsilon)^{\bot}} \nabla f(\signal(i))}_F^2 = \nonumber \\  &\vectornormbig{\mathcal{P}_{\mathcal{X}^\ast\setminus \mathcal{X}_i} \nabla f(\signal(i))}_F^2 + \vectornormbig{\mathcal{P}_{(\mathcal{X}^\ast\setminus \mathcal{X}_i)^{\bot}} \nabla f(\signal(i))}_F^2 \label{eq:appr_act:01}
\end{align} Here, we use the notation defined in the proof of Lemma 6. Since $ \mathcal{P}_{\mathcal{D}_i} \nabla f(\signal(i)) $ is the best rank-$ \rank $ approximation to $ \nabla f(\signal(i)) $, we have:
\begin{align}
&\vectornormbig{\mathcal{P}_{\mathcal{D}_i} \nabla f(\signal(i)) - \nabla f(\signal(i))}_F^2 \leq \nonumber \\ &\vectornormbig{\mathcal{P}_{\mathcal{X}^\ast\setminus \mathcal{X}_i} \nabla f(\signal(i)) - \nabla f(\signal(i))}_F^2 \Leftrightarrow \nonumber \\
&\vectornormbig{\mathcal{P}_{\mathcal{D}_i^\bot} \nabla f(\signal(i))}_F^2 \leq \vectornormbig{\mathcal{P}_{(\mathcal{X}^\ast\setminus \mathcal{X}_i)^\bot} \nabla f(\signal(i))}_F^2 \Leftrightarrow \nonumber \\
&(1+\epsilon)\vectornormbig{\mathcal{P}_{\mathcal{D}_i^\bot} \nabla f(\signal(i))}_F^2 \leq (1+\epsilon)\vectornormbig{\mathcal{P}_{(\mathcal{X}^\ast\setminus \mathcal{X}_i)^\bot} \nabla f(\signal(i))}_F^2 \label{eq:appr_act:02}
\end{align} where $ \text{rank}(\text{span}(\text{ortho}(\mathcal{X}^\ast \setminus \mathcal{X}_i))) \leq \rank $. Using (\ref{eq:appr_act:00}) in (\ref{eq:appr_act:02}), the following series of inequalities are observed:
\begin{align}
\vectornormbig{\mathcal{P}_{(\mathcal{D}_i^\epsilon)^{\bot}} \nabla f(\signal(i))}_F^2 &\leq (1+\epsilon)\vectornormbig{\mathcal{P}_{\mathcal{D}_i^\bot} \nabla f(\signal(i))}_F^2 \nonumber \\ &\leq (1+\epsilon)\vectornormbig{\mathcal{P}_{(\mathcal{X}^\ast\setminus \mathcal{X}_i)^\bot} \nabla f(\signal(i))}_F^2 \label{eq:appr_act:03}
\end{align} Now, in (\ref{eq:appr_act:01}), we compute the series of inequalities in (\ref{eq:start})-(\ref{eq:appr_act:04a}).
\begin{figure*}[!htp]
\begin{align}
\vectornormbig{\mathcal{P}_{\mathcal{D}_i^\epsilon} \nabla f(\signal(i))}_F^2 + \vectornormbig{\mathcal{P}_{(\mathcal{D}_i^\epsilon)^{\bot}} \nabla f(\signal(i))}_F^2 &= \vectornormbig{\mathcal{P}_{\mathcal{X}^\ast\setminus \mathcal{X}_i} \nabla f(\signal(i))}_F^2 + \vectornormbig{\mathcal{P}_{(\mathcal{X}^\ast\setminus \mathcal{X}_i)^{\bot}} \nabla f(\signal(i))}_F^2 \stackrel{(\ref{eq:appr_act:02})}{\Leftrightarrow} \label{eq:start}\\
\vectornormbig{\mathcal{P}_{\mathcal{D}_i^\epsilon} \nabla f(\signal(i))}_F^2 + (1+\epsilon)\vectornormbig{\mathcal{P}_{(\mathcal{X}^\ast\setminus \mathcal{X}_i)^\bot} \nabla f(\signal(i))}_F^2 &\geq \vectornormbig{\mathcal{P}_{\mathcal{X}^\ast\setminus \mathcal{X}_i} \nabla f(\signal(i))}_F^2 + \vectornormbig{\mathcal{P}_{(\mathcal{X}^\ast\setminus \mathcal{X}_i)^{\bot}} \nabla f(\signal(i))}_F^2 \Leftrightarrow \nonumber \\
\vectornormbig{\mathcal{P}_{\mathcal{D}_i^\epsilon} \nabla f(\signal(i))}_F^2 + \epsilon\vectornormbig{\mathcal{P}_{(\mathcal{X}^\ast\setminus \mathcal{X}_i)^\bot} \nabla f(\signal(i))}_F^2 &\geq \vectornormbig{\mathcal{P}_{\mathcal{X}^\ast\setminus \mathcal{X}_i} \nabla f(\signal(i))}_F^2 \Leftrightarrow \nonumber \\
\vectornormbig{\mathcal{P}_{\mathcal{D}_i^\epsilon} \nabla f(\signal(i))}_F^2 + \vectornormbig{\mathcal{P}_{\mathcal{X}_i} \nabla f(\signal(i))}_F^2 + \epsilon\vectornormbig{\mathcal{P}_{(\mathcal{X}^\ast\setminus \mathcal{X}_i)^\bot} \nabla f(\signal(i))}_F^2 &\geq \vectornormbig{\mathcal{P}_{\mathcal{X}^\ast\setminus \mathcal{X}_i} \nabla f(\signal(i))}_F^2 + \vectornormbig{\mathcal{P}_{\mathcal{X}_i} \nabla f(\signal(i))}_F^2 \Leftrightarrow \nonumber \\
\vectornormbig{\mathcal{P}_{\mathcal{S}_i} \nabla f(\signal(i))}_F^2 + \epsilon\vectornormbig{\mathcal{P}_{(\mathcal{X}^\ast\setminus \mathcal{X}_i)^\bot} \nabla f(\signal(i))}_F^2 &\geq \vectornormbig{\mathcal{P}_{\mathcal{S}_i^\ast} \nabla f(\signal(i))}_F^2 \Leftrightarrow \nonumber \\
\vectornormbig{\mathcal{P}_{\mathcal{S}_i \setminus \mathcal{S}_i^\ast} \nabla f(\signal(i))}_F^2 + \epsilon\vectornormbig{\mathcal{P}_{(\mathcal{X}^\ast\setminus \mathcal{X}_i)^\bot} \nabla f(\signal(i))}_F^2 &\geq \vectornormbig{\mathcal{P}_{\mathcal{S}_i^\ast \setminus \mathcal{S}_i} \nabla f(\signal(i))}_F^2 \Leftrightarrow \nonumber \\
\vectornormbig{\mathcal{P}_{\mathcal{S}_i \setminus \mathcal{S}_i^\ast} \sensing^\ast(\obs - \sensing \signal(i))}_F^2 + \epsilon\vectornormbig{\mathcal{P}_{(\mathcal{X}^\ast\setminus \mathcal{X}_i)^\bot} \sensing^\ast(\obs - \sensing \signal(i))}_F^2 &\geq \vectornormbig{\mathcal{P}_{\mathcal{S}_i^\ast \setminus \mathcal{S}_i} \sensing^\ast(\obs - \sensing \signal(i))}_F^2 \Leftrightarrow \nonumber \\
\vectornormbig{\mathcal{P}_{\mathcal{S}_i \setminus \mathcal{S}_i^\ast} \sensing^\ast(\obs - \sensing \signal(i))}_F + \sqrt{\epsilon}\vectornormbig{\mathcal{P}_{(\mathcal{X}^\ast\setminus \mathcal{X}_i)^\bot} \sensing^\ast(\obs - \sensing \signal(i))}_F &\geq \vectornormbig{\mathcal{P}_{\mathcal{S}_i^\ast \setminus \mathcal{S}_i} \sensing^\ast(\obs - \sensing \signal(i))}_F \label{eq:appr_act:04a}
\end{align}
\hrulefill
\end{figure*}
Focusing on $ \vectornormbig{\mathcal{P}_{\mathcal{X}^\ast\setminus \mathcal{X}_i}^\bot \sensing^\ast(\obs - \sensing \signal(i))}_F $, we observe:
\begin{align}
&\vectornormbig{\mathcal{P}_{(\mathcal{X}^\ast\setminus \mathcal{X}_i)^\bot} \sensing^\ast(\obs - \sensing \signal(i))}_F = \nonumber \\ 
&\vectornormbig{\mathcal{P}_{(\mathcal{X}^\ast\setminus \mathcal{X}_i)^\bot} \sensing^\ast(\sensing \bestsignal + \noise  - \sensing \signal(i))}_F \leq \nonumber \\ &\vectornormbig{\mathcal{P}_{(\mathcal{X}^\ast\setminus \mathcal{X}_i)^\bot} \sensing^\ast\sensing (\bestsignal - \signal(i))}_F + \vectornormbig{\mathcal{P}_{\mathcal{X}^\ast\setminus \mathcal{X}_i}^\bot \sensing^\ast \noise}_F \leq \nonumber \\ 
&\vectornormbig{\sensing^\ast\sensing (\bestsignal - \signal(i))}_F + \vectornormbig{\sensing^\ast \noise}_F \leq 2\lambda
\label{eq:appr_act:04}
\end{align} 

Moreover, we know the following hold true from Lemma \ref{lemma:act_subspace_exp}:
\begin{align}
\vectornormbig{\mathcal{P}_{\mathcal{S}_i \setminus \mathcal{S}_i^\ast} &\sensing^\ast \sensing (\bestsignal - \signal(i)) + \mathcal{P}_{\mathcal{S}_i \setminus \mathcal{S}_i^\ast} \sensing^\ast \noise}_F \nonumber \\ &\leq 2\delta_{3\rank}\vectornormbig{\bestsignal - \signal(i)}_F + \vectornormbig{\mathcal{P}_{\mathcal{S}_i \setminus \mathcal{S}_i^\ast} \sensing^\ast \noise}_F \label{eq:appr_act:06}
\end{align} and
\begin{align}
\vectornormbig{\mathcal{P}_{\mathcal{S}_i^\ast \setminus \mathcal{S}_i} &\sensing^\ast \sensing (\bestsignal - \signal(i)) + \mathcal{P}_{\mathcal{S}_i^\ast \setminus \mathcal{S}_i} \sensing^\ast \noise}_F \nonumber \\ &\geq \vectornormbig{\mathcal{P}_{\mathcal{S}_i^\ast \setminus \mathcal{S}_i} (\bestsignal - \signal(i))}_F - 2\delta_{2\rank}\vectornormbig{\signal(i) - \bestsignal}_F \nonumber \\ &- \vectornormbig{ \mathcal{P}_{\mathcal{S}_i^\ast \setminus \mathcal{S}_i}\sensing^\ast \noise}_F \label{eq:appr_act:07}
\end{align} Combining (\ref{eq:appr_act:04})-(\ref{eq:appr_act:07}) in (\ref{eq:appr_act:04a}), we obtain:
\begin{align}
\vectornormbig{\mathcal{P}_{\mathcal{S}_i^\ast \setminus \mathcal{S}_i}&\bestsignal}_F = \vectornormbig{\mathcal{P}_{\mathcal{X}^\ast \setminus \mathcal{S}_i}\bestsignal}_F \nonumber \\ 
&\leq \big(2\delta_{2\rank} + 2\delta_{3\rank}\big)\vectornormbig{\signal(i) - \bestsignal}_F + \sqrt{2(1+\delta_{2\rank})}\vectornormbig{\noise}_2 \nonumber \\ &+ 2\lambda \sqrt{\epsilon} \nonumber \end{align}

\subsection{Proof of Theorem \ref{thm:mALPS0_appr}}

To prove Theorem \ref{thm:mALPS0_appr}, we combine the following series of lemmas for each step of Algorithm 1.

\begin{lemma}{\label{lemma:appr_greedy}}[Error norm reduction via gradient descent] Let $ \mathcal{S}_i \leftarrow \text{ortho}(\mathcal{X}_i \cup \mathcal{D}_i^{\epsilon}) $ be a set of orthonormal, rank-1 matrices that span a rank-2$ \rank $ subspace in $ \mathbb{R}^{\dimension} $. Then (\ref{eq:appr_mALPS5:01}) holds.
\begin{figure*}[!htp]
\begin{align}
\vectornormbig{\boldsymbol{V}(i) - \bestsignal}_F &\leq \Bigg [ \Big( 1+ \frac{\delta_{3\rank}}{1-\delta_{2\rank}}\Big)\Big(2\delta_{2\rank} + 2\delta_{3\rank} + \delta_{\rank})\Big) + \frac{2\delta_{2\rank}}{1-\delta_{2\rank}}\Bigg] \vectornormbig{\signal(i) - \bestsignal}_F \nonumber \\
&+ \Big[ \big(1+ \frac{\delta_{3\rank}}{1-\delta_{2\rank}}\big)\sqrt{2(1+\delta_{2\rank})} + \frac{\sqrt{1+\delta_{2\rank}}}{1-\delta_{2\rank}}\Big]\vectornormbig{\noise}_2 + \big(1+ \frac{\delta_{3\rank}}{1-\delta_{2\rank}} \big)2\lambda \sqrt{\epsilon}. \label{eq:appr_mALPS5:01}
\end{align}
\hrulefill
\end{figure*}
\end{lemma}

\begin{proof} We observe the following:
\begin{align}
\vectornormbig{\boldsymbol{V}(i) - \bestsignal}_F^2 = \vectornormbig{\mathcal{P}_{\mathcal{S}_i}(\boldsymbol{V}(i) - \bestsignal)}_F^2 + \vectornormbig{\mathcal{P}_{\mathcal{S}_i^\bot}(\boldsymbol{V}(i) - \bestsignal)}_F^2 \label{eq:appr_greedy:00}
\end{align}
The following equations hold true:
\begin{align}
\vectornormbig{\mathcal{P}_{\mathcal{S}_i^\bot}(\boldsymbol{V}(i) - \bestsignal)}_F^2 &= \vectornormbig{\mathcal{P}_{\mathcal{S}_i^\bot} \bestsignal}_F^2 = \vectornormbig{\mathcal{P}_{\mathcal{X}^\ast \setminus \mathcal{S}_i} \bestsignal}_F^2 \nonumber 
\end{align} Furthermore, we compute:
\begin{align}
&\vectornormbig{\mathcal{P}_{\mathcal{S}_i}(\boldsymbol{V}(i) - \bestsignal)}_F = \vectornormbig{\mathcal{P}_{\mathcal{S}_i}(\signal(i) - \frac{\mu_i}{2}\mathcal{P}_{\mathcal{S}_i}\nabla f(\signal(i)) - \bestsignal)}_F \nonumber \\
&= \vectornormbig{\mathcal{P}_{\mathcal{S}_i}(\signal(i) - \bestsignal) - \mu_i \mathcal{P}_{\mathcal{S}_i}\sensing^\ast \sensing (\signal(i) - \bestsignal) + \mu_i \mathcal{P}_{\mathcal{S}_i}\sensing^\ast \noise }_F \nonumber \\
&\leq \vectornormbig{(\id - \mu_i \mathcal{P}_{\mathcal{S}_i}\sensing^\ast \sensing \mathcal{P}_{\mathcal{S}_i} \mathcal{P}_{\mathcal{S}_i}(\signal(i) - \bestsignal)}_F \nonumber \\ &+ \mu_i \vectornormbig{\mathcal{P}_{\mathcal{S}_i}\sensing^\ast \sensing \mathcal{P}_{\mathcal{S}_i^\bot}(\signal(i) - \bestsignal)}_F + \mu_i \vectornormbig{\mathcal{P}_{\mathcal{S}_i}\sensing^\ast \noise}_F \nonumber \\
&\stackrel{(i)}{\leq} \frac{2\delta_{2\rank}}{1-\delta_{2\rank}}\vectornormbig{\mathcal{P}_{\mathcal{S}_i}(\signal(i) - \bestsignal)}_F + \frac{\delta_{3\rank}}{1-\delta_{2\rank}} \vectornormbig{\mathcal{P}_{\mathcal{S}_i^\bot}(\signal(i) - \bestsignal)}_F \nonumber \\ &+ \frac{\sqrt{1+\delta_{2\rank}}}{1-\delta_{2\rank}}\vectornormbig{\noise}_2 \label{eq:appr_greedy:02}
\end{align} where $ (i) $ is due to Lemmas \ref{lemma:1}, \ref{lemma:3}, \ref{lemma:4} and $ \frac{1}{1+\delta_{2\rank}} \leq \mu_i \leq \frac{1}{1-\delta_{2\rank}} $.

Using the subadditivity property of the square root in (\ref{eq:appr_greedy:00}), (\ref{eq:appr_greedy:02}), Lemma \ref{lemma:appr_act_subspace_exp} and the fact that $ \vectornormbig{\mathcal{P}_{\mathcal{S}_i}(\signal(i) - \bestsignal)}_F \leq \vectornormbig{\signal(i) - \bestsignal}_F $, we obtain:
\begin{align}
&\vectornormbig{\boldsymbol{V}(i) - \bestsignal}_F \leq \vectornormbig{\mathcal{P}_{\mathcal{S}_i}(\boldsymbol{V}(i) - \bestsignal)}_F + \vectornormbig{\mathcal{P}_{\mathcal{S}_i^\bot}(\boldsymbol{V}(i) - \bestsignal)}_F \nonumber \\
&\leq \hat{\rho} \vectornormbig{\signal(i) - \bestsignal}_F + \big(1+ \frac{\delta_{3\rank}}{1-\delta_{2\rank}} \big)\sqrt{\epsilon}\vectornormbig{\mathcal{P}_{\mathcal{X}^\ast \setminus \mathcal{X}_i}^\bot \sensing^\ast \noise}_F \nonumber \\
&+ \Big[ \big(1+ \frac{\delta_{3\rank}}{1-\delta_{2\rank}}\big)\sqrt{2(1+\delta_{2\rank})} + \frac{\sqrt{1+\delta_{2\rank}}}{1-\delta_{2\rank}}\Big]\vectornormbig{\noise}_2 \label{eq:mALPS0_appr:00}
\end{align} where $\hat{\rho} := \Big( 1+ \frac{\delta_{3\rank}}{1-\delta_{2\rank}}\Big)\Big(2\delta_{2\rank} + 2\delta_{3\rank}\Big) + \frac{2\delta_{2\rank}}{1-\delta_{2\rank}}$
\end{proof}

We exploit Lemma \ref{lemma:appr_comb_selection} to obtain the following inequalities:
\begin{align}
\vectornormbig{\widehat{\boldsymbol{W}}_i - \bestsignal}_F &= \vectornormbig{\widehat{\boldsymbol{W}}_i - \boldsymbol{V}(i) + \boldsymbol{V}(i) - \bestsignal}_F \nonumber \\ &\leq \vectornormbig{\widehat{\boldsymbol{W}}_i - \boldsymbol{V}(i)}_F + \vectornormbig{\boldsymbol{V}(i) - \bestsignal}_F \nonumber \\ &\leq (1+\epsilon)\vectornormbig{\boldsymbol{W}(i) - \boldsymbol{V}(i)}_F + \vectornormbig{\boldsymbol{V}(i) - \bestsignal}_F \nonumber \\ &\leq (2 + \epsilon)\vectornormbig{\boldsymbol{V}(i) - \bestsignal}_F \label{eq:mALPS0_appr:01}
\end{align} where the last inequality holds since $ \boldsymbol{W}(i) $ is the best rank-$ \rank $ matrix estimate of $ \boldsymbol{V}(i) $ and, thus, $ \vectornormbig{\boldsymbol{W}(i) - \boldsymbol{V}(i)}_F \leq \vectornormbig{\boldsymbol{V}(i) - \bestsignal}_F $.

Following similar motions for steps 6 and 7 in Matrix ALPS I, we obtain:
\begin{align}
\vectornormbig{\signal(i+1) - \bestsignal}_F &\leq \big(1 + \frac{2\delta_{\rank}}{1-\delta_{\rank}} + \frac{\delta_{2\rank}}{1-\delta_{\rank}} \big)\vectornormbig{\widehat{\boldsymbol{W}}_i - \bestsignal}_F \nonumber \\ &+ \frac{\sqrt{1+\delta_{\rank}}}{1-\delta_{\rank}}\vectornormbig{\noise}_2 \label{eq:mALPS0_appr:04}
\end{align} Combining (\ref{eq:mALPS0_appr:04}), (\ref{eq:mALPS0_appr:01}) and (\ref{eq:mALPS0_appr:00}), we obtain the desired inequality.

\bibliographystyle{unsrt}
\bibliography{recipes}

\begin{thebibliography}{10}

\bibitem{baraniuk2010low}
R.G. Baraniuk, V.~Cevher, and M.B. Wakin.
\newblock Low-dimensional models for dimensionality reduction and signal
  recovery: A geometric perspective.
\newblock {\em Proceedings of the IEEE}, 98(6):959--971, 2010.

\bibitem{candès2009exact}
E.J. Cand{\`e}s and B.~Recht.
\newblock Exact matrix completion via convex optimization.
\newblock {\em Foundations of Computational Mathematics}, 9(6):717--772, 2009.

\bibitem{SVP}
R.~Meka, P.~Jain, and I.~S. Dhillon.
\newblock Guaranteed rank minimization via singular value projection.
\newblock In {\em NIPS Workshop on Discrete Optimization in Machine Learning},
  2010.

\bibitem{TyagiCevherRidge}
H.~Tyagi and V.~Cevher.
\newblock Learning ridge functions with randomized sampling in high dimensions.
\newblock In {\em Acoustics, Speech and Signal Processing (ICASSP), 2012 IEEE
  International Conference on}, pages 2025--2028. IEEE, 2012.

\bibitem{TyagiCevherRidgeII}
H.~Tyagi and V.~Cevher.
\newblock Learning non-parametric basis independent models from point queries
  via low-rank methods.
\newblock {\em Technical report, EPFL}, 2012.

\bibitem{liuuniversal}
Y.K. Liu.
\newblock Universal low-rank matrix recovery from pauli measurements.
\newblock 2011.

\bibitem{hemant2012active}
H.~Tyagi and V.~Cevher.
\newblock Active learning of multi-index function models.
\newblock In {\em Advances in Neural Information Processing Systems 25}, pages
  1475--1483, 2012.

\bibitem{candes2012solving}
E.J. Candes and X.~Li.
\newblock Solving quadratic equations via phaselift when there are about as
  many equations as unknowns.
\newblock {\em arXiv preprint arXiv:1208.6247}, 2012.

\bibitem{netflix}
J.~Bennett and S.~Lanning.
\newblock The netflix prize.
\newblock In {\em In KDD Cup and Workshop in conjunction with KDD}, 2007.

\bibitem{candes2011robust}
E.J. Candes, X.~Li, Y.~Ma, and J.~Wright.
\newblock Robust principal component analysis?
\newblock {\em Journal of the ACM}, 58(3), 2011.

\bibitem{KyrillidisCevherSSP}
A.~Kyrillidis and V.~Cevher.
\newblock Matrix alps: Accelerated low rank and sparse matrix reconstruction.
\newblock {\em Technical report, EPFL}, 2012.

\bibitem{sparcs}
A.E. Waters, A.C. Sankaranarayanan, and R.G. Baraniuk.
\newblock Sparcs: Recovering low-rank and sparse matrices from compressive
  measurements.
\newblock In {\it NIPS}, 2011.

\bibitem{brecht2010}
M.~Fazel, B.~Recht, and P.~A. Parrilo.
\newblock Guaranteed minimum rank solutions to linear matrix equations via
  nuclear norm minimization.
\newblock {\em SIAM Review}, 52(3):471--501, 2010.

\bibitem{Liu2009}
Z.~Liu and L.~Vandenberghe.
\newblock Interior-point method for nuclear norm approximation with application
  to system identification.
\newblock {\em SIAM J. Matrix Anal. Appl.}, 31:1235--1256, November 2009.

\bibitem{Fazel2010}
K.~Mohan and M.~Fazel.
\newblock Reweighted nuclear norm minimization with application to system
  identification.
\newblock In {\em American Control Conference (ACC)}. IEEE, 2010.

\bibitem{SVT}
Jian-Feng Cai, Emmanuel~J. Cand\`{e}s, and Zuowei Shen.
\newblock A singular value thresholding algorithm for matrix completion.
\newblock {\em SIAM J. on Optimization}, 20:1956--1982, March 2010.

\bibitem{ParallelRecht}
B.~Recht and C.~Re.
\newblock Parallel stochastic gradient algorithms for large-scale matrix
  completion.
\newblock {\em Preprint}, 2011.

\bibitem{ALM}
Z.~Lin, M.~Chen, and Y.~Ma.
\newblock The augmented lagrange multiplier method for exact recovery of
  corrupted low-rank matrices.
\newblock {\em arXiv preprint arXiv:1009.5055}, 2010.

\bibitem{APG}
J.~Wright L. Wu M.~Chen Z.~Lin, A.~Ganesh and Y.~Ma.
\newblock Fast convex optimization algorithms for exact recovery of a corrupted
  low-rank matrix.
\newblock {\em UIUC Technical Report UILU-ENG-09-2214}.

\bibitem{natarajan1995sparse}
B.K. Natarajan.
\newblock Sparse approximate solutions to linear systems.
\newblock {\em SIAM journal on computing}, 24(2):227--234, 1995.

\bibitem{admira2010}
K.~Lee and Y.~Bresler.
\newblock Admira: Atomic decomposition for minimum rank approximation.
\newblock {\em IEEE Trans. on Information Theory}, 56(9):4402--4416, 2010.

\bibitem{Goldfarb:2011}
D.~Goldfarb and S.~Ma.
\newblock Convergence of fixed-point continuation algorithms for matrix rank
  minimization.
\newblock {\em Found. Comput. Math.}, 11:183--210, April 2011.

\bibitem{beck2011linearly}
A.~Beck and M.~Teboulle.
\newblock A linearly convergent algorithm for solving a class of
  nonconvex/affine feasibility problems.
\newblock {\em Fixed-Point Algorithms for Inverse Problems in Science and
  Engineering}, pages 33--48, 2011.

\bibitem{KyrillidisCevherRecipes}
A.~Kyrillidis and V.~Cevher.
\newblock Recipes on hard thresholding methods.
\newblock In {\em Computational Advances in Multi-Sensor Adaptive Processing},
  Dec. 2011.

\bibitem{clash}
A.~Kyrillidis and V.~Cevher.
\newblock Combinatorial selection and least absolute shrinkage via the
  \textsc{Clash} algorithm.
\newblock In {\em IEEE International Symposium on Information Theory}, July
  2012.

\bibitem{findingstructure}
N.~Halko, P.~G. Martinsson, and J.~A. Tropp.
\newblock Finding structure with randomness: Probabilistic algorithms for
  constructing approximate matrix decompositions.
\newblock {\em SIAM Rev.}, 53:217--288, May 2011.

\bibitem{Bertsekas}
D.~Bertsekas.
\newblock {\em Nonlinear programming}.
\newblock Athena Scientific, 1995.

\bibitem{horn1990matrix}
R.~A. Horn and C.~R. Johnson.
\newblock {\em Matrix analysis}.
\newblock Cambridge Univ. Press, 1990.

\bibitem{cohen06}
A.~Cohen, W.~Dahmen, and R.~DeVore.
\newblock Compressed sensing and best k-term approximation.
\newblock {\em J. Amer. Math. Soc}, 22(1):211--231, 2009.

\bibitem{Tro04:Greed-Good}
J.~A. Tropp.
\newblock Greed is good: {A}lgorithmic results for sparse approximation.
\newblock {\em IEEE Trans. on Information Theory}, 50(10):2231--2242, Oct.
  2004.

\bibitem{cevher2011alps}
V.~Cevher.
\newblock An alps view of sparse recovery.
\newblock In {\em Acoustics, Speech and Signal Processing (ICASSP), 2011 IEEE
  International Conference on}, pages 5808--5811. IEEE, 2011.

\bibitem{cosamp}
D.~Needell and J.A. Tropp.
\newblock Cosamp: Iterative signal recovery from incomplete and inaccurate
  samples.
\newblock {\em Applied and Computational Harmonic Analysis}, 26(3):301--321,
  2009.

\bibitem{SP}
W.~Dai and O.~Milenkovic.
\newblock Subspace pursuit for compressive sensing signal reconstruction.
\newblock {\em IEEE Trans. on Information Theory}, 55:2230--2249, May 2009.

\bibitem{HTP}
S.~Foucart.
\newblock Hard thresholding pursuit: an algorithm for compressed sensing.
\newblock {\em SIAM Journal on Numerical Analysis}, 49(6):2543--2563, 2011.

\bibitem{Blumensath_iterativehard}
T.~Blumensath and M.~E. Davies.
\newblock Iterative hard thresholding for compressed sensing.
\newblock {\em Appl. Comp. Harm. Anal}, 27(3):265--274, 2009.

\bibitem{garg2009gradient}
R.~Garg and R.~Khandekar.
\newblock Gradient descent with sparsification: an iterative algorithm for
  sparse recovery with restricted isometry property.
\newblock In {\em ICML}. ACM, 2009.

\bibitem{NIHT}
T.~Blumensath and M.~E. Davies.
\newblock Normalized iterative hard thresholding: Guaranteed stability and
  performance.
\newblock {\em J. Sel. Topics Signal Processing}, 4(2):298--309, 2010.

\bibitem{AIHT}
T.~Blumensath.
\newblock Accelerated iterative hard thresholding.
\newblock {\em Signal Process.}, 92:752--756, March 2012.

\bibitem{tannernormalized}
J.~Tanner and K.~Wei.
\newblock Normalized iterative hard thresholding for matrix completion.
\newblock {\em Preprint}, 2012.

\bibitem{coifman2001noiselets}
R.~Coifman, F.~Geshwind, and Y.~Meyer.
\newblock Noiselets.
\newblock {\em Applied and Computational Harmonic Analysis}, 10(1):27--44,
  2001.

\bibitem{foucart2010sparse}
S.~Foucart.
\newblock Sparse recovery algorithms: sufficient conditions in terms of
  restricted isometry constants.
\newblock In {\em Proceedings of the 13th International Conference on
  Approximation Theory}, 2010.

\bibitem{nesterov2007gradient}
Y.~Nesterov.
\newblock Gradient methods for minimizing composite objective function. core
  discussion papers 2007076, universit{\'e} catholique de louvain.
\newblock {\em Center for Operations Research and Econometrics (CORE)}, 2007.

\bibitem{nesterov}
Y.~Nesterov.
\newblock {\em Introductory lectures on convex optimization}.
\newblock Kluwer Academic Publishers, 1996.

\bibitem{drineas1}
P.~Drineas, A.~Frieze, R.~Kannan, S.~Vempala, and V.~Vinay.
\newblock Clustering large graphs via the singular value decomposition.
\newblock {\em Machine Learning}, 56(1):9--33, 2004.

\bibitem{drineas2}
P.~Drineas, R.~Kannan, and M.~W. Mahoney.
\newblock Fast monte carlo algorithms for matrices ii: Computing a low-rank
  approximation to a matrix.
\newblock {\em SIAM J. Comput.}, 36:158--183, July 2006.

\bibitem{deshpande1}
A.~Deshpande, L.~Rademacher, S.~Vempala, and G.~Wang.
\newblock Matrix approximation and projective clustering via volume sampling.
\newblock In {\em Proceedings of the seventeenth annual ACM-SIAM symposium on
  Discrete algorithm}, SODA '06, pages 1117--1126, New York, NY, USA, 2006.
  ACM.

\bibitem{deshpande2}
A.~Deshpande and S.~Vempala.
\newblock Adaptive sampling and fast low-rank matrix approximation.
\newblock {\em Electronic Colloquium on Computational Complexity (ECCC)},
  13(042), 2006.

\bibitem{OptSpace}
R.H. Keshavan, A.~Montanari, and S.~Oh.
\newblock Matrix completion from a few entries.
\newblock {\em IEEE Trans. on Information Theory}, 56(6):2980--2998, 2010.

\bibitem{GROUSE}
L.~Balzano, R.~Nowak, and B.~Recht.
\newblock Online identification and tracking of subspaces from highly
  incomplete information.
\newblock In {\em Communication, Control, and Computing (Allerton), 2010 48th
  Annual Allerton Conference on}, pages 704--711. IEEE, 2010.

\bibitem{GRASTA}
J.~He, L.~Balzano, and J.~C.~S. Lui.
\newblock Online robust subspace tracking from partial information.
\newblock {\em arXiv:1109.3827}, 2011.

\bibitem{RTRMC}
N.~Boumal and P.A. Absil.
\newblock Rtrmc: A riemannian trust-region method for low-rank matrix
  completion.
\newblock In {\em NIPS}, 2011.

\bibitem{LMatFit}
Z.~Wen, W.~Yin, and Y.~Zhang.
\newblock Solving a low-rank factorization model for matrix completion by a
  nonlinear successive over-relaxation algorithm.
\newblock {\em Rice University CAAM Technical Report TR10-07. Submitted}, 2010.

\bibitem{propack}
R.~M. Larsen.
\newblock Propack: Software for large and sparse svd calculations.
\newblock {\em \rm{http://soi.stanford.edu/\~rmunk/PROPACK}}.

\bibitem{nonuclear}
X.~Shi and P.S. Yu.
\newblock Limitations of matrix completion via trace norm minimization.
\newblock {\em ACM SIGKDD Explorations Newsletter}, 12(2):16--20, 2011.

\end{thebibliography}

\end{document}